\pdfoutput=1

\documentclass[10pt]{article} 

\usepackage{xspace}
\usepackage{url}
\usepackage{mathtools}
\usepackage{amssymb}
\usepackage{amsthm}
\usepackage{empheq}
\usepackage{latexsym}
\usepackage[shortlabels]{enumitem}
\usepackage{eurosym}
\usepackage{dsfont}
\usepackage{appendix}
\usepackage{color} 
\usepackage[unicode]{hyperref}
\usepackage{frcursive}
\usepackage[utf8]{inputenc}
\usepackage[T1]{fontenc}
\usepackage{geometry}
\usepackage{multirow}
\usepackage{todonotes}
\usepackage{lmodern}
\usepackage{anyfontsize}
\usepackage{stmaryrd}
\usepackage{natbib}
\usepackage{cleveref}
\usepackage[english]{babel}
\usepackage[english=british]{csquotes}
\usepackage{mathtools}
\usepackage{accents}
\usepackage{color}
\usepackage{comment}
\usepackage{mdframed}
\usepackage{hyperref}
\usepackage{cases}
\usepackage{bm}
\usepackage[mathscr]{euscript}

\mdfsetup{middlelinecolor=blue,middlelinewidth=2pt,linewidth=0pt,backgroundcolor=blue!15,,roundcorner=10pt}

\allowdisplaybreaks

\setcitestyle{numbers,open={[},close={]}}

\definecolor{red}{rgb}{0.7,0.15,0.15}
\definecolor{green}{rgb}{0,0.5,0}
\definecolor{blue}{rgb}{0,0,0.7}
\hypersetup{colorlinks, linkcolor={red}, citecolor={green}, urlcolor={blue}}
			
\makeatletter \@addtoreset{equation}{section}

\newtheorem{theorem}{Theorem}[section]
\newtheorem{assumption}[theorem]{Assumption}
\newtheorem{corollary}[theorem]{Corollary}
\newtheorem{example}[theorem]{Example}

\newtheorem{lemma}[theorem]{Lemma}
\newtheorem{proposition}[theorem]{Proposition}

\newtheorem{definition}[theorem]{Definition}
\newtheorem{remark}[theorem]{Remark}

\newcommand\cA{\mathcal A}
\newcommand\cB{\mathcal B}
\newcommand\cC{\mathcal C}

\newcommand\cE{\mathcal E}
\newcommand\cF{\mathcal F}
\newcommand\cG{\mathcal G}
\newcommand\cH{\mathcal H}
\newcommand\cI{\mathcal I}

\newcommand\cK{\mathcal K}
\newcommand\cL{\mathcal L}
\newcommand\cM{\mathcal M}
\newcommand\cN{\mathcal N}
\newcommand\cO{\mathcal O}
\newcommand\cP{\mathcal P}
\newcommand\cQ{\mathcal Q}

\newcommand\cS{\mathcal S}
\newcommand\cT{\mathcal T}
\newcommand\cU{\mathcal U}

\newcommand\cX{\mathcal X}

\newcommand\cZ{\mathcal Z}

\newcommand{\smalltext}[1]{\text{\fontsize{4}{4}\selectfont$#1$}}
\newcommand{\tinytext}[1]{\text{\fontsize{3}{3}\selectfont$#1$}}

\newcommand{\vertiii}[1]{{\left\vert\kern-0.25ex\left\vert\kern-0.25ex\left\vert #1 \right\vert\kern-0.25ex\right\vert\kern-0.25ex\right\vert}}
\newcommand\fH{\mathfrak H}
\newcommand\sgn{\text{sgn}}
\newcommand\ff{\mathfrak f}
\newcommand\fg{\mathfrak g}

\def \E{\mathbb{E}}
\def \F{\mathbb{F}}
\def \G{\mathbb{G}}
\def \H{\mathbb{H}}

\def \L{\mathbb{L}}

\def \N{\mathbb{N}}

\def \P{\mathbb{P}}
\def \Q{\mathbb{Q}}
\def \R{\mathbb{R}}
\def \S{\mathbb{S}}

\def\Ac{\mathcal{A}}

\def\Pc{\mathcal{P}}

\newcommand{\brs}[1]{\prescript{\ast}{}{#1}}

\newcommand{\bcdot}{\boldsymbol{\cdot}}

\newcommand{\1}{\mathbf{1}}

\def\d{\mathrm{d}}

\DeclareMathOperator*{\esssup}{ess\,sup}

\begin{document}

\title{Reflections on BSDEs}

\author{Dylan {\sc Possama\"{i}}\footnote{ETH Z\"urich, Department of Mathematics, Switzerland, dylan.possamai@math.ethz.ch. This author gratefully acknowledges the support of the SNF project MINT 205121-219818.}\and Marco {\sc Rodrigues}\footnote{ETH Z\"urich, Department of Mathematics, Switzerland, marco.rodrigues@math.ethz.ch.}}

\date{December 12, 2024}

\maketitle

\begin{abstract}

We prove well-posedness results for backward stochastic differential equations (BSDEs) and reflected BSDEs with an optional obstacle process in the case of appropriately weighted $\mathbb{L}^2$-data when the generator is integrated with respect to a possibly purely discontinuous process. This leads to a unified treatment of discrete-time and continuous-time (reflected) BSDEs. We compare our well-posedness results with the current literature and highlight that our results are sharp and cannot be improved within the framework presented here. Finally, we provide sufficient conditions for a comparison principle.

\vspace{5mm}
\vspace{5mm}
\end{abstract}


\section{Introduction}\label{sec::introduction}

The primary motivation for this work is to develop a unified approach to discrete-time and continuous-time (reflected) BSDEs with jumps. The formulation we employ here permits the integration of the generator with a potentially purely discontinuous process. Informally, for the time being, the (reflected) BSDE studied here is of the form
\begin{numcases}{}
	\displaystyle Y_t = \xi_T + \int_t^T f_s\big(Y_s,Y_{s-},Z_s, U_s(\cdot)\big) \d C_s - \int_t^T Z_s \d X_s - \int_t^T\int_E U_s(x)(\mu-\nu)(\d s,\d x) - \int_t^T \d N_s + \int_t^T \d K_s, \nonumber \\
	\displaystyle Y \geq \xi, \nonumber \\
	\displaystyle \int_0^T(Y_{s-} - \xi_{s-}) \d K^r_s + \int_0^T(Y_s-\xi_s)\d K^\ell_s = 0, \; \text{where} \; K^\ell \coloneqq \sum_{s \in [0,\cdot]}(K_{s+}-K_s) \; \text{and} \; K^r \coloneqq K- K^\ell_-. \nonumber
\end{numcases}
Here, $T$ is a stopping time, the process $\xi$ is the obstacle, $f$ is the generator with stochastic Lipschitz coefficient $\alpha$, $X$ is the driving martingale, $\mu$ the driving random measure, and $C$ is a non-decreasing, predictable process that dominates the predictable quadratic variation $\langle X \rangle$ of $X$ and the predictable compensator $\nu$ of $\mu$. Note that the above system reduces to a BSDE in case $\xi = -\infty$ on $[0,T)$. We are then seeking to find a class of processes within which there exists a unique collection $(Y,Z,U,N,K)$ of adapted processes that solve the above system. In this work, we prove well-posedness of these BSDEs and reflected BSDEs under appropriately integrable $\L^2$-data $(\xi,f)$ in case $\alpha^2\Delta C \leq \Phi \in (0,1)$.

\medskip
Although not yet strictly referred to as BSDEs, the first appearance of these types of objects in case the generator $f$ is linear was in the works of \citeauthor*{bismut1973analyse} \cite{bismut1973analyse,bismut1973conjugate}, and even earlier by \citeauthor*{davis1973dynamic} \cite{davis1973dynamic}. In the former, BSDEs were introduced as adjoint equations that originated from an application of the Pontryagin stochastic maximum principle, and in the latter, they were used to study stochastic control problems with drift control. What has been recognised already in \cite{davis1973dynamic} is that value process of certain stochastic control problems can be recasted as the $Y$-component of a BSDE, which through comparison principles allows one to deduce characterisations of optimal control responses to stochastic systems. The first systematic study of BSDEs in the non-linear setting with Lipschitz generator was carried out by \citeauthor{pardoux1990adapted} \cite{pardoux1990adapted}, and the same authors later connected BSDEs to quasi-linear PDEs through Feynman--Kac--type formulas in \cite{pardoux1992backward}. The seminal survey article by \citeauthor*{el1997backward} \cite{el1997backward} collected properties of BSDEs and showed how they may be applied to solve problems in mathematical finance. Reflected BSDEs were introduced later by \citeauthor*{el1997reflected} \cite{el1997reflected} directly in a Lipschitz setting, where the immediate connection to optimal stopping problems and obstacle problems for parabolic PDEs was drawn. Concomitantly to these early contributions, reflected BSDEs have also been applied to hedging problems of American options by \citeauthor*{el1997non} \cite{el1997non}, and \citeauthor*{el1997reflected2} \cite{el1997reflected2}, and to an optimal control problem with consumption by \citeauthor*{el1998optimization} \cite{el1998optimization}, see also \citeauthor*{bally2002reflected} \cite{bally2002reflected}.

\medskip
With time, BSDEs and reflected BSDEs were applied in many areas. In finance, as mentioned above, for pricing of financial derivatives, see \citeauthor*{el1997backward} \cite{el1997backward} and \cite{el1997non,el1997reflected2}, or for utility maximisation problems in \citeauthor*{hu2005utility} \cite{hu2005utility}. One can also use BSDEs to construct recursive utilities as in \citeauthor*{duffie1992asset} \cite{duffie1992asset,duffie1992stochastic}. There have been works on applications to zero-sum games by \citeauthor*{hamadene1995zero} \cite{hamadene1995zero} and to Dynkin games by \citeauthor*{cvitanic1996backward} \cite{cvitanic1996backward}. Recently, based on some new backward propagation of chaos techniques appearing in \citeauthor*{lauriere2019backward} \cite{lauriere2019backward}, BSDEs have also been applied to mean-field games by \citeauthor*{possamai2021non} \cite{possamai2021non} to deduce convergence rates of the $N$-player game to its mean-field counterpart.

\medskip
Ever since the seminal works \cite{pardoux1990adapted} and \cite{el1997reflected}, the theory of BSDEs and reflected BSDEs has expanded rapidly, and there have been various forms of generalisation of well-posedness results for these systems. \citeauthor*{kobylanski1997resultats} \cite{kobylanski1997resultats} and \citeauthor*{tevzadze2008solvability} \cite{tevzadze2008solvability} studied well-posedness of BSDEs with generators that are quadratic in the $z$-variable, see also the works of \citeauthor*{jackson2022characterization} \cite{jackson2022characterization}, \citeauthor*{zheng2023one} \cite{zheng2023one} and the references therein, and \citeauthor*{kazi2015quadratic1} \cite{kazi2015quadratic1,kazi2016quadratic} as well as \citeauthor*{karoui2016quadratic} \cite{karoui2016quadratic}, \citeauthor*{jeanblanc2012robust} \cite{jeanblanc2012robust}, or \citeauthor*{matoussi2019exponential} \cite{matoussi2019exponential} for BSDEs with jumps in the quadratic case. Reflected BSDEs with quadratic growth have also been considered by \citeauthor*{kobylanski2002reflected} \cite{kobylanski2002reflected}, \citeauthor*{lepeltier2007reflected} \cite{lepeltier2007reflected}, and \citeauthor*{bayraktar2012quadratic} \cite{bayraktar2012quadratic}. Another possible direction of generalisation is to consider BSDEs and reflected BSDEs on random time horizon. Here the first well-posedness result for BSDEs was established by \citeauthor*{peng1991probabilistic} \cite{peng1991probabilistic}, and then extended by \citeauthor*{darling1997backwards} \cite{darling1997backwards}. Other works on BSDEs with random terminal time include \citeauthor*{briand1998stability} \cite{briand1998stability} and \citeauthor*{royer2004bsdes} \cite{royer2004bsdes}. More recently, \citeauthor*{lin2020second} \cite{lin2020second} complemented this theory and proved well-posedness of random horizon BSDEs, 2BSDEs and reflected BSDEs with $\L^p$-data to which we also refer the interested reader for more references on the subject of random horizon systems. Other than the aforementioned works, there have also been works by \citeauthor*{rozkosz2012solutions} \cite{rozkosz2012solutions}, \citeauthor*{klimsiak2012reflected2} \cite{klimsiak2012reflected2} and \citeauthor*{klimsiak2019reflected} \cite{klimsiak2019reflected} on reflected BSDEs with $\L^p$-data and by \citeauthor*{alsheyab2023optimal} \cite{alsheyab2023optimal} on random horizon reflected BSDEs. Furthermore, \citeauthor*{el2023infinite} \cite{el2023infinite}, \citeauthor*{li2023multi} \cite{li2023multi} and \citeauthor*{qian2023multi} \cite{qian2023multi} study multidimensional reflected BSDEs.

\medskip
What most of the above references have in common is that they consider one driving Brownian motion of the system. Works that considered also a driving Poisson random measure include \citeauthor*{barles1997backward} \cite{barles1997backward}, \citeauthor*{royer2006backward} \cite{royer2006backward}, \citeauthor*{quenez2013bsdes} \cite{quenez2013bsdes}, \citeauthor*{becherer2019monotone} \cite{becherer2019monotone} in the BSDE case, and \citeauthor*{hamadene2003reflected} \cite{hamadene2003reflected}, \citeauthor*{crepey2008reflected} \cite{crepey2008reflected}, \citeauthor*{quenez2014reflected} \cite{quenez2014reflected}, \citeauthor*{perninge2023optimal} \cite{perninge2023optimal,perninge2024optimal} in the reflected BSDE case, to name but a few. On the other hand, there are very few results that go beyond the case of Brownian motion to more general martingales. In the BSDE case there is work by \citeauthor*{buckdahn1993backward} \cite{buckdahn1993backward}, \citeauthor*{elkaroui1997general} \cite{elkaroui1997general}, \citeauthor*{carbone2008backward} \cite{carbone2008backward}, \citeauthor*{confortola2014backward} \cite{confortola2014backward}, \citeauthor*{bandini2015existence} \cite{bandini2015existence}, \citeauthor*{cohen2012existence} \cite{cohen2012existence} and the more recent work by \citeauthor*{papapantoleon2018existence} \cite{papapantoleon2018existence}, to which we also refer the reader interested in further history for BSDEs. The results obtained in \cite{papapantoleon2018existence} have been applied to well-posedness results for backward stochastic Volterra integral equations with jumps by \citeauthor*{popier2021backward} \cite{popier2021backward}, to an optimal reinsurance problem by \citeauthor*{brachetta2022optimal} \cite{brachetta2022optimal}, and to a stochastic control problem involving Lévy processes by \citeauthor*{dinunno2022stochastic} \cite{dinunno2022stochastic}. Reflected BSDEs with jumps have also been considered by \citeauthor*{nie2022reflected} \cite{nie2022reflected}, see also the references therein. However, the well-posedness result in \cite{nie2022reflected} relies on assumptions that are too restrictive for the applications we have in mind since, for example, the integrator $C$ in \cite{nie2022reflected} is assumed to be continuous. This immediately excludes piecewise constant integrators, which we want to cover to some extent at least. In other recent contributions, \citeauthor*{aksamit2023generalized} \cite{aksamit2023generalized} and \citeauthor*{li2022well} \cite{li2022well} study `generalised' BSDEs and reflected BSDEs with a view towards applications to pricing of vulnerable options. Additionally, \citeauthor*{gu2023mean} \cite{gu2023mean,gu2023reflected} and \citeauthor*{lin2024mean} \cite{lin2024mean} study reflected BSDEs driven by a marked point process.

\medskip
In the reflected BSDE case, the regularity of the obstacle has also been lifted over the years. There are works considering obstacles that are  c\`adl\`ag, see \citeauthor*{hamadene2002reflected} \cite{hamadene2002reflected}, \citeauthor*{lepeltier2005penalization} \cite{lepeltier2005penalization}, \citeauthor*{hamadene2003reflected} \cite{hamadene2003reflected}, or merely measurable, see \citeauthor*{peng2010reflected} \cite{peng2005smallest,peng2010reflected}, \citeauthor*{klimsiak2013bsdes} \cite{klimsiak2013bsdes,klimsiak2015reflected} and \citeauthor*{klimsiak2019reflected} \cite{klimsiak2019reflected}. See also the works of \citeauthor*{kobylanski2012optimal} \cite{kobylanski2012optimal} on a general approach to optimal stopping problems. Recently, in a series of two inspiring papers, \citeauthor*{grigorova2017reflected} \cite{grigorova2017reflected} and \citeauthor*{grigorova2020optimal} \cite{grigorova2020optimal} considered reflected BSDEs driven by Brownian motion and a Poisson random measure whose obstacle is merely an optional process. Other than proving well-posedness of the corresponding reflected BSDE, the aforementioned reference also draws the connection to the corresponding optimal stopping problem with respect to the induced $f$-expectation. Results related to the aforementioned works were obtained by \citeauthor{baadi2017reflected} \cite{baadi2017reflected,baadi2018reflected}, \citeauthor*{akdim2020strong} \cite{akdim2020strong} and \citeauthor*{bouhadou2022rbsdes} \cite{bouhadou2022rbsdes}. The case of predictable obstacles has also been studied by \citeauthor*{bouhadou2018non} \cite{bouhadou2018non,bouhadou2018optimal}.

\medskip
Although not entirely related to what we study here, we would like to mention that there have also been works on doubly reflected BSDEs, where the BSDEs are constrained to stay within an upper and lower obstacle process. Here the first well-posedness study was carried out by \citeauthor*{cvitanic1996backward} \cite{cvitanic1996backward}, where they also connected the $Y$-component of the corresponding doubly reflected BSDE to the value of a Dynkin game. Other results were then obtained by \citeauthor*{hamadene1997double} \cite{hamadene1997double}, \citeauthor*{lepeltier2004backward} \cite{lepeltier2004backward}, \citeauthor*{hamadene2005bsdes} \cite{hamadene2005bsdes,hamadene2006bsdes}, \citeauthor*{hamadene2006mixed} \cite{hamadene2006mixed}, \citeauthor*{crepey2008reflected} \cite{crepey2008reflected}, \citeauthor*{pham2013some} \cite{pham2013some}, \citeauthor*{essaky2013generalized} \cite{essaky2013generalized}, \citeauthor*{dumitrescu2016generalized} \cite{dumitrescu2016generalized}, \citeauthor*{grigorova2018doubly} \cite{grigorova2018doubly}, \citeauthor*{nie2022reflected} \cite{nie2022reflected}, \citeauthor*{klimsiak2013bsdes} \cite{klimsiak2013bsdes,klimsiak2015reflected,klimsiak2021non}, \citeauthor*{klimsiak2020reflected} \cite{klimsiak2020reflected}, \citeauthor*{klimsiak2021reflected} \cite{klimsiak2021reflected,klimsiak2022nonlinear}, \citeauthor*{arharas2021reflected} \cite{arharas2021reflected}, \citeauthor*{baadi2022reflected} \cite{baadi2022reflected}, \citeauthor*{li2024propagation} \cite{li2024propagation} and \citeauthor*{li2023backward} \cite{li2023backward}.

\medskip
Our main contributions are two well-posedness results, one for BSDEs and another for reflected BSDEs. These well-posedness results, as presented, turn out to be sharp within the framework we lay out and cannot be improved with our methods, due to counterexamples to existence and uniqueness in \citeauthor*{confortola2014backward} \cite{confortola2014backward}. This also refines the BSDE results in \cite{papapantoleon2018existence}. The method of proof we use is based on a fixed-point argument and \emph{a priori} estimates that we establish by techniques that are inspired by \cite{elkaroui1997general}, although we do not use exponential weights like in \cite{papapantoleon2018existence}, but stochastic exponential weights like in \cite{bandini2015existence,cohen2012existence}. With the techniques in \cite{elkaroui1997general}, we can exploit the $\L^2$-structure of our problem, allowing us to circumvent a reliance on It\^o's formula. Our BSDE results thus extend straightforwardly to BSDEs with a multi-dimensional generator and terminal condition. In the reflected BSDE case, we will also use a fixed-point argument and \emph{a priori} estimates. Classically, these estimates are derived by an application of It\^o’s formula. However, in the generality we are aiming for, this would necessitate imposing additional assumptions on the integrator $C$. We thus approach this problem differently, and apply it in combination with the methods in \cite{elkaroui1997general} and \cite{grigorova2020optimal} to deduce the desired \emph{a priori} estimates. To the best of our knowledge, the well-posedness result we will present in the reflected BSDE case is the first of its kind. Let us mention here that after completing the first version of this manuscript, we became aware that \citeauthor*{papapantoleon2024existence} \cite{papapantoleon2024existence} independently obtained similar results for BSDEs.

\medskip
The link between discrete-time BSDEs and control theory has already been mentioned in \citeauthor*{cohen2012existence} \cite{cohen2010general2, cohen2011backward, cohen2012existence}. In continuous-time, BSDEs and reflected BSDEs are intimately connected to control problems with drift control only, as the dynamic programming equation in this context is semi-linear (thus quasi-linear). However, for stochastic control problems with drift and volatility control, the dynamic programming equation is fully non-linear and can thus not be analysed by classical BSDEs. This fact was the starting point for the new notion of second-order BSDEs (2BSDEs) introduced by \citeauthor*{cheridito2007second} \cite{cheridito2007second} and \citeauthor*{soner2011quasi} \cite{soner2011quasi,soner2012wellposedness}. Here, the $Y$-component of these 2BSDE corresponds to the classical value process of a control problem with drift and volatility control. This fact has been recently applied to principal--agent problems in \cite{cvitanic2018dynamic}. For an excellent and comprehensive introduction to BSDEs, we can refer the interested reader to the books by \citeauthor*{touzi2013optimal} \cite{touzi2013optimal} and \citeauthor*{zhang2017backward} \cite{zhang2017backward}. The latter reference also covers 2BSDEs. In the previously mentioned seminal works \cite{soner2011quasi,soner2012wellposedness} on 2BSDEs the terminal random variable $\xi_T$ had to satisfy strong regularity conditions for existence and uniqueness to hold. This assumption has been lifted by \citeauthor*{possamai2018stochastic} \cite{possamai2018stochastic}. In future work, we seek to combine the results of this work and the techniques in \cite{possamai2018stochastic} to show well-posedness of 2BSDEs with jumps that go beyond the case of Poisson random measures in \citeauthor*{kazi2015second} \cite{kazi2015second,kazi2015second2}.

\medskip
The remainder of this paper is structured as follows: in \Cref{sec::preliminaries}, we recall preliminaries on (vector) stochastic integration and orthogonal decompositions of martingales. We also fix the data and formulate the BSDE and reflected BSDE. In \Cref{sec::main_results}, we formulate our main results, separately, for the reflected BSDE first, and then for the BSDE. We also compare our assumptions with the current literature. In \Cref{sec::optimal_stopping}, we revisit the Snell envelope and optimal stopping theory, with which we solve the reflected BSDE in case of a generator not depending on the solution. In \Cref{sec::a_priori}, we establish the \emph{a priori} estimates, which we will use in the contraction argument, separately in the reflected BSDE case first, and then in the BSDE case. We then establish the two well-posedness results in \Cref{sec::existence_rbsde}. Finally, we prove a comparison result for our BSDE in \Cref{sec::comparison}. The appendices contain proofs of technical results and some auxiliary results we make use of throughout this work. 

\medskip
{\small\textbf{Notations:} {throughout this work, we fix a positive integer $m$. Let $\N$ and $\R$ denote the non-negative integers and real numbers, respectively. For $(a,b) \in [-\infty,\infty]^2$, we write $a \lor b \coloneqq \max\{a,b\}$, $a \land b \coloneqq \min\{a,b\}$ and $a^+ \coloneqq a \lor 0 = \max\{a,0\}$. We write $|a|$ for the modulus of $a \in [-\infty,\infty]$, and for $b \in \R^m$, we write $\|b\|$ for the Euclidean norm of $b \in \R^m$. For a finite-dimensional matrix $M$, we denote by $M^\top$ its transpose. For a subset $V$ of a Hilbert space $H$, we write $V^\perp$ for its orthogonal in $H$; for two subspaces $W$ and $W^\prime$ of $H$ with $W \cap W^\prime = \{0_H\}$, we write $W \oplus W^\prime$ for the internal direct sum of $W$ and $W^\prime$ in $H$.  For a set $\Omega$ and $A \subseteq \Omega$, we denote by $\mathbf{1}_A$ its indicator function defined on $\Omega$. We abuse notation and sometimes also write $\mathbf{1}_{\{x \in A\}}$ for $\mathbf{1}_{A}(x)$. For a nonempty set $\cZ$, we denote the Dirac-measure at $z \in \cZ$ by $\boldsymbol{\delta}_z$. For two measurable spaces $(\Omega,\cF)$ and $(\Omega^\prime,\cF^\prime)$, we denote by $\cF \otimes \cF^\prime$ the product-$\sigma$-algebra on the product space $\Omega \times \Omega^\prime$. For $t \in (0,\infty]$, a limit of the form $s \uparrow\uparrow t$ means that $s \rightarrow t$ along $s < t$. Analogously, for $t \in [0,\infty)$, a limit of the form $s \downarrow\downarrow t$ means that $s \rightarrow t$ along $s > t$. For a stochastic process $Y$ indexed by $[0,\infty)$ or $[0,\infty]$, we let $Y^T \coloneqq Y_{\cdot \land T}$.  Let $\mathsf{y} : [0,\infty) \rightarrow \R$ be a l\`adl\`ag function, that is, $\mathsf{y}$ admits limits from the right on $[0,\infty)$ and from the left on $(0,\infty)$. We define $\mathsf{y}_{t-} \coloneqq \lim_{s \uparrow\uparrow t} \mathsf{y}_s$, $t \in (0,\infty)$, and analogously $\mathsf{y}_{t+} \coloneqq \lim_{s \downarrow\downarrow t} \mathsf{y}_s$, $t \in [0,\infty)$. Then $\Delta \mathsf{y} : [0,\infty) \rightarrow \R$ is defined by $\Delta \mathsf{y}_t \coloneqq \mathsf{y}_t - \mathsf{y}_{t-}$ if $t \in (0,\infty)$ and $\Delta\mathsf{y}_0 = 0$. Similarly, we define $\Delta^+ \mathsf{y} : [0,\infty) \rightarrow \R$ by $\Delta^+ \mathsf{y}_t \coloneqq \mathsf{y}_{t+} - \mathsf{y}_t$ for $t \in [0,\infty)$. If $\mathsf{y}$ is additionally right-continuous (thus c\`adl\`ag) and non-decreasing, the functions $\mathsf{y}^c$ and $\mathsf{y}^d$ denote the continuous part and purely discontinuous part of $\mathsf{y}$, respectively. They are defined through the formulas $\mathsf{y}^d_t \coloneqq \sum_{s \in (0,t]} \Delta\mathsf{y}_s$ and $\mathsf{y}^c \coloneqq \mathsf{y} - \mathsf{y}^d$. Note that $\mathsf{y}^c_0 = \mathsf{y}_0$. In case $\mathsf{y}$ is defined on $[0,\infty]$, we additionally define $\mathsf{y}_{\infty-} \coloneqq \lim_{s \uparrow\uparrow \infty}\mathsf{y}_s$, $\Delta \mathsf{y}_\infty = \mathsf{y}_\infty - \mathsf{y}_{\infty-}$, $\mathsf{y}^d_{\infty} = \sum_{s \in (0,\infty]} \Delta \mathsf{y}_s$ and $\mathsf{y}^c_\infty \coloneqq \mathsf{y}_\infty - \mathsf{y}^d_\infty$.}}

\section{Preliminaries and formulation of the reflected BSDE}\label{sec::preliminaries}

This section will lay the foundations for the analysis that follows. We recall the construction and some properties of the vector stochastic integral and compensated stochastic integral with respect to an integer-valued random measure. We then present the assumptions we impose on the data and the formulation of our reflected BSDE.

\subsection{Stochastic basis}\label{sec::stochastic_basis}

We fix once and for all a probability space $(\Omega,\cG,\P)$ and a right-continuous filtration $\G = (\cG_t)_{t \in [0,\infty)}$. We denote by $\cG_\infty \coloneqq \cG_{\infty-}$ the $\sigma$-algebra on $\Omega$ generated by the sets in $\cup_{t \in [0,\infty)} \cG_t$. We denote the $\P$-augmentation of $\G$ by $\G^\P = (\cG^\P_t)_{t \in [0,\infty)}$, that is, each $\cG^\P_t$, $t \in [0,\infty)$, is generated by $\cG_t \lor \cN^\P$, where $\cN^\P$ is the collection of $(\P,\cG)$--null sets. The universal completion of a $\sigma$-algebra $\cA$ is the $\sigma$-algebra $\cA^U \coloneqq \cap_{\P^\prime\in\Pc(\Omega,\Ac)} \cA^{\P^\prime}$, where the intersection is over the set $\Pc(\Omega,\Ac)$ of all probability measures $\P^\prime$ on $(\Omega,\cA)$. We will also assume for simplicity that $\cG^U_\infty \subseteq \cG$. Unless stated otherwise, probabilistic notions requiring a filtration or a probability measure will implicitly refer to $\G$ or $\P$, respectively. 

\begin{remark}\label{rem::measurability_of_sup}
	The reason we suppose that $\cG^U_\infty$ is included in $\cG$ is that this ensures the $\cG$-measurability of $\sup_{s \in [0,\infty]} \xi_s$ for a product-measurable process $\xi : \Omega \times [0,\infty] \longrightarrow [-\infty,\infty]$, see {\rm\cite[Proposition 2.21]{karoui2013capacities}}.
\end{remark}

For two stopping times $S$ and $T$, we denote by $\cT_{S,T}$ the collection of all stopping times $\tau$ satisfying $\P[S \leq \tau \leq T] = 1$. Note that $\cT_{S,T}$ is empty if $\P[S > T] > 0$. We denote by $\cG_T$ the $\sigma$-algebra of all $A \in \cG_\infty$ for which $A \cap \{T \leq t\} \in \cG_t$ for all $t \in [0,\infty)$, and by $\cG_{T-}$ the $\sigma$-algebra generated by $\cG_0$ and and all sets of the form $A \cap \{t < T\}$ for $A \in \cG_t$ and $t \in [0,\infty)$.  If $\cC$ is a collection of processes indexed by $[0,\infty)$, we define $\cC_\mathrm{loc}$ as the collection of processes $X = (X_t)_{t \in [0,\infty)}$ for which there is a sequence of stopping times $(\tau_n)_{n \in \N}$ that is $\P$--a.s. increasing to infinity, with $X_{\cdot \land \tau_\smalltext{n}} \in \cC$ for each $n \in \N$.

\medskip
A real-valued martingale indexed by $[0,\infty)$ on our filtered probability space always has a $\P$-modification for which \emph{all} its paths are real-valued and right-continuous, and \emph{$\P$--almost all} its paths have left-limits on $(0,\infty)$ (see \citeauthor*{weizsaecker1990stochastic} \cite[Theorem 3.2.6]{weizsaecker1990stochastic}).  We will always choose such a modification of a martingale. We denote by $\cM$ the space of uniformly integrable martingales and by $\cH^2$ the space of square-integrable martingales indexed by $[0,\infty)$. We denote by $\cH^2_0$ the space of all $M \in \cH^2$ for which $M_0 = 0$, $\P$--almost surely. Note that $\cH^2_0$ is a closed subspace of $\cH^2$. For $(M,N) \in \cH^2_\mathrm{loc} \times \cH^2_\mathrm{loc}$, we denote by $\langle M, N \rangle$ the predictable quadratic covariation between $M$ and $N$ in the sense of \cite[Theorem I.4.2]{jacod2003limit} and let $\langle M \rangle \coloneqq \langle M,M \rangle$. We endow $\cH^2$ with the scalar product $( M , N )_{\cH^{\smalltext{2}}} \coloneqq \E[M_0N_0] + \E[\langle M, N \rangle_\infty] = \E[M_\infty N_\infty],$ which turns it into a Hilbert space by identifying processes whose paths agree $\P$--almost surely. We denote the norm associated to $(\cdot,\cdot)_{\cH^\smalltext{2}}$ by $\|\cdot\|_{\cH^\smalltext{2}}$. A stable subspace of $\cH^2$ is a closed linear subspace $\cX$ of $\cH^2$ such that $\mathbf{1}_A M_{\cdot \land T} \in \cX$ for each $T \in \cT_{0,\infty}$, each $A \in \cG_0$ and each $M \in \cX$. A very thorough study of stable subspaces can be found in \citeauthor*{cohen2015stochastic} \cite{cohen2015stochastic}, \citeauthor*{jacod1979calcul} \cite{jacod1979calcul}, and \citeauthor*{protter2005stochastic} \cite{protter2005stochastic}. Let us mention though that by \cite[Proposition 4.26]{jacod1979calcul}, if $M \in \cH^2$ is orthogonal to a stable subspace $\cX$ for the scalar product $(\cdot,\cdot)_{\cH^\smalltext{2}}$, then $\langle M,N\rangle = 0$ for all $N \in \cX$. 

\medskip
An element $M \in \cM_\mathrm{loc}$ has by \cite[Theorem I.4.18]{jacod2003limit} a, up to $\P$-evanescence, unique decomposition $M = M_0 + M^c +M^d$, where $M^c \in \cH^2_\mathrm{loc}$ has $\P$--a.s. continuous paths and satisfies $M^c_0 = 0$, while $M^d \in \cM_\mathrm{loc}$ is purely discontinuous in the sense that $M^d_0 = 0$ and $M^d N \in \cM_\mathrm{loc}$ for every $N \in \cH^2_\mathrm{loc}$ with $\P$--a.s. continuous paths. The processes $M^c$ and $M^d$ are referred to as the continuous and purely discontinuous local martingale parts of $M$, respectively. It is immediate that if $M \in \cH^2_\mathrm{loc}$, we also have that $M^d \in \cH^2_\mathrm{loc}$, and that we can write $\langle M \rangle = \langle M^c \rangle + \langle M^d \rangle$. Thus $(M^c,M^d) \in \cH^2 \times\cH^2$ if $M \in \cH^2$ by \cite[Th\'eor\`eme 2.34]{jacod1979calcul}. Lastly, we denote by $[X,Y]$ the optional quadratic covariation between semi-martingales $X$ and $Y$ in the sense of \cite[Definition I.4.45]{jacod2003limit}. In particular, if $M \in \cH^2$, then $[M]_\infty = [M,M]_\infty$ is integrable, $[M] - \langle M \rangle$ is a uniformly integrable martingale, and $[M^c] = \langle M^c \rangle$ since $[M]_t = \langle M^c \rangle_t + \sum_{s \in (0,t]} (\Delta M_s)^2, \; t \in [0,\infty),$ $\P$--a.s. Note that this implies $[M] = [M^c] + [M^d] = \langle M^c \rangle + [M^d]$ since $[M^d] = \sum_{s \in (0,\cdot]}(\Delta M_s)^2$.

\medskip
We now introduce the optional and predictable $\sigma$-algebras induced by our filtration $\G$. The optional $\sigma$-algebra $\overline\cO(\G)$ on $\Omega \times [0,\infty]$ is generated by all $\G$-adapted processes $\xi : \Omega \times [0,\infty] \longrightarrow \R$ that are right-continuous on $[0,\infty)$ and admit left-hand limits on $(0,\infty]$. The optional $\sigma$-algebra on $\Omega \times [0,\infty)$ is given by the trace--$\sigma$-algebra $\cO(\G) \coloneqq \overline\cO(\G) \cap (\Omega \times [0,\infty))$. We have that $\cO(\G)$ is generated by $\G$-adapted, real-valued, c\`adl\`ag processes defined on $\Omega \times [0,\infty)$. The predictable $\sigma$-algebra $\overline\cP(\G)$ on $\Omega \times [0,\infty]$ is generated by all $\G$-adapted processes $\xi : \Omega \times [0,\infty] \longrightarrow \R$ that are continuous on $[0,\infty]$. The predictable $\sigma$-algebra on $\Omega \times [0,\infty)$ is the trace-$\sigma$-algebra $\cP(\G) \coloneqq \overline\cP(\G) \cap (\Omega \times [0,\infty))$, which is also generated by all $\G$-adapted, real-valued, continuous processes defined on $\Omega \times [0,\infty)$. If no confusion may arise, we simply write $\cP \coloneqq \cP(\G)$, $\overline\cP \coloneqq \overline\cP(\G)$, $\cO \coloneqq \cO(\G)$ and $\overline\cO \coloneqq \overline\cO(\G)$. We agree to use the following convention: if not stated otherwise, a process indexed by $[0,\infty]$ is optional (resp. predictable), if it is measurable with respect to $\overline\cO$ (resp. $\overline\cP$), and a process indexed by $[0,\infty)$ is optional (resp. predictable), if it is measurable with respect to $\cO$ (resp. $\cP$). Finally, a stopping time $T$ is predictable, if $\llbracket 0,T\llbracket \coloneqq \{(\omega,t) \in \Omega \times [0,\infty) : t < T(\omega)\}$ is in $\cP$, and we denote the collection of predictable stopping times by $\cT^p_{0,\infty}$.

\medskip
The following result appears in \cite[Proposition 1.1]{jacod1979calcul} and \cite[Lemma I.1.19 and Lemma I.2.17]{jacod2003limit}. It will allow us to use results proved under the usual conditions in our setting where the filtration is not assumed to be $\P$-complete. Of course, a similar result holds for the optional or predictable $\sigma$-algebra on $\Omega \times [0,\infty]$.

\begin{lemma}\label{lem::compl_filt}
	$(i)$ Suppose that $T$ is a $\G^\P$-stopping time $($resp. $\G^\P$-predictable stopping time$)$. There exists a $\G$-stopping time $($resp. $\G$-predictable stopping time$)$ $T^\prime$ such that $T = T^\prime$, $\P$--almost surely.
	
	\medskip
	$(ii)$ Suppose that $X$ is $\cO(\G^\P)$-measurable $($resp. $\cP(\G^\P)$-measurable$)$, then there exists an $\cO(\G)$-measurable $($resp. $\cP(\G)$-measurable$)$ process $X^\prime$ such that $X = X^\prime$ up to $\P$-indistinguishability.
\end{lemma}

\subsection{A lot of integrals}\label{sec::integrals}

This part is purely for completeness as we recall the construction and some results of (stochastic) integration theory. The integrals in this work are always constructed on $[0,\infty)$, and the corresponding value at infinity of the (stochastic) integrals are determined by taking the limit $[0,\infty) \ni t \uparrow\uparrow \infty$. This allows us to consider them as optional processes indexed by $[0,\infty]$ which are additionally left-continuous at infinity.

\subsubsection{Lebesgue--Stieltjes integral}\label{subsec::integrals}
We have collected some results on Lebesgue--Stieltjes integrals in \Cref{lem::leb_stj}.\footnote{While we do not rely on \Cref{lem::leb_stj} in this work, we have included it in order to correct \cite[Lemma B.1]{papapantoleon2018existence}.} Suppose that $C = (C_t)_{t \in [0,\infty)}$ is an optional process for which $\P$--almost every path is right-continuous, non-decreasing and $[0,\infty)$--valued. Let $f = (f_u)_{s \in [0,\infty)}$ be an optional process with values in $[0,\infty]$, or with values in $[-\infty,\infty]$ and satisfying $\int_{[0,t]} |f_u| \d C_u < \infty$, $\P$--a.s., $t \in [0,\infty)$. Here, the measure $\d C_u$ charges $\{0\}$ with mass $C_0$. We denote by $\int_0^\cdot f_u \d C_u = (\int_0^t f_u \d C_u)_{t \in [0,\infty)}$ the, up to $\P$--indistinguishability, unique, optional process with $\P$-a.s. right-continuous paths satisfying  $\int_0^t f_u \d C_u = \int_{[0,t]} f_u \d C_u$, $t \in [0,\infty),$ $\P$--a.s. Note that $\int_0^{t-} f_u \d C_u = \int_{[0,t)} f_u \d C_u, \; \text{$\P$--a.s.}$ We denote by $\int_0^\infty f_u \d C_u$ the $\P$--a.s. unique $\cG_\infty$-measurable random variable satisfying $\int_0^\infty f_u \d C_u = \lim_{t \uparrow\uparrow \infty} \int_0^t f_u \d C_u = \int_{[0,\infty)} f_u \d C_u$, $\P$--a.s., in case $f$ is non-negative or $\int_{[0,\infty)} |f_u| \d C_u < \infty$, $\P$--a.s.
For two stopping times $S$ and $T$, we then use the convention $\int_S^T f_u \d C_u \coloneqq \int_0^T f_u \d C_u - \int_0^{S \land T} f_u \d C_u = \int_{[0,\infty)} \mathbf{1}_{(S,T]}(u) f_u \d C_u,$ $\P$--a.s. We also note that in case $C_0 = 0$, $\P$--a.s., we have $\int_{[0,t]} f_u \d C_u = \int_{(0,t]} f_u \d C_u$. This then implies $\int_S^T f_u \d C_u = \int_{(0,\infty)} \mathbf{1}_{(S,T]}(u) f_u \d C_u$, $\P$--a.s. In case $C$ and $f$ are both predictable, the integral process $\int_0^\cdot f_u \d C_u$ can be chosen to be predictable as well. Finally, even when $C = (C_t)_{t \in [0,\infty]}$ has a well-defined value $C_\infty$ at infinity, which does not necessarily correspond to its left limit $C_{\infty-}$, we never include $\infty$ in the domain of the integration. So we always have $\int_S^T f_u \d C_u = \int_{[0,\infty)}\1_{(S,T]}(u) f_u \d C_u$ and $\int_{S-}^T f_u \d C_u = \int_{[0,\infty)}\1_{[S,T]}(u) f_u \d C_u$ up to a $\P$--null set.

\subsubsection{Vector stochastic integral}\label{sec::vector_si}

In this part, we recall the $\L^2$-theory of the vector stochastic integral and refer the reader to \cite[Section III.6]{jacod2003limit} or \cite[Section IV.2]{jacod1979calcul} for details. Let $X = (X)_{t \in [0,\infty)}$ be an $\R^m$-valued process with components in $\cH^2_\mathrm{loc}$ with $X_0 = 0$, $\P$--a.s., and denote by $\H^{2,0}(X)$ the linear space of $\R^m$-valued, predictable processes $Z=(Z_t)_{t \in [0,\infty)}$ for which each component $Z^i$ satisfies
\begin{equation*}
	\E\bigg[ \int_{(0,\infty)} |Z^i_s|^2 \d \langle X^i, X^i \rangle_s\bigg] < \infty,\; i\in\{1,\dots,m\}.
\end{equation*}
Let $\cL^{2,0}(X)$ denote the linear subspace of $\cH^2$ consisting of the component-wise stochastic integrals $\sum_{i = 1}^m \int_{(0,\cdot]} Z^i_s \d X^i_s$ for $Z \in \H^{2,0}(X)$. We endow $\H^{2,0}(X)$ with the semi-norm
\begin{equation*}
	\|Z\|^2_{\H^{\smalltext{2}\smalltext{,}\smalltext{0}}(X)} \coloneqq \E \bigg[ \sum_{i= 1}^m\sum_{j=1}^m \int_{(0,\infty)} Z^i_s Z^j_s \d \langle X^i,X^j \rangle_s \bigg],
\end{equation*}
which is finite on $\H^{2,0}(X)$ by the Kunita--Watanabe inequality. The map
\begin{equation}\label{eq::isometry}
	\H^{2,0}(X) \ni Z \longmapsto \sum_{i = 1}^m \int_{(0,\cdot]} Z^i_s \d X^i_s \in \cH^2,
\end{equation}
is an isometry between the semi-normed space $\H^{2,0}(X)$ and the Hilbert space $\cH^2$ since
\begin{equation*}
	\bigg\langle \sum_{i = 1}^m \int_{(0,\cdot]} Z^i_s \d X^i_s, \sum_{j = 1}^m \int_{(0,\cdot]} Z^j_s \d X^j_s \bigg\rangle = \sum_{i=1}^m\sum_{j=1}^m \int_{(0,\cdot]} Z^i_s Z^j_s \d \langle X^i, X^j \rangle_s.
\end{equation*}
Note that for $(Z,Z^\prime) \in \H^{2,0}(X) \times \H^{2,0}(X)$ with $\|Z-Z^\prime\|_{\H^{\smalltext{2}\smalltext{,}\smalltext{0}}}=0$, we have $\int_{(0,\cdot]} Z_s \d X_s = \int_{(0,\cdot]} Z^\prime_s \d X_s$, up to $\P$-indistinguishability. In general, the space $\H^{2,0}(X)$ is not complete, and therefore $\cL^{2,0}(X)$ is not closed in $\cH^2$. In what follows, we will construct the completion of $\H^{2,0}(X)$, which leads to the notion of the vector stochastic integral whose collection thereof is a stable subspace of $\cH^2$, and thus, in particular, closed.

\medskip
Let $C = (C_t)_{t \in [0,\infty)}$ be a predictable process which is $\P$--a.s. right-continuous, non-decreasing and starts from zero. Consider a predictable process $c = (c^{i,j})_{(i,j) \in \{1,\ldots,m\}^\smalltext{2}}$ with values in the space of positive semi-definite matrices that satisfies
\begin{equation}\label{eq::dis_x_circ}
	\langle X^i, X^j \rangle_\cdot = \int_{(0,\cdot]} c^{i,j}_s \d C_s, \, \text{$\P$--a.s.}
\end{equation}

\begin{remark}
	We borrow the following construction of the pair $(c,C)$ from {\rm\cite{nutz2015robust}}. Consider $C \coloneqq \sum_{i = 1}^m \langle X^i, X^i \rangle$, and let
	\begin{equation*}
		c^{i,j}_t \coloneqq \hat c^{i,j}_t \mathbf{1}_{\{\hat c^{i,j}_s \in \S^\smalltext{m}_\smalltext{+}\}}, \; \text{\rm where} \; \hat c^{i,j}_t \coloneqq \limsup_{n \rightarrow \infty} \frac{\langle X^i,X^j \rangle_t - \langle X^i,X^j\rangle_{(t-1/n)\lor 0}}{C_t - C_{(t-1/n) \lor 0}},
	\end{equation*}
	and $\S^m_+$ is the space of positive semi-definite, real-valued $m \times m$ matrices, and where we used the convention $0 \coloneqq 0 / 0$.
\end{remark}

The linear space of all predictable processes $Z = (Z_t)_{t \in [0,\infty)}$ with values in $\R^m$ satisfying
\begin{equation*}
	\|Z\|^2_{\H^\smalltext{2}(X)} \coloneqq \E \bigg[ \int_{(0,\infty)}  \sum_{i=1}^m\sum_{j=1}^m Z^i_s c^{i,j}_s Z^j_s  \d C_s \bigg] < \infty,
\end{equation*}
is denoted by $\H^2(X)$. Note that $\H^2(X)$ does not depend on the choice of the pair $(c,C)$ for which \eqref{eq::dis_x_circ} holds and that $\|\cdot\|_{\H^\smalltext{2}(X)}$ is a semi-norm which coincides with $\|\cdot\|_{\H^{\smalltext{2}\smalltext{,}\smalltext{0}}}$ on $\H^{2,0}(X)$. The space $\H^2(X)$ together with $\|\cdot\|_{\H^\smalltext{2}(X)}$ is the semi-norm completion of $\H^{2,0}(X)$ by \cite[Th\'eor\`eme 4.35]{jacod1979calcul}.\footnote{To be precise, $\mathbb{H}^{2}(X)$, together with $\|\cdot \|_{\mathbb{H}^{\text{\fontsize{4}{4}\selectfont $2$}}(X)}$, forms a complete semi-normed space containing $\mathbb{H}^{2,0}(X)$ as a dense subset.} The isometry in \eqref{eq::isometry} thus extends uniquely to an isometry
\begin{equation}\label{eq::isometry2}
	\H^2(X) \ni Z \longmapsto \int_{(0,\cdot]} Z_s \d X_s \in \cH^2,
\end{equation}
between the semi-normed space $\H^2(X)$ and the Hilbert space $\cH^2$. Since by continuity, for two elements $Z$ and $Z^\prime$ in $\H^2(X)$ satisfying $\|Z - Z^\prime\|_{\H^\smalltext{2}(X)} = 0$, we have $\int_{(0,\cdot]} Z_s \d X_s = \int_{(0,\cdot]} Z^\prime_s \d X_s,$ we can turn \eqref{eq::isometry2} into an isometry between Banach spaces after identifying processes in $\H^2(X)$ with $\|Z-Z^\prime\|_{\H^\smalltext{2}(X)} = 0$. For each $Z$ in $\H^2(X)$, $\int_{(0,\cdot]} Z_s \d X_s = ( \int_{(0,t]} Z_s \d X_s)_{t \in [0,\infty)},$ is the \emph{vector stochastic integral} of $Z$ with respect to $X$. It is the, up to $\P$-indistinguishability, unique process in $\cH^2$ that starts from zero, $\P$--a.s., and satisfies
\begin{equation*}
	\bigg\langle \int_{(0,\cdot]} Z_s \d X_s , N \bigg\rangle_\cdot = \int_{(0,\cdot]}\bigg( \sum_{i = 1}^m H^i_s c^{N,i}_s \bigg) \d C_s, \; \text{$\P$--a.s.,}
\end{equation*}
where $c^{N,i}$ is a predictable process with $\langle N, X^i \rangle_\cdot = \int_{(0,\cdot]} c^{N,i}_s \d C_s$. Similarly to the stochastic integral of one-dimensional processes, for $Z \in \H^2(X)$ and $T \in \cT_{0,\infty}$, the predictable process $\Omega \times [0,\infty) \ni (\omega,t) \longmapsto Z_s(\omega)\mathbf{1}_{[0,T(\omega)]}(s) \in \R^m$ is in $\H^2(X)$ and $\int_{(0,\cdot]} Z_s \mathbf{1}_{[0,T]}(s) \d X_s = \int_{(0,\cdot \land T]} Z_s \d X_s,$ $\P$--a.s. Moreover, for $W$ in $\H^2 \big(\int_{(0,\cdot]}Z_s \d X_s \big)$, the product $W Z$ is in $\H^2(X)$ and
		\begin{equation*}
			\int_{(0,\cdot]} W_sZ_s \d X_s = \int_{(0,\cdot]} W_s \d \bigg( \int_{(0,\cdot]} Z_u \d X_u \bigg)_s.
		\end{equation*}
As the space of vector stochastic integrals $\cL^2(X) \coloneqq \big\{ \int_{(0,\cdot]} Z_s \d X_s : Z \in \H^2(X) \big\}$ is the image of an isometry defined on a Banach space, it is a closed subspace of $\cH^2$. Additionally, it is also a stable subspace of $\cH^2$ (see also \cite[Definition 4.4 and Theorem 4.35]{jacod1979calcul}). Let us stress that for $Z \in \H^2(X)$, we have the following equalities
\begin{equation*}
	\|Z\|^2_{\H^\smalltext{2}(X)} = \E \bigg[ \int_{(0,\infty)}  \sum_{i=1}^m\sum_{j=1}^m Z^i_s c^{i,j}_s Z^j_s  \d C_s \bigg] = \bigg\| \int_{(0,\cdot]} Z_s \d X_s \bigg\|^2_{\cH^\smalltext{2}}.
\end{equation*}

We close this part by agreeing on adopting a useful convention that will ease the notation in what follows. First, we agree to write $\int_0^t Z_s \d X_s \coloneqq \int_{(0,t]} Z_s \d X_s, \; t \in [0,\infty).$ Since the vector stochastic integral is in $\cH^2$, it will have a $\P$--a.s. unique $\cG_\infty$-measurable, real-valued, square-integrable limit at infinity which we denote by $
 	\int_0^\infty Z_s \d X_s.$ For two stopping times $S$ and $T$, we use the convention $\int_S^T Z_u \d X_u \coloneqq \int_0^T Z_u \d X_u - \int_0^{S\land T} Z_u \d X_u.$

\subsubsection{Stochastic integral with respect to a compensated integer-valued random measure}\label{sec::compensated_si}

Here, we recall the construction of the stochastic integral with respect to compensated integer-valued random measures in the sense of \cite{cohen2015stochastic,jacod1979calcul,jacod2003limit}. Before doing so, we need to introduce some terminology and notation. The \emph{graph} of a stopping time $T$ is the set $\llbracket T \rrbracket \coloneqq \{(\omega,t) \in \Omega \times [0,\infty): T(\omega)= t\},$ and an optional set $D \subseteq \Omega \times [0,\infty)$ is \emph{thin} if there exists a sequence of stopping times $(T_n)_{n \in \N}$ such that $D = \cup_{n \in \N} \llbracket T_n \rrbracket.$ If $\llbracket T_m \rrbracket \cap \llbracket T_n \rrbracket = \varnothing$ for $m \neq n$, then $(T_n)_{n \in \N}$ is referred to as an \emph{exhausting sequence} of stopping times for $D$. Note that every thin set admits an exhausting sequence of stopping times (see \cite[Lemma I.1.31]{jacod2003limit}). Recall from \Cref{sec::stochastic_basis} that a \emph{predictable stopping time} is a stopping time $T$ for which $\llbracket 0,T\llbracket$ is a predictable subset of $\Omega \times [0,\infty)$.

\medskip
We now turn to random measures. Let $(E,\cE)$ be a Blackwell space in the sense of \citeauthor*{dellacherie1978probabilities} \cite[Definition III.24]{dellacherie1978probabilities} and let $\widetilde\Omega \coloneqq \Omega \times [0,\infty) \times E$.\footnote{One can think for simplicity of $E$ being a Polish space together with its Borel-$\sigma$-algebra $\cE = \cB(E)$.} We consider two $\sigma$-algebrae on $\widetilde\Omega$, the predictable one given by $\widetilde\cP \coloneqq \cP \otimes \cE$ and the optional one given by $\widetilde\cO \coloneqq \cO \otimes \cE$. Let $\mu = \{\mu(\omega; \d t, \d x) : \omega \in \Omega\}$ be a random measure on $\cB([0,\infty))\otimes \cE$. For an $\cG \otimes \cB([0,\infty)) \otimes \cE$-measurable function $U : \widetilde\Omega \longrightarrow \R$, we define the process $U \star \mu = (U \star \mu_t)_{t \in [0,\infty)}$ by
\begin{equation}\label{eq::def_int_rm}
	U \star \mu_t(\omega) \coloneqq  
	\begin{cases}
		\displaystyle \int_{(0,t] \times E} U_s(\omega;x) \mu(\omega; \d s, \d x), \; \text{if} \; \int_{(0,t] \times E} |U_s(\omega;x)| \mu(\omega; \d s, \d x) < \infty, \\
		\infty, \; \text{otherwise.}
	\end{cases}
\end{equation}
We suppose that $\mu$ is an integer-valued random measure in the sense of \cite{jacod2003limit}, that is, there exists a $\widetilde\cP$-measurable function $V > 0$ satisfying $\E[V \star\mu_\infty]<\infty$, an $\cO$--measurable, $E$--valued process $\varrho = (\varrho)_{t \in [0,\infty)}$, and a thin set $D$ with an exhausting sequence of stopping times $(T_n)_{n \in \N}$ such that
\begin{equation*}
	\mu(\omega; \d t, \d x) = \sum_{(\omega,s) \in D} \boldsymbol{\delta}_{(s,\varrho_\smalltext{s}(\omega))}(\d t, \d x) = \sum_{n \in \N} \mathbf{1}_{\{T_n < \infty\}}(\omega)\boldsymbol{\delta}_{(T_\smalltext{n}(\omega),\varrho_{\smalltext{T}_\tinytext{n}\smalltext{(}\smalltext{\omega}\smalltext{)}}(\omega))}(\d t, \d x).
\end{equation*}
Here $\boldsymbol{\delta}_{(s,z)}(\d t, \d x)$ denotes the Dirac measure at $(s,z)$. Note that $U \star \mu$ is an optional processes for any $\widetilde\cO$-measurable function $U : \widetilde\Omega \longrightarrow \R$.
\begin{example}
	An example of an integer-valued random measure is the jump measure of an adapted, c\`adl\`ag process $X$
	\begin{equation*}
		\mu(\omega; \d t, \d x) = \sum_{s \in (0,\infty)} \mathbf{1}_{\{\Delta X_\smalltext{s}(\omega) \neq 0\}} {\boldsymbol{\delta}}_{(s,\Delta X_\smalltext{s}(\omega))}(\d t, \d x).
	\end{equation*}
\end{example}

The predictable compensator of $\mu$ is the random measure $\nu = \{\nu(\omega; \d t, \d x) : \omega \in \Omega\}$, which is (up to $\P$--null sets) uniquely characterised by the following: $\nu(\omega; \{0\} \times E) = 0$ for each $\omega \in \Omega$, the process $U \star \nu$ is $\cP$-measurable and satisfies $\E[U \star \mu_\infty] = \E[U \star \nu_\infty]$, for each non-negative, $\widetilde\cP$-measurable function $U$ (see \cite[Proposition II.1.28]{jacod2003limit}). Moreover, we choose a version of $\nu$ that satisfies $\nu(\omega; \{t\} \times E) \leq 1$ for every $(\omega,t) \in \Omega \times [0,\infty)$ and such that $\{(\omega,t) \in \Omega \times [0,\infty): \nu(\omega; \{t\} \times E) > 0\}$ can be exhausted by a sequence of predictable stopping times (see \cite[Proposition II.1.17]{jacod2003limit}). Next, there exists (see \cite[Theorem II.1.8]{jacod2003limit} together with \cite[Lemma 6.5.10]{weizsaecker1990stochastic}) a right-continuous, $\P$--a.s. non-decreasing, predictable process $C = (C_t)_{t \in [0,\infty)}$ and a transition kernel $K$ from $(\Omega \times [0,\infty),\cP)$ to $(E,\cE)$ such that
\begin{equation}
	\nu(\omega; \d s, \d x) = K_s(\omega; \d x) \d C_s(\omega),\; \text{$\P$--a.e.}
\end{equation} 
For any $\cG \otimes \cB([0,\infty)) \otimes \cE$-measurable function $U : \widetilde\Omega \longrightarrow \R$, we define the process $\widehat U = (\widehat U_t)_{t \in [0,\infty)}$ by
\begin{equation*}
	\widehat U_t(\omega) \coloneqq
	\begin{cases}
		\displaystyle \int_{\{t\} \times E} U_t(\omega;x) \nu(\omega; \d t, \d x), \; \text{if} \; \int_{\{t\} \times E} |U_t(\omega;x)| \nu(\omega; \d t, \d x) < \infty, \\
		\infty, \; \text{otherwise,}
	\end{cases}
\end{equation*}
and $\widetilde U = (\widetilde U_t)_{t \in [0,\infty)}$ by $\widetilde U_t(\omega) \coloneqq U_t(\omega,\varrho_t(\omega)) \mathbf{1}_D(\omega,t) - \widehat U_t(\omega).$ Note that $\widetilde U$ is an optional process and that for each $\omega \in \Omega$, the collection $\{t \in [0,\infty): \widetilde U_t(\omega) \neq 0\}$ is at most countable, and thus $\{(\omega,t) \in \Omega \times [0,\infty) : \widetilde U_t(\omega) \neq 0\}$ admits an exhausting sequence of stopping times (see \cite[Theorem IV.88]{dellacherie1978probabilities}). The sum $\sum_{s \in (0,\infty)} |\tilde U_s|^2$ is thus well-defined and $\cG$-measurable.

\medskip
We now turn to the construction of the compensated stochastic integral with respect to $\mu$. We denote by $\H^2(\mu)$ the linear space of $\widetilde\cP$-measurable functions $U : \widetilde\Omega \longrightarrow \R$ satisfying $\E [ \sum_{s \in (0,\infty)} |\widetilde U_s|^2 ] < \infty.$ For each $U \in \H^2(\mu)$, there exists a, up to $\P$--indistinguishability, unique purely discontinuous $U \star \tilde\mu \in \cM_\mathrm{loc}$ whose jumps are given by $\Delta (U \star \tilde\mu) = \widetilde U$ up to $\P$-evanescence (see \cite[Theorem I.4.56]{jacod2003limit}). Note that $(U + U^\prime) \star \tilde\mu = U \star \tilde\mu + U^\prime \star \tilde\mu$, $\P$--a.s., for $(U,U^\prime) \in \H^2(\mu) \times \H^2(\mu)$. Since
\begin{equation*}
	[U \star \tilde\mu]_\cdot = \sum_{s \in (0,\cdot]} |\Delta (U \star \tilde\mu)_s|^2 = \sum_{s \in (0,\cdot]} |\widetilde U_s|^2, \; \text{$\P$--a.s.,}
\end{equation*}
we find that $[U \star \tilde\mu]_\infty \in \L^1(\cG_\infty)$ and thus $U \star \tilde\mu \in \cH^2$ (see \cite[Proposition I.4.50]{jacod2003limit}). By \cite[Theorem II.1.33]{jacod2003limit}, the predictable quadratic variation of the process $U \star \tilde\mu$ is given by
\begin{align*}
	\langle U \star \tilde\mu \rangle_t(\omega) &= (U - \widehat U)^2\star \nu_t(\omega) + \sum_{s \in (0,t]} \big( 1-\zeta_s(\omega) \big) |\widehat U_s(\omega)|^2 \\
	&= \int_{(0,\cdot]} \int_{E}\big(U_s(\omega; x) - \widehat U_s(\omega)\big)^2 K_s(\omega; \d x) \d C_s(\omega) + \int_{(0,t]} \big(1 - \zeta_s(\omega)\big) \bigg( \int_{E} U_s(\omega;x) K_s(\omega; \d x) \bigg)^2 \Delta C_s(\omega) \d C_s(\omega) \\
	&= \int_{(0,\cdot]} \Bigg( \int_{E}\big(U_s(\omega; x) - \widehat U_s(\omega)\big)^2 K_s(\omega; \d x)  + \big(1 - \zeta_s(\omega)\big) \bigg( \int_{E} U_s(\omega;x) K_s(\omega; \d x) \bigg)^2 \Delta C_s(\omega) \Bigg) \d C_s(\omega) \\
	&\eqqcolon \int_{(0,\cdot]} \big(\vertiii{U_s(\omega;\cdot)}_s(\omega) \big)^2 \d C_s(\omega), \; t \in [0,\infty), \; \text{$\P$--a.e. $\omega \in \Omega$},
\end{align*}
where $\zeta = (\zeta_s)_{s \in [0,\infty)}$ is the predictable process defined by $\zeta_s(\omega) \coloneqq \nu(\omega ; \{s\} \times E) \in [0,1]$. Note that we can choose a predictable version of $\Delta C$ that is $[0,\infty)$-valued, and agrees with the jump process of $C$ up to a $\P$--null set. Indeed, we define
\begin{equation*}
	\Delta C_t \coloneqq {\Delta C}^\prime_t \mathbf{1}_{\{0 \leq{\Delta C}^\smalltext{\prime}_\smalltext{t} < \infty\}}, \;\text{where}\;  {\Delta C}^\prime_t(\omega) \coloneqq \limsup_{n \uparrow\uparrow \infty} \big(C_t - C_{(t-1/n)\lor 0}\big).
\end{equation*}
With this construction of the jump process of $C$, the map $\Omega \times [0,\infty) \ni (\omega,s) \longmapsto \vertiii{U_s(\omega;\cdot)}_s(\omega) \in [0,\infty]$ becomes predictable. Conversely, if we start with a $\widetilde\cP$-measurable function $U$ such that $\|U\|^2_{\H^\smalltext{2}(\mu)} \coloneqq \E [ \int_{(0,\infty)} \vertiii{U_s(\cdot)}_s ^2 \d C_s ] < \infty,$ then $U \in \H^2(\mu)$ and thus $U \star \tilde\mu \in \cH^2$ (see \cite[Theorem II.1.33]{jacod2003limit}). Note that for $(U,V) \in \H^2(\mu) \times \H^2(\mu)$ satisfying $\|U - V\|_{\H^\smalltext{2}(\mu)} = 0$, we have $U \star \tilde\mu = V \star \tilde\mu$. We therefore identify $U$ und $V$ in $\H^2(\mu)$ in this case, which turns $\H^2(\mu)$ into a normed space. The space of compensated stochastic integrals $\cK^2(\mu) \coloneqq \{U \star \tilde\mu : U \in \H^2(\mu)\}$ is a stable subspace of $\cH^2$ by \cite[Proposition 3.71 and Theorem 4.46]{jacod1979calcul} and thus closed in $\cH^2$. We end up with the following result, whose proof we defer to \Cref{app::proofs_preliminaries}.

\begin{proposition}\label{prop::completeness_ivm}
	The space $\H^2(\mu)$ endowed with the norm $\|\cdot\|_{\H^\smalltext{2}(\mu)}$ is a Banach space. Moreover, for each $U \in \H^2(\mu)$,
	\begin{equation*}
		\|U\|^2_{\H^\smalltext{2}(\mu)} = \E \bigg[ \int_{(0,\infty)} \big(\vertiii{U_s(\cdot)}_s \big)^2 \d C_s \bigg] = \E\big[\langle U \star \tilde\mu\rangle_\infty \big] = \E\big[ [ U \star \tilde\mu]_\infty \big].
	\end{equation*}
\end{proposition}

For $(\omega,t) \in \Omega \times [0,\infty)$, let $\mathfrak{H}_{\omega,t}$ denote the collection of $\cE$-measurable maps $\cU : E \longrightarrow \R$ satisfying $\vertiii{\cU(\cdot)}_t(\omega) < \infty$. Define $\fH$ as the collection of $\widetilde\cP$-measurable functions $U: \widetilde\Omega \longrightarrow \R$ such that $U_t(\omega; \cdot) \in \fH_{\omega,t}$ for each $(\omega,t) \in \Omega \times [0,\infty)$. Since for $U \in \H^2(\mu)$, we have $\vertiii{U_t(\omega;\cdot)}_t(\omega) < \infty$ for $\d \P \times \d C$--a.e. $(\omega,t) \in \Omega \times [0,\infty)$, we can define $U^\prime_t(\omega;x) \coloneqq U_t(\omega;x) \mathbf{1}_{N^c}(\omega,t),$ where $N = \{(\omega,t) \in \Omega \times [0,\infty) : \vertiii{U_t(\omega;\cdot)}_t(\omega) = \infty\} \in \cP$, which yields a version of $U$ in $\H^2(\mu)$ which is also in $\fH$. We thus always choose a version of $U \in \H^2(\mu)$ that is also in $\fH$. The space $\fH$ will be fundamental in our formulation of reflected BSDEs in \Cref{sec::formulation}.

\medskip
Let us close this part by a agreeing on a convention similar to the one we made about vector stochastic integrals. Note first that since $U \star \tilde \mu \in \cH^2$ for $U \in \H^2(X)$, there's a well-defined limit at infinity $U \star \mu_\infty$. In the sequel, we will denote the process $U \star\tilde\mu$ by $\int_0^t \int_{E} U_s(x) \d \tilde\mu(\d s, \d x) \coloneqq U \star \tilde\mu_t,$ $t \in [0,\infty]$, and for two stopping times $S$ and $T$, we write
\begin{equation*}
	\int_S^T \int_{E} U_s(x) \d \tilde\mu(\d s, \d x) \coloneqq \int_0^T \int_{E} U_s(x) \d \tilde\mu(\d s, \d x) - \int_0^{S \land T} \int_{E} U_s(x) \d \tilde\mu(\d s, \d x) = U\star\tilde\mu_T - U\star\tilde\mu_{S \land T}.
\end{equation*}

\subsection{Orthogonal decomposition}

Let $X$ be the $\R^m$-valued process with components in $\cH^2_\mathrm{loc}$ from \Cref{sec::vector_si} and $\mu$ be the integer-valued random measure from \Cref{sec::compensated_si}. We are not working under the assumption of martingale representation, and thus want to find conditions on $X$ and $\mu$ that allow us to decompose a square-integrable martingale uniquely along $X$, $\mu$ and another square-integrable martingale $N$ appropriately orthogonal to $X$ and $\mu$. We mentioned in \Cref{sec::vector_si} and \Cref{sec::compensated_si} that the spaces $\cL^2(X)$ and $\cK^2(\mu)$ of stochastic integrals with respect to $X$ and $\mu-\nu$, respectively, are stable (and thus closed) subspaces of $\cH^2$. We give sufficient conditions on $X$ and $\mu$ under which $\cL^2(X) \cap \cK^2(\mu)$ is the null space in $\cH^2$. This allows us to write
\[
	\cH^2 = \cL^2(X) \oplus \cK^2(\mu) \oplus \cH^{2,\perp}(X,\mu) \; \text{with} \;\cH^{2,\perp}(X,\mu) \coloneqq \big(\cL^2(X) \oplus \cK^2(\mu)\big)^\perp.
\]
This part is based on \cite{papapantoleon2018existence} and we borrow notations from \cite{jacod2003limit}. Let $M_{\mu}$ be the Dol\'eans measure defined by
\begin{equation*}
	M_{\mu} [W] \coloneqq\int W(\omega,s,x)M_{\mu}(\d \omega, \d s, \d x) \coloneqq \E [W \star \mu_\infty],
\end{equation*}
for each $\cG \otimes \cB\big([0,\infty)\big) \otimes \cE$-measurable function $W : \widetilde\Omega \longrightarrow [0,\infty]$. Recall that there exists a $\widetilde\cP$ measurable function $V > 0$ with $\E[V \star \mu_\infty] < \infty$. Thus the restriction of $M_\mu$ to $\widetilde\cP$ is a $\sigma$-finite measure. By the Radon–Nikod\'ym theorem, there exists for every $\cG \otimes \cB([0,\infty)) \otimes \cE$-measurable function $W : \widetilde \Omega \longrightarrow [0,\infty]$, a $M_{\mu}$--a.e. unique, $\widetilde\cP$-measurable function $M_{\mu}[W| \widetilde\cP] \coloneqq W^\prime : \widetilde\Omega \longrightarrow [0,\infty]$ satisfying $M_{\mu}[W^\prime U] = M_{\mu}[WU],$ for each $\widetilde\cP$-measurable function $U : \widetilde\Omega \longrightarrow [0,\infty]$. For a general measurable function $W : \widetilde\Omega \longrightarrow [-\infty,\infty]$, we use the same convention as in \cite{jacod2003limit} to define $M_{\mu}[W| \widetilde\cP]$, namely
\begin{equation*}
	M_{\mu}\big[W\big| \widetilde\cP\big]\coloneqq
	\begin{cases}
		M_{\mu}\big[\max\{W,0\}\big| \widetilde\cP\big] - M_{\mu}\big[\max\{-W,0\}\big| \widetilde\cP\big], \; \text{on the set where $M_{\mu}\big[|W|\big| \widetilde\cP\big] < \infty$}, \\
		\infty, \; \text{otherwise.}
	\end{cases}
\end{equation*}
The following is the main result about orthogonal decompositions along $X$ and $\mu$, whose proof we defer to \Cref{app::proofs_preliminaries}.
\begin{proposition}\label{prop::orthogonal_decomposition}
Suppose that $M_{\mu}[\Delta X^i | \widetilde\cP] = 0$ for every $i \in \{1,\ldots,m\}$. For each $M \in \cH^2$, there exists a unique pair $(Z,U) \in \H^2(X) \times \H^2(\mu)$ such that $N =(N_t)_{t \in [0,\infty)} \in \cH^2$ defined by
	\begin{equation*}
		N_t \coloneqq M_t - \int_0^t Z_s \d X_s - \int_0^t\int_E U_s(x) \tilde\mu(\d s, \d x),
	\end{equation*}
	satisfies $\langle N, X^i \rangle = 0$ for each $i \in \{1,\ldots,m\}$ and $M_{\mu}[\Delta N | \widetilde\cP] = 0$.
\end{proposition}
\begin{corollary}
	Under the assumptions of {\rm\Cref{prop::orthogonal_decomposition}}, we have $\cH^2 = \cL^2(X) \oplus \cK^2(\mu) \oplus \cH^{2,\perp}(X,\mu),$ where $\cH^{2,\perp}(X,\mu) \coloneqq \big(\cL^2(X) \oplus \cK^2(\mu)\big)^\perp.$ In particular, $\cH^{2,\perp}(X,\mu) = \big\{N \in \cH^2 : M_{\mu}[\Delta N | \widetilde\cP] = 0,\; \langle N, X^i \rangle = 0 \text{ \rm for $i \in \{1,\ldots,m\}$}\big\}.$
\end{corollary}

\subsection{Data and the corresponding weighted spaces}\label{sec::weighted_spaces}

In this section we fix the data of the reflected BSDE and the weighted spaces in which we will construct its solution. The obstacle and terminal condition are described by a single optional process $\xi$ as in \cite{grigorova2017reflected,grigorova2020optimal}. From \Cref{rem::after_formulation}, this is without loss of generality. Throughout this work, we fix once and for all the data $(X, \mu,\G,T,\xi,f,C)$, where
\begin{enumerate}[{\bf(D1)}, leftmargin=0.9cm]
	\item \label{data::filtration} $\G$ is the filtration in the stochastic basis;
	\item \label{data::martingale} $X = (X_t)_{t \in [0,\infty)}$ is an $\R^m$-valued process whose components are in $\cH^2_\mathrm{loc}$ with $X_0 = 0$, $\P$--a.s., $\mu$ is an integer-valued random measure on $\R_+ \times E$, where $(E,\cE)$ is some Blackwell space, and $M_\mu[\Delta X^i | \widetilde\cP] = 0$ for each $i \in \{1,\ldots,m\}$;
	\item $C = (C_t)_{t \in [0,\infty)}$ is a real-valued, predictable process with $\P$--a.s. right-continuous and non-decreasing paths starting from zero that satisfies 
		\begin{equation*}
			\d\langle X^i, X^j\rangle_s(\omega) = c^{i,j}_s\d C_s(\omega), \; \text{and} \; \nu(\omega;\d s, \d x) = K_s(\omega,\d x)\d C_s(\omega), \; \text{for $\P$--a.e. $\omega \in \Omega$, for each $(i,j) \in \{1,\ldots,d\}^2$},
		\end{equation*}
		where each $c^{i,j}=(c^{i,j}_t)_{t \in [0,\infty)}$ is a predictable process with values in the set of positive, semi-definite, symmetric matrices, and $K$ is a transition kernel from $(\Omega \times [0,\infty),\cP)$ into $(E,\cE)$;
	\item $T$ is a $\G$--stopping time;
	\item \label{data::expectation_sup}$\xi = (\xi_t)_{t \in [0,\infty]}$ is a $[-\infty,\infty)$-valued, optional process satisfying\footnote{Recall from \Cref{rem::measurability_of_sup} that the expectation is well-defined as $\sup_{s \in [0,T)}\xi^+_s = \sup_{s \in [0,\infty]}\xi^+_s\mathbf{1}_{\{s < T\}}$.} $\E[|\xi_T|^2] + \E \big[ \sup_{s \in [0,T)} |\xi^+_{s}|^2 \big] < \infty;$
	\item \label{data::generator}$f : \bigsqcup_{(\omega,t) \in \Omega \times [0,\infty)} \big(\R \times \R \times \R^m \times \fH_{\omega,t}\big) \longrightarrow \R$ is such that for each $(y,\mathrm{y},z,u) \in \R \times \R \times \R^m \times \mathfrak{H}$, the map\footnote{The symbol $\bigsqcup$ denotes the disjoint union, and therefore each $f_t(\omega,\cdot,\cdot,\cdot,\cdot)$ is a map from $\R \times \R \times \R^m \times \fH_{\omega,t}$ into $\R$.}
		\begin{equation*}
		\Omega \times [0,\infty) \ni (\omega,t) \longmapsto f_t(\omega,y,\mathrm{y},z,u_t(\omega; \cdot) ) \in \R,
	\end{equation*}
	is optional and $f$ is $(r,\theta^X,\theta^\mu)$--Lipschitz continuous on $\llbracket 0, T \rrbracket \coloneqq \{(\omega,t) \in \Omega \times [0,\infty) : t \leq T(\omega) \}$ in the sense that
			\begin{align*}
		\big|f_t(\omega,y,\mathrm{y},z,u_t(\omega; \cdot)) - f_t(\omega,y^\prime,\mathrm{y}^\prime,z^\prime,u_t^\prime(\omega; \cdot))\big|^2 
		&\leq r_t(\omega) |y-y^\prime|^2 + \mathrm{r}_t(\omega) |\mathrm{y}-\mathrm{y}^\prime|^2 \\
		&\quad + \theta^X_t(\omega) \|c_t^{1/2}(\omega)(z-z^\prime)\|^2 + \theta^\mu_t(\omega) \Big(\vertiii{u_t(\omega;\cdot) - u^\prime_t(\omega;\cdot)}_t(\omega) \Big)^2,
	\end{align*}
			for $\P \otimes \d C$--a.e. $(\omega,t) \in \llbracket 0, T \rrbracket$, where $r = (r_t)_{t \in [0,\infty)}$, $\mathrm{r} = (\mathrm{r}_t)_{t \in [0,\infty)}$, $\theta^X = (\theta^X_t)_{t \in [0,\infty)}$ and $\theta^\mu =(\theta^\mu_t)_{t \in [0,\infty)}$ are $[0,\infty)$-valued, predictable processes, and $c^{1/2}$ is the unique square-root matrix-valued process of $c$;\footnote{See \citeauthor*{horn2013matrix} \cite[page 439]{horn2013matrix}.}
		\item \label{data::integ_f0} the optional process $f_\cdot(0,0,0,\mathbf{0})$\footnote{Here $\mathbf{0}$ denotes the zero element of the space $\H^2(\mu)$.} satisfies
		\begin{equation*}
			\E \bigg[\bigg(\int_0^T |f_s(0,0,0,\mathbf{0})| \d C_s \bigg)^2 \bigg] < \infty;
		\end{equation*}
	\item \label{data::process_A} the non-negative, predictable process $\alpha = (\alpha_t)_{t \in [0,\infty)}$ defined through $\alpha^2_t = \max\{\sqrt{r_t},\sqrt{\mathrm{r}_t},\theta^X_t,\theta^\mu_t\}$ satisfies $\alpha_t(\omega) > 0$ for $\P \otimes \d C$--a.e. $(\omega,t) \in \llbracket 0,T\rrbracket$, and the predictable process $A = (A_t)_{t \in [0,\infty)}$ defined by $A_t \coloneqq \int_0^{t \land T} \alpha^2_s \d C_s$ is real-valued and satisfies $\Delta A \leq \Phi$, up to $\P$-evanescence, for some $\Phi \in [0,\infty)$.
\end{enumerate}

\begin{remark}\label{rem::after_data}
$(i)$ If we start with a process $\tilde\xi$ that is only defined on $\{(\omega,t) \in \Omega \times [0,\infty] : t \leq T(\omega)\}$, then we ask here that the process $\xi = (\xi_t)_{t \in [0,\infty]}$ defined by $\xi_t(\omega) \coloneqq \tilde\xi_t(\omega) \mathbf{1}_{[0,T(\omega)]}(t) + \tilde\xi_{T(\omega)}(\omega) \mathbf{1}_{(T(\omega),\infty]}(t)$ is optional. The process $\tilde\xi$ is still the lower barrier and $\tilde\xi_T$ is the terminal condition of the reflected {\rm BSDE}.

\medskip
$(ii)$ Note that \ref{data::martingale} allows us to decompose each $M \in \cH^2$ uniquely into
\begin{equation*}
	M_t = M_0 + \int_0^t Z_s \d X_s + \int_0^t \int_E U_s(x) \mu(\d s, \d x) + N_t, \; t \in [0,\infty), \; \text{$\P$\rm--a.s.},
\end{equation*}
for $(Z,U,N) \in \H^2(X) \times \H^2(\mu) \times \cH^{2,\perp}(X,\mu)$ by {\rm\Cref{prop::orthogonal_decomposition}}.
\end{remark}

The process $\cE(\beta A)$ denotes the stochastic exponential of $\beta A$, that is, $\cE(\beta A) = \big(\cE(\beta A)_t\big)_{t \in [0,\infty)}$ is the unique right-continuous, adapted process satisfying
\begin{equation*}
	\cE(\beta A)_t = 1 + \int_0^t \cE(\beta A)_{s-}\beta \d A_s, \; t \in [0,\infty), \; \text{$\P$--a.s.}
\end{equation*}
As $A$ is $\P$--a.s. non-decreasing, it follows from \cite[Theorem I.4.61]{jacod2003limit} that
\begin{equation*}
	\displaystyle \cE(\beta A)_t = \mathrm{e}^{\beta A_t} \prod_{s \in (0,t]}(1+\gamma\Delta A_s)\mathrm{e}^{-\gamma\Delta A_s}, \; t \in [0,\infty), \; \text{$\P$--a.s.}
\end{equation*}
Therefore $\cE(\beta A)$ is $\P$--a.s. non-decreasing as well and satisfies $1\leq\cE(\beta A)\leq\mathrm{e}^{\beta A}$ up to $\P$--indistinguishability.
We now introduce the (weighted) classical spaces in which we will construct the solution to the reflected BSDE. Although these spaces depend on $(\alpha,C,T)$, we will suppress the dependence on $(\alpha,C)$ to ease the notation. For $\beta \in [0,\infty)$

\begin{itemize}[leftmargin=0.9cm]
	\item $\L^p_\beta(\cF)$, for $p \in [1,\infty)$ and a sub--$\sigma$-algebra $\cF \subseteq \cG$, denotes the space of real-valued, $\cF$-measurable random variables $\zeta$ satisfying 
	\[
	\|\zeta\|^{2}_{\L^{\smalltext{2}}_{\smalltext{\beta}}} \coloneqq \E\big[\big|\cE(\beta A)^{1/2}_T\zeta\big|^2\big] < \infty;
	\]
	\item $\cH^2_{T,\beta}$ denotes the Banach space of real-valued martingales $M = (M_t)_{t \in [0,\infty)}$ in $\cH^2$ satisfying $M = M_{\cdot \land T}$ and 
	\[
	\|M\|^2_{\cH^{\smalltext{2}}_{\smalltext{T}\smalltext{,}\smalltext{\beta}}} \coloneqq \E[M^2_0] + \E\bigg[\int_{(0,\infty)} \cE(\beta A)_s \d \langle M^T \rangle_s\bigg] = \E[M^2_0] + \E\bigg[\int_0^T \cE(\beta A)_s \d \langle M \rangle_s\bigg] < \infty;
	\]
	\item $\cS^2_{T,\beta}$ denotes the Banach space\footnote{That $\cS^2_{T,\beta}$ is a Banach space follows with the arguments described in \cite[IV.21, pp. 82--83]{dellacherie1982probabilities}.} of real-valued, optional processes $Y = (Y_t)_{t \in [0,\infty]}$ satisfying $Y = Y_{\cdot \land T}$ and
		\begin{equation*}
			\|Y\|^2_{\cS^{\smalltext{2}}_{\smalltext{T}\smalltext{,}\smalltext{\beta}}} \coloneqq \E \bigg[ \sup_{s \in [0,T]} \big| \cE(\beta A)^{1/2}_s Y_s \big|^2\bigg] < \infty;
		\end{equation*}
	\item $\H^2_{T,\beta}$ denotes the Banach space of real-valued, optional processes $\phi = (\phi_t)_{t \in [0,\infty]}$ satisfying $\phi = \phi_{\cdot \land T}$ and
		\begin{equation*}
			\|\phi\|^2_{\H^{\smalltext{2}}_{\smalltext{T}\smalltext{,}\smalltext{\beta}}} \coloneqq \E \bigg[ \int_{(0,\infty)} \cE(\beta A)_s |\phi_s|^2 \d C^T_s\bigg] = \E \bigg[ \int_0^T \cE(\beta A)_s |\phi_s|^2 \d C_s\bigg] < \infty;
		\end{equation*}
	\item $\H^2_{T,\beta}(X)$ denotes the Banach space of $\R^m$-valued, predictable processes $Z$ in $\H^2(X)$ satisfying $Z = Z \mathbf{1}_{\llbracket 0, T \rrbracket}$ and
		\begin{equation*}
			\|Z\|^2_{\H^\smalltext{2}_{\smalltext{T}\smalltext{,}\smalltext{\beta}}(X)} \coloneqq \big\|\cE(\beta A)^{1/2}Z\mathbf{1}_{\llbracket 0, T\rrbracket}\big\|^2_{\H^\smalltext{2}(X)} = \E\bigg[\int_0^T \cE(\beta A)_s \sum_{i=1}^m\sum_{j=1}^m Z^i_s c^{i,j}_s Z^j_s \d C_s\bigg] < \infty;
		\end{equation*}
	\item $\H^2_{T,\beta}(\mu)$ denotes the Banach space of real-valued, $\widetilde\cP$-predictable processes $U$ in $\H^2(\mu)$ satisfying $U = U \mathbf{1}_{\llbracket 0, T \rrbracket}$ and
		\begin{equation*}
			\|U\|^2_{\H^\smalltext{2}_{\smalltext{T}\smalltext{,}\smalltext{\beta}}(\mu)} \coloneqq \big\|\cE(\beta A)^{1/2}U\mathbf{1}_{\llbracket 0, T \rrbracket}\big\|^2_{\H^\smalltext{2}(\mu)} = \E\bigg[\int_0^T \cE(\beta A)_s \d \langle U \star \tilde\mu \rangle_s\bigg] < \infty;
		\end{equation*}
	\item $\cH^{2,\perp}_{0,T,\beta}(X,\mu)$ denotes the closed subspace of real-valued martingales $N = (N_t)_{t \in [0,\infty)}$ in $\cH^2_{T,\beta}$ with $N \in \cH^2_0 \cap \cH^{2,\perp}(X,\mu);$
	\item $\cI^2_{T,\beta}$ denotes the space of $[0,\infty)$-valued, optional processes $K = (K_t)_{t \in [0,\infty]}$ whose paths are $\P$--a.s. right-continuous and non-decreasing, satisfying $K = K_{\cdot \land T}$, $\E[K^2_T] < \infty$ and
	\begin{equation*}
		\|K\|^2_{\cI^\smalltext{2}_{\smalltext{T}\smalltext{,}\smalltext{\beta}}} \coloneqq \E\bigg[\bigg(\int_{[0,\infty)}\cE(\beta A)^{1/2}_s \d K^T_s\bigg)^2\bigg] = \E\bigg[\bigg(\int_0^T\cE(\beta A)^{1/2}_s \d K_s\bigg)^2\bigg] < \infty,
	\end{equation*}
	with convention $K_{0-} \coloneqq 0$.
	\end{itemize}
Finally, for $\beta = 0$, we simply write $\L^p \coloneqq \L^p_0(\cG)$, $\L^p(\cF) \coloneqq \L^p_0(\cF)$, $\cH^2_T \coloneqq \cH^2_{T,0}$, $\cS^2_T \coloneqq \cS^2_{T,0}$, $\H^2_T \coloneqq \H^2_{T,0}$, $\H^2_T(X) \coloneqq \H^2_{T,0}(X)$, $\H^2_T(\mu) \coloneqq \H^2_{T,0}(\mu)$, $\cH^{2,\perp}_{0,T}(X,\mu) \coloneqq \cH^{2,\perp}_{0,T,0}(X,\mu)$ and $\cI^2_{T} \coloneqq \cI^2_{T,0}$.

\subsection{Formulation of the reflected BSDE}\label{sec::formulation}

In this work, we consider reflected BSDEs driven by c\`adl\`ag martingales and integer-valued random measures on a possibly unbounded time horizon. It turns out that for the analysis that follows, it is convenient to construct the solution directly on $[0,\infty]$, although the driving martingales and integer-valued random measures are only defined on $[0,\infty)$. As we have seen before, we can and will assign a value to the (stochastic) integrals at infinity by taking the limit $t \uparrow\uparrow \infty$, whenever this makes sense. Inspired by the work of \citeauthor*{grigorova2020optimal} \cite{grigorova2020optimal}, we will not suppose any regularity on the paths of the obstacle process $\xi$. We thus have to consider its left-limit process $\overline\xi = (\overline\xi_t)_{t \in [0,\infty]}$ defined by
\begin{equation*}
	\overline\xi_t \coloneqq
		\xi_0\mathbf{1}_{\{t = 0\}}+\limsup_{s \uparrow\uparrow t} \xi_s\mathbf{1}_{\{t\in (0,\infty]\}}, 
\end{equation*}
which is $(\cG^U_t)_{t \in [0,\infty]}$-predictable by \Cref{prop::pred_left_limit}. Note that this process is $\P$--indistinguishable from a $(\cG_t)_{t \in [0,\infty]}$-predictable process by \Cref{lem::compl_filt} or \cite[Proposition 1.1]{jacod1979calcul}. The solution to the reflected BSDE with generator $f$ and obstacle process $\xi$ is a collection of processes $(Y,Z,U,N,K^r,K^\ell)$ satisfying the following conditions
\begin{enumerate}[{\bf (R1)}, leftmargin=0.9cm]
	\item \label{cond::integ_mart}$(Z,U,N) \in \H^2_T(X) \times \H^2_T(\mu) \times \cH^{2,\perp}_{0,T}(X,\mu)$, and $Y = (Y_t)_{t \in [0,\infty]}$ is optional with $\P$--a.s. l\`adl\`ag paths;
	\item \label{cond::integ_y_and_k} $(K^r,K^\ell) \in \cI^2_T \times  \cI^2_T$ with $K^r_0 = 0$, $\P$--a.s.;
	\item \label{cond::refl_eq} $(Y,Z,U,N,K^r,K^\ell)$ satisfies\footnote{We use a predictable version $Y_-$ of the left-limit process of $Y$ here (see for example Equation 2.3.3 and Lemma 6.1.3 in \cite{weizsaecker1990stochastic}).}
		\begin{equation*}
			\E\bigg[\int_0^T \big|f_s\big(Y_s,Y_{s-},Z_s,U_s(\cdot)\big)\big|\d C_s\bigg] < \infty,
		\end{equation*}
	and $\P$--a.s., for each $t \in [0,\infty]$,
		\begin{align*}
			Y_t &= \xi_T + \int_t^T f_s\big(Y_s,Y_{s-},Z_s,U_s(\cdot)\big) \d C_s - \int_t^T Z_s \d X_s - \int_t^T\int_E U_s(x)\tilde\mu(\d s, \d x) - \int_t^T\d N_s+ K^r_T - K^r_t + K^{\ell}_{T-} - K^{\ell}_{t-};
		\end{align*}
	\item \label{cond::y_equals_xi}$Y_T = \xi_T$, and $Y = Y_{\cdot \land T}$;
	\item \label{cond::y_greater_xi}$Y_{\cdot \land T} \geq \xi_{\cdot \land T}$;
	\item \label{cond::increasing_proc} $K^r$ is predictable, $K^\ell$ satisfies $K^\ell_T = K^\ell_{T-}$, $\P$--a.s., and is purely discontinuous, that is, $K^\ell_t = K^\ell_0 + \sum_{s \in (0,t] } \Delta K^\ell_s, \; t \in [0,\infty], \; \text{$\P$--a.s.},$
		 and the following Skorokhod condition holds, with $K^\ell_{0-} \coloneqq 0$,
		\[
		 \big(Y_{T-} - \overline{\xi}_T\big) \Delta K^r_T + \int_{(0,T)} \big(Y_{s-} - \overline{\xi}_s\big) \d K^r_s + \int_{[0,T)} (Y_s - \xi_s) \d K^\ell_s = 0, \;\text{$\P$--a.s.};
		\]
	
	\item \label{cond::repre_Y} $\displaystyle Y_S = {\esssup_{\tau \in \cT_{\smalltext{S}\smalltext{,}\smalltext{T}}}}^{\cG_\smalltext{S}} \E\bigg[ \xi_\tau + \int_S^\tau f_s\big(Y_s,Y_{s-},Z_s,U_s(\cdot)\big) \d C_s \bigg| \cG_S\bigg]$, $\P$--a.s., $S \in \cT_{0,T}$.
\end{enumerate}

Let us comment on the form of the generator.

\begin{remark}\label{rem::discussion_y}
	To the best of our knowledge, and except in {\rm\cite{papapantoleon2018existence}}, the literature only considers the case where the generator depends on $Y_{s-}$ and not on $Y_s$. When the integrator $C$ does not jump and is thus continuous, the dependence on $Y_s$ or $Y_{s-}$ does not matter as $\{s :Y_s \neq Y_{s-}\}$ will be at most countable and thus of $\d C_s$-measure zero. When the process $C$ can jump, the dependence matters, and we include both cases for the following reasons:
	
	\begin{itemize}
	\item[$(i)$] a dependence on $Y_{s-}$ in the generator has been considered in numerical schemes, see among others {\rm\citeauthor*{briand2001donsker} \cite{briand2001donsker}, \citeauthor*{briand2002robustness} \cite{briand2002robustness}, \citeauthor*{briand2021donsker} \cite{briand2021donsker}, \citeauthor*{cheridito2013bs} \cite{cheridito2013bs}, \citeauthor*{madan2015convergence} \cite{madan2015convergence}}, {\rm\citeauthor*{possamai2015weak} \cite{possamai2015weak}}, and {\rm\citeauthor*{papapantoleon2021stability} \cite{papapantoleon2019stability,papapantoleon2021stability}}$;$
	
	\item[$(ii)$] a linear {\rm BSDE} only seems to allow for an explicit representation of its $Y$-component as a conditional expectation if the linearity of the generator depends not on $Y_{s-}$ but on $Y_s$. To see this, one can adapt the techniques of {\rm\cite[Lemma 2.2]{carbone2008backward}} to our setup. Similarly, in the reflected {\rm BSDE} case, a $Y_s$-dependence in the generator already appears in the following seemingly simple example. Consider the optimal stopping problem $\sup_{\tau \in \cT_{0,T}} \E[\xi_\tau/D_\tau]$, where the discounting process is given by $D = \cE(\int_0^\cdot r_s \d C_s)$ for some predictable process $r = (r_s)_{s \in [0,\infty)}$ that is $C$-integrable. The value process $V_t = \esssup_{\tau \in \cT_{t,T}} \E[\xi_\tau D_t/D_\tau]$ of this optimal stopping problem, after applying some standard transformations and then Mertens's decomposition, is the $Y$-component of a reflected {\rm BSDE} with obstacle $\xi$ and generator of the form 
		\begin{equation*}
			f_s(Y_s) = \frac{r_s Y_s}{1+r_s\Delta C_s}.
		\end{equation*}
		As our main motivation is to develop $($reflected$)$ {\rm BSDEs} to analyse certain discrete- and continuous-time problems in a unified manner, we stress that the dependence on $Y_s$ in the generator is thus crucial.
	\end{itemize}
\end{remark}

We now comment on the formulation of the reflected BSDE. 
\begin{remark}\label{rem::after_formulation}
	$(i)$ That $\xi$ plays the role of the lower barrier and terminal condition is without loss of generality since for any lower barrier $L = (L_t)_{t \in [0,\infty)}$ and terminal condition $\zeta$, we can define the obstacle $\xi$ as $\xi_t \coloneqq L_t \mathbf{1}_{[0,T)}(t) + \zeta \mathbf{1}_{[T,\infty]}(t),$ $t\geq 0$. With our formulation, it is clear that on $[0,T)$, $\xi$ is the lower barrier $L$ and $\xi_T$ is the terminal condition $\zeta$.
	
	\medskip
	$(ii)$ With the conventions we agreed upon in {\rm\Cref{sec::integrals}}, the integral processes appearing in \ref{cond::refl_eq} and \ref{cond::repre_Y} never include the points $0$ or $\infty$ in the domain of integration.
	
	\medskip
	$(iii)$ Note that the forward dynamics of $Y$ are, $\P${\rm--a.s.},
	\begin{align*}
			Y_t = Y_0 - \int_0^{t \land T} f_s\big(Y_s,Y_{s-},Z_s,U_s(\cdot)\big) \d C_s  + \int_0^{t \land T} Z_s \d X_s + \int_0^{t \wedge T} \int_E U_s(x)\tilde\mu(\d s, \d x) + N_{t \land T} - K^r_{t \land T} - K^{\ell}_{(t \land T)-}, \; t \in [0,\infty].
		\end{align*}
	In particular, $Y$ is a $\P${\rm--a.s.} l\`adl\`ag optional semimartingale.
	
	\medskip
	$(iv)$ Conditions \ref{cond::y_equals_xi}, \ref{cond::y_greater_xi} and \ref{cond::repre_Y} are equivalent to
	\begin{equation*}
		Y_S = {\esssup_{\tau \in \cT_{S,\infty}}}^{\cG_\smalltext{S}} \E \bigg[  \xi_{\tau \land T} + \int_S^{\tau \land T} f_s\big(Y_{s-},Y_s, Z_s, U_s(\cdot)\big) \d C_s \bigg| \cG_S \bigg], \; \text{$\P$\rm--a.s.}, \; S \in \cT_{0,\infty}.
	\end{equation*}
	
	\medskip
	$(v)$ We will see in {\rm\Cref{lem::cond_repre}} that if $(\alpha Y,\alpha Y_-,Z,U,N) \in \H^2_{T,\hat\beta} \times \H^2_{T,\hat\beta} \times \H^2_{T,\hat\beta}(X) \times \H^2_{T,\hat\beta}(\mu)$, then \ref{cond::integ_mart} up to \ref{cond::increasing_proc} imply \ref{cond::repre_Y}.
	
	\medskip
	$(vi)$
	If $\xi$ is {\rm$\P$--a.s.} left--upper semicontinuous along stopping times, then $K^r$ is continuous. This is similar to {\rm\cite[Remark 2.4]{grigorova2020optimal}}.

	\medskip
	$(vii)$ If $\xi$ is $\P${\rm--a.s.} right-continuous, then so is $Y$. Indeed, let $Y^{(+)}$ be the right-continuous, optional process that {\rm$\P$--a.s.} agrees with the process of right-hand side limits of $Y$ $($see {\rm\cite[Appendix I, Remark 5(b)]{dellacherie1982probabilities}}$)$. For $\tau \in \cT_{0,T}$, we have $Y_{\tau+} - Y_\tau = -( K^\ell_{\tau \land T} - K^\ell_{{\tau \land T}-}) \leq 0$, {\rm$\P$--a.s.}, since $K^\ell$ is {$\P$--a.s.} non-decreasing, and therefore $Y_{\tau+} \leq Y_\tau$ up to a {\rm$\P$--null set}. This then also implies $Y_\tau \geq Y_{\tau+} \geq \xi_{\tau+} = \xi_\tau$ up to a {\rm$\P$--null set}. Following {\rm\cite[Remark 3.3]{grigorova2017reflected}}, we then find with \ref{cond::y_greater_xi} and \ref{cond::increasing_proc} that 
		\begin{align*}
			1 = \P[Y_\tau \geq \xi_\tau] = \P[Y_\tau > \xi_\tau] + \P[Y_\tau = \xi_\tau] &= \P[\Delta K^\ell_\tau = 0, Y_\tau > \xi_\tau] + \P[Y_\tau = Y_{\tau+},Y_\tau = \xi_\tau] \\
			&= \P[Y_{\tau+} = Y_\tau, Y_\tau > \xi_\tau] + \P[Y_{\tau+} = Y_\tau, Y_\tau = \xi_\tau] = \P[Y_\tau=Y_{\tau+}].
		\end{align*}
		Hence $Y_\tau = Y_{\tau+} = Y^{(+)}_\tau$, $\P${\rm--a.s.}, which implies that $Y = Y_{\cdot\land T}$ is $\P$--indistinguishable from $Y^{(+)}_{\cdot\land T}$ by the optional cross-section theorem in {\rm\cite[Theorem IV.84, p. 137]{dellacherie1978probabilities}} $($or see {\rm\Cref{prop::optional_ineq}}$)$. Incidentally, since $K^\ell$ is purely discontinuous and the source of right-hand side jumps of $Y$, this implies $K^\ell = 0$ up to a {\rm$\P$--null set}.
		
	\medskip
	$(viii)$ Instead of considering two predictable processes $(K^r,K^\ell)$ in the above formulation of the reflected {\rm BSDE}, we could simply consider a single $\P$--{\rm a.s.} non-decreasing, predictable process $K$  satisfying $K_0 = 0$ and $K = K_{\cdot \land T}$, $\P$--{\rm a.s.}, defined by $K = K^r + K^\ell_-$ up to $\P$--indistinguishability. The processes $(K,K^r,K^\ell)$ are then related to each other as follows
		\begin{equation*}
			K^r_t = K_t- K^\ell_{t-}, \; \text{\rm where} \; K^\ell_t = \sum_{s \in [0,t]} (K_{s+} - K_s), \; t \in [0,\infty], \; \text{$\P$\rm--a.s.},\;  \text{with} \; K^\ell_{0-} = 0.
		\end{equation*}
\end{remark}

\section{Main results}\label{sec::main_results}

This section contains the main results of our work. We present them first in the reflected BSDE and then in the (non-reflected) BSDE setting. Although the formulation we chose for reflected BSDEs includes BSDEs as special cases, it turns out that the \emph{a priori} estimates in \Cref{sec::a_priori} can be improved. We therefore report the results separately.

\subsection{Existence and uniqueness for reflected BSDEs}\label{sec::main_results_rbsde}

Before stating our main well-posedness result for reflected BSDEs, we introduce some notation to ease the presentation. For $(\Psi,\beta) \in [0,\infty) \times (0,\infty)$, let
\begin{align*}
	\ff^\Psi(\beta) &\coloneqq \inf_{\gamma\in (0,\beta)} \bigg\{\frac{(1+\beta\Psi)}{\gamma(\beta-\gamma)}\bigg\} = \frac{4(1+\beta\Psi)}{\beta^2},\\
	\fg^\Psi(\beta) &\coloneqq \inf_{\gamma \in (0,\beta)} \bigg\{ \frac{(1+\gamma\Psi)}{\gamma(\beta-\gamma)}\bigg\} = \frac{4}{\beta^2}\1_{\{\Psi = 0\}} + \frac{\Psi^2  \sqrt{1+\beta\Psi}}{\big(1+\beta\Psi -\sqrt{1+\beta\Psi}\big)\big(\sqrt{1+\beta\Psi}-1\big)}\1_{\{\Psi > 0\}}.
\end{align*}
Here the equalities follow from \Cref{lem::analysis_ff_fg}. We define
\begin{align*}
	M_1^\Psi(\beta) &\coloneqq \ff^\Psi(\beta) + \frac{4}{\beta} + \max\bigg\{1,\frac{(1+\beta\Psi)}{\beta}\bigg\}\bigg(\frac{5}{\beta} +\frac{4}{\beta}(1+\beta\Psi)^{1/2} + \beta \fg^\Psi(\beta)\bigg),\\ M_2^\Psi (\beta)&\coloneqq \ff^\Psi(\beta) + \bigg( \frac{5}{\beta} + \frac{4}{\beta}(1+\beta\Psi)^{1/2} + \beta \fg^\Psi(\beta) \bigg), \\
	M_3^\Psi (\beta) &\coloneqq \frac{4}{\beta} + \max\bigg\{1,\frac{(1+\beta\Psi)}{\beta}\bigg\}\bigg(\frac{5}{\beta} +\frac{4}{\beta}(1+\beta\Psi)^{1/2} + \beta \fg^\Psi(\beta)\bigg),\; \beta \in (0,\infty).
\end{align*}

The constants $M^\Phi_1(\beta)$, $M^\Phi_2(\beta)$ and $M^\Phi_3(\beta)$ will appear when we construct the contraction mappings on the weighted solution spaces of the reflected BSDE. Being able to keep them strictly less than one will allow us to use a fixed-point argument to deduce well-posedness. We now turn to the integrability conditions we need to impose on the obstacle $\xi$ and the generator $f$ to make our method of proof work, in particular, to ensure that the contraction mappings will be well-defined. Let $\brs{\xi} = (\brs{\xi}_t)_{t \in [0,\infty]}$ be the process defined by
\begin{equation}\label{eq::left_limit_xi}
	\brs{\xi}_0 \coloneqq 0, \; \text{and} \; \brs{\xi}_t \coloneqq \lim_{t^{\smalltext{\prime}} \uparrow\uparrow t}\bigg\{\sup_{s \in [t^\smalltext{\prime},\infty]} |\xi^+_s \mathbf{1}_{\{s < T\}}|\bigg\}, \; t \in (0,\infty],
\end{equation}
which is $\cG^U_\infty \otimes \cB([0,\infty])$-measurable by \cite[Proposition 2.21]{karoui2013capacities}. Although this process depends on the stopping time $T$, we suppress this to ease the notation. The following definition contains the main integrability condition that we impose.
\begin{definition}\label{def::standard_data}
The collection $(X,\mu,\G,T,\xi,f,C)$ is \emph{standard data for $\hat\beta \in [0,\infty)$}, if the pair $(\xi,f)$ satisfies
\begin{equation*}
		\|\xi_T\|_{\L^\smalltext{2}_\smalltext{\hat\beta}} + \|\alpha \brs{\xi}\|_{\H^\smalltext{2}_{\smalltext{T}\smalltext{,}\smalltext{\hat\beta}}} + \bigg\|\frac{f(0,0,0,\mathbf{0})}{\alpha}\bigg\|_{\H^\smalltext{2}_{\smalltext{T}\smalltext{,}\smalltext{\hat\beta}}}  < \infty.
	\end{equation*}
\end{definition}

\begin{remark}
	Our integrability assumption on $\xi$ is slightly different than the one imposed in {\rm\cite{grigorova2017reflected,grigorova2020optimal}}, which reads
		\begin{equation*}
		\E\bigg[{\esssup_{\tau \in \cT_{\smalltext{0}\smalltext{,}\smalltext{T}}}}^{\cG_\smalltext{T}} |\xi_\tau|^2\bigg] < \infty,
	\end{equation*}
	in a bounded horizon, Brownian--Poisson framework. Our assumption is not stronger than the one above, if we set ourselves into their framework. On the contrary, we actually only need to consider the positive part of $\xi$ on $[0,T)$, which is more general than the integrability condition in {\rm\cite{grigorova2017reflected,grigorova2020optimal}}. We will have a more thorough comparison with the literature in {\rm\Cref{sec::comparison_literature}}.
\end{remark}

Let us mention here a sufficient condition to have $\|\alpha \brs{\xi}\|_{\H^\smalltext{2}_{\smalltext{T}\smalltext{,}\smalltext{\hat\beta}}}$ finite. We defer the proof to \Cref{app::proofs_preliminaries}.

\begin{lemma}\label{lem::integrability_xi}
	Let $(\hat\beta,\beta^\star) \in [0,\infty)^2$ with $\hat\beta < \beta^\star$. Then
	\begin{equation*}
		\|\alpha \brs{\xi}\|_{\H^\smalltext{2}_{\smalltext{T}\smalltext{,}\smalltext{\hat\beta}}}^2 
		\leq \frac{(1+\beta^\star \Phi)(1+\hat\beta \Phi)}{(\beta^\star-\hat\beta)}\|\xi^+_\cdot \mathbf{1}_{\{\cdot < T\}}\|^2_{\cS^\smalltext{2}_{\smalltext{T}\smalltext{,}\smalltext{\beta^\star}}}.
	\end{equation*}
	In particular, if $\|\xi^+_\cdot \mathbf{1}_{\{\cdot < T\}}\|_{\cS^\smalltext{2}_{\smalltext{T}\smalltext{,}\smalltext{\beta^\star}}}$ is finite, then so is $\|\alpha \brs{\xi}\|_{\H^\smalltext{2}_{\smalltext{T}\smalltext{,}\smalltext{\hat\beta}}}$.
\end{lemma}

We now turn to our main well-posedness result for reflected BSDEs. It covers the case where the generator depends, additionally to $Z$ and $U(\cdot)$, on both $Y_s$ and $Y_{s-}$, just on $Y_s$ or just on $Y_{s-}$. The proof is deferred to \Cref{sec::existence_rbsde} as it is based on the optimal stopping and Snell envelope theory we revisit in \Cref{sec::optimal_stopping} and the \emph{a priori} estimates we establish in \Cref{sec::a_priori}.

\begin{theorem}\label{thm::main}
	Suppose that $(X,\G,T,\xi,f,C)$ is standard data for some $\hat\beta \in (0,\infty)$.	
	\begin{itemize}
	\item[$(i)$] If $M_1^\Phi(\hat\beta) < 1$, there exists a solution $(Y,Z,U,N,K^r,K^\ell)$ to the reflected {\rm BSDE} satisfying \ref{cond::integ_mart} up to \ref{cond::increasing_proc} such that $(Y,\alpha Y,\alpha Y_-,Z,U,N)$ is in $\cS^2_T \times \H^2_{T,\hat\beta} \times \H^2_{T,\hat\beta} \times \H^2_{T,\hat\beta}(X) \times \H^2_{T,\hat\beta}(\mu) \times \cH^{2,\perp}_{0,T,\hat\beta}(X,\mu)$. 
	
	\item[$(ii)$] If $M_2^\Phi(\hat\beta) < 1$ and $f$ does not depend on the component $Y_{s-}$, then there exists a solution $(Y,Z,U,N,K^r,K^\ell)$ to the reflected {\rm BSDE} satisfying \ref{cond::integ_mart} up to \ref{cond::increasing_proc} such that $(\alpha Y,Z,U,N)$ is in $\H^2_{T,\hat\beta} \times \H^2_{T,\hat\beta}(X) \times \H^2_{T,\hat\beta}(\mu) \times \cH^{2,\perp}_{0,T,\hat\beta}(X,\mu)$. Moreover, $Y$ is in $\cS^2_T$.
	
	\item[$(iii)$] If $M_3^\Phi(\hat\beta) < 1$ and $f$ does not depend on the component $Y_s$, then there exists a solution $(Y,Z,U,N,K^r,K^\ell)$ to the reflected {\rm BSDE} satisfying \ref{cond::integ_mart} up to \ref{cond::increasing_proc} such that $(Y,\alpha Y_-,Z,U,N)$ is in $\cS^2_T \times \H^2_{T,\hat\beta} \times \H^2_{T,\hat\beta}(X) \times \H^2_{T,\hat\beta}(\mu) \times \cH^{2,\perp}_{0,T,\hat\beta}(X,\mu)$. 
	\end{itemize}
	In all three cases, the triple $(Y,K^r,K^\ell)$ is unique up to $\P$--indistinguishability and $(Z,U,N)$ is unique in $\H^2_{T,\hat\beta}(X) \times \H^2_{T,\hat\beta}(\mu) \times \cH^{2,\perp}_{0,T,\hat\beta}(X,\mu)$. Furthermore, \ref{cond::repre_Y} holds. 
	If additionally, $\xi^+\mathbf{1}_{[0,T)} \in \cS^2_{T,\beta}$ for some $\beta \in (0,\hat\beta)$, then $(K^r,K^\ell) \in \cI^2_{T,\beta} \times \cI^2_{T,\beta}$.
\end{theorem}

\begin{remark}\label{rem::reflection_stopping_norm}
	The fixed-point argument used in the proof of {\rm\Cref{thm::main}} relies on the usage of the $\cS^2_T$--norm in cases $(i)$ and $(iii)$. However, one can use the following alternative norm\footnote{For more details on this norm, see \cite[IV.21, pp. 82--83]{dellacherie1982probabilities}.}
	\begin{equation}\label{eq::def_norm_tau}
		\|Y\|^2_{\cT^\smalltext{2}_\smalltext{T}} \coloneqq \sup_{\tau \in \cT_{\smalltext{0}\smalltext{,}\smalltext{T}}}\E\big[|Y_\tau|^2\big]
	\end{equation}
	in the contraction argument.
	The statements of {\rm\Cref{thm::main}} and {\rm\Cref{cor::main_result_reflected_bsde}} remain unchanged, except that we would be able to replace the constants $M_1^\Phi(\hat\beta)$ and $M_3^\Phi (\hat\beta)$ by
	\begin{align*}
	M_1^\Phi(\hat\beta) = \ff^\Phi(\hat\beta) + \frac{1}{\hat\beta} + \max\bigg\{1,\frac{(1+\hat\beta\Phi)}{\hat\beta}\bigg\}\bigg(\frac{5}{\hat\beta} +\frac{4}{\hat\beta}(1+\hat\beta\Phi)^{1/2} + \hat\beta \fg^\Phi(\hat\beta)\bigg)
	\end{align*}
	and
	\begin{align*}
	M_3^\Phi (\hat\beta) = \frac{1}{\hat\beta} + \max\bigg\{1,\frac{(1+\hat\beta\Phi)}{\hat\beta}\bigg\}\bigg(\frac{5}{\hat\beta} +\frac{4}{\hat\beta}(1+\hat\beta\Phi)^{1/2} + \hat\beta \fg^\Phi(\hat\beta)\bigg),
\end{align*}
	respectively. For the {\rm BSDE} case, see {\rm\Cref{rem::bsde_stopping_norm}}. We refer to {\rm\Cref{rem::apriori_stopping_norm}} and {\rm\ref{rem::proof_main_stopping_norm}} for more details.
\end{remark}

\Cref{thm::main} and the analysis of the contraction constants in \Cref{lem::contraction_constants_reflected_bsde} yield the following immediate result.

\begin{corollary}\label{cor::main_result_reflected_bsde}
	Suppose that $\Phi < 1$. For each $i \in \{1,2,3\}$, there exists $\beta^\star_i \in (0,\infty)$ such that $M^\Phi_i(\hat\beta) <1$ for every $\hat\beta \in (\beta^\star_i,\infty)$. Moreover
	\begin{itemize}
		\item[$(i)$] if $(X,\G,T,\xi,f,C)$ is standard data for $\hat\beta \in (\beta^\star_1,\infty)$, then {\rm\Cref{thm::main}.$(i)$} holds$;$
		\item[$(ii)$] if $(X,\G,T,\xi,f,C)$ is standard data for $\hat\beta \in (\beta^\star_2,\infty)$, then {\rm\Cref{thm::main}.$(ii)$} holds$;$
		\item[$(iii)$] if $(X,\G,T,\xi,f,C)$ is standard data for $\hat\beta \in (\beta^\star_3,\infty)$, then {\rm\Cref{thm::main}.$(iii)$} holds.
	\end{itemize}
\end{corollary}

\subsection{Existence and uniqueness for BSDEs}\label{sec::main_results_bsde}

We now discuss the existence and uniqueness of the non-reflected BSDE with generator $f$ and terminal condition $\xi_T$. More precisely, we look for a unique quadruple $(Y,Z,U,N)$ within a class of processes satisfying
\begin{enumerate}[{\bf (B1)}, leftmargin=0.9cm]
	\item\label{BSDE1} $(Z,U,N) \in \H^2_T(X) \times \H^2_T(\mu) \times \cH^{2,\perp}_{0,T}(X,\mu)$;
	\item\label{BSDE2} $Y = (Y_t)_{t \in [0,\infty]}$ is optional with $\P$--a.s. l\`adl\`ag paths,\footnote{Note that $Y$ is then \emph{a fortiori} $\mathbb{P}$--a.s. c\`adl\`ag.}
		\begin{equation*}
			\E\bigg[\int_0^T \big|f_s\big(Y_s,Y_{s-},Z_s,U_s(\cdot)\big)\big|\d C_s\bigg] < \infty,
		\end{equation*}
		\begin{equation*}
			Y_t = \xi_T + \int_t^T f_s\big(Y_s,Y_{s-},Z_s,U_s(\cdot)\big) \d C_s - \int_t^T Z_s \d X_s - \int_t^T \int_E U_s(x)\tilde\mu(\d s, \d x) - \int_t^T \d N_s, \; t \in [0,\infty], \; \text{$\P$--a.s.}
		\end{equation*}
\end{enumerate}
To deduce existence and uniqueness within a class of processes of the BSDE above, one could just redefine the obstacle $\xi$ to be $-\infty$ on $[0,T)$, note that then $K^r = 0$ and $K^\ell_{-} = 0$, the representation \ref{cond::repre_Y} turns into
\begin{equation*}
	Y_S = \E\bigg[\zeta + \int_S^T f_s\big(Y_s,Y_{s-},Z_s,U_s(\cdot)\big)\d C_s\bigg| \cG_S\bigg], \; \text{$\P$--a.s.}, \; S \in \cT_{0,\infty},
\end{equation*}
and then refer to \Cref{thm::main} for the conditions that provide existence and uniqueness in case there exists $\hat\beta \in (0,\infty)$ with
\begin{equation}\label{eq::integrability_bsde}
	\|\xi_T\|_{\L^\smalltext{2}_\smalltext{\hat\beta}} + \bigg\|\frac{f(0,0,0,\mathbf{0})}{\alpha}\bigg\|_{\H^\smalltext{2}_{\smalltext{T}\smalltext{,}\smalltext{\hat\beta}}} < \infty.
\end{equation}

However, it is worthwhile to redo the \emph{a priori} estimates in \Cref{sec::a_priori} in this case as the contraction constants improve significantly. Let us also emphasise here that the techniques we employ to establish well-posedness for BSDEs do not depend on It\^o's formula, and extending them to BSDEs with a multi-dimensional generator and terminal condition is straightforward. The constants we want to control in the contraction argument to prove well-posedness of \ref{BSDE1}--\ref{BSDE2} become
\begin{align*}\label{eq::tilde_pi_psi_1} 
	\widetilde M_1^\Psi(\beta) \coloneqq \ff^\Psi(\beta) + \frac{4}{\beta} + \max\bigg\{1,\frac{(1+\beta\Psi)}\beta\bigg\}\bigg(\frac{1}{\beta} + \beta \fg^\Psi(\beta)\bigg),\;	\widetilde M_2^\Psi(\beta) \coloneqq \ff^\Psi(\beta) + \bigg(\frac{1}{\beta} + \beta \fg^\Psi(\beta)\bigg),
\end{align*}
\begin{equation*}
	\widetilde M_3^\Psi(\beta) \coloneqq \frac{4}{\beta} + \max\bigg\{1,\frac{(1+\beta\Psi)}\beta\bigg\}\bigg(\frac{1}{\beta} + \beta \fg^\Psi(\beta)\bigg).
\end{equation*}

The following is our main well-posedness result for BSDEs. We defer its proof to the end of \Cref{sec::existence_rbsde}.

\begin{theorem}\label{thm::main_bsde}
	Suppose that $(X,\G,T,\xi,f,C)$ is standard data for some $\hat\beta \in (0,\infty)$ and $\xi = -\infty$ on $[0,T)$.
	
	\begin{itemize}
	\item[$(i)$] If $\widetilde M_1^\Phi(\hat\beta) < 1$, then there exists a solution $(Y,Z,U,N)$ to the {\rm BSDE} \ref{BSDE1}--\ref{BSDE2} such that $(Y,\alpha Y,\alpha Y_-,Z,U,N)$ is in $\cS^2_T \times \H^2_{T,\hat\beta} \times \H^2_{T,\hat\beta} \times \H^2_{T,\hat\beta}(X) \times \H^2_{T,\hat\beta}(\mu) \times \cH^{2,\perp}_{0,T,\hat\beta}(X,\mu)$. 

	\item[$(ii)$] If $\widetilde M_2^\Phi(\hat\beta) < 1$ and the generator $f$ does not depend on $Y_{s-}$, then there exists a solution $(Y,Z,U,N)$ to the {\rm BSDE} \ref{BSDE1}--\ref{BSDE2} such that $(\alpha Y,Z,U,N)$ is in $\H^2_{T,\hat\beta} \times \H^2_{T,\hat\beta}(X) \times \H^2_{T,\hat\beta}(\mu) \times \cH^{2,\perp}_{0,T,\hat\beta}(X,\mu)$. 

	\item[$(iii)$] If $\widetilde M_3^\Phi(\hat\beta) < 1$ and the generator $f$ does not depend on $Y_s$, then there exists a solution $(Y,Z,U,N)$ to the {\rm BSDE} \ref{BSDE1}--\ref{BSDE2} such that $(Y,\alpha Y_-,Z,U,N)$ is in $\cS^2_T \times \H^2_{T,\hat\beta} \times \H^2_{T,\hat\beta}(X) \times \H^2_{T,\hat\beta}(\mu) \times \cH^{2,\perp}_{0,T,\hat\beta}(X,\mu)$.
	\end{itemize}
	
	In all three cases, $Y$ is in $\cS^2_{T,\hat\beta}$ and unique up to $\P$--indistinguishability, and $(Z,U,N)$ is unique in $\H^2_{T,\hat\beta}(X) \times \H^2_{T,\hat\beta}(\mu) \times \cH^{2,\perp}_{0,T,\hat\beta}(X,\mu)$.
\end{theorem}

\begin{remark}\label{rem::bsde_stopping_norm}
	Similar to the discussion in \Cref{rem::reflection_stopping_norm}, we can also replace here in the contraction argument the norm $\|\cdot\|_{\cS^\smalltext{2}_\smalltext{T}}$ by the norm $\|\cdot\|_{\cT^\smalltext{2}_\smalltext{T}}$ introduced in \eqref{eq::def_norm_tau}. Incidentally, we would be able to replace the constants $\widetilde M_1^\Phi(\hat\beta)$ and $\widetilde M_3^\Phi(\hat\beta)$ in the statement of {\rm\Cref{thm::main_bsde}} and {\rm\Cref{cor::main_result_bsde}} by
	\begin{align*}
	\widetilde M_1^\Phi(\hat\beta) = \ff^\Phi(\hat\beta) + \frac{1}{\hat\beta} + \max\bigg\{1,\frac{1+\hat\beta\Phi}{\hat\beta}\bigg\}\bigg(\frac{1}{\hat\beta} + \hat\beta \fg^\Phi(\hat\beta)\bigg)
\end{align*}
and
\begin{equation*}
	\widetilde M_3^\Phi(\hat\beta) = \frac{1}{\hat\beta} + \max\bigg\{1,\frac{1+\hat\beta\Phi}{\hat\beta}\bigg\}\bigg(\frac{1}{\hat\beta} + \hat\beta \fg^\Phi(\hat\beta)\bigg),
\end{equation*}
respectively. We refer to {\rm\Cref{rem::apriori_stopping_norm}} and {\rm\ref{rem::proof_main_stopping_norm}} for more details.
\end{remark}

Combining \Cref{thm::main_bsde} with the analysis of the contraction constants in \Cref{lem::contraction_constants_bsde}, we find the following.

\begin{corollary}\label{cor::main_result_bsde}
Suppose that $\Phi < 1$ and that $\xi = -\infty$ on $[0,T)$. For each $i \in \{1,2,3\}$, there exists $\beta^\star_i \in (0,\infty)$ such that $\widetilde M^\Phi_i(\hat\beta) <1$ for every $\hat\beta \in (\beta^\star_i,\infty)$. Moreover
	\begin{itemize}
		\item[$(i)$] if $(X,\G,T,\xi,f,C)$ is standard data for $\hat\beta \in (\beta^\star_1,\infty)$, then {\rm\Cref{thm::main_bsde}.$(i)$} holds$;$
		\item[$(ii)$] if $(X,\G,T,\xi,f,C)$ is standard data for $\hat\beta \in (\beta^\star_2,\infty)$, then {\rm\Cref{thm::main_bsde}.$(ii)$} holds$;$
		\item[$(iii)$] if $(X,\G,T,\xi,f,C)$ is standard data for $\hat\beta \in (\beta^\star_3,\infty)$, then {\rm\Cref{thm::main_bsde}.$(iii)$} holds.
	\end{itemize}
\end{corollary}

\begin{remark}\label{rem::main_bsde}
	$(i)$ Here, the condition we need to impose on $\Phi$ to get well-posedness for sufficiently integrable data is weaker than the claimed condition $\Phi < 1/(18\mathrm{e})$ in {\rm\cite[Corollary 3.6]{papapantoleon2018existence}} in case the generator depends on $\big(Y_s,Z_s,U_s(\cdot)\big)$. There is actually a slight issue that occurs when deriving the \emph{a priori} estimates in {\rm\cite{papapantoleon2018existence}}. By correcting the issue appearing in the derivation of the \emph{a priori} estimates, it turns out that the constant $\Phi$ needs to be even lower than $1/(18\mathrm{e})$ for the contraction argument to go through. We will come back to this in {\rm\Cref{rem::discussion_constants}}.
	
	\medskip
	$(ii)$ The second well-posedness result in {\rm\cite[Theorem 3.23]{papapantoleon2018existence}} relies on the generator depending on $\big(Y_{s-},Z_s,U_s(\cdot)\big)$. The proof is also based on a fixed-point argument, but the \emph{a priori} estimates are derived by an application of It\^o's formula and not by the more direct approach as in the first part of {\rm \cite{papapantoleon2018existence}}. Although the integrability condition {\bf(H4)} imposed on $f$ in {\rm\cite{papapantoleon2018existence}} under which well-posedness is established is not comparable to our integrability condition \eqref{eq::integrability_bsde}, their result relies on the more restrictive conditions {\bf(H5)} and {\bf(H6)} imposed on the integrator $C$ and the driving martingale $X$ in {\rm\cite{papapantoleon2018existence}}. 
	
	\medskip
	$(iii)$ It turns out that within the framework we are working in, the condition $\Phi < 1$ to establish well-posedness of  \emph{BSDEs} with jumps is not only sufficient, but also necessary in the following sense: if we ask ourselves whether a condition of the form $\Phi < a$ for some $a \in (1,\infty)$ would still allow us to have a general well-posedness result, then the is answer is no as can be seen by the counterexamples to existence and uniqueness established in {\rm\citeauthor*{confortola2014backward} \cite[Section 4.3]{confortola2014backward}}. See also the discussion in {\rm\citeauthor*{papapantoleon2018existence} \cite[Section 3.3.1]{papapantoleon2018existence}}. 
	
	\medskip
	$(iv)$ An extension to $d$-dimensional {\rm BSDEs}, for $d \in \N$, is straightforward as we will never use It\^o's formula to derive the \emph{a priori} estimates in the {\rm BSDE} case. In this case, the generator $f = (f^1,\ldots,f^d)$ is $\R^d$-valued and the system of {\rm BSDEs} will then take the form
		\begin{align*}
			Y^i_t = \xi^i_T + \int_t^T f^i_s\big(Y_s, Y_{s-}, Z_s, U_s(\cdot)\big) \d C_s - \int_t^T Z^i_s \d X_s - \int_t^T\int_E U^i_s(x)(\mu-\mu^p)(\d s,\d x) - \int_t^T N^i_s,\; i\in\{1,\dots,d\}, 
		\end{align*}
		where $Y = (Y^1,\ldots,Y^d)^\top$, $Z = (Z^1,\ldots,Z^d)^\top$, $U = (U^1,\ldots,U^d)^\top$ and $N = (N^1,\ldots,N^d)^\top$. To adapt our method of proof, we need to replace in \ref{data::generator}, in the definition of the weighted norms of {\rm\Cref{sec::weighted_spaces}}, and in the proof of {\rm\Cref{thm::main_bsde}} the absolute value $|\cdot|$ by the Euclidean norm $\|\cdot\|_{\R^\smalltext{d}}$. All computations that we will carry out in {\rm\Cref{sec::a_priori}} in the {\rm BSDE} case and in the proof of {\rm\Cref{thm::main_bsde}} will still go through.
\end{remark}

\subsection{Comparison with the literature and some consequences}\label{sec::comparison_literature}

In this part we compare our well-posedness results with other results in the literature. We are mostly interested in a comparison of the integrability conditions imposed on the data, and, in case the integrator $C$ jumps, whether some a condition similar to our condition $\alpha\Delta C \leq \Phi \in [0,1)$ is needed to ensure well-posedness. However, we restrict ourselves to works which are closest to ours, although there are far more well-posedness results out there. In particular, all well-posedness results we mention in this part, except one, consider $\L^2$-data and Lipschitz-continuous generators.

\subsubsection{When the obstacle is predictable}

The case of a predictable obstacle process $\xi$ together with a notion of `predictable reflected BSDE' was studied by \citeauthor*{bouhadou2018non} \cite{bouhadou2018non}. While we do allow, of course, for an obstacle $\xi$ that is merely predictable, the solution to the reflected BSDE in \cite{bouhadou2018non} consists of predictable processes. Specifically, the processes $Y$ in \cite{bouhadou2018non} is predictable. Their study closely relies on the theory of predictable strong supermartingales by \citeauthor*{meyer1976cours} \cite[page 388]{meyer1976cours} and the corresponding predictable Snell envelopes by \citeauthor*{karoui1978arret} \cite{karoui1978arret}. Although we do not cover the well-posedness result of \cite{bouhadou2018non} by simply taking predictable projections of our solution processes, it would be intriguing to explore whether this can be achieved by our techniques in this work in combination with the results in \cite{karoui1978arret} and \cite{meyer1976cours}.

\subsubsection{When Wiener meets Poisson}

We start with comparing our results in the reflected BSDE case to \citeauthor*{grigorova2017reflected} \cite{grigorova2017reflected} and \citeauthor*{grigorova2020optimal} \cite{grigorova2020optimal}, in which they show well-posedness of bounded horizon reflected BSDEs in a Brownian--Poisson framework whose obstacle, as in our case, is merely an optional process. Let us translate their setup into ours to see that we cover their well-posedness result. Let $\d C_s = \d s$, let $T$ be a deterministic and finite time horizon, let $X$ be a Brownian motion and let $\mu$ be a Poisson random measure, that is, the predictable compensator $\nu$ of $\mu$ disintegrates as $\nu(\d s, \d x) = F(\d x) \d s$ for some $\sigma$-finite measure $F : (E,\cE) \longrightarrow [0,\infty]$, see also \cite{jacod1979calcul,jacod2003limit}. Since $C$ is continuous, the dependence on $Y_s$ or $Y_{s-}$ in the generator does not matter, so we drop one of the arguments in the generator that involves $Y$. Suppose that $f$ has a universal and deterministic Lipschitz coefficient $\alpha \in (0,\infty)$. The integrability condition on the obstacle $\xi$ and the generator $f$ in \cite{grigorova2017reflected,grigorova2020optimal} that ensures well-posedness is
\begin{equation*}
	\E\bigg[{\esssup_{\tau \in \cT_{\smalltext{0}\smalltext{,}\smalltext{T}}}}^{\cG_\smalltext{T}}|\xi_\tau|^2\bigg] + \E \bigg[ \int_0^T |f_s(0,0,\mathbf{0})|^2 \d s \bigg] < \infty.
\end{equation*}
We now show that our integrability condition in \Cref{def::standard_data} is satisfied for any $\hat\beta \in (0,\infty)$, although it looks rather different at first sight. Note that we can choose $\Phi = 0$ since the integrator $C$, and thus $A$, never jumps. Moreover, $\cE(\hat\beta A) = \mathrm{e}^{\hat\beta A}$, and since $1 \leq \mathrm{e}^{\hat\beta A} \leq \mathrm{e}^{\hat\beta\alpha^2T}$, we immediately find
\begin{equation*}
	\E\bigg[\int_0^T \mathrm{e}^{\hat\beta A_s} \frac{|f_s(0,0,\mathbf{0})|^2}{\alpha^2_s}\d C_s\bigg] \leq \frac{\mathrm{e}^{\hat\beta \alpha^\smalltext{2}T}}{\alpha^2} \E \bigg[\int_0^T |f_s(0,0,\mathbf{0})|^2\mathrm{d}s\bigg] < \infty,
\end{equation*}
\begin{equation*}
	\E\bigg[ \int_0^T \mathrm{e}^{\hat\beta A_s} |\!\brs{\xi}_s|^2 \d A_s \bigg]
	\leq \alpha^2\mathrm{e}^{\hat\beta \alpha^\smalltext{2} T}\E\bigg[ \int_0^T \sup_{u \in [s,T]}|\xi_u|^2 \d s \bigg] 
	\leq \alpha^2T\mathrm{e}^{\hat\beta \alpha^\smalltext{2} T}\E\bigg[ \sup_{u \in [0,T]}|\xi_u|^2 \bigg]  
	\leq 4\alpha^2T\mathrm{e}^{\hat\beta \alpha^\smalltext{2} T}\E \bigg[ {\esssup_{\tau \in \cT_{\smalltext{0}\smalltext{,}\smalltext{T}}}}^{\cG_\smalltext{T}} |\xi_\tau|^2 \bigg] < \infty,
\end{equation*}
for any $\hat\beta \in (0,\infty)$. Here the last inequality follows by an application of \Cref{prop::equiv_integ_xi}. Furthermore, $\H^2_{T} = \H^2_{T,\hat\beta}$, $\H^2_{T}(X) = \H^2_{T,\hat\beta}(X)$, $\H^2_{T}(\mu) = \H^2_{T,\hat\beta}(\mu)$ and $\cH^2_{T}(X,\mu) = \cH^{2,\perp}_{0,T,\hat\beta}(X,\mu)$ since $\mathrm{e}^{\hat\beta A}$ is bounded. For $\hat\beta$ large enough, we have $M^\Phi_i(\hat\beta) < 1$ for $i \in \{1,2,3\}$, and thus our \Cref{thm::main} provides well-posedness of the reflected BSDE considered in \cite{grigorova2017reflected,grigorova2020optimal}. However, we want to mention that the class of processes in which uniqueness holds in \cite{grigorova2020optimal} is not completely clear, at least to us. We will discuss this further in \Cref{rem::missing_pieces}.

\subsubsection{When the random measure is a marked point process}

We would like to draw attention to the recent work of \citeauthor*{foresta2021optimal} \cite{foresta2021optimal} on reflected BSDEs driven by Brownian motion $X$ and a marked point process $\mu$, that is, $\mu$ is an integer-valued measure such that
\begin{equation*}
	\mu(\omega; \d t, \d x) = \sum_{n \in \N} \mathbf{1}_{\{T_n < \infty\}}(\omega)\boldsymbol{\delta}_{(T_\smalltext{n}(\omega),\varrho_{\smalltext{T}_\tinytext{n}\smalltext{(}\smalltext{\omega}\smalltext{)}}(\omega))}(\d t, \d x),
\end{equation*}
for a sequence of stopping times $(T_n)_{n \in \N}$ that satisfies $T_n \leq T_{n+1}$, $\P$--a.s., and $T_n < T_{n+1}$, $\P$--a.s. on $\{T_n < \infty\}$. The well-posedness result of the reflected BSDE considered in \cite[Theorem 4.1]{foresta2021optimal} can be covered by our \Cref{thm::main} in case of sufficiently integrable data. We start by translating their setup into our notation. Let $C_t = t + C^\prime_t$, where $C^\prime$ is some continuous process for which we can write the predictable compensator $\nu$ of $\mu$ as $\nu(\d s,\d x) = K^\prime_s(\omega;\d x) \d C^\prime_s$. Since $A$ is continuous, we can choose $\Phi = 0$, which then implies that $M^\Phi(\beta) \longrightarrow 0$ for $\beta \longrightarrow \infty$. Moreover, the stochastic exponential weight $\cE(\hat\beta A)$ reduces to $\mathrm{e}^{\hat\beta A}$. Suppose that the generator is of the form
\begin{equation*}
	f_s\big(\omega,y,\mathrm{y},z,u_s(\omega;\cdot)\big) =f^1_s(\omega,y,z) + f^2_s\big(\omega,y,u_s(\omega;\cdot)\big)\bigg(1- \frac{\d s}{\d C_s}\bigg),
\end{equation*}
 with deterministic and time-independent Lipschitz coefficients, so $\alpha \in (0,\infty)$.
The integrability condition in \cite{foresta2021optimal} reads
\begin{align}\label{eq::integ_foresta}
	\E\bigg[ \mathrm{e}^{\beta C^\prime_T} |\xi_T|^2\bigg] + \E\bigg[\sup_{s \in [0,T)} \mathrm{e}^{(\beta +\delta)C^\prime_s} |\xi_s|^2\bigg] + \E\bigg[\int_0^T \mathrm{e}^{\beta C^\prime_s}|f^1_s(0,0)|^2 \d s\bigg] + \E\bigg[\int_0^T \mathrm{e}^{\beta C^\prime_s}|f^2_s(0,\mathbf{0})|^2 \d C^\prime_s\bigg] <\infty.
\end{align}
It is then straightforward to check that for $\hat\beta = \beta/\alpha^2$, which then also satisfies $\hat\beta A = \hat\beta \alpha^2 C = \beta C$, we have $\|\xi_T\|_{\L^\smalltext{2}_\smalltext{\hat\beta}} + \big\|\frac{f(0,0,0,\mathbf{0})}{\alpha}\big\|_{\H^\smalltext{2}_{\smalltext{T}\smalltext{,}\smalltext{\hat\beta}}}  < \infty.$ That $\|\alpha \brs{\xi}\|_{\H^\smalltext{2}_{\smalltext{T}\smalltext{,}\smalltext{\hat\beta}}}$ is finite follows from \Cref{lem::integrability_xi}. We thus conclude that if \eqref{eq::integ_foresta} holds for $\beta$, our \Cref{thm::main} provides well-posedness of the reflected BSDE considered in \cite{foresta2021optimal}.

\subsubsection{When the horizon is finite but random}

Inspired by applications to random horizon principal--agent problems, \citeauthor*{lin2020second} \cite{lin2020second} proved well-posedness of random horizon BSDEs, 2BSDEs and reflected BSDEs. Although we cannot cover their results in general as they work with $\L^p$-data for $p > 1$ and the novel norms they use do not fit with our setup, we can compare to some extent their well-posedness result in the reflected BSDE case for $p = 2$. In our notation, the setup studied in \cite{lin2020second} is the following: there is no integer-valued random measure $\mu$, the process $X$ is a Brownian motion, the obstacle process $\xi$ is optional and c\`adl\`ag, the stopping time $T$ is finite, the integrator $C$ satisfies $\d C_s = \d s$, the Lipschitz-continuous generator $f$ has deterministic and time-independent Lipschitz coefficients, thus $\alpha \in (0,\infty)$, and moreover, $f$ is monotone in the $y$-variable, see \cite[Assumption 3.1.$(ii)$]{lin2020second}. As before, we drop one of the arguments in the generator that depends on $Y$. Note that we can choose $\Phi = 0$, so that $M^\Phi_2(\beta) \longrightarrow 0$ as $\beta \longrightarrow\infty$. This ensures that we can provide well-posedness for data that is sufficiently integrable. We now turn to the integrability condition in \cite{lin2020second} that provides well-posedness in an $\L^2$-setting. Let $\cQ_\alpha(\P)$ be the collection of probability measures $\Q^\lambda$ on $(\Omega,\cG)$ that satisfies
\begin{equation*}
	\frac{\d \Q^\lambda|_{\cG_\smalltext{t}}}{\d \P|_{\cG_\smalltext{t}}} = \cE\bigg(\int_0^\cdot \lambda_s \d X_s\bigg)_t, \; t \in [0,\infty),
\end{equation*}
for some predictable process $\lambda = (\lambda_s)_{s \in [0,\infty)}$ with $|\lambda_s| \leq \alpha$. Suppose now that there exists $\varepsilon \in (0,\infty)$ and $\hat\beta \in (0,\infty)$ sufficiently large such that 
\begin{equation*}
	\sup_{\Q \in \cQ_\smalltext{\alpha}(\P)}\E^\Q\Big[|\mathrm{e}^{\frac{\hat\beta}{2}\alpha^\smalltext{2} T}\xi_T|^{2+\varepsilon}\Big] + \sup_{\Q \in \cQ_\smalltext{\alpha}(\P)}\E^\Q\bigg[\sup_{s \in [0,\infty)}|\mathrm{e}^{\frac{\hat\beta}{2}\alpha^\smalltext{2} T}\xi^+_{s \land T}|^{2+\varepsilon}\bigg] + \sup_{\Q \in \cQ_\smalltext{\alpha}(\P)}\E^\Q\bigg[\bigg(\int_0^T\mathrm{e}^{\hat\beta\alpha^\smalltext{2} s}\frac{|f_s(0,0)|^2}{\alpha^2}\d s\bigg)^{\frac{2+\varepsilon}{2}}\bigg] < \infty.
\end{equation*}
An argument similar to the one in the previous paragraph shows that $\|\alpha\brs{\xi}\|_{\H^\smalltext{2}_{\smalltext{T}\smalltext{,}\smalltext{\hat\beta}}} < \infty$ as $\P \in \cQ_\alpha(\P)$. Suppose now that $\hat\beta$ is large enough, so that the conditions in \cite[Theorem 3.9]{lin2020second} and in our \Cref{thm::main_bsde}.$(ii)$ are satisfied. Fix $\beta < \hat\beta$ such that the difference $\hat\beta-\beta$ is small. By \Cref{thm::main_bsde}.$(ii)$, there exists a unique solution $(Y,Z,N,K^r,K^\ell)$ to \ref{cond::integ_mart}--\ref{cond::repre_Y} such that $(\alpha Y,Z,N) \in \H^2_{T,\beta} \times \H^2_{T,\beta}(X) \times \cH^{2,\perp}_{T,\beta}(X)$. Recall here that due to the obstacle $\xi$ having c\`adl\`ag paths, the process $K^\ell$ vanishes, see \Cref{rem::after_formulation}.$(vii)$, we thus write $K \coloneqq K^r$. Let us now compare our solution to the one constructed in \cite{lin2020second}. The solution $(Y^\prime,Z^\prime,N^\prime,K^\prime)$ constructed using Theorem 3.4 in \cite{lin2020second} satisfies \ref{cond::integ_mart}--\ref{cond::increasing_proc}. Moreover, we also have $(Y^\prime,Z^\prime,N^\prime) \in \cS^2_T \times \H^2_{T,\beta}(X) \times \cH^{2,\perp}_{T,\beta}(X)$, and also $Y^\prime \in \cS^2_{T,\beta^\smalltext{\dag}}$ for each $\beta^\dag \in (\beta,\hat\beta)$. The latter property implies that $\alpha Y^\prime \in \H^2_{T,\beta}$ since
\begin{align*}
	\|\alpha Y^\prime\|^2_{T,\beta} = \E\bigg[\int_0^T \mathrm{e}^{\beta A_s} |Y^\prime_s|^2 \d A_s\bigg] &= \E\bigg[\sup_{s \in [0,T]}|\mathrm{e}^{\frac{\beta^\smalltext{\dag}}{2} A_s}Y^\prime_s|^2 \int_0^T \mathrm{e}^{(\beta-\beta^\dag)A_s} \d A_s\bigg] \\
	&\leq \E\bigg[\sup_{s \in [0,T]}|\mathrm{e}^{\frac{\beta^\smalltext{\dag}}{2} A_s}Y^\prime_s|^2 \int_0^\infty \mathrm{e}^{(\beta-\beta^\dag)s} \d s\bigg] = \frac{1}{(\beta-\beta^\dag)} \|Y^\prime\|^2_{\cS^\smalltext{2}_{\smalltext{T}\smalltext{,}\smalltext{\beta^\dag}}}.
\end{align*}
With our uniqueness statement in \Cref{thm::main}.$(ii)$, we conclude that $(Y,K) = (Y^\prime,K^\prime)$, up to $\P$--indistinguishability, and that $(Z,N) = (Z^\prime,N^\prime)$ in $\H^2_{T,\beta}(X) \times \cH^{2,\perp}_{T,\beta}(X)$. Since each $\Q \in \cQ_\alpha(\P)$ is locally absolutely continuous with respect to $\P$, it is straightforward to check that our solution $(Y,Z,N,K)$ also coincides with $(Y^\prime,Z^\prime,N^\prime,K^\prime)$ with respect to the norm used in \cite{lin2020second}. We thus conclude that our solution is in their solution space.

\subsubsection{When the generator has stochastic Lipschitz coefficients}

\citeauthor*{perninge2021note} \cite{perninge2021note} also studies reflected BSDEs in a Brownian setting on an infinite horizon, and with stochastic Lipschitz coefficients, where the $Y$-component will converge to zero at infinity. The stochasticity in the Lipschitz coefficient actually only appears in $Z$-component of the generator. Here the integrability conditions imposed on the data are not comparable to ours. However, the range of applications seem to be more restrictive than in our setup for the following reason: in our notation, the process $X$ is a Brownian motion, and the stochastic Lipschitz coefficient $\theta^X$ in \cite{perninge2021note} is supposed to be an adapted and continuous process and should satisfy $\E[\cE(\int_0^\cdot \zeta_s\d X_s)_t] = 1$ for each $t \in [0,\infty)$ and optional process $\zeta$ satisfying $|\zeta|^2 \leq \theta^X$. In particular, the process $\sqrt{\theta^X}$ itself should satisfy this condition \emph{a fortiori}. Thus the well-posedness result in \cite{perninge2021note} is not for arbitrary Lipschitz generators with stochastic Lipschitz coefficients.

\subsubsection{When the non-reflected BSDE is driven by arbitrary martingales}

Let us close this section by a comparison of our BSDE results to the works of \citeauthor*{bandini2015existence} \cite{bandini2015existence}, \citeauthor*{cohen2012existence} \cite{cohen2012existence} and \citeauthor*{papapantoleon2018existence} \cite{papapantoleon2018existence}, as we feel that these works are closest to the BSDE formulation we chose here. In \cite{bandini2015existence}, the integrability condition imposed on the data is not comparable to ours. We can thus not cover the well-posedness result in \cite{bandini2015existence} in general. What is surprising nonetheless is that by translating \cite{bandini2015existence} into our notation, we see that the condition $\sqrt{r_s}\Delta C_s \leq \Phi \in [0,1/\sqrt{2})$, which only involves the Lipschitz coefficient of $f$ with respect to the $Y$-component, is sufficient for the contraction argument to go through in case of sufficiently integrable data, see \cite[Theorem 4.1]{bandini2015existence}. However, the setup in \cite{bandini2015existence} is simpler than the one we study here, as the only driving force in the BSDE is a random measure $\mu$ with finite activity, that is, $\{t \in [0,t^\prime] : \mu(\omega; \{t\} \times E) = 1\}$ is finite for each $(\omega,t^\prime) \in \Omega \times [0,\infty)$. 

\medskip
The BSDE considered in \cite{cohen2012existence} is rather different from ours. We initially fix a driving martingale $X$ and an integrator $C$ so that $\d\langle X \rangle_s$ is absolutely continuous with respect to $\d C_s$. Contrary to our case, in \cite{cohen2012existence} the integrator $C$ is fixed in the beginning, and the driving martingales of the BSDE are a sequence of orthogonal martingales that are constructed from a general martingale representation theorem relying on the assumption that the underlying probability space is separable. In \cite{cohen2012existence}, the integrator $C$ may have no relation at all to the the predictable quadratic variations of the driving martingales. We are thus not able to link our well-posedness result to theirs. However, they suppose that $C$ is deterministic and strictly increasing, which immediately excludes piecewise constant integrators. The condition ensuring well-posedness in case of sufficiently integrable data is similar to \cite{bandini2015existence}, namely, in our notation, $\sqrt{r_s}\Delta C_s \leq \Phi \in[0,1)$ is enough for the arguments to go through, see \cite[Lemma 5.5 and Theorem 6.1]{cohen2012existence}. This again only involves the Lipschitz coefficient with respect to the $Y$-component of the generator. 

\medskip
Lastly, the formulation of BSDEs we chose in this work already appeared in \cite{papapantoleon2018existence}, with the slight difference that we allow our driving martingale $X$ to have components in $\cH^2_\text{\rm loc}$ and not in $\cH^2$, thus enabling us to apply our results in case $X$ is a Brownian motion, and the random measure $\mu$ in this work is a general integer-valued random measure in the sense of \cite{jacod1979calcul,jacod2003limit}, which is not necessarily induced by the jump measure of a process in $\cH^2$ as in \cite{papapantoleon2018existence}. Moreover, we allow the generator to depend on both $Y_s$ and its left-limit $Y_{s-}$. Our integrability condition on the data is also weaker, as we use stochastic exponential weights $\cE(\beta A)$, while the well-posedness result in \cite{papapantoleon2018existence} relies on exponential weights $\mathrm{e}^{\beta A}$. Interestingly, this change of the weights allows us to build a contraction map under the condition $\alpha^2_s \Delta C_s \leq \Phi \in [0,1) $, while in \cite{papapantoleon2018existence} the more restrictive condition $\alpha^2_s \Delta C_s \leq \Phi < 1/(18\mathrm{e})$ is needed. We thus cover \cite[Theorem 3.5]{papapantoleon2018existence}. As already mentioned in \Cref{rem::main_bsde}.$(i)$, there is a small issue in the proof of the \emph{a priori} estimates in \cite{papapantoleon2018existence} to which we will come back in \Cref{rem::discussion_constants}.

\section{Optimal stopping and Mertens' decomposition}\label{sec::optimal_stopping}

In this section, we solve the reflected BSDE in case the generator does not depend on $(y,\mathrm{y},z,u)$, which we assume throughout. We imposed the following integrability condition on $f$ and $\xi$\footnote{Recall that $\xi^+ = \max\{\xi,0\}$.}
\begin{equation}\label{eq::integrability_optimal_stopping}
	\E[|\xi_T|^2] + \E\bigg[\sup_{u \in [0,T)}|\xi^+_u|^2\bigg] + \E\Bigg[\bigg(\int_0^T|f_u|\d C_u \bigg)^2\Bigg] < \infty.
\end{equation}

 From \ref{cond::repre_Y}, it is clear that the first component of the solution is related to the optimal stopping problem
\begin{equation}\label{eq::optimal_stopping}
	V(S) \coloneqq {\esssup_{\tau \in \cT_{\smalltext{S}\smalltext{,}\smalltext{\infty}}}}^{\cG_\smalltext{S}} \E \bigg[ \xi_{\tau \land T} + \int_0^{\tau \land T} f_u \d C_u \bigg| \cG_S \bigg], \; \text{$\P$--a.s.}, \; S \in \cT_{0,\infty}.
\end{equation}
Note that the conditional expectations are well-defined in $[-\infty,\infty)$. The fact that we can actually find a solution to \ref{cond::integ_mart}--\ref{cond::repre_Y} is \emph{a priori} not clear, since we cannot directly employ the classical results of optimal stopping and the Snell envelope theory, as the gains process is not necessarily non-negative (see \cite{el1981aspects,maingueneau1978temps}) or in $\cS^2_T$ (see \cite{grigorova2017reflected,grigorova2020optimal}). We thus need to go through a series of technical lemmata to modify our optimal stopping problem first. Since we do not need their proofs for the analysis that follows, we defer them to \Cref{sec::optimal_stopping_proofs_technical}.

\medskip
For the following lemma, we fix a martingale $M= (M_t)_{t \in [0,\infty]}$ satisfying
	\begin{equation*}
		M_S = \E\bigg[\xi_T + \int_0^T f_s \d C_s\bigg|\cG_S\bigg], \; \text{$\P$--a.s.}, \; S \in \cT_{0,\infty}.
	\end{equation*}
	Note that $V(S) \geq M_S$, $\P$--a.s., $S \in \cT_{0,\infty}$. Let $J = (J_t)_{t \in [0,\infty]}$ and $L = (L_t)_{t \in [0,\infty]}$ be the optional processes defined by
	\begin{equation*}
		J_t \coloneqq \xi_{t \land T} + \int_0^{t \land T} f_s \d C_s, \; \text{and} \; L_t = J_t \lor \big(M_t - \mathbf{1}_{\{t < T\}}\big),\ t\geq 0.
	\end{equation*}

We now rewrite the optimal stopping problem \eqref{eq::optimal_stopping} using the auxiliary process $L$. This idea stems from the proof of \cite[Proposition 6.3.2]{zhang2017backward}.

\begin{lemma}\label{lem::snell_reformulation}
	The process $L$ is in $\cS^2_T$, satisfies $L_\cdot = L_{\cdot \land T}$, up to $\P$--indistinguishability, and
	\begin{equation*}
		V(S) = {\esssup_{\tau \in \cT_{\smalltext{S}\smalltext{,}\smalltext{\infty}}}}^{\cG_S} \E[ L_\tau | \cG_S ], \; \text{$\P$\rm--a.s.}, \; S \in \cT_{0,\infty}.
	\end{equation*}
\end{lemma}

The previous lemma now allows us to deduce the following.

\begin{lemma}\label{lem::supermartingale_family}
	The family $(V(S))_{S \in \cT_{\smalltext{0}\smalltext{,}\smalltext{\infty}}}$ satisfies the following properties
	\begin{enumerate}
		\item[$(i)$] $V(S) \in \L^2(\cG_S)$ and $V(S) \geq \E[V(U)|\cG_S]$, $\P${\rm--a.s.}, for all $(S,U) \in \cT_{0,\infty} \times \cT_{0,\infty}$ with $\P[S \leq U] = 1$.
		\item[$(ii)$] $V(S) = V(U)$, $\P${\rm--a.s.} on $\{U = S\}$, for all $(S,U) \in \cT_{0,\infty} \times \cT_{0,\infty}$.
	\end{enumerate}
\end{lemma}

The next result shows that we can aggregate the family $(V(S))_{S \in \cT_{0,\infty}}$ into a process.
\begin{lemma}\label{lem::snell_aggregation}
	There exists a, up to $\P$--indistinguishability, unique optional process $V = (V_t)_{t \in [0,\infty]} \in \cS^2_T$ satisfying $V_S = V(S)$, $\P$--{\rm a.s.}, for each $S \in \cT_{0,\infty}$. Moreover, $V_\cdot = V_{\cdot \land T}$, up to $\P$--indistinguishability.
\end{lemma}

By \Cref{lem::supermartingale_family}, the process $V = (V_t)_{t \in [0,\infty]}$ constructed in \Cref{lem::snell_aggregation} is a strong optional supermartingale in the sense of \cite[Appendix I]{dellacherie1982probabilities}, and thus $\P$--almost all its paths are l\`adl\`ag (see \cite[Appendix I, Theorem 4]{dellacherie1982probabilities}). 
The next lemma allows us to deduce the Skorokhod condition using our modified optimal stopping problem.

\begin{lemma}\label{lem::snell_indicator}
	$(i)$ For each $S \in \cT_{0,\infty}$, we have $\mathbf{1}_{\{V_\smalltext{S} = L_\smalltext{S}\}} = \mathbf{1}_{\{V_\smalltext{S} = J_\smalltext{S}\}}$, $\P${\rm--a.s.}
	
	\medskip
	$(ii)$ For each $S \in \cT^p_{0,\infty}$, we have $\mathbf{1}_{\{V_{\smalltext{S}\smalltext{-}} = \bar L_\smalltext{S}\}} = \mathbf{1}_{\{V_{\smalltext{S}\smalltext{-}} = \bar J_\smalltext{S}\}}$, $\P${\rm--a.s.}\footnote{Recall that $\cT^p_{0,\infty}$ is the collection of predictable stopping times and $\overline L$ is defined by $\overline L_0 \coloneqq L_0$ and $\overline L_t \coloneqq \limsup_{s \uparrow\uparrow \infty} L_s$ for $t \in (0,\infty]$. The process $\overline J$ is defined analogously.}
\end{lemma}

We have now established the necessary technical results that allow us to apply the arguments laid out in \cite{grigorova2020optimal} to construct the solution to the reflected BSDE for the generator $f$ that does not depend on $(y,z,u)$. First, define $Y = (Y_t)_{t \in [0,\infty]}$ by $Y_t \coloneqq V_t - \int_0^{t \land T} f_s \d C_s$ $t\geq 0$. Then $Y_\cdot = Y_{\cdot \land T}$, up to $\P$-indistinguishability, and
\begin{equation*}
	Y_S = V_S - \int_0^{S \land T} f_s \d C_s = {\esssup_{\tau \in \cT_{\smalltext{S}\smalltext{,}\smalltext{\infty}}}}^{\cG_\smalltext{S}} \E\bigg[\xi_{\tau \land T} + \int_S^{\tau \land T} f_s \d C_s \bigg| \cG_S\bigg], \; \text{$\P$--a.s.}, \; S \in \cT_{0,\infty}.
\end{equation*}

We know by now that $V$ is a strong optional supermartingale in the sense of \cite[Appendix I]{dellacherie1982probabilities}. We can therefore apply Mertens' decomposition to construct the solution to the reflected BSDE.

\begin{proposition}\label{lem::snell_decomposition}
	There exists a unique triple
$(Z,U,N) \in \mathbb{H}^{2}_{T}(X) \times \mathbb{H}^{2}_{T}(
\mu ) \times \mathcal H^{2,\perp}_{0,T}(X,\mu )$ and a, up to $\mathbb{P}$--indistinguish\-ability, unique pair $(K^{r},K^{\ell}) \in \mathcal I^{2}_{T}
\times \mathcal I^{2}_{T}$ such that $K^{r}$ is predictable and starts
$\mathbb{P}$--\textup{a.s.} from zero, $K^\ell$ has $\P$--{\rm a.s.} purely discontinuous paths and satisfies $K^\ell_T = K^\ell_{T-}$, \textnormal{$\P$--a.s.}, and
		\begin{equation*}
			Y_t = \xi_T + \int_t^T f_s \d C_s - \int_t^T Z_s\d X_s - \int_t^T\int_E U_s(x)\tilde\mu(\d s, \d x) - \int_t^T \d N_s + K^r_T - K^r_t + K^\ell_{T-} - K^\ell_{t-}, \; t \in [0,\infty], \; \text{$\P${\rm--a.s.}},
		\end{equation*}
		holds with convention $K^\ell_{0-}\coloneqq 0$. Moreover,
	\[
		 \big(Y_{T-} - \overline{\xi}_T\big) \Delta K^r_T + \int_{(0,T)} \big(Y_{s-} - \overline{\xi}_s\big) \d K^r_s + \int_{[0,T)} (Y_s - \xi_s) \d K^\ell_s = 0, \;\text{$\P$\rm--a.s.}
	\]
\end{proposition}

\begin{proof}
	We use $V_\cdot = V_{\cdot \land T}$, up to $\P$--indistinguishability, and \cite[Lemma 3.2]{grigorova2020optimal}\footnote{Their results still applies to our infinite, but right-closed, horizon $[0,\infty]$.} (which is based on Mertens' unique decomposition of a strong supermartingale) to find a martingale $M = (M_t)_{t \in [0,\infty]}$ with $M = M_{\cdot \land T}$, $M_T \in \L^2(\cG_T)$ and $M_0 = 0$, $\P$--a.s., and two processes $(K^r,K^\ell) \in \cI^2_T \times \cI^2_T$ such that $K^r$ is predictable and satisfies $K^r_0 = 0$, $\P$--a.s., $K^\ell$ has $\P$--a.s. purely discontinuous paths and satisfies $K^\ell_T = K^\ell_{T-}$, $\P$--a.s., with convention $K^\ell_{0-} \coloneqq 0$,
	\begin{equation*}
		Y_t + \int_0^{t \land T} f_s \d C_s = V_t = V_0  + M_t - K^r_t - K^\ell_{t-}, \; t \in [0,\infty], \; \text{$\P$--a.s.},
	\end{equation*}
	and $\Delta K^\ell_\tau = \mathbf{1}_{\{V_\smalltext{\tau} = L_\smalltext{\tau}\}} \Delta K^\ell_\tau$, $\P$--a.s., for each $\tau \in \cT_{0,\infty}$, and $\Delta K^r_\tau = \mathbf{1}_{\{V_{\smalltext{\tau}\smalltext{-}} = \overline L_\smalltext{\tau}\}} \Delta K^r_\tau$, $\P$--a.s., for each $\tau \in \cT^p_{0,\infty}$. Moreover, \cite[Lemma 3.3]{grigorova2020optimal} implies
	\begin{equation*}
		\int_{(0,T)} \mathbf{1}_{\{V_{s-} > \overline L_s\}} \d K^{r,c}_s = 0, \; \text{$\P$--a.s.}, \; \text{where} \; K^{r,c} = K^r - \sum_{s \in (0,\cdot]} \Delta K^r_s.
	\end{equation*}
	Therefore, by \Cref{lem::snell_indicator},
	\begin{equation*}
		(Y_{T-} - \overline\xi_T)\Delta K^r_T + \int_{(0,T)} (Y_{s-} - \overline\xi_s) \d K^r_s + \int_{[0,T)} (Y_s - \xi_s) \d K^\ell_s = 0, \; \text{$\P$--a.s.}
	\end{equation*}
	Since $M_\infty$ is $\cG_\infty$-measurable, and because $\cG_\infty = \cG_{\infty-} = \sigma(\cup_{t \in [0,\infty)}\cG_t)$, the martingale $M$ is $\P$--a.s. left-continuous at infinity. We can therefore decompose $(M_t)_{t \in [0,\infty)}$ as
	\begin{equation}\label{eq::snell_decomposition}
		M_t = \int_0^t Z_s \d X_s + \int_0^t\int_E U_s(x)\tilde\mu(\d s, \d x) + N_t, \; t \in [0,\infty), \; \text{$\P$--a.s.},
	\end{equation}
	for unique $(Z,U,N) \in \H^2_T(X) \times \H^2_T(\mu) \times \cH^{2,\perp}_{0,T}(X,\mu)$ with $N_0 = 0$, $\P$--a.s., such that the value $M_\infty$ is the $\P$--a.s. limit at infinity of the processes on the right-hand-side of \eqref{eq::snell_decomposition}, that is,
	\begin{equation*}
		M_\infty = \int_0^{\infty} Z_s \d X_s + \int_0^{\infty}\int_E U_s(x)\tilde\mu(\d s, \d x) + N_\infty, \; \text{$\P$--a.s.}
	\end{equation*}
	It remains to write the dynamics of $Y$ in backward form.
	
	\medskip
	We now present the argument that implies uniqueness separately as it might seem odd that we do not impose $K^\ell_0 = 0$, $\P$--almost surely. Suppose that $(M^\prime,K^{r,\prime},K^{\ell,\prime})$ is a triplet satisfying the same properties as $(M,K^r,K^\ell)$, and such that
	\begin{equation*}
		Y_t + \int_0^{t \land T} f_s \d C_s = V_t = V_0 +M^\prime_t - K^{r,\prime}_t - K^{\ell,\prime}_{t-}, \; t \in [0,\infty], \; \text{$\P$--a.s.}
	\end{equation*}
	It follows that $M_t - K^r_t - K^\ell_{t-} = M^\prime_t - K^{r,\prime}_t - K^{\ell,\prime}_{t-}, \; t \in [0,\infty], \; \text{$\P$--a.s.},$ hence $M-M^\prime$ is a martingale that is $\P$--indistinguishable from a predictable process with $\P$--a.s. locally finite variation paths. Thus  $M_t = M^\prime_t$, $t \in [0,\infty]$, $\P$--a.s., by \cite[Corollary I.3.16]{jacod2003limit}, and therefore
	\begin{equation*}
		K^r_t + K^\ell_{t-} = K^{r,\prime}_t + K^{\ell,\prime}_{t-}, \; t \in [0,\infty], \; \text{$\P$--a.s.}
	\end{equation*}
	Since $K^\ell$ and $K^{\ell,\prime}$ are $\P$--a.s. purely discontinuous, we can write
	\begin{equation*}
		K^r_t - K^{r,\prime}_t = K^\ell_{t-} - K^{\ell,\prime}_{t-} = K^\ell_0 - K^{\ell,\prime}_0 + \sum_{s \in (0,t)} (\Delta K^\ell_s - \Delta K^{\ell,\prime}_s), \; t \in (0,\infty], \; \text{$\P$--a.s.}
	\end{equation*}
	As $K^r - K^{r,\prime}$ is $\P$--a.s. right-continuous, it follows by taking limits from the right that
	\begin{equation*}
		K^r_t - K^{r,\prime}_t = K^\ell_0 - K^{\ell,\prime}_0 + \sum_{s \in (0,t]} (\Delta K^\ell_s - \Delta K^{\ell,\prime}_s), \; t \in [0,\infty], \; \text{$\P$--a.s.}
	\end{equation*}
	Hence, $0 = K^r_0 - K^{r,\prime}_0 = K^\ell_0 - K^{\ell,\prime}_0 \; \text{and} \; \Delta K^\ell_s = \Delta K^{\ell,\prime}_s, \; s \in (0,\infty], \; \text{$\P$--a.s.},$ which implies $K^\ell = K^{\ell,\prime}$, up to $\P$--indistinguishability, and thus also $K^r = K^{r,\prime}$, up to $\P$--indistinguishability. This completes the proof.
\end{proof}

\section{\emph{A priori} estimates}\label{sec::a_priori}

We now devote ourselves to \emph{a priori} estimates, which will allow us to construct a contraction mapping on the weighted normed spaces introduced in \Cref{sec::weighted_spaces}. Classically, these sort of estimates are derived by an application of It\^o's formula to the square of the (weighted) difference of two solutions to the (reflected) BSDE (see \cite{bouchard2018unified,grigorova2017reflected,grigorova2020optimal}). In our generality, we were not able to deduce appropriate estimates solely using this tool, and thus we will approach the estimates differently. We first derive bounds on conditional expectations of the martingale processes using both It\^o's formula and a pathwise version of Doob's martingale inequality (\Cref{lem::cond_doob}). Then we employ the ideas of \cite[page 27]{el1997backward2} to get the desired weighted norm estimates. For $i \in \{1,2\}$, let $f^i = (f^i_u)_{u \in [0,\infty)}$ be an optional processes satisfying
\begin{equation*}
	\E\bigg[\bigg(\int_0^T |f^i_u| \d C_u\bigg)^2\bigg] < \infty,
\end{equation*}
and suppose that we are given an optional process $y^i = (y^i_t)_{t \in [0,\infty]}$ satisfying
	\begin{subnumcases}{}
		 y^i_t = \xi_T + \int_t^T f^i_s \d C_s - \int_t^T\d \eta^i_s + k^{r,i}_T - k^{r,i}_t  + k^{\ell,i}_{T-} - k^{\ell,i}_{t-}, \; t \in [0,\infty], \; \text{$\P$--a.s.},\label{eq::differential_y} \\
		 y^i = y^i_{\cdot \land T} \geq \xi_{\cdot \land T}, \; \text{$\P$--a.s.}, \\
		 y^i_{S} = {\esssup_{\tau \in \cT_{\smalltext{S}\smalltext{,}\smalltext{\infty}}}}^{\cG_\smalltext{S}} \E \bigg[ \xi_{\tau \land T} +\int_S^{\tau \land T} f^i_u\d C_u \bigg| \cG_S \bigg],\; \text{$\P$--a.s.}, \; S \in \cT_{0,\infty}, \label{eq::rep_y} \\
		 \big(y^i_{T-} - \overline{\xi}_T\big)\Delta k^{r,i}_T + \int_{(0,T)} \big(y^i_{s-} - \overline{\xi}_s\big) \d k^{r,i}_s + \int_{[0,T)} (y^i_s - \xi_s) \d k^{\ell,i}_s = 0, \; \text{$\P$--a.s.}, \label{eq::a_priori_skor}
	\end{subnumcases}
	for some $\eta^i \in \cH^2_T$ and $(k^{r,i}, k^{\ell,i}) \in \cI^{2}_T \times \cI^{2}_T$ with $k^{r,i}$ predictable and starting $\P$--a.s. from zero, and $k^{\ell,i}_T = k^{\ell,i}_{T-}$ up to a $\P$--null set. Here we use the convention $k^{\ell,i}_{0-} \coloneqq 0$. Let 
	\begin{equation*}
		\delta y \coloneqq y^1 - y^2, \; \delta \eta \coloneqq \eta^1 - \eta^2, \; \delta k^r \coloneqq k^{r,1} - k^{r,2}, \; \delta k^\ell \coloneqq k^{\ell,1} - k^{\ell,2},\; \text{and} \; \delta f \coloneqq f^1 - f^2.
	\end{equation*}

To ease the notation, we denote by $L$ the infimum of the function $\ell$ defined on $\{(\varepsilon,\kappa) \in (0,\infty)^2 : 0 < 1-4\kappa \leq 1\}$ by
\begin{equation*}
	\ell(\varepsilon,\kappa) \coloneqq \frac{\max\{1 + \varepsilon+4\kappa + 12/\varepsilon,12/\varepsilon + 2/\kappa\}}{1-4\kappa}.
\end{equation*}

Recall from \ref{data::expectation_sup} that the obstacle $\xi$ satisfies $\E[|\xi_T|^2] + \E\big[ \sup_{s \in [0,T)} |\xi^+_s|^2  \big] < \infty.$ The following is the main result of this section.
\begin{proposition}\label{prop::apriori}
	Let $\beta \in (0,\infty)$. Then\footnote{Recall the definition of $\ff^\Phi$, $\fg^\Phi$, $M^\Phi_1$, $M^\Phi_2$ and $M^\Phi_3$ at the beginning of \Cref{sec::main_results_rbsde}.}
	\begin{equation*}
			\|\delta y\|^2_{\cS^\smalltext{2}_\smalltext{T}} \leq \frac{4}{\beta} \bigg\| \frac{\delta f}{\alpha} \bigg\|^2_{\H^{\smalltext{2}}_{\smalltext{T}\smalltext{,}\smalltext{\beta}}},
			\;
			\|\alpha \delta y\|^2_{\H^{\smalltext{2}}_{\smalltext{T}\smalltext{,}\smalltext{\beta}}} 
			\leq \ff^\Phi(\beta) \bigg\| \frac{\delta f}{\alpha} \bigg\|^2_{\H^{\smalltext{2}}_{\smalltext{T}\smalltext{,}\smalltext{\beta}}},
			\;
			\|\alpha \delta y_-\|^2_{\H^\smalltext{2}_{\smalltext{T}\smalltext{,}\smalltext{\beta}}} 
			\leq (1+\beta\Phi) \fg^\Phi(\beta)\bigg\|\frac{\delta f}{\alpha}\bigg\|^2_{\mathbb{H}^\smalltext{2}_{\smalltext{T}\smalltext{,}\smalltext{\beta}}},
	\end{equation*}
	\begin{align*}
		\E\big[(\delta y_0)^2\big] + \frac{\beta}{(1+\beta\Phi)}\|\alpha\delta y_{-}\|^2_{\H^\smalltext{2}_{\smalltext{T}\smalltext{,}\smalltext{\beta}}} + \|\delta\eta\|^2_{\cH^\smalltext{2}_{\smalltext{T}\smalltext{,}\smalltext{\beta}}} + \E\bigg[\int_0^T \cE(\beta A)_s \d[\delta k^r]_s\bigg] &+ \E\bigg[\int_0^T \cE(\beta A)_s \d[\delta k^\ell]_s\bigg]  \\
		&\quad \leq \bigg( \frac{5}{\beta} + \frac{4}{\beta}(1+\beta\Phi)^{1/2} + \beta \fg^\Phi(\beta) \bigg)\bigg\|\frac{\delta f}{\alpha}\bigg\|^2_{\H^\smalltext{2}_{\smalltext{T}\smalltext{,}\smalltext{\beta}}},
	\end{align*}

	and
	\begin{equation*}
		\|\delta y\|^2_{\cS^\smalltext{2}_\smalltext{T}} + \|\alpha \delta y\|^2_{\H^{\smalltext{2}}_{\smalltext{T}\smalltext{,}\smalltext{\beta}}} + \|\alpha \delta y_-\|^2_{\H^\smalltext{2}_{\smalltext{T}\smalltext{,}\smalltext{\beta}}} + \|\delta\eta\|^2_{\cH^{\smalltext{2}}_{\smalltext{T}\smalltext{,}\smalltext{\beta}}} 
		\leq M^\Phi_1(\beta) \bigg\| \frac{\delta f}{\alpha} \bigg\|^2_{\H^{\smalltext{2}}_{\smalltext{T}\smalltext{,}\smalltext{\beta}}},
		\;
		\|\alpha \delta y\|^2_{\H^{\smalltext{2}}_{\smalltext{T}\smalltext{,}\smalltext{\beta}}} + \|\delta\eta\|^2_{\cH^{\smalltext{2}}_{\smalltext{T}\smalltext{,}\smalltext{\beta}}} 
		\leq M^\Phi_2(\beta) \bigg\| \frac{\delta f}{\alpha} \bigg\|^2_{\H^{\smalltext{2}}_{\smalltext{T}\smalltext{,}\smalltext{\beta}}},
	\end{equation*}
	\begin{equation*}
		\|\delta y\|^2_{\cS^\smalltext{2}_\smalltext{T}} + \|\alpha \delta y_-\|^2_{\H^\smalltext{2}_{\smalltext{T}\smalltext{,}\smalltext{\beta}}} + \|\delta\eta\|^2_{\cH^{\smalltext{2}}_{\smalltext{T}\smalltext{,}\smalltext{\beta}}} 
		\leq M^\Phi_3(\beta) \bigg\| \frac{\delta f}{\alpha} \bigg\|^2_{\H^{\smalltext{2}}_{\smalltext{T}\smalltext{,}\smalltext{\beta}}}.
	\end{equation*}
	Moreover, for each $i \in \{1,2\}$,
	\begin{equation*}
		\|y^i\|^2_{\cS^\smalltext{2}_\smalltext{T}} \leq 12 \, \Bigg(\|\xi_T\|^2_{\L^\smalltext{2}} + \Big\|\sup_{u \in [0,T)} \xi^+_u\Big\|^2_{\L^\smalltext{2}} +  \frac{1}{\beta}\bigg\|\frac{f^i}{\alpha}\bigg\|^2_{\H^\smalltext{2}_{\smalltext{T}\smalltext{,}\smalltext{\beta}}}\Bigg),
	\end{equation*}
	\begin{equation*}
		\|\alpha y^i\|_{\mathbb{H}^{\smalltext{2}}_{\smalltext{T}\smalltext{,}\smalltext{\beta}}}^2 
		\leq 3 \Bigg( \frac{(1+\beta\Phi)}{\beta} \|\xi_T\|^2_{\L^\smalltext{2}_\smalltext{\beta}} + \|\alpha\brs{\xi}\|^2_{\H^\smalltext{2}_{\smalltext{T}\smalltext{,}\smalltext{\beta}}} +  \ff^\Phi(\beta) \bigg\| \frac{f^i}{\alpha} \bigg\|^2_{\H^{\smalltext{2}}_{\smalltext{T}\smalltext{,}\smalltext{\beta}}} \Bigg),
	\end{equation*}
	\begin{equation*}
		\|\alpha y^i_-\|^2_{\H^\smalltext{2}_{\smalltext{T}\smalltext{,}\smalltext{\beta}}} 
		\leq 3 \Bigg( \frac{(1+\beta\Phi)}{\beta} \|\xi_T\|^2_{\L^\smalltext{2}_\smalltext{\beta}} + \|\alpha\brs{\xi}\|^2_{\H^\smalltext{2}_{\smalltext{T}\smalltext{,}\smalltext{\beta}}} +  (1+\beta\Phi) \fg^\Phi(\beta) \bigg\| \frac{f^i}{\alpha} \bigg\|^2_{\H^{\smalltext{2}}_{\smalltext{T}\smalltext{,}\smalltext{\beta}}} \Bigg),
	\end{equation*}
	and
	\begin{align*}
	&\E\big[|y^i_0|^2\big] + \frac{\beta}{(1+\beta\Phi)}\|\alpha y^i_{-}\|^2_{\H^\smalltext{2}_{\smalltext{T}\smalltext{,}\smalltext{\beta}}} + \|\eta^i\|^2_{\cH^\smalltext{2}_{\smalltext{T}\smalltext{,}\smalltext{\beta}}} + \E\bigg[\int_0^T \cE(\beta A)_s \d[k^{r,i}]_s\bigg] + \E\bigg[\int_0^T \cE(\beta A)_s \d[k^{\ell,i}]_s\bigg] \\
	&\quad\leq L\Bigg( \|\xi_T\|^2_{\L^\smalltext{2}} + \|\xi^+\1_{\llbracket 0,T\rrparenthesis}\|^2_{\cS^\smalltext{2}_\smalltext{T}} + \frac{1}{\beta} \bigg\|\frac{f^i}{\alpha}\bigg\|^2_{\H^\smalltext{2}_{\smalltext{T}\smalltext{,}\smalltext{\beta}}} + \beta\bigg(\|\xi_T\|^2_{\L^\smalltext{2}_\smalltext{\beta}} + \| \alpha\brs{\xi}\|^2_{\H^\smalltext{2}_{\smalltext{T}\smalltext{,}\smalltext{\beta}}} + \fg^\Phi(\beta) \bigg\|\frac{f^i}{\alpha}\bigg\|^2_{\H^\smalltext{2}_{\smalltext{T}\smalltext{,}\smalltext{\beta}}}\bigg)\Bigg).
\end{align*}
\end{proposition}
	
\medskip
The proof of the preceding proposition will be based on the two following lemmata whose proofs we defer to \Cref{subsec::proofs_tech_lemmas}. Here we use the convention $\zeta_{0-} \coloneqq 0$ for a process $\zeta =(\zeta_t)_{t \in [0,\infty]}$, and we recall from \Cref{subsec::integrals} that we never include the point $\infty$ in the domain of integration.

\begin{lemma}\label{lem::bound_delta_y_s2}
	The following inequalities hold
	\begin{equation}\label{eq::cond_bound_delta_y_s}
		|\delta y_S| \leq \E\bigg[\int_S^T |\delta f_u| \d C_u \bigg| \cG_S\bigg],  \; |y^i_S| \leq \E\bigg[|\xi_T| + \sup_{u \in [S,\infty]}|\xi^+_u \mathbf{1}_{\{u<T\}}| + \int_S^T |f^i_u| \d C_u \bigg| \cG_S\bigg], \; \text{$\P$\rm--a.s.}, \; S \in \cT_{0,\infty}, \; i \in \{1,2\},
	\end{equation}
	\begin{equation}\label{eq::cond_bound_delta_y_s_minus}
		|\delta y_{S-}| \leq \E\bigg[\int_{S-}^T |\delta f_u| \d C_u \bigg| \cG_{S-}\bigg],  \; |y^i_{S-}| \leq \E\bigg[|\xi_T| + \brs{\xi}_S + \int_{S-}^T |f^i_u| \d C_u \bigg| \cG_{S-}\bigg], \; \text{$\P$\rm--a.s.}, \; S \in \cT^p_{0,\infty}, \; i \in \{1,2\}.\footnote{Recall from \Cref{sec::stochastic_basis} that $\cT^p_{0,\infty}$ denotes the collection of predictable stopping times.}
	\end{equation}
	Moreover
	\begin{equation}\label{eq::exp_bound_delta_y_s}
		\|\delta y\|^2_{\cS^\smalltext{2}_\smalltext{T}} \leq 4 \E\bigg[\bigg( \int_0^T|\delta f_u| \d C_u\bigg)^2\bigg] < \infty, \; \|y^i\|^2_{\cS^\smalltext{2}_\smalltext{T}} \leq 12  \E\bigg[|\xi_T|^2 + \sup_{u \in [0,T)}|\xi^+_u|^2 + \bigg(\int_0^T |f^i_u| \d C_u\bigg)^2\bigg] < \infty, \; i \in \{1,2\}.
	\end{equation}
\end{lemma}

\begin{lemma}\label{lem::cond_estimates}
	The following inequalities hold:
		\begin{align}\label{eq::cond_deltaeta0}
			|\delta y_S|^2 + \E\bigg[\int_S^T \d\langle\delta\eta\rangle_u \bigg| \cG_S\bigg] &+ \E\bigg[\int_S^T \d[\delta k^r]_u \bigg| \cG_S\bigg] + \E\bigg[\int_{S-}^T \d[\delta k^\ell]_u \bigg| \cG_S\bigg] \nonumber\\
	&\quad \leq 2 \E\bigg[\int_S^T \delta y_u \delta f_u \d C_u \bigg| \cG_S\bigg] + \E\bigg[\int_S^T (\delta f_u)^2\d[C]_u \bigg| \cG_S\bigg],\; \text{$\P$\rm--a.s.}, \; S \in \cT_{0,\infty},
		\end{align}
		\begin{align}\label{eq::cond_deltaeta_pred0}
			|\delta y_{S-}|^2 + \E\bigg[\int_{S-}^T \d\langle\delta\eta\rangle_u \bigg| \cG_{S-}\bigg] &+ \E\bigg[\int_{S-}^T \d[\delta k^r]_u \bigg| \cG_{S-}\bigg] + \E\bigg[\int_{S-}^T \d[\delta k^\ell]_u \bigg| \cG_{S-}\bigg] \nonumber\\
	&\quad \leq 2 \E\bigg[\int_{S-}^T \delta y_u \delta f_u \d C_u \bigg| \cG_{S-}\bigg] + \E\bigg[\int_{S-}^T (\delta f_u)^2\d[C]_u \bigg| \cG_{S-}\bigg],\; \text{$\P$\rm--a.s.}, \; S \in \cT^p_{0,\infty}.
		\end{align}
		
		Moreover, for each $i \in \{1,2\}$,
		\begin{align}\label{eq::cond_eta1}
			 |y^i_S|^2 + \E\bigg[ \int_S^T \d \langle \eta^i \rangle_u \bigg| \cG_S \bigg] &+ \E\bigg[\int_S^T \d[k^{r,i}]_u \bigg| \cG_S\bigg] + \E\bigg[\int_{S-}^T \d[k^{\ell,i}]_u \bigg| \cG_S\bigg] \nonumber\\
			 &\quad\leq L \E\bigg[ |\xi_T|^2 + \sup_{u \in [S,\infty]} |\xi^+_u \mathbf{1}_{\{u < T\}}|^2 + \bigg(\int_S^T |f^i_u| \d C_u\bigg)^2 \bigg| \cG_S \bigg],\; \text{$\P$\rm--a.s.}, \; S \in \cT_{0,\infty},
		\end{align}
		\begin{align}\label{eq::cond_eta1_pred}
			 |y^i_{S-}|^2 + \E\bigg[ \int_{S-}^T \d \langle \eta^i \rangle_u \bigg| \cG_{S-} \bigg] &+ \E\bigg[\int_{S-}^T \d[k^{r,i}]_u \bigg| \cG_{S-}\bigg] + \E\bigg[\int_{S-}^T \d[k^{\ell,i}]_u \bigg| \cG_{S-}\bigg] \nonumber\\
			 &\quad\leq L \E\bigg[ |\xi_T|^2 + |\brs{\xi}_S|^2 + \bigg(\int_{S-}^T |f^i_u| \d C_u\bigg)^2 \bigg| \cG_{S-} \bigg],\; \text{$\P$\rm--a.s.}, \; S \in \cT^p_{0,\infty}.
		\end{align}
\end{lemma}

\medskip
We turn to the proof of the \emph{a priori} estimates.

\begin{proof}[Proof of Proposition \ref{prop::apriori}]\label{proof::apriori}
	To ease the presentation, let us abuse notation sligthly in this proof and denote by $\E[W_\cdot|\cG_\cdot]$ and $\E[W_\cdot|\cG_{\cdot-}]$ the optional and predictable projections, respectively, of a non-negative, product-measurable process $W = (W_t)_{t \in [0,\infty]}$. We refer to \cite[Section VI.2 and Appendix I]{dellacherie1982probabilities} for their existence and properties.
	
	\medskip
	We suppose, without loss of generality, that $\| \frac{f^1}{\alpha}\|_{\H^{\smalltext{2}}_{\smalltext{T}\smalltext{,}\smalltext{\beta}}}$, $\| \frac{f^2}{\alpha}\|_{\H^{\smalltext{2}}_{\smalltext{T}\smalltext{,}\smalltext{\beta}}}$ and $\| \frac{\delta f}{\alpha}\|_{\H^{\smalltext{2}}_{\smalltext{T}\smalltext{,}\smalltext{\beta}}}$ are all finite; otherwise the stated inequalities trivially hold. We start with some introductory calculations. First, let $(\gamma,\beta) \in (0,\infty)^2$ with $\gamma < \beta$, and recall that the stochastic exponential $\cE(\gamma A)$ of the non-decreasing, predictable process $\gamma A$ satisfies
\begin{equation}\label{eq::sde_stoch_exp}
	\cE(\gamma A)_t = 1 + \int_0^t \cE(\gamma A)_{s-}  \d (\gamma A)_s = 1 + \int_0^t \cE(\gamma A)_{s-} \gamma \d A_s, \; t \in [0,\infty), \; \text{$\P$--a.s.}
\end{equation}
In particular, $\cE(\gamma A)$ is predictable and (see \cite[page 134]{jacod2003limit})
\begin{equation}\label{eq::prod_formula_stoch_exp}
	\displaystyle \cE(\gamma A)_t = \mathrm{e}^{\gamma A_\smalltext{t}} \prod_{s \in (0,t]}(1+\gamma\Delta A_s)\mathrm{e}^{-\gamma\Delta A_\smalltext{s}}, \; t \in [0,\infty), \; \text{$\P$--a.s.}
\end{equation}
Since $A$ is $\P$--a.s. non-decreasing, so is $\cE(\gamma A)$ by \eqref{eq::sde_stoch_exp} and \eqref{eq::prod_formula_stoch_exp}. Note also that
\begin{equation}\label{eq::jump_stoch_exp}
	\cE(\gamma A)_t = \cE(\gamma A)_{t-} + \Delta\cE(\gamma A)_t = \cE(\gamma A)_{t-} + \cE(\gamma A)_{t-}\gamma\Delta A_t = \cE(\gamma A)_{t-}(1+\gamma\Delta A_t), \; t \in (0,\infty), \; \text{$\P$--a.s.}
\end{equation}
By \cite[Lemma 4.4]{cohen2012existence} or \Cref{lem::stoch_exp_rules}.$(i)$, we have $\cE(\gamma A)^{-1} = \cE(-\overline{\gamma A})$, $\P$--a.s., where $\overline{\gamma A}$ is the predictable process satisfying
\begin{equation*}
	\overline{\gamma A} = \gamma A - \sum_{s \in (0,\cdot]} \frac{(\gamma\Delta A_s)^2}{1+\gamma\Delta A_s}, \; \text{$\P$--a.s.},
\end{equation*}
\begin{equation}\label{eq::overline_A}
	(\overline{\gamma A})^c = (\gamma A)^c = \gamma A^c, \; \Delta\overline{\gamma A} = \gamma\Delta A - \frac{(\gamma\Delta A)^2}{1+\gamma\Delta A} = \frac{\gamma\Delta A}{1+\gamma\Delta A}, \; \text{$\P$--a.s.}
\end{equation}
In particular, $\overline{\gamma A}$ is $\P$--a.s. non-decreasing.
Let $F = \big(F(t)\big)_{t \in [0,\infty]}$ be the process defined by
\begin{equation}\label{eq::def_F}
	F(t) \coloneqq \int_t^T |\delta f_s| \d C_s \coloneqq \int_0^T |\delta f_s|\d C_s - \int_0^{t \land T}|\delta f_s|\d C_s = \int_{(0,\infty)} |\delta f_s|\mathbf{1}_{(t,T]}(s) \d C_s, \; \text{$\P$--a.s.}
\end{equation}
For $t \in (0,\infty]$, we have $F(t-) = \int_{(0,\infty)} |\delta f_s| \mathbf{1}_{[t,T]}(s) \d C_s$, $\P$--a.s., and
\begin{align}\label{eq::bound_on_F_squared_minus}
	|F(t-)|^2 
	\leq \int_{(0,\infty)} \frac{1}{\cE(\gamma A)_s} \mathbf{1}_{[t,T]}(s) \d A_s \int_{(0,\infty)} \cE(\gamma A)_s \frac{|\delta f_s|^2}{\alpha^2_s}  \mathbf{1}_{[t,T]}(s) \d C_s, \; \text{$\P$--a.s.},
\end{align}
by the Cauchy--Schwarz inequality. The first integral on the right-hand side above can be bounded as follows
\begin{align}\label{eq::integ_stoch_exp_inverse}
	\int_{(0,\infty)} \frac{1}{\cE(\gamma A)_s} \mathbf{1}_{[t,T]}(s) \d A_s 
	&= \lim_{t^\prime \uparrow\uparrow \infty}\int_{(0,\infty)} \frac{1}{\cE(\gamma A)_s} \mathbf{1}_{[t,t^\prime \land T]}(s) \d A_s \nonumber\\
	&= \lim_{t^\prime \uparrow\uparrow \infty}\int_{(0,\infty)} \frac{1}{\cE(\gamma A)_{s-}}\frac{1}{(1+\gamma\Delta A_s)} \mathbf{1}_{[t,t^\prime \land T]}(s) \d A_s \nonumber\\
	&= \lim_{t^\prime \uparrow\uparrow \infty}\int_{(0,\infty)} \cE(-\overline{\gamma A})_{s-}\frac{1}{(1+\gamma\Delta A_s)} \mathbf{1}_{[t,t^\prime \land T]}(s) \d A_s \nonumber\\
	&= \lim_{t^\prime \uparrow\uparrow \infty}\int_{(0,\infty)} \cE(-\overline{\gamma A})_{s-}\frac{1}{(1+\gamma\Delta A_s)} \mathbf{1}_{[t,t^\prime \land T]}(s) \d A^c_s + \sum_{s \in [t,\infty)} \cE(-\overline{\gamma A})_{s-}\mathbf{1}_{\{s \leq t^\prime \land T\}}\frac{\Delta A_s}{(1+\gamma\Delta A_s)}  \nonumber\\
	&= \lim_{t^\prime \uparrow\uparrow \infty}\frac{1}{\gamma}\int_{(0,\infty)} \cE(-\overline{\gamma A})_{s-} \mathbf{1}_{[t,t^\prime \land T]}(s) \d (\overline{\gamma A})^c_s + \frac{1}{\gamma}\sum_{s \in [t,\infty)} \cE(-\overline{\gamma A})_{s-}\mathbf{1}_{\{s \leq t^\prime \land T\}}\Delta\overline{\gamma A}_s \nonumber\\
	&= \lim_{t^\prime \uparrow\uparrow \infty}\frac{1}{\gamma} \int_{(0,\infty)} \cE(-\overline{\gamma A})_{s-}\mathbf{1}_{[t,t^\prime \land T]}(s) \d \overline{\gamma A}_s \nonumber\\
	&= -\frac{1}{\gamma} \lim_{t^\prime \uparrow\uparrow \infty}\bigg(\cE(-\overline{\gamma A})_{t^\prime \land T} - \cE(-\overline{\gamma A})_{{t \land T}-}\bigg) 
	\leq \frac{1}{\gamma} \frac{1}{\cE(\gamma A)_{t \land T-}}.
\end{align}
Here, the fourth line follows from \eqref{eq::overline_A} and 
\begin{equation}\label{eq::dA_cont_jumps}
	\frac{1}{(1+\gamma\Delta A)} \d A^c = \d A^c
\end{equation}
since $A^c$ is continuous, and thus $\d A^c$ does not charge any points on $(0,\infty)$.
Thus
\begin{equation*}
	|F(t-)|^2 \leq \frac{1}{\gamma} \frac{1}{\cE(\gamma A)_{t\land T-}} \int_{(0,\infty)} \cE(\gamma A)_s \frac{|\delta f_s|^2}{\alpha^2_s}\mathbf{1}_{[t,T]}(s) \d C_s,\; t \in (0,\infty), \; \text{$\P$--a.s.},
\end{equation*}
and therefore
\begin{align}\label{eq::integral_F_minus}
	\int_0^T \cE(\beta A)_{t-} |F(t-)|^2 \d A_t &= \int_{(0,\infty)} \mathbf{1}_{(0,T]}(t) \cE(\beta A)_{t-} |F(t-)|^2 \d A_t \nonumber\\
	&\leq \frac{1}{\gamma}\int_{(0,\infty)} \mathbf{1}_{(0,T]}(t) \cE(\beta A)_{t-}  \frac{1}{\cE(\gamma A)_{t-}} \int_{(0,\infty)} \cE(\gamma A)_s \frac{|\delta f_s|^2}{\alpha^2_s} \mathbf{1}_{[t,T]}(s) \d C_s \d A_t \nonumber\\
	&= \frac{1}{\gamma}\int_{(0,\infty)}\int_{(0,\infty)}\mathbf{1}_{\{0<t\leq s \leq T\}} \cE(\beta A)_{t-}  \frac{1}{\cE(\gamma A)_{t-}} \cE(\gamma A)_s \frac{|\delta f_s|^2}{\alpha^2_s} \d C_s \d A_t \nonumber\\
	&= \frac{1}{\gamma}\int_{(0,\infty)}\int_{(0,\infty)}\mathbf{1}_{\{0<t\leq s \leq T\}} \cE(\beta A)_{t-}  \frac{1}{\cE(\gamma A)_{t-}} \cE(\gamma A)_s \frac{|\delta f_s|^2}{\alpha^2_s} \d A_t \d C_s \nonumber\\
	&= \frac{1}{\gamma}\int_{(0,\infty)} \cE(\gamma A)_s \frac{|\delta f_s|^2}{\alpha^2_s} \int_{(0,\infty)}\mathbf{1}_{\{0<t\leq s \leq T\}} \cE(\beta A)_{t-}  \frac{1}{\cE(\gamma A)_{t-}} \d A_t \d C_s \nonumber\\
	&= \frac{1}{\gamma}\int_0^T \cE(\gamma A)_s \frac{|\delta f_s|^2}{\alpha^2_s} \int_0^s \cE(\beta A)_{t-}  \frac{1}{\cE(\gamma A)_{t-}} \d A_t \d C_s, \; \text{$\P$--a.s.}
\end{align}
Here, the fourth line follows from Tonelli's theorem. \Cref{lem::stoch_exp_rules}.$(ii)$ yields
\begin{equation*}
	\int_0^s \cE(\beta A)_{t-}  \frac{1}{\cE(\gamma A)_{t-}} \d A_t  = \int_0^s \cE(\widehat A^{\beta,\gamma})_{t-} \d A_t,\; s \in [0,\infty), \; \text{$\P$--a.s.},
\end{equation*}
where $\widehat A^{\beta,\gamma} = (\widehat A^{\beta,\gamma}_t)_{t \in [0,\infty)}$ is the predictable process satisfying
\begin{equation}\label{eq::widehat_A_beta_gamma}
	\widehat A^{\beta,\gamma} = (\beta-\gamma)A^c + \sum_{s \in (0,\cdot]} (\beta-\gamma)\frac{\Delta A_s}{1+\gamma\Delta A_s}, \; \text{$\P$--a.s.}
\end{equation}
Since $\beta-\gamma > 0$, the process $\widehat A^{\beta,\gamma}$ is $\P$--a.s. non-decreasing. Similar to the derivation of \eqref{eq::integ_stoch_exp_inverse}, we find
\begin{align}\label{eq::widehat_A_integral}
	\int_0^s \cE\big(\widehat A^{\beta,\gamma}\big)_{t-} \d A_t &= \int_{(0,s]} \cE\big(\widehat A^{\beta,\gamma}\big)_{t-} \d A^c_t + \sum_{t \in (0,s]} \cE\big(\widehat A^{\beta,\gamma}\big)_{t-} \Delta A_t \nonumber\\
	&= \frac{1}{(\beta-\gamma)}\int_{(0,s]} \cE\big(\widehat A^{\beta,\gamma}\big)_{t-} (1+\gamma\Delta A_t) \d \big(\widehat A^{\beta,\gamma}\big)^c_t + \frac{1}{(\beta-\gamma)}\sum_{t \in (0,s]} \cE\big(\widehat A^{\beta,\gamma}\big)_{t-} (1+\gamma\Delta A_t) \Delta \widehat A^{\beta,\gamma}_t \nonumber\\
	&= \frac{1}{(\beta-\gamma)}\int_{(0,s]} \cE\big(\widehat A^{\beta,\gamma}\big)_{t-} (1+\gamma\Delta A_t) \d \widehat A^{\beta,\gamma}_t \nonumber\\
	&\leq \frac{(1+\gamma\Phi)}{(\beta - \gamma)} \int_{(0,s]} \cE\big(\widehat A^{\beta,\gamma}\big)_{t-} \d \widehat A^{\beta,\gamma}_t \nonumber\\
	&= \frac{(1+\gamma\Phi)}{(\beta - \gamma)} \big(\cE\big(\widehat A^{\beta,\gamma}\big)_s - 1\big) \leq \frac{(1+\gamma\Phi)}{(\beta - \gamma)} \cE\big(\widehat A^{\beta,\gamma}\big)_s = \frac{(1+\gamma\Phi)}{(\beta - \gamma)} \cE(\beta A)_s \frac{1}{\cE(\gamma A)_s},\; s \in (0,\infty), \; \text{$\P$--a.s.}
\end{align}
Here, we used $(1+\gamma\Delta A)\d A^c = \d A^c$ and \eqref{eq::widehat_A_beta_gamma} in the second line, and the definition of $\widehat A^{\beta,\gamma}$ in the last equality. By substituting \eqref{eq::widehat_A_integral} into \eqref{eq::integral_F_minus}, we obtain
\begin{equation*}
	\int_0^T \cE(\beta A)_{t-} |F(t-)|^2 \d A_t 
	\leq \frac{(1+\gamma\Phi)}{\gamma(\beta - \gamma)} \int_0^T \cE(\beta A)_s \frac{|\delta f_s|^2}{\alpha^2_s} \d C_s, \; \text{$\P$--a.s.}
\end{equation*}
Consequently, this implies
\begin{align*}
	\int_0^T \cE(\beta A)_t |F(t-)|^2 \d A_t &= \int_0^T \cE(\beta A)_{t-}(1+\beta \Delta A_t) |F(t-)|^2 \d A_t  \\
	&\leq (1+\beta\Phi) \int_0^T \cE(\beta A)_{t-} |F(t-)|^2 \d A_t \leq \frac{(1+\beta\Phi)(1+\gamma\Phi)}{\gamma(\beta-\gamma)} \int_0^T \cE(\beta A)_s \frac{|\delta f_s|^2}{\alpha^2_s} \d C_s, \; \text{$\P$--a.s.}
\end{align*}
Next, we note that we also have
\begin{align}\label{eq::bound_on_F_squared}
	|F(t)|^2 
	\leq \int_{(0,\infty)} \frac{1}{\cE(\gamma A)_s} \mathbf{1}_{(t,T]}(s) \d A_s \int_{(0,\infty)} \cE(\gamma A)_s \frac{|\delta f_s|^2}{\alpha^2_s}  \mathbf{1}_{(t,T]}(s) \d C_s, \; t \in [0,\infty), \; \text{$\P$--a.s.},
\end{align}
and similarly to \eqref{eq::integ_stoch_exp_inverse}, we find
\begin{equation}\label{eq::integ_stoch_exp_inverse2}
	\int_{(0,\infty)} \frac{1}{\cE(\gamma A)_s} \mathbf{1}_{(t,T]}(s) \d A_s \leq \frac{1}{\gamma} \frac{1}{\cE(\gamma A)_{t \land T}}, \; t \in [0,\infty), \; \text{$\P$--a.s.}
\end{equation}
This also follows by taking right-hand limits along $t$ in \eqref{eq::integ_stoch_exp_inverse}. We insert \eqref{eq::integ_stoch_exp_inverse2} into \eqref{eq::bound_on_F_squared} and find
\begin{equation}\label{eq::F_t_squared}
	|F(t)|^2 \leq \frac{1}{\gamma} \frac{1}{\cE(\gamma A)_{t \land T}} \int_{(0,\infty)} \cE(\gamma A)_s \frac{|\delta f_s|^2}{\alpha^2_s}  \mathbf{1}_{(t,T]}(s) \d C_s, \; t \in [0,\infty), \; \text{$\P$--a.s.}
\end{equation}
This yields
\begin{align*}
	\int_0^T \cE(\beta A)_t |F(t)|^2 \d A_t 
	&= \int_{(0,\infty)} \cE(\beta A)_t |F(t)|^2 \mathbf{1}_{(0,T]}(t) \d A_t \\ 
	&\leq \frac{1}{\gamma}\int_{(0,\infty)} \cE(\beta A)_{t}  \frac{1}{\cE(\gamma A)_{t}} \mathbf{1}_{(0,T]}(t) \int_{(0,\infty)} \cE(\gamma A)_s \frac{|\delta f_s|^2}{\alpha^2_s} \mathbf{1}_{(t,T]}(s) \d C_s \d A_t \nonumber\\
	&= \frac{1}{\gamma}\int_{(0,\infty)}\int_{(0,\infty)}\mathbf{1}_{\{0<t < s \leq T\}} \cE(\beta A)_{t}  \frac{1}{\cE(\gamma A)_{t}} \cE(\gamma A)_s \frac{|\delta f_s|^2}{\alpha^2_s} \d C_s \d A_t \nonumber\\
	&= \frac{1}{\gamma}\int_{(0,\infty)}\int_{(0,\infty)}\mathbf{1}_{\{0<t < s \leq T\}} \cE(\beta A)_{t}  \frac{1}{\cE(\gamma A)_{t}} \cE(\gamma A)_s \frac{|\delta f_s|^2}{\alpha^2_s} \d A_t \d C_s \nonumber\\
	&= \frac{1}{\gamma}\int_{(0,\infty)} \cE(\gamma A)_s \frac{|\delta f_s|^2}{\alpha^2_s} \int_{(0,\infty)}\mathbf{1}_{\{0<t< s \leq T\}} \cE(\beta A)_{t}  \frac{1}{\cE(\gamma A)_{t}} \d A_t \d C_s \nonumber\\
	&= \frac{1}{\gamma}\int_0^T \cE(\gamma A)_s \frac{|\delta f_s|^2}{\alpha^2_s} \int_0^{s-} \cE(\beta A)_{t}  \frac{1}{\cE(\gamma A)_{t}} \d A_t \d C_s \\
	&=  \frac{1}{\gamma}\int_0^T \cE(\gamma A)_s \frac{|\delta f_s|^2}{\alpha^2_s} \int_0^{s-} \frac{\cE(\beta A)_{t-}}{\cE(\gamma A)_{t-}}\frac{(1+\beta\Delta A_t)}{(1+\gamma\Delta A_t)} \d A_t \d C_s \\
	&\leq  \frac{1}{\gamma}\frac{(1+\beta\Phi)}{(1+\gamma\Phi)} \int_0^T \cE(\gamma A)_s \frac{|\delta f_s|^2}{\alpha^2_s} \int_0^{s-} \frac{\cE(\beta A)_{t-}}{\cE(\gamma A)_{t-}}\d A_t \d C_s \\
	&=  \frac{1}{\gamma}\frac{(1+\beta\Phi)}{(1+\gamma\Phi)}\int_0^T \cE(\gamma A)_s \frac{|\delta f_s|^2}{\alpha^2_s} \int_0^{s-} \cE\big(\widehat A^{\beta,\gamma}\big)_{t-} \d A_t \d C_s \\
	&\leq \frac{(1+\beta\Phi)}{\gamma(\beta - \gamma)}\int_0^T \cE(\gamma A)_s \frac{|\delta f_s|^2}{\alpha^2_s}  \cE(\beta A)_{s-} \frac{1}{\cE(\gamma A)_{s-}} \d C_s \\
	&= \frac{(1+\beta\Phi)}{\gamma(\beta - \gamma)}\int_0^T \cE(\beta A)_{s-}  \frac{|\delta f_s|^2}{\alpha^2_s} \frac{\cE(\gamma A)_{s-}(1+\gamma\Delta A_s)}{\cE(\gamma A)_{s-}}  \d C_s \\
	&\leq \frac{(1+\beta\Phi)}{\gamma(\beta - \gamma)}\int_0^T \cE(\beta A)_{s-}(1+\beta\Delta A_s)  \frac{|\delta f_s|^2}{\alpha^2_s} \d C_s = \frac{(1+\beta\Phi)}{\gamma(\beta - \gamma)}\int_0^T \cE(\beta A)_s \frac{|\delta f_s|^2}{\alpha^2_s}   \d C_s, \; \text{$\P$--a.s.}
\end{align*}
Here, we used Tonelli's theorem in the fourth line, \eqref{eq::jump_stoch_exp} in the seventh line, the fact that $(1+\beta \Delta A)/(1+\gamma \Delta A) \leq (1+\beta\Phi)/(1+\gamma\Phi)$ since $x \longmapsto (1+\beta x)/(1+\gamma x)$ is increasing on $[0,\Phi]$ in the eighth line, \Cref{eq::widehat_A_integral} in the tenth line, and \Cref{eq::jump_stoch_exp} again in the second-to-last and last line. Consequently, this implies
\begin{equation*}
	\int_0^T \cE(\beta A)_{t-} |F(t)|^2 \d A_t 
	\leq \int_0^T \cE(\beta A)_{t-} |F(t-)|^2 \d A_t 
	\leq \frac{(1+\gamma\Phi)}{\gamma(\beta - \gamma)}\int_0^T \cE(\beta A)_s  \frac{|\delta f_s|^2}{\alpha^2_s}  \d C_s, \; \text{$\P$--a.s.},
\end{equation*}
since $F(t) \leq F(t-)$ for $t \in (0,\infty]$. To summarise, we found
\begin{align}\label{eq::stoch_exp_F}
	\int_0^T \cE(\beta A)_t |F(t)|^2 \d A_t 
	&\leq \frac{(1+\beta\Phi)}{\gamma(\beta - \gamma)}\int_0^T \cE(\beta A)_s  \frac{|\delta f_s|^2}{\alpha^2_s}  \d C_s, \nonumber\; \int_0^T \cE(\beta A)_{t-} |F(t)|^2 \d A_t 
	\leq \frac{(1+\gamma\Phi)}{\gamma(\beta - \gamma)}\int_0^T \cE(\beta A)_s  \frac{|\delta f_s|^2}{\alpha^2_s}  \d C_s, \nonumber\\
\nonumber	&\int_0^T \cE(\beta A)_t |F(t-)|^2 \d A_t 
	\leq \frac{(1+\beta\Phi)(1+\gamma\Phi)}{\gamma(\beta-\gamma)} \int_0^T \cE(\beta A)_s \frac{|\delta f_s|^2}{\alpha^2_s} \d C_s, \\ 
	&\qquad\int_0^T \cE(\beta A)_{t-} |F(t-)|^2 \d A_t 
	\leq \frac{(1+\gamma\Phi)}{\gamma(\beta - \gamma)} \int_0^T \cE(\beta A)_s \frac{|\delta f_s|^2}{\alpha^2_s} \d C_s, \; \text{$\P$--a.s.}
\end{align}

\medskip
We turn to the stated bounds. The $\cS^2_T$-bounds
\begin{equation}\label{eq::S2T_bounds_proof}
	\|\delta y\|^2_{\cS^\smalltext{2}_\smalltext{T}} \leq \frac{4}{\beta} \bigg\| \frac{\delta f}{\alpha} \bigg\|^2_{\H^{\smalltext{2}}_{\smalltext{T}\smalltext{,}\smalltext{\beta}}} \; \text{and} \; \|y^i\|^2_{\cS^\smalltext{2}_\smalltext{T}} \leq 12 \, \Bigg(\|\xi_T\|^2_{\L^\smalltext{2}} + \Big\|\sup_{u \in [0,T)} \xi^+_u\Big\|^2_{\L^\smalltext{2}} +  \frac{1}{\beta}\bigg\|\frac{f^i}{\alpha}\bigg\|^2_{\H^\smalltext{2}_{\smalltext{T}\smalltext{,}\smalltext{\beta}}}\Bigg)
\end{equation}
follow from \Cref{lem::bound_delta_y_s2} together with
\begin{align}
	\E\bigg[\bigg(\int_0^T |\delta f_s| \d C_s \bigg)^2\bigg] &\leq \E\bigg[\bigg(\int_0^T \frac{1}{\cE(\beta A)_s}\d A_s\bigg)\bigg(\int_0^T \cE(\beta A)_s\frac{|\delta f_s|^2}{\alpha^2_s} \d C_s \bigg)\bigg] \leq \frac{1}{\beta} \bigg\|\frac{\delta f}{\alpha}\bigg\|^2_{\H^\smalltext{2}_{\smalltext{T}\smalltext{,}\smalltext{\beta}}},\label{eq::bound_integral_delta_f}\\
	\E\bigg[\bigg(\int_0^T |f^i_s| \d C_s \bigg)^2\bigg] &\leq \E\bigg[\bigg(\int_0^T \frac{1}{\cE(\beta A)_s}\d A_s\bigg)\bigg(\int_0^T \cE(\beta A)_s\frac{|f^i_s|^2}{\alpha^2_s} \d C_s \bigg)\bigg] \leq \frac{1}{\beta} \bigg\|\frac{f^i}{\alpha}\bigg\|^2_{\H^\smalltext{2}_{\smalltext{T}\smalltext{,}\smalltext{\beta}}}. \label{eq::bound_integral_f}
\end{align}
Here we first use the Cauchy--Schwarz inequality and then \eqref{eq::integ_stoch_exp_inverse2}.
The $\H^2$-bound for $\delta y$ follows from
\begin{align*}
	\|\alpha \delta y\|^2_{\H^\smalltext{2}_{\smalltext{T}\smalltext{,}\smalltext{\beta}}} = \E \bigg[ \int_0^T \cE(\beta A)_t |\delta y_t|^2 \alpha^2_t \d C_t \bigg] 
	= \E \bigg[ \int_0^T \cE(\beta A)_t |\delta y_{t}|^2 \d A_t \bigg] 
&\leq \E \bigg[ \int_0^T \E \big[ \cE(\beta A)_t |F(t)|^2 \big| \cG_{t} \big] \d A_t \bigg] \\
	&= \E \bigg[ \int_0^T \cE(\beta A)_t |F(t)|^2 \d A_t \bigg]
	\leq \frac{(1+\beta\Phi)}{\gamma(\beta - \gamma)} \bigg\| \frac{\delta f}{\alpha} \bigg\|^2_{\H^{\smalltext{2}}_{\smalltext{T}\smalltext{,}\smalltext{\beta}}},
\end{align*}
where the first inequality follows from optional projection together with \Cref{lem::bound_delta_y_s2}, the step from the first to the second line follows from \cite[Theorem VI.57]{dellacherie1982probabilities}, and the last inequality follows from \eqref{eq::stoch_exp_F}. Hence,
\begin{equation}\label{eq::inequality_delta_y_H_2}
	\|\alpha \delta y\|^2_{\H^\smalltext{2}_{\smalltext{T}\smalltext{,}\smalltext{\beta}}} 
	\leq \inf_{\gamma \in (0,\beta)}\bigg\{\frac{(1+\beta\Phi)}{\gamma(\beta-\gamma)}\bigg\}\bigg\| \frac{\delta f}{\alpha} \bigg\|^2_{\H^{\smalltext{2}}_{\smalltext{T}\smalltext{,}\smalltext{\beta}}} = \frac{4(1+\beta\Phi)}{\beta^2}\bigg\| \frac{\delta f}{\alpha} \bigg\|^2_{\H^{\smalltext{2}}_{\smalltext{T}\smalltext{,}\smalltext{\beta}}} = \ff^\Phi(\beta) \bigg\| \frac{\delta f}{\alpha} \bigg\|^2_{\H^{\smalltext{2}}_{\smalltext{T}\smalltext{,}\smalltext{\beta}}}.
\end{equation}

We similarly find
\begin{align}
	\|\alpha \delta y_-\|^2_{\H^\smalltext{2}_{\smalltext{T}\smalltext{,}\smalltext{\beta}}} 
	&= \E \bigg[ \int_0^T \cE(\beta A)_t |\delta y_{t-}|^2 \d A_t \bigg] 
\leq (1+\beta\Phi)\E \bigg[ \int_0^T \cE(\beta A)_{t-}|\delta y_{t-}|^2 \d A_t \bigg] \label{eq::exp_minus_y_minus}\\
	&= (1+\beta\Phi)\E \bigg[ \int_0^T \cE(\beta A)_{t-} |F(t-)|^2 \d A_t \bigg]
	\leq (1+\beta\Phi)\inf_{\gamma \in (0,\beta)}\bigg\{\frac{(1+\gamma\Phi)}{\gamma(\beta - \gamma)}\bigg\} \bigg\|\frac{\delta f}{\alpha}\bigg\|^2_{\mathbb{H}^\smalltext{2}_{\smalltext{T}\smalltext{,}\smalltext{\beta}}}
	= (1+\beta\Phi) \fg^\Phi(\beta) \bigg\|\frac{\delta f}{\alpha}\bigg\|^2_{\mathbb{H}^\smalltext{2}_{\smalltext{T}\smalltext{,}\smalltext{\beta}}}, \nonumber
\end{align}
where we now use the predictable projection together with \Cref{lem::bound_delta_y_s2} and  \cite[Theorem VI.57]{dellacherie1982probabilities}. Assuming for the moment that $f^2 = 0$, we find analogously
\begin{align*}
	\|\alpha y^1\|^2_{\H^\smalltext{2}_{\smalltext{T}\smalltext{,}\smalltext{\beta}}} = \E \bigg[ \int_0^T \cE(\beta A)_t |y^1_t|^2 \alpha^2_t \d C_t \bigg] 
	&= \E \bigg[ \int_0^T \cE(\beta A)_t |y^1_{t}|^2 \d A_t \bigg] \\
	&\leq 3\E\bigg[ \int_0^T \E \bigg[\cE(\beta A)_t|\xi_T|^2 + \cE(\beta A)_t\sup_{u \in [t,\infty]}|\xi^+_u \mathbf{1}_{\{u < T\}}|^2 + \cE(\beta A)_t |F(t)|^2 \bigg| \cG_{t} \bigg] \d A_t \bigg] \\
	&= 3\Bigg(\E \bigg[ \int_0^T \cE(\beta A)_t |\xi_T|^2 \d A_t \bigg] + \E \bigg[ \int_0^T \cE(\beta A)_t \sup_{u \in [t,\infty]}|\xi^+_u \mathbf{1}_{\{u < T\}}|^2 \d A_t \bigg] \\
	&\quad+ \E \bigg[ \int_0^T \cE(\beta A)_t |F(t)|^2 \d A_t \bigg] \Bigg) \\
	&= \frac{3}{\beta}\E \bigg[ |\xi_T|^2 \int_0^T \cE(\beta A)_{t-}(1+\beta\Delta A_t) \d (\beta A)_t \bigg] \\
	&\quad+3\Bigg( \E \bigg[ \int_0^T \cE(\beta A)_t \sup_{u \in [t,\infty]}|\xi^+_u \mathbf{1}_{\{u < T\}}|^2 \d A_t \bigg] + \E \bigg[ \int_0^T \cE(\beta A)_t |F(t)|^2 \d A_t \bigg] \Bigg) \\
	&\leq 3\frac{(1+\beta\Phi)}{\beta}\E \bigg[ |\xi_T|^2 \int_0^T \cE(\beta A)_{t-} \d (\beta A)_t \bigg]\\
	&\quad+ 3\Bigg( \E \bigg[ \int_0^T \cE(\beta A)_t \sup_{u \in [t,\infty]}|\xi^+_u \mathbf{1}_{\{u < T\}}|^2 \d A_t \bigg] +\E \bigg[ \int_0^T \cE(\beta A)_t |F(t)|^2 \d A_t \bigg] \Bigg) \\
	&\leq 3\Bigg(\frac{(1+\beta\Phi)}{\beta}\E \Big[ \cE(\beta A)_T|\xi_T|^2 \Big] + \E \bigg[ \int_0^T \cE(\beta A)_t \sup_{u \in [t,\infty]}|\xi^+_u \mathbf{1}_{\{u < T\}}|^2 \d A_t \bigg] \\
	&\quad+ \E \bigg[ \int_0^T \cE(\beta A)_t |F(t)|^2 \d A_t \bigg] \Bigg) \\
	&\leq 3 \Bigg( \frac{(1+\beta\Phi)}{\beta} \|\xi_T\|^2_{\L^\smalltext{2}_\smalltext{\beta}} + \|\alpha\brs{\xi}\|^2_{\H^\smalltext{2}_{\smalltext{T}\smalltext{,}\smalltext{\beta}}} +  \inf_{\gamma \in (0,\beta)}\bigg\{\frac{(1+\beta\Phi)}{\gamma(\beta - \gamma)}\bigg\} \bigg\| \frac{f^1}{\alpha} \bigg\|^2_{\H^{\smalltext{2}}_{\smalltext{T}\smalltext{,}\smalltext{\beta}}} \Bigg) \\
	&= 3 \Bigg( \frac{(1+\beta\Phi)}{\beta} \|\xi_T\|^2_{\L^\smalltext{2}_\smalltext{\beta}} + \|\alpha\brs{\xi}\|^2_{\H^\smalltext{2}_{\smalltext{T}\smalltext{,}\smalltext{\beta}}} +  \ff^\Phi(\beta) \bigg\| \frac{f^1}{\alpha} \bigg\|^2_{\H^{\smalltext{2}}_{\smalltext{T}\smalltext{,}\smalltext{\beta}}} \Bigg),
\end{align*}
where the second line follows from the optional projection together with \Cref{lem::bound_delta_y_s2} and $(a^2+b^2+c^3) \leq 3(a^2+b^2+c^2)$, the third line follows from \cite[Theorem VI.57]{dellacherie1982probabilities}, and the last line follows from \eqref{eq::stoch_exp_F}.
An analogous argument leads to the $\H^2$-bound on $y^2$. By using predictable projection, we find
\begin{align*}
	\|\alpha y^1_-\|^2_{\H^\smalltext{2}_{\smalltext{T}\smalltext{,}\smalltext{\beta}}} = \E \bigg[ \int_0^T \cE(\beta A)_t |y^1_{t-}|^2 \alpha^2_t \d C_t \bigg] 
	&= \E \bigg[ \int_0^T \cE(\beta A)_t |y^1_{t-}|^2 \d A_t \bigg] \\
	&\leq 3\E\bigg[ \int_0^T \E \bigg[\cE(\beta A)_t\bigg(|\xi_T|^2 +\lim_{t^\prime \uparrow\uparrow t}\sup_{u \in [t^\prime,\infty]}|\xi^+_u \mathbf{1}_{\{u < T\}}|^2 +  |F(t-)|^2\bigg) \bigg| \cG_{t-} \bigg] \d A_t \bigg] \\
	&= 3\Bigg(\E \bigg[ \int_0^T \cE(\beta A)_t |\xi_T|^2 \d A_t \bigg] + \E \bigg[ \int_0^T \cE(\beta A)_t \lim_{t^\prime \uparrow\uparrow t}\sup_{u \in [t^\prime,\infty]}|\xi^+_u \mathbf{1}_{\{u < T\}}|^2 \d A_t \bigg] \\
	&\quad+ \E \bigg[ \int_0^T \cE(\beta A)_t |F(t-)|^2 \d A_t \bigg] \Bigg) \\
	&\leq 3 \Bigg( \frac{(1+\beta\Phi)}{\beta} \|\xi_T\|^2_{\L^\smalltext{2}_\smalltext{\beta}} + \|\alpha\brs{\xi}\|^2_{\H^\smalltext{2}_{\smalltext{T}\smalltext{,}\smalltext{\beta}}} +  (1+\beta\Phi) \inf_{\gamma \in (0,\beta)}\bigg\{\frac{(1+\gamma\Phi)}{\gamma(\beta-\gamma)}\bigg\} \bigg\| \frac{f^1}{\alpha} \bigg\|^2_{\H^{\smalltext{2}}_{\smalltext{T}\smalltext{,}\smalltext{\beta}}} \Bigg) \\
	&= 3 \Bigg( \frac{(1+\beta\Phi)}{\beta} \|\xi_T\|^2_{\L^\smalltext{2}_\smalltext{\beta}} + \|\alpha\brs{\xi}\|^2_{\H^\smalltext{2}_{\smalltext{T}\smalltext{,}\smalltext{\beta}}} +  (1+\beta\Phi) \fg^\Phi(\beta) \bigg\| \frac{f^1}{\alpha} \bigg\|^2_{\H^{\smalltext{2}}_{\smalltext{T}\smalltext{,}\smalltext{\beta}}} \Bigg).
\end{align*}
An analogous argument leads to the $\H^2$-bound on $y^2_-$. 

\medskip
We turn to the bounds on the martingale $\delta\eta$ and first note that for a $\P$--a.s. c\`adl\`ag process $V = (V_t)_{t \in [0,\infty)}$ with $\P$--a.s. non-decreasing paths starting from zero, we have
\begin{align}\label{eq::weighted_norm_V}
	\int_0^T \cE(\beta A)_s \d V_s 
	&= \int_0^T \bigg(1+\int_0^s \cE(\beta A)_{t-}\d(\beta A)_t\bigg) \d V_s \nonumber\\
	&= V_T + \int_0^T \int_0^s \cE(\beta A)_{t-}\beta\d A_t \d V_s \nonumber\\
	&= V_T + \beta \int_{(0,\infty)} \int_{(0,\infty)} \mathbf{1}_{\{0<t\leq s \leq T\}} \cE(\beta A)_{t-} \d A_t \d V_s \nonumber\\
	&= V_T + \beta \int_{(0,\infty)} \int_{(0,\infty)} \mathbf{1}_{\{0<t\leq s \leq T\}} \cE(\beta A)_{t-} \d V_s\d A_t \nonumber\\
	&= V_T + \beta \int_{(0,\infty)} \mathbf{1}_{\{0<t\leq T\}} \cE(\beta A)_{t-} \int_{(0,\infty)} \mathbf{1}_{\{t\leq s \leq T\}} \d V_s\d A_t \nonumber\\
	&= V_T + \beta \int_0^T \cE(\beta A)_{t-} \int_{t-}^T \d V_s\d A_t \nonumber\\
	&= V_T + \beta \int_0^T \cE(\beta A)_{t-} \int_{t-}^T \d V_s\d A_t, \; \text{$\P$--a.s.}
\end{align}
Here the fourth line follows from Tonelli's theorem. Thus, by letting $V \coloneqq \langle\delta\eta\rangle + [\delta k^r] + [\delta k^\ell]$, we get
\begin{align}\label{eq::decompeta}
	&\|\delta\eta\|^2_{\cH^\smalltext{2}_{\smalltext{T}\smalltext{,}\smalltext{\beta}}} + \E\bigg[\int_0^T \cE(\beta A)_s \d [\delta k^r]_s\bigg] + \E\bigg[\int_0^T \cE(\beta A)_s \d [\delta k^\ell]_s\bigg] \nonumber\\
	&\quad= \E\bigg[\int_0^T \cE(\beta A)_s \d \langle\delta\eta \rangle_s\bigg] + \E\bigg[\int_0^T \cE(\beta A)_s \d [\delta k^r]_s\bigg] + \E\bigg[\int_0^T \cE(\beta A)_s \d [\delta k^\ell]_s\bigg] \nonumber\\
	&\quad= \E\big[ \langle \delta\eta \rangle_T + [\delta k^r]_T + [\delta k^\ell]_T \big] + \beta \E \bigg[ \int_0^T \cE(\beta A)_{t-} \int_{t-}^T \d (\langle \delta\eta \rangle + [\delta k^r] + [\delta k^\ell])_s \d A_t \bigg].
\end{align}
We now apply \cite[Theorem VI.57]{dellacherie1982probabilities}, the predictable projection together with \Cref{lem::cond_estimates}, and find
\begin{align*}
	&\E \bigg[ \int_0^T \cE(\beta A)_{t-} \int_{t-}^T \d (\langle \delta\eta \rangle + [\delta k^r] + [\delta k^\ell])_s \d A_t \bigg] \\
	&\leq 2 \E\bigg[ \int_0^T \cE(\beta A)_{t-} \int_{t-}^T |\delta y_s \delta f_s| \d C_s\d A_t \bigg] + \E\bigg[ \int_0^T \cE(\beta A)_{t-} \int_{t-}^T (\delta f_s)^2 \d [C]_s \d A_t \bigg]  - \E\bigg[ \int_0^T \cE(\beta A)_{t-}|\delta y_{t-}|^2 \d A_t \bigg].
\end{align*}
We insert this inequality into \eqref{eq::decompeta} and find, after a rearrangement of the terms, that
\begin{align}\label{eq::inequality_norms_1}
	&\beta \E\bigg[ \int_0^T \cE(\beta A)_{t-}|\delta y_{t-}|^2 \d A_t \bigg] + \|\delta\eta\|^2_{\cH^\smalltext{2}_{\smalltext{T}\smalltext{,}\smalltext{\beta}}} + \E\bigg[\int_0^T \cE(\beta A)_s \d [\delta k^r]_s\bigg] + \E\bigg[\int_0^T \cE(\beta A)_s \d [\delta k^\ell]_s\bigg] \nonumber\\
	&\leq \E\big[ \langle \delta\eta \rangle_T + [\delta k^r]_T + [\delta k^\ell]_T \big] + 2\beta \E\bigg[ \int_0^T \cE(\beta A)_{t-} \int_{t-}^T |\delta y_s \delta f_s| \d C_s\d A_t \bigg] + \beta \E\bigg[ \int_0^T \cE(\beta A)_{t-} \int_{t-}^T (\delta f_s)^2 \d [C]_s \d A_t \bigg].
\end{align}
We now bound the second expectation on the last line. By the Cauchy--Schwartz-inequality and since $A = \int_0^\cdot \alpha^2_s \d C_s$, we find
\begin{align*}
	\int_{t-}^T |\delta y_s \delta f_s|\d C_s 
	&= \int_{t-}^T \big(\cE(\gamma A)^{-1/2}_s|\delta y_s|\alpha_s\big) \bigg(\cE(\gamma A)^{1/2}_s\frac{|\delta f_s|}{\alpha_s}\bigg)\d C_s \\
	&\leq \bigg(\int_{t-}^T \cE(\gamma A)^{-1}_s |\delta y_s|^2 \d A_s\bigg)^{1/2}\bigg(\int_{t-}^T \cE(\gamma A)_s\frac{|\delta f_s|^2}{\alpha^2_s}\d C_s\bigg)^{1/2}.
\end{align*}
Therefore, for $\gamma \in (0,\beta)$, we find
\begin{align}\label{eq6}
	&\E\bigg[ \int_0^T \cE(\beta A)_{t-}\int_{t-}^T |\delta y_s \delta f_s|\d C_s \d A_t \bigg] \nonumber \\
	&\leq \E\bigg[ \int_0^T \cE(\beta A)^{1/2}_{t-}\cE(\gamma A)^{1/2}_{t-}\bigg(\int_{t-}^T \cE(\gamma A)^{-1}_s |\delta y_s|^2 \d A_s\bigg)^{1/2} \cE(\beta A)^{1/2}_{t-}\frac{1}{\cE(\gamma A)^{1/2}_{t-}}\bigg(\int_{t-}^T \cE(\gamma A)_s\frac{|\delta f_s|^2}{\alpha^2_s}\d C_s\bigg)^{1/2} \d A_t \bigg] \nonumber \\
	& \leq \E\bigg[ \int_0^T \cE(\beta A)_{t-}\cE(\gamma A)_{t-}\int_{t-}^T \cE(\gamma A)^{-1}_s |\delta y_s|^2 \d A_s \d A_t \bigg]^{1/2} \E\bigg[ \gamma \frac{1}{\gamma} \int_0^T\cE(\beta A)_{t-}\frac{1}{\cE(\gamma A)_{t-}}\int_{t-}^T \cE(\gamma A)_s\frac{|\delta f_s|^2}{\alpha^2_s}\d C_s \d A_t \bigg]^{1/2} \nonumber\\
	& \leq \E\bigg[ \int_0^T \cE(\beta A)_{t-}\cE(\gamma A)_{t-}\int_{t-}^T \cE(\gamma A)^{-1}_s |\delta y_s|^2 \d A_s \d A_t \bigg]^{1/2} \E\bigg[\gamma \frac{(1+\gamma\Phi)}{\gamma(\beta-\gamma)} \int_0^T \cE(\beta A)_s \frac{|\delta f_s|^2}{\alpha^2_s}\d C_s\bigg]^{1/2} \nonumber\\
	& = \E\bigg[ \int_0^T \cE(\beta A)_{t-}\cE(\gamma A)_{t-}\int_{t-}^T \cE(\gamma A)^{-1}_s |\delta y_s|^2 \d A_s \d A_t \bigg]^{1/2} \bigg(\frac{(1+\gamma\Phi)}{(\beta-\gamma)}\bigg)^{1/2} \bigg\|\frac{\delta f}{\alpha} \bigg\|_{\H^\smalltext{2}_{\smalltext{T}\smalltext{,}\smalltext{\beta}}}.
\end{align}
Here the third line follows from the Cauchy--Schwarz inequality, and the fourth line follows from the equalities starting on the second line of \eqref{eq::integral_F_minus} together with \eqref{eq::widehat_A_integral}. Next, note that
\begin{align}\label{eq7}
	&\int_0^T \cE(\beta A)_{t-}\cE(\gamma A)_{t-}\int_{t-}^T \cE(\gamma A)^{-1}_s |\delta y_s|^2 \d A_s \d A_t \nonumber\\
	&\quad = \int_{(0,\infty)}\int_{(0,\infty)} \mathbf{1}_{\{t \leq s \leq T\}} \cE(\beta A)_{t-}\cE(\gamma A)_{t-} \cE(\gamma A)^{-1}_s |\delta y_s|^2 \d A_s \d A_t \nonumber\\
	&\quad = \int_{(0,\infty)}\int_{(0,\infty)} \mathbf{1}_{\{t \leq s \leq T\}} \cE(\beta A)_{t-}\cE(\gamma A)_{t-} \cE(\gamma A)^{-1}_s |\delta y_s|^2 \d A_t \d A_s \nonumber\\
	&\quad = \int_0^T \cE(\gamma A)^{-1}_s |\delta y_s|^2 \int_0^s \cE(\beta A)_{t-}\cE(\gamma A)_{t-} \d A_t \d A_s \nonumber\\
	&\quad = \int_0^T \cE(\gamma A)^{-1}_s |\delta y_s|^2 \int_0^s \cE\big(\underbrace{(\beta+\gamma) A + \beta\gamma[A]}_{\displaystyle\overline A \coloneqq}\big)_{t-} \d A_t \d A_s \nonumber\\
	&\quad = \int_0^T \cE(\gamma A)^{-1}_s |\delta y_s|^2 \bigg( \int_0^s \cE(\overline A)_{t-} \d A^c_t + \sum_{0 < t \leq s} \cE(\overline A)_{t-} \Delta A_t \bigg) \d A_s \nonumber\\
	&\quad = \int_0^T \cE(\gamma A)^{-1}_s |\delta y_s|^2 \bigg( \int_0^s \cE(\overline A)_{t-} \frac{1}{(\beta + \gamma)} \d\overline A^c_t + \sum_{0 < t \leq s} \cE(\overline A)_{t-} \frac{1}{(\beta + \gamma + \beta\gamma\Delta A_t)} \Delta \overline A_t \bigg) \d A_s \nonumber\\
	&\quad = \int_0^T \cE(\gamma A)^{-1}_s |\delta y_s|^2 \bigg( \int_0^s \cE(\overline A)_{t-} \frac{1}{(\beta + \gamma + \beta\gamma\Delta A_t)} \d\overline A^c_t + \sum_{0 < t \leq s} \cE(\overline A)_{t-} \frac{1}{(\beta + \gamma + \beta\gamma\Delta A_t)} \Delta \overline A_t \bigg) \d A_s \nonumber\\
	&\quad = \int_0^T \cE(\gamma A)^{-1}_s |\delta y_s|^2  \int_0^s \cE(\overline A)_{t-} \frac{1}{(\beta + \gamma + \beta\gamma\Delta A_t)} \d\overline A_t \d A_s \nonumber\\
	&\quad \leq \int_0^T \cE(\gamma A)^{-1}_s |\delta y_s|^2 \frac{1}{(\beta + \gamma)} \int_0^s \cE(\overline A)_{t-} \d\overline A_t \d A_s \nonumber\\
	&\quad \leq \frac{1}{\beta+\gamma} \int_0^T \cE(\gamma A)^{-1}_s |\delta y_s|^2 \cE(\overline A)_{s} \d A_s \nonumber\\
	&\quad = \frac{1}{\beta+\gamma} \int_0^T \cE(\gamma A)^{-1}_s |\delta y_s|^2 \cE(\beta A)_s \cE(\gamma A)_s \d A_s \nonumber\\
	&\quad = \frac{1}{\beta+\gamma} \int_0^T \cE(\beta A)_s |\delta y_s|^2  \d A_s.
\end{align}
Here the seventh line follows from $\overline A^c = (\beta + \gamma) A^c$ and $\Delta \overline A = (\beta + \gamma)\Delta A + \beta\gamma(\Delta A)^2 = (\beta + \gamma + \beta\gamma\Delta A)\Delta A$, the eigth line follows from the fact that $\d\overline A^c$ puts no mass on the points $s \in (0,\infty)$ where $\Delta A_s \neq 0$, and the tenth line follows from $\beta + \gamma \leq \beta + \gamma + \beta\gamma \Delta A$. This then yields
\begin{align}\label{eq8}
	&\E\bigg[ \int_0^T \cE(\beta A)_{t-}\cE(\gamma A)_{t-}\int_{t-}^T \cE(\gamma A)^{-1}_s |\delta y_s|^2 \d A_s \d A_t \bigg]^{1/2} \leq \bigg(\frac{1}{\beta+\gamma}\bigg)^{1/2}\E\bigg[\int_0^T\cE(\beta A)_s |\delta y_s|^2 \d A_s\bigg]^{1/2} \nonumber\\
	&\quad = \bigg(\frac{1}{\beta + \gamma}\bigg)^{1/2} \|\alpha\delta y\|_{\H^\smalltext{2}_{\smalltext{T}\smalltext{,}\smalltext{\beta}}} 
	\leq \bigg(\frac{1}{\beta + \gamma}\bigg)^{1/2} \big(\ff^\Phi(\beta)\big)^{1/2}\bigg\|\frac{\delta f}{\alpha}\bigg\|_{\H^\smalltext{2}_{\smalltext{T}\smalltext{,}\smalltext{\beta}}}
	= \bigg(\frac{1}{\beta + \gamma}\bigg)^{1/2} \bigg(\frac{4(1+\beta\Phi)}{\beta^2}\bigg)^{1/2}\bigg\|\frac{\delta f}{\alpha}\bigg\|_{\H^\smalltext{2}_{\smalltext{T}\smalltext{,}\smalltext{\beta}}}.
\end{align}
Here the inequality on the second line follows from \eqref{eq::inequality_delta_y_H_2}. We now combine \Cref{eq6} and \Cref{eq8} and find that
\begin{equation*}
	\E\bigg[ \int_0^T \cE(\beta A)_{t-}\int_{t-}^T |\delta y_s \delta f_s| \d C_s \d A_t \bigg] 
	\leq \bigg(\frac{1}{\beta + \gamma}\bigg)^{1/2} \bigg(\frac{4(1+\beta\Phi)}{\beta^2}\bigg)^{1/2} \bigg(\frac{(1+\gamma\Phi)}{(\beta-\gamma)}\bigg)^{1/2} \bigg\|\frac{\delta f}{\alpha} \bigg\|^2_{\H^\smalltext{2}_{\smalltext{T}\smalltext{,}\smalltext{\beta}}}.
\end{equation*}
Since $\gamma \in (0,\beta)$ was arbitrary, we find
\begin{align}\label{eq9}
	\E\bigg[ \int_0^T \cE(\beta A)_{t-}\int_{t-}^T |\delta y_s \delta f_s|\d C_s \d A_t \bigg] 
	&\leq \Bigg( \inf_{\gamma \in (0,\beta)}\bigg\{\frac{(1+\gamma\Phi)}{(\beta^2-\gamma^2)}\bigg\} \frac{4(1+\beta\Phi)}{\beta^2}\Bigg)^{1/2}\bigg\|\frac{\delta f}{\alpha} \bigg\|^2_{\H^\smalltext{2}_{\smalltext{T}\smalltext{,}\smalltext{\beta}}} \nonumber\\
	&= \bigg(\frac{1}{\beta^2} \frac{4(1+\beta\Phi)}{\beta^2}\bigg)^{1/2}\bigg\|\frac{\delta f}{\alpha} \bigg\|^2_{\H^\smalltext{2}_{\smalltext{T}\smalltext{,}\smalltext{\beta}}} = \frac{2}{\beta^2}(1+\beta\Phi)^{1/2}\bigg\|\frac{\delta f}{\alpha} \bigg\|^2_{\H^\smalltext{2}_{\smalltext{T}\smalltext{,}\smalltext{\beta}}}.
\end{align}

\medskip
We now bound the fourth summand on the second line in \eqref{eq::inequality_norms_1}. Since
\begin{align*}
	\int_{t-}^T(\delta f_s)^2\d[C]_s 
	&= \sum_{t \leq s \leq T}\mathbf{1}_{\{s < \infty\}}(\delta f_s)^2 (\Delta C_s)^2 \nonumber\\
	&= \sum_{t \leq s \leq T}\mathbf{1}_{\{s < \infty\}}(|\delta f_s|\Delta C_s)^2 
	\leq \bigg(\sum_{t \leq s \leq T}\mathbf{1}_{\{s < \infty\}}|\delta f_s|\Delta C_s\bigg)^2 
	\leq \bigg(\int_{t-}^T|\delta f_s|\d C_s\bigg)^2 = |F(t-)|^2,
\end{align*}
it follows from \eqref{eq::stoch_exp_F} that
\begin{equation}\label{eq10}
	\E \bigg[ \int_0^T \cE(\beta A)_{t-}\int_{t-}^T(\delta f_s)^2\d[C]_s \d A_t \bigg] 
	\leq \E \bigg[ \int_0^T \cE(\beta A)_{t-}|F(t-)|^2 \d A_t \bigg]
	\leq \inf_{\gamma \in (0,\beta)}\bigg\{\frac{(1+\gamma\Phi)}{\gamma(\beta-\gamma)}\bigg\} \bigg\|\frac{\delta f}{\alpha}\bigg\|^2_{\H^\smalltext{2}_{\smalltext{T}\smalltext{,}\smalltext{\beta}}}
\end{equation}
by arbitrariness of $\gamma \in (0,\beta)$.

\medskip
We now turn to the first expectation on the second line in \eqref{eq::inequality_norms_1}. We find
\begin{align}\label{eq12}
	\E \big[ \langle \delta\eta \rangle_T + [\delta k^r]_T + [\delta k^\ell]_T \big] 
	&\leq -\E\big[(\delta y_0)^2\big] + 2 \E\bigg[\int_0^T\delta y_{s}\delta f_s \d C_s\bigg]  + \E\bigg[\int_0^T(\delta f_s)^2\d[C]_s\bigg] \nonumber\\
	&\leq -\E\big[(\delta y_0)^2\big] + 2 \E\bigg[\sup_{s \in [0,T]}|\delta y_s|\int_0^T|\delta f_s| \d C_s\bigg]  + \E\bigg[\bigg(\int_0^T|\delta f_s|\d C_s\bigg)^2\bigg] \nonumber\\
	&\leq -\E\big[(\delta y_0)^2\big] + 2 \E\Big[\sup_{s \in [0,T]}|\delta y_s|^2\Big]^{1/2}\E\bigg[\bigg(\int_0^T |\delta f_s|\d C_s\bigg)^2\bigg]^{1/2} + \E\bigg[\bigg(\int_0^T|\delta f_s|\d C_s\bigg)^2\bigg] \nonumber\\
	&\leq -\E\big[(\delta y_0)^2\big] + 2 \frac{2}{\beta^{1/2}} \bigg\|\frac{\delta f}{\alpha}\bigg\|_{\H^\smalltext{2}_{\smalltext{T}\smalltext{,}\smalltext{\beta}}}  \frac{1}{\beta^{1/2}}\bigg\|\frac{\delta f}{\alpha}\bigg\|_{\H^\smalltext{2}_{\smalltext{T}\smalltext{,}\smalltext{\beta}}} + \frac{1}{\beta} \bigg\|\frac{\delta f}{\alpha} \bigg\|^2_{\H^\smalltext{2}_{\smalltext{T}\smalltext{,}\smalltext{\beta}}} \nonumber\\
	&= -\E\big[(\delta y_0)^2\big] + \frac{5}{\beta} \bigg\|\frac{\delta f}{\alpha}\bigg\|^2_{\H^\smalltext{2}_{\smalltext{T}\smalltext{,}\smalltext{\beta}}}
\end{align}
Here, the second-to-last line follows from \eqref{eq::S2T_bounds_proof}. We substitute \eqref{eq12}, \eqref{eq10} and \eqref{eq9} into \eqref{eq::inequality_norms_1} and find
\begin{align*}
	&\E\big[(\delta y_0)^2\big] + \frac{\beta}{(1+\beta\Phi)}\|\alpha\delta y_{-}\|^2_{\H^\smalltext{2}_{\smalltext{T}\smalltext{,}\smalltext{\beta}}} + \|\delta\eta\|^2_{\cH^\smalltext{2}_{\smalltext{T}\smalltext{,}\smalltext{\beta}}} + \E\bigg[\int_0^T \cE(\beta A)_s \d[\delta k^r]_s\bigg] + \E\bigg[\int_0^T \cE(\beta A)_s \d[\delta k^\ell]_s\bigg] \\
	&\quad\leq \E\big[(\delta y_0)^2\big] + \beta\E\bigg[\int_0^T\cE(\beta A)_{s-}|\delta y_{s-}|^2 \d A_s\bigg] + \|\delta\eta\|^2_{\cH^\smalltext{2}_{\smalltext{T}\smalltext{,}\smalltext{\beta}}} + \E\bigg[\int_0^T \cE(\beta A)_s \d[\delta k^r]_s\bigg] + \E\bigg[\int_0^T \cE(\beta A)_s \d[\delta k^\ell]_s\bigg] \\
	&\quad\leq \bigg( \frac{5}{\beta} + (2\beta)\frac{2}{\beta^2}(1+\beta\Phi)^{1/2} + \beta \inf_{\gamma \in (0,\beta)}\bigg\{\frac{(1+\gamma\Phi)}{\gamma(\beta-\gamma)}\bigg\}  \bigg)\bigg\|\frac{\delta f}{\alpha}\bigg\|^2_{\H^\smalltext{2}_{\smalltext{T}\smalltext{,}\smalltext{\beta}}} \\
	&\quad= \bigg( \frac{5}{\beta} + \frac{4}{\beta}(1+\beta\Phi)^{1/2} + \beta \fg^\Phi(\beta) \bigg)\bigg\|\frac{\delta f}{\alpha}\bigg\|^2_{\H^\smalltext{2}_{\smalltext{T}\smalltext{,}\smalltext{\beta}}}.
\end{align*}
In particular, this implies
\begin{align*}
	\min\bigg\{1,\frac{\beta}{(1+\beta\Phi)}\bigg\}\Big(\|\alpha\delta y_{-}\|^2_{\H^\smalltext{2}_{\smalltext{T}\smalltext{,}\smalltext{\beta}}} + \|\delta\eta\|^2_{\cH^\smalltext{2}_{\smalltext{T}\smalltext{,}\smalltext{\beta}}}\Big)
	&\leq \frac{\beta}{(1+\beta\Phi)}\|\alpha\delta y_{-}\|^2_{\H^\smalltext{2}_{\smalltext{T}\smalltext{,}\smalltext{\beta}}} + \|\delta\eta\|^2_{\cH^\smalltext{2}_{\smalltext{T}\smalltext{,}\smalltext{\beta}}} \\
	&\leq \bigg(\frac{5}{\beta} +\frac{4}{\beta}(1+\beta\Phi)^{1/2} + \beta \fg^\Phi(\beta)\bigg)\bigg\|\frac{\delta f}{\alpha}\bigg\|^2_{\H^\smalltext{2}_{\smalltext{T}\smalltext{,}\smalltext{\beta}}},
\end{align*}
and therefore
\begin{align*}
	&\|\delta y\|^2_{\cS^\smalltext{2}_\smalltext{T}} + \|\alpha\delta y_{-}\|^2_{\H^\smalltext{2}_{\smalltext{T}\smalltext{,}\smalltext{\beta}}} + \|\delta\eta\|^2_{\cH^\smalltext{2}_{\smalltext{T}\smalltext{,}\smalltext{\beta}}} \\ 
	&\quad\leq \Bigg(\frac{4}{\beta} + \max\bigg\{1,\frac{(1+\beta\Phi)}{\beta}\bigg\}\bigg(\frac{5}{\beta} +\frac{4}{\beta}(1+\beta\Phi)^{1/2} + \beta \fg^\Phi(\beta)\bigg)\Bigg)\bigg\|\frac{\delta f}{\alpha}\bigg\|^2_{\H^\smalltext{2}_{\smalltext{T}\smalltext{,}\smalltext{\beta}}} 
	= M^\Phi_3(\beta)\bigg\|\frac{\delta f}{\alpha}\bigg\|^2_{\H^\smalltext{2}_{\smalltext{T}\smalltext{,}\smalltext{\beta}}}.
\end{align*}
The inequalities involving $M^\Phi_1(\beta)$ or $M^\Phi_2(\beta)$ follow immediately.

\medskip
Finally, we turn to the inequality involving $\eta^i$. Similar to before, we choose $V = \langle\delta\eta^i\rangle + [k^{r,i}] + [k^{r,\ell}]$ in \eqref{eq::weighted_norm_V} and find
\begin{align}\label{eq::inequality_eta_i_proof}
	&\|\eta^i\|^2_{\cH^\smalltext{2}_{\smalltext{T}\smalltext{,}\smalltext{\beta}}} + \E\bigg[\int_0^T \cE(\beta A)_s \d [k^{r,i}]_s\bigg] + \E\bigg[\int_0^T \cE(\beta A)_s \d [k^{\ell,i}]_s\bigg] \nonumber\\
	&\quad= \E\big[ \langle \eta^i \rangle_T + [k^{r,i}]_T + [k^{\ell,i}]_T \big] + \beta \E \bigg[ \int_0^T \cE(\beta A)_{t-} \int_{t-}^T \d (\langle\eta^i \rangle + [k^{r,i}] + [k^{\ell,i}])_s \d A_t \bigg].
\end{align}
To bound the first expectation on the last line, we apply \eqref{eq::cond_eta1} and find
\begin{align*}
	\E\big[ \langle \eta^i \rangle_T + [k^{r,i}]_T + [k^{\ell,i}]_T \big] 
	&\leq -\E\big[|y^i_0|^2\big] + L \E\bigg[ |\xi_T|^2 + \sup_{u \in [0,\infty]} |\xi^+_u \mathbf{1}_{\{u < T\}}|^2 + \bigg(\int_0^T |f^i_u| \d C_u\bigg)^2 \bigg] \\
	&\leq -\E\big[|y^i_0|^2\big] + L \|\xi_T\|^2_{\L^\smalltext{2}} + L\|\xi^+\1_{\llbracket 0,T\rrparenthesis}\|^2_{\cS^\smalltext{2}_\smalltext{T}} + \frac{L}{\beta} \bigg\|\frac{f^i}{\alpha}\bigg\|^2_{\H^\smalltext{2}_{\smalltext{T}\smalltext{,}\smalltext{\beta}}},
\end{align*}
where the last inequality follows from \eqref{eq::bound_integral_f}. To bound the second expectation on the last line of \eqref{eq::inequality_eta_i_proof}, we apply \eqref{eq::cond_eta1_pred} and then \eqref{eq::stoch_exp_F} and find
\begin{align*}
	&\E\bigg[ \int_0^T \cE(\beta A)_{t-} \int_{t-}^T \d (\langle\eta^i \rangle + [k^{r,i}] + [k^{\ell,i}])_s \d A_t \bigg] \\
	&\quad \leq - \E\bigg[\int_0^T \cE(\beta A)_{t-} |y^i_{t-}|^2\d A_t\bigg] + L\E\bigg[ \int_0^T \cE(\beta A)_{t-} \bigg(|\xi_T|^2 + |\brs{\xi}_t|^2 +  \bigg(\int_{t-}^T |f^i_s|\d C_s\bigg)^2\bigg) \d A_t \bigg] \\
	&\quad\leq - \E\bigg[\int_0^T \cE(\beta A)_{t-} |y^i_{t-}|^2\d A_t\bigg] +  L\|\xi_T\|^2_{\L^\smalltext{2}_\smalltext{\beta}} + L\| \alpha\brs{\xi}\|^2_{\H^\smalltext{2}_{\smalltext{T}\smalltext{,}\smalltext{\beta}}} + L\inf_{\gamma \in (0,\infty)} \bigg\{\frac{(1+\gamma\Phi)}{\gamma(\beta-\gamma)}\bigg\} \bigg\|\frac{f^i}{\alpha}\bigg\|^2_{\H^\smalltext{2}_{\smalltext{T}\smalltext{,}\smalltext{\beta}}}.
\end{align*}
We now substitute this back into \eqref{eq::inequality_eta_i_proof} and find after a rearrangement of the terms that
\begin{align*}
	&\E\big[|y^i_0|^2\big] + \frac{\beta}{(1+\beta\Phi)}\|\alpha y^i_{-}\|^2_{\H^\smalltext{2}_{\smalltext{T}\smalltext{,}\smalltext{\beta}}} + \|\eta^i\|^2_{\cH^\smalltext{2}_{\smalltext{T}\smalltext{,}\smalltext{\beta}}} + \E\bigg[\int_0^T \cE(\beta A)_s \d[k^{r,i}]_s\bigg] + \E\bigg[\int_0^T \cE(\beta A)_s \d[k^{\ell,i}]_s\bigg] \\
	&\quad\leq \E\big[|y^i_0|^2\big] + \beta\E\bigg[\int_0^T\cE(\beta A)_{s-}|y^i_{s-}|^2 \d A_s\bigg] + \|\eta^i\|^2_{\cH^\smalltext{2}_{\smalltext{T}\smalltext{,}\smalltext{\beta}}} + \E\bigg[\int_0^T \cE(\beta A)_s \d[k^{r,i}]_s\bigg] + \E\bigg[\int_0^T \cE(\beta A)_s \d[k^{\ell,i}]_s\bigg] \\
	&\quad\leq L\Bigg( \|\xi_T\|^2_{\L^\smalltext{2}} + \|\xi^+\1_{\llbracket 0,T\rrparenthesis}\|^2_{\cS^\smalltext{2}_\smalltext{T}} + \frac{1}{\beta} \bigg\|\frac{f^i}{\alpha}\bigg\|^2_{\H^\smalltext{2}_{\smalltext{T}\smalltext{,}\smalltext{\beta}}} + \beta\bigg(\|\xi_T\|^2_{\L^\smalltext{2}_\smalltext{\beta}} + \| \alpha\brs{\xi}\|^2_{\H^\smalltext{2}_{\smalltext{T}\smalltext{,}\smalltext{\beta}}} + \inf_{\gamma \in (0,\infty)} \bigg\{\frac{(1+\gamma\Phi)}{\gamma(\beta-\gamma)}\bigg\} \bigg\|\frac{f^i}{\alpha}\bigg\|^2_{\H^\smalltext{2}_{\smalltext{T}\smalltext{,}\smalltext{\beta}}}\bigg)\Bigg) \\
	&= L\Bigg( \|\xi_T\|^2_{\L^\smalltext{2}} + \|\xi^+\1_{\llbracket 0,T\rrparenthesis}\|^2_{\cS^\smalltext{2}_\smalltext{T}} + \frac{1}{\beta} \bigg\|\frac{f^i}{\alpha}\bigg\|^2_{\H^\smalltext{2}_{\smalltext{T}\smalltext{,}\smalltext{\beta}}} + \beta\bigg(\|\xi_T\|^2_{\L^\smalltext{2}_\smalltext{\beta}} + \| \alpha\brs{\xi}\|^2_{\H^\smalltext{2}_{\smalltext{T}\smalltext{,}\smalltext{\beta}}} + \fg^\Phi(\beta) \bigg\|\frac{f^i}{\alpha}\bigg\|^2_{\H^\smalltext{2}_{\smalltext{T}\smalltext{,}\smalltext{\beta}}}\bigg)\Bigg).
\end{align*}
This completes the proof.
\end{proof}

Although the \emph{a priori} estimates in \Cref{prop::apriori} also hold for non-reflected BSDEs, we obtain sharper bounds by redoing them in this special case. Furthermore, as previously indicated in \Cref{rem::main_bsde}.$(iv)$, the techniques we employ below do not rely on It\^o's formula, and an extension to BSDEs with multi-dimensional generator and terminal condition is straightforward. As before, we use the convention $\zeta_{0-} \coloneqq 0$ for a process $\zeta =(\zeta_t)_{t \in [0,\infty]}$.

\begin{proposition}\label{prop::apriori_bsde}
	Suppose that $\xi = -\infty$ on $[0,T)$. Then
	\begin{align*}
		|\delta y_S|^2 +\E\bigg[\int_S^T \d \langle\delta\eta\rangle_u \bigg| \cG_S\bigg] &= \E\bigg[\bigg(\int_S^T |\delta f_u| \d C_u\bigg)^2\bigg| \cG_S\bigg], \; \text{$\P$\rm--a.s.}, \; S \in \cT_{0,\infty},\\
		|\delta y_{S-}|^2 +\E\bigg[\int_{S-}^T \d \langle\delta\eta\rangle_u \bigg| \cG_{S-}\bigg] &= \E\bigg[\bigg(\int_{S-}^T |\delta f_u| \d C_u\bigg)^2\bigg| \cG_{S-}\bigg], \; \text{$\P$\rm--a.s.}, \; S \in \cT^p_{0,\infty}.
	\end{align*}
	Moreover, for $\beta \in (0,\infty)$,
	\begin{equation*}
		\frac{\beta}{(1+\beta\Phi)}\|\alpha\delta y_{-}\|^2_{\H^\smalltext{2}_{\smalltext{T}\smalltext{,}\smalltext{\beta}}} + \|\delta\eta\|^2_{\cH^\smalltext{2}_{\smalltext{T}\smalltext{,}\smalltext{\beta}}} 
		\leq \bigg(\frac{1}{\beta} + \beta \fg^\Phi(\beta)\bigg) \bigg\|\frac{\delta f}{\alpha}\bigg\|^2_{\H^\smalltext{2}_{\smalltext{T}\smalltext{,}\smalltext{\beta}}},
	\end{equation*}
	and thus\footnote{Recall the definition of $\widetilde M^\Phi_1$, $\widetilde M^\Phi_2$ and $\widetilde M^\Phi_3$ from \Cref{sec::main_results_bsde}.}
	\begin{equation*}
		\|\delta y\|^2_{\cS^\smalltext{2}_\smalltext{T}} + \|\alpha\delta y\|^2_{\H^\smalltext{2}_{\smalltext{T}\smalltext{,}\smalltext{\beta}}} + \|\alpha\delta y_{-}\|^2_{\H^\smalltext{2}_{\smalltext{T}\smalltext{,}\smalltext{\beta}}} + \|\delta\eta\|^2_{\cH^\smalltext{2}_{\smalltext{T}\smalltext{,}\smalltext{\beta}}} 
		\leq \widetilde M^\Phi_1(\beta) \bigg\|\frac{\delta f}{\alpha}\bigg\|^2_{\H^\smalltext{2}_{\smalltext{T}\smalltext{,}\smalltext{\beta}}},
		\;
		\|\alpha\delta y\|^2_{\H^\smalltext{2}_{\smalltext{T}\smalltext{,}\smalltext{\beta}}} + \|\delta\eta\|^2_{\cH^\smalltext{2}_{\smalltext{T}\smalltext{,}\smalltext{\beta}}} 
		\leq \widetilde M^\Phi_2(\beta) \bigg\|\frac{\delta f}{\alpha}\bigg\|^2_{\H^\smalltext{2}_{\smalltext{T}\smalltext{,}\smalltext{\beta}}},
	\end{equation*}
	\begin{equation*}
		\|\delta y\|^2_{\cS^\smalltext{2}_\smalltext{T}} + \|\alpha\delta y_{-}\|^2_{\H^\smalltext{2}_{\smalltext{T}\smalltext{,}\smalltext{\beta}}} + \|\delta\eta\|^2_{\cH^\smalltext{2}_{\smalltext{T}\smalltext{,}\smalltext{\beta}}} 
		\leq \widetilde M^\Phi_3(\beta) \bigg\|\frac{\delta f}{\alpha}\bigg\|^2_{\H^\smalltext{2}_{\smalltext{T}\smalltext{,}\smalltext{\beta}}}.
	\end{equation*}
\end{proposition}

\begin{proof}
	As in the proof of \Cref{prop::apriori}, we slightly abuse notation and denote by $\E[W_\cdot| \cG_\cdot]$ and $\E[W_\cdot|\cG_{\cdot-}]$ the optional and predictable projection, respectively, of a non-negative, measurable process $W = (W_t)_{t \in [0,\infty]}$.
	
	\medskip
	Note first that$ \int_t^T \d(\delta\eta)_s = -\delta y_t + \int_t^T \delta f_s \d C_s = - \E\big[\int_t^T \delta f_s \d C_s \big| \cG_t\big] + \int_t^T \delta f_s \d C_s,$ which implies
	\begin{align}\label{eq::cond_delta_eta_bsde}
		\E\bigg[\int_t^T \d\langle\delta\eta\rangle_s\bigg|\cG_t\bigg] = \E\bigg[\bigg(\int_t^T \d(\delta\eta)_s\bigg)^2\bigg|\cG_t\bigg] &= \E\bigg[\bigg(\int_t^T \delta f_s \d C_s\bigg)^2 \bigg| \cG_t\bigg] - \bigg(\E\bigg[\int_t^T \delta f_s \d C_s \bigg| \cG_t\bigg]\bigg)^2 \nonumber\\
		&= \E\bigg[\bigg(\int_t^T \delta f_s \d C_s\bigg)^2 \bigg| \cG_t\bigg] - (\delta y_t)^2, \; t \in [0,\infty], \; \text{$\P$--a.s.}
	\end{align}
	A similar argument, but now using the predictable projection, implies
	\begin{align}\label{eq::cond_delta_eta_bsde_minus}
		(\delta y_{t-})^2 + \E\bigg[\int_{t-}^T \d\langle\delta\eta\rangle_s\bigg|\cG_{t-}\bigg] = \E\bigg[\bigg(\int_{t-}^T |\delta f_s| \d C_s\bigg)^2 \bigg| \cG_{t-}\bigg], \; t \in [0,\infty], \; \text{$\P$--a.s.}
	\end{align}
	As in the proof of \Cref{prop::apriori}, we have
	\begin{align*}
		\E\bigg[\int_0^T \cE(\beta A)_s \d \langle\delta\eta\rangle_s\bigg] 
		&= \E[\langle\delta\eta\rangle_T] + \beta \E\bigg[\int_0^T \cE(\beta A)_{s-}\int_{t-}^T \d\langle\delta\eta\rangle_s\d A_t\bigg] \\
		&= \E[\langle\delta\eta\rangle_T] + \beta \E\bigg[\int_0^T \cE(\beta A)_{s-}\bigg(\int_{t-}^T |\delta f_s|\d C_s\bigg)^2\d A_t\bigg] - \beta \E\bigg[\int_0^T \cE(\beta A)_{s-}|\delta y_{t-}|^2\d A_t\bigg].
	\end{align*}
	We now rearrange the terms and find for $\gamma \in (0,\beta)$ arbitrary that
	\begin{align*}
		\beta \E\bigg[\int_0^T \cE(\beta A)_{s-}|\delta y_{s-}|^2\d A_s\bigg] + \E\bigg[\int_0^T \cE(\beta A)_s \d \langle\delta\eta\rangle_s\bigg] 
		&= \E[\langle\delta\eta\rangle_T] + \beta \E\bigg[\int_0^T \cE(\beta A)_{s-}\bigg(\int_{t-}^T |\delta f_s|\d C_s\bigg)^2\d A_t\bigg] \\
		&\leq \bigg(\frac{1}{\beta} + \beta\frac{(1+\gamma\Phi)}{\gamma(\beta-\gamma)}\bigg) \bigg\|\frac{\delta f}{\alpha}\bigg\|^2_{\H^\smalltext{2}_{\smalltext{T}\smalltext{,}\smalltext{\beta}}}.
	\end{align*}
	Here the inequality follows from \eqref{eq::F_t_squared}, \eqref{eq::stoch_exp_F}, \eqref{eq::cond_delta_eta_bsde} and \eqref{eq::cond_delta_eta_bsde_minus}. Hence
	\begin{equation*}
		\beta \E\bigg[\int_0^T \cE(\beta A)_{s-}|\delta y_{s-}|^2\d A_s\bigg] + \E\bigg[\int_0^T \cE(\beta A)_s \d \langle\delta\eta\rangle_s\bigg] \leq \bigg(\frac{1}{\beta} + \beta \fg^\Phi(\beta)\bigg) \bigg\|\frac{\delta f}{\alpha}\bigg\|^2_{\H^\smalltext{2}_{\smalltext{T}\smalltext{,}\smalltext{\beta}}}.
	\end{equation*}
	 Since
	\begin{align*}
		\frac{\beta}{(1+\beta\Phi)}\|\alpha\delta y_{-}\|^2_{\H^\smalltext{2}_{\smalltext{T}\smalltext{,}\smalltext{\beta}}} = \frac{\beta}{(1+\beta\Phi)} \E\bigg[\int_0^T \cE(\beta A)_{s}|\delta y_{t-}|^2\d A_t\bigg] 
		&= \frac{\beta}{(1+\beta\Phi)} \E\bigg[\int_0^T \cE(\beta A)_{s-}(1+\beta\Delta A_s)|\delta y_{t-}|^2\d A_t\bigg] \\
		&\leq \beta \E\bigg[\int_0^T \cE(\beta A)_{s-}|\delta y_{t-}|^2\d A_t\bigg],
	\end{align*}
	we thus have
	\begin{equation*}
		\min\bigg\{1,\frac{\beta}{(1+\beta\Phi)}\bigg\} \bigg(\|\alpha\delta y_{-}\|^2_{\H^\smalltext{2}_{\smalltext{T}\smalltext{,}\smalltext{\beta}}} + \|\delta\eta\|^2_{\cH^\smalltext{2}_{\smalltext{T}\smalltext{,}\smalltext{\beta}}}\bigg) 
		\leq \frac{\beta}{(1+\beta\Phi)}\|\alpha\delta y_{-}\|^2_{\H^\smalltext{2}_{\smalltext{T}\smalltext{,}\smalltext{\beta}}} + \|\delta\eta\|^2_{\cH^\smalltext{2}_{\smalltext{T}\smalltext{,}\smalltext{\beta}}} 
		\leq \bigg(\frac{1}{\beta} + \beta \fg^\Phi(\beta)\bigg) \bigg\|\frac{\delta f}{\alpha}\bigg\|^2_{\H^\smalltext{2}_{\smalltext{T}\smalltext{,}\smalltext{\beta}}},
	\end{equation*}
	which implies
	\begin{equation*}
		\|\alpha\delta y_{-}\|^2_{\H^\smalltext{2}_{\smalltext{T}\smalltext{,}\smalltext{\beta}}} + \|\delta\eta\|^2_{\cH^\smalltext{2}_{\smalltext{T}\smalltext{,}\smalltext{\beta}}}  
		\leq \max\bigg\{1,\frac{(1+\beta\Phi)}{\beta}\bigg\}\bigg(\frac{1}{\beta} + \beta \fg^\Phi(\beta)\bigg) \bigg\|\frac{\delta f}{\alpha}\bigg\|^2_{\H^\smalltext{2}_{\smalltext{T}\smalltext{,}\smalltext{\beta}}}
	\end{equation*}
	Together with $\|\delta y\|^2_{\cS^\smalltext{2}_\smalltext{T}} \leq \frac{4}{\beta} \big\|\frac{\delta f}{\alpha}\big\|^2_{\H^\smalltext{2}_{\smalltext{T}\smalltext{,}\smalltext{\beta}}}$ from \Cref{prop::apriori}, 
	we find
	\begin{equation*}
		\|\delta y\|^2_{\cS^\smalltext{2}_\smalltext{T}} + \|\alpha\delta y_{-}\|^2_{\H^\smalltext{2}_{\smalltext{T}\smalltext{,}\smalltext{\beta}}} + \|\delta\eta\|^2_{\cH^\smalltext{2}_{\smalltext{T}\smalltext{,}\smalltext{\beta}}} \leq \bigg(\frac{4}{\beta} + \max\bigg\{1,\frac{(1+\beta\Phi)}{\beta}\bigg\}\bigg(\frac{1}{\beta} + \beta \fg^\Phi(\beta)\bigg)\bigg) \bigg\|\frac{\delta f}{\alpha}\bigg\|^2_{\H^\smalltext{2}_{\smalltext{T}\smalltext{,}\smalltext{\beta}}} = \widetilde M^\Phi_3(\beta)\bigg\|\frac{\delta f}{\alpha}\bigg\|^2_{\H^\smalltext{2}_{\smalltext{T}\smalltext{,}\smalltext{\beta}}}.
	\end{equation*}
	Finally, since by \Cref{prop::apriori} $
		\|\alpha \delta y\|^2_{\H^{\smalltext{2}}_{\smalltext{T}\smalltext{,}\smalltext{\beta}}} 
		\leq \ff^\Phi(\beta) \big\| \frac{\delta f}{\alpha} \big\|^2_{\H^{\smalltext{2}}_{\smalltext{T}\smalltext{,}\smalltext{\beta}}},
$
	we find the remaining two bounds
	\begin{equation*}
		\|\delta y\|^2_{\cS^\smalltext{2}_\smalltext{T}} + \|\alpha \delta y\|^2_{\H^{\smalltext{2}}_{\smalltext{T}\smalltext{,}\smalltext{\beta}}} + \|\alpha\delta y_{-}\|^2_{\H^\smalltext{2}_{\smalltext{T}\smalltext{,}\smalltext{\beta}}} + \|\delta\eta\|^2_{\cH^\smalltext{2}_{\smalltext{T}\smalltext{,}\smalltext{\beta}}} 
		\leq \widetilde M^\Phi_1(\beta) \bigg\| \frac{\delta f}{\alpha} \bigg\|^2_{\H^{\smalltext{2}}_{\smalltext{T}\smalltext{,}\smalltext{\beta}}},\;
		\|\alpha \delta y\|^2_{\H^{\smalltext{2}}_{\smalltext{T}\smalltext{,}\smalltext{\beta}}} + \|\delta\eta\|^2_{\cH^\smalltext{2}_{\smalltext{T}\smalltext{,}\smalltext{\beta}}} 
		\leq \widetilde M^\Phi_2(\beta) \bigg\| \frac{\delta f}{\alpha} \bigg\|^2_{\H^{\smalltext{2}}_{\smalltext{T}\smalltext{,}\smalltext{\beta}}}.
	\end{equation*}
	This completes the proof.
\end{proof}

\begin{remark}\label{rem::apriori_stopping_norm}
	We note here that with {\rm\eqref{eq::cond_bound_delta_y_s}}, {\rm\eqref{eq::bound_integral_delta_f}} and {\rm\eqref{eq::bound_integral_f}}, we find
	\begin{equation*}
		\|\delta y\|^2_{\cT^\smalltext{2}_\smalltext{T}} \leq \frac{1}{\beta} \bigg\|\frac{\delta f}{\alpha}\bigg\|^2_{\H^\smalltext{2}_{\smalltext{T}\smalltext{,}\smalltext{\beta}}} \; \text{and} \;
		\|y^i\|^2_{\cT^\smalltext{2}_\smalltext{T}} 
		\leq 3 \, \Bigg(\|\xi_T\|^2_{\L^\smalltext{2}} + \Big\|\sup_{u \in [0,T)} \xi^+_u\Big\|^2_{\L^\smalltext{2}} +  \frac{1}{\beta}\bigg\|\frac{f^i}{\alpha}\bigg\|^2_{\H^\smalltext{2}_{\smalltext{T}\smalltext{,}\smalltext{\beta}}}\Bigg), \; i \in \{1,2\}.
	\end{equation*}
	Hence, we could substitute the $\cS^2_T$--norm $\|\cdot\|_{\cS^\smalltext{2}_\smalltext{T}}$ with the $\cT^2_T$--norm\footnote{We refer to \Cref{rem::reflection_stopping_norm} for the definition of $\|\cdot\|_{\cT^\smalltext{2}_\smalltext{T}}$.} $\|\cdot\|_{\cT^\smalltext{2}_\smalltext{T}}$ in both {\rm\Cref{prop::apriori}} and {\rm\ref{prop::apriori_bsde}}. This would involve adjusting the constants $M_1^\Phi(\beta)$, $M_3^\Phi(\beta)$, $\widetilde{M}_1^\Phi(\beta)$ and $\widetilde{M}_3^\Phi(\beta)$ in a similar manner as discussed in {\rm\Cref{rem::reflection_stopping_norm}} and {\rm\ref{rem::bsde_stopping_norm}}.
\end{remark}

\begin{remark}\label{rem::discussion_constants}
	Since our existence and uniqueness result is in spirit similar to the one presented in {\rm\cite{papapantoleon2018existence}}, we want to comment here on the constant appearing in their contraction argument. Note that the weights used in {\rm\cite{papapantoleon2018existence}} are exponential functions and not stochastic exponentials as in our setting. {\rm Inequality $(3.24)$} in {\rm\cite{papapantoleon2018existence}} reads
	\begin{equation*}
		\|\delta\eta\|^2_{\cH^\smalltext{2}_{\smalltext{T}\smalltext{,}\smalltext{\beta}}} \leq \beta \E \bigg[ \int_0^T \mathrm{e}^{\beta A_t} \int_{t}^T \d \langle \delta\eta \rangle_s \d A_t \bigg] + \E [ \langle \delta\eta \rangle_T ].
	\end{equation*}
	This differs from our {\rm Inequality \eqref{eq::decompeta}} since $\{t\}$ is not included in the domain of integration of the innermost integral. However, the above inequality is derived by applying Tonelli's theorem to
	\begin{equation*}
		\int_0^T \mathrm{e}^{\beta A_s} \d \langle \delta\eta \rangle_s 
		\leq \beta \int_0^T\int_0^s \mathrm{e}^{\beta A_t} \d A_t \d \langle \delta\eta \rangle_s + \langle \delta\eta \rangle_T,
	\end{equation*} 
	and as we saw in the proofs of this section, changing the order of integration necessitates including $\{t\}$ in the domain of integration of the innermost integral. Consequently, a weighted bound on $F(t-)$ rather than on $F(t)$ $($see \eqref{eq::def_F}$)$ is required. Hence, an additional term of the form $\mathrm{e}^{(\gamma\lor\beta)\Phi}$ should appear in the contraction constant $M^\Phi(\hat\beta)$ in {\rm\cite{papapantoleon2018existence}}. It now seems that no closed-form expression for this contraction constant can be derived. Therefore, one naturally has to resort to employing numerical schemes.
\end{remark}

For the sake completeness, we close this section with the following weighted bound on the increasing processes $(k^r,k^\ell)$. We have refrained from including this estimate in \Cref{prop::apriori_bsde} as we will not need it in the contraction argument. To state the bound, we define the function $\mathfrak{j} : [0,\infty)^3 \longrightarrow [0,\infty)$ by the formula
\begin{equation*}
	\mathfrak{j}(\gamma,\beta,\Psi) \coloneqq \max\Bigg\{\frac{\gamma}{\beta-\gamma}, \frac{(\sqrt{1+\gamma \Psi} - 1)\sqrt{1+\beta \Psi}}{\sqrt{1+\beta \Psi} - \sqrt{1+\gamma \Psi}}\Bigg\} = \max\Bigg\{\frac{\gamma}{\beta-\gamma}, \frac{(\sqrt{1+\gamma \Psi} - 1)\sqrt{1+\beta \Psi}\big(\sqrt{1+\beta \Psi}+\sqrt{1+\gamma \Psi}\big)}{(\beta-\gamma)\Psi}\Bigg\}.
\end{equation*}
Here we use the convention $0 \coloneqq 0/0$. 

\begin{proposition}\label{prop::weighted_k_estimate}
	Let $(\gamma,\beta) \in (0,\infty)^2$ with $\gamma < \beta$. For each $i \in \{1,2\}$,
	\begin{align}\label{eq::weighted_k_inequality}
		\|k^{r,i}\|^2_{\cI^\smalltext{2}_{\smalltext{T}\smalltext{,}\smalltext{\gamma}}} + \|k^{\ell,i}\|^2_{\cI^\smalltext{2}_{\smalltext{T}\smalltext{,}\smalltext{\gamma}}} 
		&\leq 3 \Bigg( \|\eta^i\|^2_{\cH^\smalltext{2}_{\smalltext{T}\smalltext{,}\smalltext{\gamma}}} + 108 \max\{1, \mathfrak{j}(\gamma,\beta,\Phi)\}^2\Big(\|\xi_T\|^2_{\L^\smalltext{2}_\smalltext{\gamma}} + \|\xi^+_\cdot\mathbf{1}_{\{\cdot < T\}}\|^2_{\cS^\smalltext{2}_{\smalltext{T}\smalltext{,}\smalltext{\gamma}}} \Big) \nonumber\\
		&\quad + \big(108 \max\{1, \mathfrak{j}(\gamma,\beta,\Phi)\}^2 + 1\big)\frac{(1+\gamma\Phi)}{\beta-\gamma}\bigg\|\frac{f^i}{\alpha}\bigg\|^2_{\H^\smalltext{2}_{\smalltext{T}\smalltext{,}\smalltext{\beta}}} \Bigg).
	\end{align}
\end{proposition}

\begin{remark}
	To formulate a bound on $\|k^{r,i}\|_{\cI^\smalltext{2}_{\smalltext{T}\smalltext{,}\smalltext{\gamma}}}$ and $\|k^{\ell,i}\|_{\cI^\smalltext{2}_{\smalltext{T}\smalltext{,}\smalltext{\gamma}}}$ solely by terms involving $f^i$ and $\xi^i$, one can combine the previous result with the bound on $\|\eta^i\|_{\cH^\smalltext{2}_{\smalltext{T}\smalltext{,}\smalltext{\gamma}}}$ from {\rm\Cref{prop::apriori}}.
\end{remark}

\begin{proof}
	We fix $i \in \{1,2\}$ and thus drop the superscript $i$ for ease of notation. Moreover, we suppose without loss of generality that the right side of \eqref{eq::weighted_k_inequality} is finite. In a first step, we fix $t \in [0,\infty)$ and apply It\^o's formula for optional semi-martingales to the function $f(x,z) = xz$ for $(x,z) \in \R^2$, see \cite[Theorem A.3]{grigorova2017reflected} or \cite[page 538]{lenglart1980tribus}, and find
	\begin{align}\label{eq::ito_stoch_exp_y}
		\cE(\gamma A)^{1/2}_t y_t 
		&= \cE(\gamma A)^{1/2}_0 y_0 + \int_{(0,t]}y_{s-}\d\cE(\gamma A)^{1/2}_s - \int_{(0,t]} \cE(\gamma A)^{1/2}_{s-} f_s \d C_s + \int_{(0,t]} \cE(\gamma A)_{s-}\d \eta_s - \int_{(0,t]} \cE(\gamma A)^{1/2}_{s-}\d k^r_s \nonumber\\
		&\quad +\sum_{s \in (0,t]} \Big( \cE(\gamma A)^{1/2}_s y_s - \cE(\gamma A)^{1/2}_{s-} y_{s-} - y_{s-}\Delta(\cE(\gamma A)^{1/2})_s - \cE(\gamma A)^{1/2}_{s-}\Delta y_s \Big) \nonumber -\int_{[0,t)}\cE(\gamma A)^{1/2}_s \d k^\ell_s \nonumber\\
		&\quad +\sum_{s \in [0,t)} \Big(\cE(\gamma A)^{1/2}_{s+} y_{s+} - \cE(\gamma A)^{1/2}_s y_s - y_s \big(\cE(\gamma A)^{1/2}_{s+} - \cE(\gamma A)^{1/2}_{s}\big) - \cE(\gamma A)^{1/2}_s(y_{s+} - y_s) \Big) \nonumber\\
		&= \cE(\gamma A)^{1/2}_0 y_0 + \int_{(0,t]}y_{s-}\d\cE(\gamma A)^{1/2}_s - \int_{(0,t]} \cE(\gamma A)^{1/2}_{s-} f_s \d C_s + \int_{(0,t]} \cE(\gamma A)_{s-}\d \eta_s - \int_{(0,t]} \cE(\gamma A)^{1/2}_{s-}\d k^r_s \nonumber\\
		&\quad +\sum_{s \in (0,t]} \Big(y_s(\cE(\gamma A)^{1/2}_s - \cE(\gamma A)^{1/2}_{s-}) - y_{s-}(\cE(\gamma A)^{1/2}_s - \cE(\gamma A)^{1/2}_{s-})\Big) \nonumber - \int_{[0,t)} \cE(\gamma A)^{1/2}_s \d k^\ell_s \nonumber\\
		&= \cE(\gamma A)^{1/2}_0 y_0 + \int_{(0,t]}y_{s-}\d\cE(\gamma A)^{1/2}_s - \int_{(0,t]} \cE(\gamma A)^{1/2}_{s-} f_s \d C_s + \int_{(0,t]} \cE(\gamma A)_{s-}\d \eta_s - \int_{(0,t]} \cE(\gamma A)^{1/2}_{s-}\d k^r_s \nonumber\\
		&\quad +\sum_{s \in (0,t]} \Big((y_s - y_{s-})(\cE(\gamma A)^{1/2}_s - \cE(\gamma A)^{1/2}_{s-})\Big) \nonumber - \int_{[0,t)} \cE(\gamma A)^{1/2}_s \d k^\ell_s \nonumber\\
		&= \cE(\gamma A)^{1/2}_0 y_0 + \int_{(0,t]}y_{s-}\d\cE(\gamma A)^{1/2}_s - \int_{(0,t]} \cE(\gamma A)^{1/2}_{s-} f_s \d C_s + \int_{(0,t]} \cE(\gamma A)_{s-}\d \eta_s - \int_{(0,t]} \cE(\gamma A)^{1/2}_{s-}\d k^r_s \nonumber\\
		&\quad +\sum_{s \in (0,t]} \Big(-f_s\Delta \cE(\gamma A)^{1/2}_s \Delta C_s + \Delta \cE(\gamma A)^{1/2}_s \Delta \eta_s - \Delta \cE(\gamma A)^{1/2}_s \Delta k^r_s\Big) \nonumber - \int_{[0,t)} \cE(\gamma A)^{1/2}_s \d k^\ell_s \nonumber\\
		&= \cE(\gamma A)^{1/2}_0 y_0 + \int_{(0,t]}y_{s-}\d\cE(\gamma A)^{1/2}_s - \int_{(0,t]} \cE(\gamma A)^{1/2}_s f_s \d C_s + \int_{(0,t]} \cE(\gamma A)^{1/2}_s\d \eta_s - \int_{(0,t]} \cE(\gamma A)^{1/2}_s\d k^r_s \nonumber\\
		&\quad - \int_{[0,t)} \cE(\gamma A)^{1/2}_s \d k^\ell_s, \; t \in [0,\infty), \; \text{$\P$--a.s.} 
	\end{align}
	Here we used the fact that $\cE(\gamma A)^{1/2}$ is predictable and non-decreasing, thus locally bounded. We now analyze the terms in \eqref{eq::ito_stoch_exp_y} one by one. First, \Cref{lem::stoch_exp_rules}.$(iii)$ implies that $\cE(\gamma A)^{1/2} = \cE(D^\gamma), \; \text{and} \; \cE(\beta A)^{1/2} = \cE(D^\beta),$ where $D^\gamma = (D^\gamma)_{t \in [0,\infty)}$ and $D^\beta = (D^\beta)_{t \in [0,\infty)}$ are the predictable processes satisfying
	\begin{equation*}
		D^\gamma_t = \frac{\gamma}{2} A^c_t + \sum_{s \in (0,t]} \big(\sqrt{1+\gamma\Delta A_s}-1\big), \; \text{and} \; D^\beta_t = \frac{\beta}{2} A^c_t + \sum_{s \in (0,t]} \big(\sqrt{1+\beta\Delta A_s}-1\big), \; t \in [0,\infty), \; \text{$\P$--a.s.}
	\end{equation*}
	Recall that $A^c$ denotes the continuous part of the process $A$. We bound the second term on the right of \eqref{eq::ito_stoch_exp_y} as follows
	\begin{align}\label{eq::sup_stoch_exp_inverse}
		\int_{(0,t]}y_{s-}\d\cE(\gamma A)^{1/2}_s 
		&= \int_{(0,t]}\cE(\beta A)^{1/2}_{s-}y_{s-}\frac{1}{\cE(\beta A)^{1/2}_{s-}}\d\cE(\gamma A)^{1/2}_s \nonumber\\
		&\leq \sup_{s \in (0,t)}\big\{\cE(\beta A)^{1/2}_{s}|y_{s}|\big\}\int_{(0,t]}\frac{1}{\cE(\beta A)^{1/2}_{s-}}\d\cE(\gamma A)^{1/2}_s \nonumber\\
		&\leq \sup_{s \in (0,t)}\big\{\cE(\beta A)^{1/2}_{s}|y_{s}|\big\}\int_{(0,t]}\frac{1}{\cE(D^\beta)_{s-}}\d\cE(D^\gamma)_s \nonumber\\
		&= \sup_{s \in (0,t)}\big\{\cE(\beta A)^{1/2}_{s}|y_{s}|\big\}\int_{(0,t]}\frac{1}{\cE(D^\beta)_{s-}}\cE(D^\gamma)_{s-}\d D^\gamma_s, \; \text{$\P$--a.s.}
	\end{align}
	By \Cref{lem::stoch_exp_rules}.$(ii)$, we can write $\cE(D^\beta)^{-1}\cE(D^\gamma) = \cE(\widehat D^{\gamma,\beta})$, where $\widehat D^{\gamma,\beta} = (\widehat D^{\gamma,\beta}_t)_{t \in [0,\infty)}$ is the predictable process satisfying
	\begin{equation*}
		\widehat D^{\gamma,\beta}_t = D^{\gamma,c}_t - D^{\beta,c}_t + \sum_{s \in (0,t]} \frac{\Delta D^\gamma_s - \Delta D^\beta_s}{1+\Delta D^\beta_s} = -\frac{(\beta-\gamma)}{2}A^c_t - \sum_{s \in (0,t]} \frac{\sqrt{1+\beta\Delta A_s} - \sqrt{1+\gamma\Delta A_s}}{\sqrt{1+\beta \Delta A_s}}, \; t \in [0,\infty), \; \text{$\P$--a.s.}
	\end{equation*}
	Note that $\Delta \widehat D^{\gamma,\beta} \geq -1$ although $\widehat D^{\gamma,\beta}$ is non-increasing. Then
	\begin{align}\label{eq::stoch_exp_ineq_l}
		\int_{(0,t]}\cE(\widehat D^{\gamma,\beta})_{s-}\d D^\gamma_s &= \int_{(0,t]} \cE(\widehat D^{\gamma,\beta})_{s-} \d D^{\gamma,c}_s + \sum_{s \in (0,t]} \cE(\widehat D^{\gamma,\beta})_{s-} \Delta D^\gamma_s \nonumber\\
		&= \frac{\gamma}{2}\int_{(0,t]} \cE(\widehat D^{\gamma,\beta})_{s-} \d A^c_s + \sum_{s \in (0,t]} \cE(\widehat D^{\gamma,\beta})_{s-} \big(\sqrt{1+\gamma\Delta A_s} - 1\big) \nonumber\\
		&= \frac{\gamma}{\beta-\gamma}\int_{(0,t]} \cE(\widehat D^{\gamma,\beta})_{s-} \frac{(\beta-\gamma)}{2} \d A^c_s + \sum_{s \in (0,t]} \cE(\widehat D^{\gamma,\beta})_{s-} \frac{(\sqrt{1+\gamma\Delta A_s} - 1)\sqrt{1+\beta\Delta A_s}}{\sqrt{1+\beta\Delta A_s} - \sqrt{1+\gamma \Delta A_s}} \Delta (-\widehat D^{\gamma,\beta})_s \nonumber\\
		&= \frac{\gamma}{\beta-\gamma}\int_{(0,t]} \cE(\widehat D^{\gamma,\beta})_{s-} \d (-\widehat D^{\gamma,c})_s + \sum_{s \in (0,t]} \cE(\widehat D^{\gamma,\beta})_{s-} \frac{(\sqrt{1+\gamma\Delta A_s} - 1)\sqrt{1+\beta\Delta A_s}}{\sqrt{1+\beta\Delta A_s} - \sqrt{1+\gamma \Delta A_s}} \Delta(- \widehat D^{\gamma,\beta})_s \nonumber\\
		&\leq \mathfrak{l}(\gamma,\beta) \int_{(0,t]} \cE(\widehat D^{\gamma,\beta})_{s-} \d (-\widehat D^\gamma)_s = \mathfrak{l}(\gamma,\beta) \bigg(1-\frac{\cE(D^\gamma)_t}{\cE(D^\beta)_t}\bigg) \leq \mathfrak{l}(\gamma,\beta), \; \text{$\P$--a.s.},
	\end{align}
	where the last inequality follows from $1 \leq \cE(D^\gamma) \leq \cE(D^\beta)$, and where
	\begin{equation*}
		\mathfrak{l}(\gamma,\beta) \coloneqq \max\Bigg\{\frac{\gamma}{\beta-\gamma}, \sup_{s \in [0,\infty)} \frac{(\sqrt{1+\gamma\Delta A_s} - 1)\sqrt{1+\beta\Delta A_s}}{\sqrt{1+\beta\Delta A_s} - \sqrt{1+\gamma \Delta A_s}}\Bigg\}. 
	\end{equation*}
	Let $[0,\infty) \ni x \longmapsto \mathfrak{k}(\gamma,\beta,x) \in [0,\infty)$ be defined by
	\begin{equation*}
		\mathfrak{k}(\gamma,\beta,x) \coloneqq  \frac{(\sqrt{1+\gamma x} - 1)\sqrt{1+\beta x}}{\sqrt{1+\beta x} - \sqrt{1+\gamma x}} = \frac{(\sqrt{1+\gamma x} - 1)\sqrt{1+\beta x}(\sqrt{1+\beta x} + \sqrt{1+\gamma x})}{(\beta-\gamma)x}.
	\end{equation*}
	Then
	\begin{align*}
		\frac{\partial}{\partial x} \mathfrak{k}(\gamma,\beta,x) 
		&= \frac{\beta  \gamma  x^2 (\sqrt{1+\beta x}+\sqrt{1+\gamma x})}{2 x^2 (\beta -\gamma ) \sqrt{1+\beta x} \sqrt{1+\gamma x}} \\
		&\quad +\frac{x (\beta +\gamma -\beta\sqrt{1+\gamma x}-\gamma  \sqrt{1+\beta x})+2 (\sqrt{1+\beta x}-1) (\sqrt{1+\gamma x}-1)}{2 x^2 (\beta -\gamma ) \sqrt{1+\beta x} \sqrt{1+\gamma x}},
	\end{align*}
	and by using $\sqrt{1+x} \leq 1+x/2$, for $x \geq -1$, we deduce $\frac{\partial}{\partial x} \mathfrak{k}(\gamma,\beta,x) \geq 0$. This implies that $x \longmapsto \mathfrak{k}(\gamma,\beta,x)$ is non-decreasing, and therefore
	\begin{equation}\label{eq::bound_l_by_j}
		\mathfrak{l}(\gamma,\beta) 
		\leq \max\bigg\{ \frac{\gamma}{\beta-\gamma}, \frac{(\sqrt{1+\gamma \Phi} - 1)\sqrt{1+\beta \Phi}}{\sqrt{1+\beta \Phi} - \sqrt{1+\gamma \Phi}} \bigg\} = \mathfrak{j}(\gamma,\beta,\Phi), \; \text{$\P$--a.s.}
	\end{equation}
	We deduce from \eqref{eq::bound_l_by_j}, \eqref{eq::stoch_exp_ineq_l} and \eqref{eq::sup_stoch_exp_inverse} that 
	\begin{equation}\label{eq::ineq_int_y_minus_stoch_exp}
		\int_{(0,t]}y_{s-}\d\cE(\gamma A)^{1/2}_s \leq \mathfrak{j}(\gamma,\beta,\Phi) \sup_{s \in (0,t)}\big\{\cE(\beta A)^{1/2}_{s}|y_{s}|\big\}, \; \text{$\P$--a.s.}
	\end{equation}
	
	\medskip
	We turn to the third term on the right of \eqref{eq::ito_stoch_exp_y}. Cauchy--Schwarz's inequality and \Cref{lem::stoch_exp_rules}.$(ii)$ yields
	\begin{align}\label{eq::cs_exp_root_f}
		\bigg(\int_{(0,t]} \cE(\gamma A)^{1/2}_s |f_s| \d C_s\bigg)^2 
		&\leq \bigg(\int_{(0,t]} \frac{1}{\cE(\beta A)_s}\cE(\gamma A)_s \d A_s\bigg)\bigg(\int_{(0,t]} \cE(\beta A)_s \frac{|f_s|^2}{\alpha^2_s} \d C_s\bigg) \nonumber\\
		&= \bigg(\int_{(0,t]} \cE(\widehat A^{\gamma,\beta})_s\d A_s\bigg)\bigg(\int_{(0,t]} \cE(\beta A)_s \frac{|f_s|^2}{\alpha^2_s} \d C_s\bigg), \; \text{$\P$--a.s.},
	\end{align}
	where $\widehat A^{\gamma,\beta}$ is the predictable process satisfying
	\begin{equation*}
		\widehat A^{\gamma,\beta}_t 
		=  -(\beta-\gamma)A^c_t - \sum_{s \in (0,t]}(\beta-\gamma)\frac{\Delta A_s}{1+\beta\Delta A_s}, \; t \in [0,\infty), \; \text{$\P$--a.s.}
	\end{equation*}
	We now explicitly bound the integral
	\begin{align}\label{eq::bound_inverse_root_exp}
		\int_{(0,t]} \cE(\widehat A^{\gamma,\beta})_s\d A_s 
		&= \int_{(0,t]} \cE(\widehat A^{\gamma,\beta})_s\d A^c_s + \sum_{s \in (0,t]}\cE(\widehat A^{\gamma,\beta})_s\Delta A_s \nonumber\\
		&= \frac{1}{(\gamma-\beta)}\int_{(0,t]} \cE(\widehat A^{\gamma,\beta})_s(1+\beta\Delta A_s) (\gamma-\beta)\d A^c_s + \frac{1}{(\gamma-\beta)}\sum_{s \in (0,t]}\cE(\widehat A^{\gamma,\beta})_s (1+\beta\Delta A_s) (\gamma-\beta)\frac{\Delta A_s}{1+\beta\Delta A_s} \nonumber\\
		&= \frac{1}{(\gamma-\beta)}\int_{(0,t]} \cE(\widehat A^{\gamma,\beta})_s(1+\beta\Delta A_s) \d\big(\widehat A^{\gamma,\beta}\big)^c_s + \frac{1}{(\gamma-\beta)}\sum_{s \in (0,t]}\cE(\widehat A^{\gamma,\beta})_s (1+\beta\Delta A_s) \Delta \widehat A^{\gamma,\beta}_s \nonumber\\
		&= \frac{1}{(\gamma-\beta)}\int_{(0,t]} \cE(\widehat A^{\gamma,\beta})_s(1+\beta\Delta A_s) \d \widehat A^{\gamma,\beta}_s \nonumber\\
		&= \frac{1}{(\gamma-\beta)}\int_{(0,t]} \cE(\widehat A^{\gamma,\beta})_{s-}(1+\Delta\widehat A^{\gamma,\beta}_s)(1+\beta\Delta A_s) \d \widehat A^{\gamma,\beta}_s \nonumber\\
		&= \frac{1}{(\gamma-\beta)}\int_{(0,t]} \cE(\widehat A^{\gamma,\beta})_{s-}(1+\gamma\Delta A_s) \d \widehat A^{\gamma,\beta}_s \nonumber\\
		&\leq \frac{(1+\gamma\Phi)}{(\gamma-\beta)}\int_{(0,t]} \cE(\widehat A^{\gamma,\beta})_{s-} \d \widehat A^{\gamma,\beta}_s \nonumber\\
		&= \frac{(1+\gamma\Phi)}{(\gamma-\beta)}\big(\cE(\widehat A^{\gamma,\beta})_t - 1\big) = \frac{(1+\gamma\Phi)}{(\beta-\gamma)}\bigg(1-\cE(\gamma A)_t\frac{1}{\cE(\beta A)_t}\bigg) \leq \frac{(1+\gamma\Phi)}{(\beta-\gamma)}, \; \text{$\P$--a.s.}
	\end{align}
	Here the second to last line follows from $\cE(\widehat A^{\gamma,\beta}) = \cE(\gamma A)/\cE(\beta A) > 0$ and from the fact that $\widehat A^{\gamma,\beta} / (\gamma-\beta)$ is non-decreasing since $\gamma-\beta < 0$ and since $\widehat A^{\gamma,\beta}$ is non-increasing, and the last line follows from $0 < \cE(\gamma A) \leq \cE(\beta A)$. Combining \eqref{eq::bound_inverse_root_exp} and \eqref{eq::cs_exp_root_f} yields 
	\begin{equation}\label{eq::ineq_exp_root_f}
		\bigg(\int_{(0,t]} \cE(\gamma A)^{1/2}_s |f_s| \d C_s\bigg)^2 
		\leq \frac{(1+\gamma\Phi)}{(\beta-\gamma)} \bigg(\int_{(0,t]} \cE(\beta A)_s \frac{|f_s|^2}{\alpha^2_s} \d C_s\bigg), \; \text{$\P$--a.s.}
	\end{equation}
	Combining \eqref{eq::ineq_int_y_minus_stoch_exp} and \eqref{eq::ito_stoch_exp_y}, then rearranging the terms and applying $(a+b+c)^2 \leq 3(a^2+b^2+c^2)$ yields
	\begin{align}\label{eq::estimate_k}
		&\bigg(\int_{(0,t]}\cE(\gamma A)^{1/2}_s \d k^r_s + \int_{[0,t)} \cE(\gamma A)^{1/2}_s \d k^\ell_s\bigg)^2 \nonumber\\
		&\leq 3\Bigg( 9\max\{1, \mathfrak{j}(\gamma,\beta,\Phi)\}^2 \sup_{s \in [0,t]}\{\cE(\gamma A)|y_t|^2\} + \bigg(\int_0^t \cE(\gamma A)^{1/2}_s |f_s| \d C_s\bigg)^2 + \bigg(\int_0^t \cE(\gamma A)^{1/2}_s \d\eta_s\bigg)^2\Bigg), \; \text{$\P$--a.s.}
	\end{align}
	We now plug \Cref{eq::ineq_exp_root_f} into \Cref{eq::estimate_k}, let $t \uparrow\uparrow \infty$ in the resulting inequality and then take the expectation and find\footnote{We refer to \Cref{sec::formulation} for the conventions we agreed upon when writing an integral of the form $\int_0^T$.}
	\begin{align*}
		\|k^r\|^2_{\cI^\smalltext{2}_{\smalltext{T}\smalltext{,}\smalltext{\gamma}}} + \|k^\ell\|^2_{\cI^\smalltext{2}_{\smalltext{T}\smalltext{,}\smalltext{\gamma}}} 
		&\leq \E\bigg[\bigg(\int_0^T\cE(\gamma A)^{1/2}_s \d k^r_s + \int_0^T \cE(\gamma A)^{1/2}_s \d k^\ell_s\bigg)^2\bigg] \\
		&\leq 3\bigg(9\max\{1,\mathfrak{j}(\gamma,\beta,\Phi)\}^2\|y\|^2_{\cS^\smalltext{2}_{\smalltext{T}\smalltext{,}\smalltext{\gamma}}} + \|\eta\|^2_{\cH^\smalltext{2}_{\smalltext{T}\smalltext{,}\smalltext{\gamma}}} + \frac{(1+\gamma\Phi)}{(\beta-\gamma)}\bigg\|\frac{f}{\alpha}\bigg\|^2_{\H^\smalltext{2}_{\smalltext{T}\smalltext{,}\smalltext{\beta}}}  \bigg).
	\end{align*}
	
	It remains to bound $\|y\|^2_{\cS^\smalltext{2}_{\smalltext{T}\smalltext{,}\smalltext{\gamma}}}$ by terms involving $f$ and $\xi$. \Cref{lem::bound_delta_y_s2} implies
	\begin{align*}
		\cE(\gamma A)^{1/2}_S |y_S| 
		&\leq \E\bigg[ |\cE(\gamma A)^{1/2}_T\xi_T| + \sup_{u \in [0,\infty]}|\cE(\gamma A)^{1/2}_u \xi^+_u\mathbf{1}_{\{u < T\}}| + \int_0^T\cE(\gamma A)^{1/2}_u |f_u|\d C_u \bigg| \cG_S\bigg] \\
		&\leq \sqrt{3}\E\Bigg[ \sqrt{|\cE(\gamma A)^{1/2}_T\xi_T|^2 + \sup_{u \in [0,\infty]}|\cE(\gamma A)^{1/2}_u \xi^+_u\mathbf{1}_{\{u < T\}}|^2 + \bigg(\int_0^T\cE(\gamma A)^{1/2}_u |f_u|\d C_u\bigg)^2} \Bigg| \cG_S\Bigg], \; \text{$\P$--a.s.}, \; S \in \cT_{0,T},
	\end{align*}
	and Doob's $\L^2$-inequality for martingales leads to
	\begin{equation*}
		\|y\|^2_{\cS^\smalltext{2}_{\smalltext{T}\smalltext{,}\smalltext{\gamma}}} \leq 12 \bigg( \|\xi_T\|^2_{\L^\smalltext{2}_\smalltext{\gamma}} + \|\xi^+_\cdot\mathbf{1}_{\{\cdot < T\}}\|^2_{\cS^\smalltext{2}_{\smalltext{T}\smalltext{,}\smalltext{\gamma}}} + \frac{(1+\gamma\Phi)}{(\beta-\gamma)}\bigg\|\frac{f}{\alpha}\bigg\|^2_{\H^\smalltext{2}_{\smalltext{T}\smalltext{,}\smalltext{\beta}}} \bigg).
	\end{equation*}
	This yields
	\begin{align*}
		\|k^r\|^2_{\cI^\smalltext{2}_{\smalltext{T}\smalltext{,}\smalltext{\gamma}}} + \|k^\ell\|^2_{\cI^\smalltext{2}_{\smalltext{T}\smalltext{,}\smalltext{\gamma}}} 
		&\leq 3 \Bigg( \|\eta\|^2_{\cH^\smalltext{2}_{\smalltext{T}\smalltext{,}\smalltext{\gamma}}} + 108 \max\{1, \mathfrak{j}(\gamma,\beta,\Phi)\}^2\Big(\|\xi_T\|^2_{\L^\smalltext{2}_\smalltext{\gamma}} + \|\xi^+_\cdot\mathbf{1}_{\{\cdot < T\}}\|^2_{\cS^\smalltext{2}_{\smalltext{T}\smalltext{,}\smalltext{\gamma}}} \Big) \\
		&\quad + \big(1+ 108 \max\{1, \mathfrak{j}(\gamma,\beta,\Phi)\}^2\big)\frac{(1+\gamma\Phi)}{\beta-\gamma}\bigg\|\frac{f}{\alpha}\bigg\|^2_{\H^\smalltext{2}_{\smalltext{T}\smalltext{,}\smalltext{\beta}}} \Bigg),
	\end{align*}
	which completes the proof.
\end{proof}

\subsection{Proof of Lemma \ref{lem::bound_delta_y_s2} and Lemma \ref{lem::cond_estimates}}\label{subsec::proofs_tech_lemmas}

\begin{proof}[Proof of \Cref{lem::bound_delta_y_s2}]
	Let us first discuss how to deduce \eqref{eq::cond_bound_delta_y_s_minus} from \eqref{eq::cond_bound_delta_y_s} and \eqref{eq::exp_bound_delta_y_s}. Suppose that $V = (V_t)_{t \in [0,\infty]}$ and $V^\prime = (V^\prime_t)_{t \in [0,\infty]}$ are two product-measurable processes whose $\P$--almost all paths admit limits from the left on $(0,\infty]$, such that $\E[\sup_{u \in [0,\infty]}|V_u|] + \E[\sup_{u \in [0,\infty]}|V^\prime_u|] < \infty$ and $\E[V_t |\cG_t] \leq \E[V^\prime_t |\cG_t], \; \text{$\P$--a.s.}, \; t \in [0,\infty]$. An application of \Cref{lem::increas_cond_mean} yields $\E[V_{t-}|\cG_{t-}] \leq \E[V^\prime_{t-}|\cG_{t-}], \; \text{$\P$--a.s.}, \; t \in [0,\infty]$. Here we use the conventions $V_{0-} \coloneqq 0$, $V^\prime_{0-} \coloneqq 0$ and $\cG_{0-} \coloneqq \cG_0$. Let us denote by $W = (W_t)_{t \in [0,\infty]}$ and $W^\prime = (W^\prime_t)_{t \in [0,\infty]}$ the processes $W_t \coloneqq V_{t-}$ and $W^\prime_t \coloneqq V^\prime_{t-}$, respectively. We write $\prescript{p}{}{W}$ and $\prescript{p}{}{W^\prime}$ for the predictable projections of $W$ and $W^\prime$. Since the predictable projection of a process with $\P$--a.s. left-continuous paths also has $\P$--a.s. left-continuous paths (see \cite[Theorem VI.47 and Remark VI.50.(f)]{dellacherie1982probabilities}), we find from
	\begin{equation*}
		\prescript{p}{}{W_t} = \E[V_{t-}|\cG_{t-}] \leq \E[V^\prime_{t-}|\cG_{t-}] = \prescript{p}{}{W^\prime_t}, \; \text{$\P$--a.s.}, \; t \in [0,\infty],
	\end{equation*}
	that $\prescript{p}{}{W_t} \leq \prescript{p}{}{W^\prime_t}$, $t \in [0,\infty]$, $\P$--almost surely. Sampling $\prescript{p}{}{W}$ and $\prescript{p}{}{W^\prime}$ at a predictable stopping time $S \in \cT_{0,\infty}$ thus yields
	\begin{equation*}
		\E[V_{S-}|\cG_{S-}] = \E[W_S|\cG_{S-}] = \prescript{p}{}{W_S} \leq \prescript{p}{}{W^\prime_S} = \E[W^\prime_S|\cG_{S-}] = \E[V^\prime_{S-}|\cG_{S-}], \; \text{$\P$--a.s.}
	\end{equation*}

	\medskip
	We now turn to the proof of \eqref{eq::cond_bound_delta_y_s} and \eqref{eq::exp_bound_delta_y_s}. Note first that for each $S \in \cT_{0,\infty}$, we have $\P$--a.s.
	\begin{align*}
		y^1_S &= {\esssup_{\tau \in \cT_{\smalltext{S}\smalltext{,}\smalltext{\infty}}}}^{\cG_\smalltext{S}} \E\bigg[ \xi_{\tau \land T} + \int_S^{\tau \land T} f^2_s \d C_s + \int_S^{\tau \land T} \delta f_u \d C_u \bigg| \cG_S\bigg] \\
		&\leq {\esssup_{\tau \in \cT_{\smalltext{S}\smalltext{,}\smalltext{\infty}}}}^{\cG_\smalltext{S}} \E\bigg[ \xi_{\tau \land T} + \int_S^{\tau \land T} f^2_s \d C_s \bigg| \cG_S \bigg] + {\esssup_{\tau \in \cT_{\smalltext{S}\smalltext{,}\smalltext{\infty}}}}^{\cG_\smalltext{S}} \E\bigg[ \int_S^{\tau \land T} |\delta f_u| \d C_u \bigg| \cG_S\bigg] \leq y^2_S + {\esssup_{\tau \in \cT_{\smalltext{S}\smalltext{,}\smalltext{\infty}}}}^{\cG_\smalltext{S}} \E\bigg[ \int_S^{\tau \land T} |\delta f_u| \d C_u \bigg| \cG_S\bigg], 
	\end{align*}
	which, since $y^i_S \in \L^2$ by \Cref{lem::supermartingale_family} and \Cref{lem::snell_aggregation}, leads by symmetry to
	\begin{equation*}
		|\delta y_S| = |y^1_S - y^2_S| \leq {\esssup_{\tau \in \cT_{\smalltext{S}\smalltext{,}\smalltext{\infty}}}}^{\cG_\smalltext{S}} \E\bigg[ \int_S^{\tau \land T} |\delta f_u| \d C_u \bigg| \cG_S\bigg] \leq \E\bigg[ \int_S^T |\delta f_u| \d C_u \bigg| \cG_S\bigg], \; \text{$\P$--a.s.}
	\end{equation*}
	We now have $|\delta y_S| \leq M_S$, $\P$--a.s., $S \in \cT_{0,\infty}$, where $M = (M_t)_{t \in [0,\infty]}$ is the martingale satisfying $M_S = \E\big[ \int_0^T |\delta f_u| \d C_u \big| \cG_S\big],$ $\P$--a.s.,  $S \in \cT_{0,\infty}.$ By \Cref{prop::optional_ineq}, we find
	\begin{equation*}
		\E\bigg[\sup_{s \in [0,\infty]} |\delta y|^2\bigg] \leq \E\bigg[\sup_{s \in [0,\infty]} |M_s|^2\bigg] \leq 4  \E[|M_\infty|^2] = 4  \E\bigg[ \bigg(\int_0^T |\delta f_u| \d C_u\bigg)^2 \bigg].
	\end{equation*}
	
	\medskip
	We turn to the inequalities containing only $y^i$ for $i \in \{1,2\}$, and drop the superscripts from now on due to the symmetry of the problem. As in the proof of \Cref{lem::supermartingale_family}, we find
	\begin{equation*}
		-\E\big[ |\xi_T| \big| \cG_S\big] - \E\bigg[\int_S^T|f_u|\d C_u\bigg|\cG_S\bigg] \leq y_S \leq \E\bigg[\sup_{u \in [S,\infty]}|\xi^+_{u \land T}| + \int_S^T |f_u|\d C_u \bigg| \cG_S\bigg], \; \text{$\P$--a.s.}, \; S \in \cT_{0,\infty},
	\end{equation*}
	and thus
	\begin{equation*}
		|y_S| \leq \E\bigg[|\xi_T| + \sup_{u \in [S,\infty]}|\xi^+_u \mathbf{1}_{\{u < T\}}| + \int_S^T |f_u|\d C_u \Bigg| \cG_S \bigg], \; \text{$\P$--a.s.}, \; S \in \cT_{0,\infty}.
	\end{equation*}
	By abusing notation, let $M = (M_t)_{t \in [0,\infty]}$ now be the martingale satisfying
	\begin{equation*}
		M_S = \E\bigg[|\xi_T| + \sup_{u \in [0,T)}|\xi^+_{u}| + \int_0^T |f_u|\d C_u \bigg| \cG_S \bigg], \; \text{$\P$--a.s.}, \; S \in \cT_{0,\infty}.
	\end{equation*}
	We derive similarly to before that
	\begin{align*}
		\E\bigg[\sup_{s \in [0,\infty]} |y_s|^2\bigg] \leq \E\bigg[\sup_{s \in [0,\infty]} |M_s|^2\bigg] \leq 4  \E[|M_\infty|^2] &= 4  \E\bigg[ \bigg(|\xi_T| + \sup_{u \in [0,T)}|\xi^+_u| + \int_0^T | f_u| \d C_u\bigg)^2 \bigg] \\
		&\leq 12  \E\bigg[ |\xi_T|^2 + \sup_{u \in [0,T)}|\xi^+_u|^2 + \bigg(\int_0^T | f_u| \d C_u\bigg)^2 \bigg],
	\end{align*}
	where in the last inequality we used $(a+b+c)^2 \leq 3(a^2+b^2+c^2)$. This completes the proof.
\end{proof}

\begin{proof}[Proof of \Cref{lem::cond_estimates}]
	As in the proof of \Cref{lem::bound_delta_y_s2}, it suffices to prove \eqref{eq::cond_deltaeta0} and \eqref{eq::cond_eta1}. We start with \eqref{eq::cond_deltaeta0}. Although we fix $(t,t^\prime) \in [0,\infty)$ with $t \leq t^\prime$ to ease the notation, the equalities and inequalities that follow should be read as holding, up to a $\P$--null set, for each pair $(t,t^\prime) \in [0,\infty)$ with $t \leq t^\prime$ unless stated otherwise. Note that the processes under consideration are all constant after time $T$ apart from $C$ and $f$. From \eqref{eq::differential_y}, we see that $\delta y$ satisfies
	\begin{equation*}
		\delta y_t = \delta y_0 - \int_0^t \delta f_s \mathbf{1}_{[0,T]}(s) \d C_s + \delta\eta_t - \delta k^r_t - \delta k^\ell_{t-}, \; t \in [0,\infty], \; \text{$\P$--a.s.}
	\end{equation*}
	To ease the notation and without loss of generality, we suppose that $\delta f_s = \delta f_s \mathbf{1}_{[0,T]}(s)$ and $\xi_s = \xi_{s \land T}$. By an application of the Gal'chouk--It\^o--Lenglart formula (see \cite[Theorem A.3 and Corollary A.2]{grigorova2017reflected} or \cite[Theorem 8.2]{gal1981optional}) on $(t,t^\prime]$, we find the (optional) semimartingale decomposition of $|\delta y|^2$ to be 
	\begin{align}\label{eq::apriori1}
		|\delta y_t|^2 &= |\delta y_{t^\prime}|^2 + 2\int_{(t,t^\prime]} \delta y_{s-} \delta f_s  \d C_s + 2 \int_{(t,t^\prime]} \delta y_{s-} \d (\delta k^r)_s - 2 \int_{(t,t^\prime]} \delta y_{s-} \d (\delta \eta)_s - \int_{(t,t^\prime]} \d [\delta \eta^c]_s \nonumber \\
			& \quad - \sum_{s \in (t,t^\prime]} (\delta y_s - \delta y_{s-})^2 + 2 \int_{[t,t^\prime)} \delta y_s \d (\delta k^\ell)_s - \sum_{s \in [t,t^\prime)} (\delta y_{s+} - \delta y_s)^2.
	\end{align}
	We decided here to write the integral bounds more clearly, as it is crucial whether one takes left-open or right-open intervals. Let us analyse the terms in the above decomposition one by one. First, note that the last term $\sum_{s \in [t,t^\prime)} (\delta y_{s+} - \delta y_s)^2$ is non-negative, and by adding zero to the last term, we find
\begin{align}
	-\sum_{s \in (t,t^\prime]} (\delta y_s - \delta y_{s-})^2 
	&= -\sum_{s \in (t,t^\prime]} (\delta y_s - \delta y_{s-})^2 - 2\sum_{s \in (t,t^\prime]} (\delta y_s - \delta y_{s-})(\delta f_s\Delta C_s) - \sum_{s \in (t,t^\prime]} (\delta f_s\Delta C_s)^2 \nonumber \\
	& \quad + 2\sum_{s \in (t,t^\prime]} (\delta y_s - \delta y_{s-})(\delta f_s\Delta C_s) + \sum_{s \in (t,t^\prime]} (\delta f_s\Delta C_s)^2 \nonumber\\
	&= -\sum_{s \in (t,t^\prime]} \big( \delta y_s - \delta y_{s-} + \delta f_s \Delta C_s \big)^2 + 2\sum_{s \in (t,t^\prime]} (\delta y_s - \delta y_{s-})(\delta f_s\Delta C_s) + \sum_{s \in (t,t^\prime]} (\delta f_s\Delta C_s)^2 \nonumber\\
	&= -\sum_{s \in (t,t^\prime]} \big( \Delta\delta\eta_s - \Delta\delta k^r_s\big)^2 + 2\sum_{s \in (t,t^\prime]} (\delta y_s - \delta y_{s-})(\delta f_s\Delta C_s) + \sum_{s \in (t,t^\prime]} (\delta f_s\Delta C_s)^2 \nonumber\\
	&= -\sum_{s \in (t,t^\prime]} \big( \Delta\delta\eta_s - \Delta\delta k^r_s\big)^2 + 2 \int_{(t,t^\prime]}(\delta y_s - \delta y_{s-})\delta f_s \d C_s + \sum_{s \in (t,t^\prime]} (\delta f_s\Delta C_s)^2 \label{eq::apriori2}.
\end{align}
By substituting this back into \eqref{eq::apriori1}, rearranging the terms, and using $\delta y_{s+} - \delta y_s = -(\delta k^\ell_s - \delta k^\ell_{s-})$ and
\begin{equation*}
	\int_{(t,t^\prime]}\d[\delta\eta]_s - 2\int_{(t,t^\prime]}\d[\delta\eta,\delta k^r]_s + \int_{(t,t^\prime]}\d[\delta k^r]_s = \int_{(t,t^\prime]}\d[\delta\eta - \delta k^r]_s = \int_{(t,t^\prime]}\d[\delta\eta^c]_s + \sum_{s \in (t,t^\prime]} \big( \Delta\delta\eta_s - \Delta\delta k^r_s\big)^2,
\end{equation*}
we find
\begin{align}\label{eq::equality_without_skorokhod}
	&|\delta y_t|^2 + \int_{(t,t^\prime]}\d[\delta\eta]_s + \int_{(t,t^\prime]}\d[\delta k^r]_s + \int_{[t,t^\prime)} \d[\delta k^\ell]_s - 2\int_{(t,t^\prime]}\d[\delta\eta,\delta k^r]_s \nonumber\\
	& = |\delta y_{t^\prime}|^2 + 2\int_{(t,t^\prime]} \delta y_{s} \delta f_s  \d C_s - 2 \int_{(t,t^\prime]} \delta y_{s-} \d (\delta \eta)_s + \sum_{s \in (t,t^\prime]} (\delta f_s\Delta C_s)^2 + 2 \int_{(t,t^\prime]} \delta y_{s-} \d (\delta k^r)_s + 2 \int_{[t,t^\prime)} \delta y_s \d (\delta k^\ell)_s.
\end{align}
The Skorokhod condition \eqref{eq::a_priori_skor} implies
\begin{equation}\label{eq::skorokhod_for_delta_y}
	\int_{(t,t^\prime]}\delta y_{s-} \d \delta k^r_s \leq 0, \; \text{and} \; \int_{[t,t^\prime)}\delta y_{s} \d \delta k^\ell_s \leq 0,
\end{equation}
which then yields

\begin{align}\label{eq::inequality_t_t_prime}
	|\delta y_t|^2 + \int_{(t,t^\prime]}\d[\delta\eta]_s & + \int_{(t,t^\prime]}\d[\delta k^r]_s + \int_{[t,t^\prime)} \d[\delta k^\ell]_s - 2\int_{(t,t^\prime]}\d[\delta\eta,\delta k^r]_s \nonumber\\
	&\quad \leq |\delta y_{t^\prime}|^2 + 2\int_{(t,t^\prime]} \delta y_{s} \delta f_s  \d C_s - 2 \int_{(t,t^\prime]} \delta y_{s-} \d (\delta \eta)_s + \sum_{s \in (t,t^\prime]} (\delta f_s\Delta C_s)^2.
\end{align}

Note that $\int_{(0,\cdot]}\delta y_{s-}\d(\delta\eta)_s$ and $\int_{(0,\cdot]}\d[\delta\eta,\delta k^r]$ are uniformly integrable martingales since $\delta y \in \cS^2_T$ by \Cref{lem::bound_delta_y_s2}, $\delta\eta \in \cH^2_T$ by assumption, $|\delta k^r| \leq k^{r,1}_T + k^{r,2}_T \in \L^2$, and $[\delta\eta,\delta k^r] = \int_{(0,\cdot]} \Delta(\delta k^r)_s\d(\delta\eta)_s$ by \cite[Proposition I.4.49]{jacod2003limit}. Indeed, since $ (k^{r,1}, k^{r,2}) \in \cI^2_T \times \cI^2_T$ and $\delta \eta \in \cH^2_T$, and thus
\begin{align*}
	\sqrt{\bigg\langle \int_{(0,\cdot]} \Delta (\delta k^r)_s \d \delta \eta_s \bigg\rangle_{\infty-}} 
	&= \sqrt{\int_{(0,\infty)} \big(\Delta (\delta k^r)_s\big)^2 \d \langle\delta \eta \rangle_s} \\
	&\leq \sqrt{2 \Big(\big(k^{r,1}_T\big)^2 + \big(k^{r,2}_T\big)^2\Big) \int_{(0,\infty)} \d \langle \delta \eta \rangle_s} 
	\leq \frac{1}{\sqrt{2}} \bigg(\big(k^{r,1}_T\big)^2 + \big(k^{r,2}_T\big)^2 + \int_{(0,\infty)} \d \langle \delta \eta \rangle_s \bigg),
\end{align*}
with an integrable right-hand side, the Burkholder--Davis--Gundy inequality implies that $\int_{(0,\cdot]} \Delta(\delta k^r)_s\d(\delta\eta)_s$ is bounded by an integrable random variable, and thus it is a uniformly integrable martingale. A similar argument, together with \Cref{lem::bound_delta_y_s2}, implies that $\int_{(0,\cdot]}\delta y_{s-}\d(\delta\eta)_s$ is a uniformly integrable martingale.
Since by \Cref{lem::bound_delta_y_s2}
\begin{equation*}
	|\delta y_{t^\prime}| \leq \E\bigg[\int_{t^\prime}^T |\delta f_s|\d C_s\bigg| \cG_{t^\prime}\bigg] = \E\bigg[\int_0^T |\delta f_s|\d C_s\bigg| \cG_{t^\prime}\bigg] - \int_0^{t^\prime \land T} |\delta f_s|\d C_s,
\end{equation*}
and since the right-hand side converges $\P$--a.s. to zero as $t^\prime$ tends to infinity, we deduce from \eqref{eq::inequality_t_t_prime} that
\begin{align}\label{eq::inequality_pathwise_t}
	|\delta y_t|^2 + \int_{(t,\infty)}\d[\delta\eta]_s & + \int_{(t,\infty)}\d[\delta k^r]_s + \int_{[t,\infty)} \d[\delta k^\ell]_s - 2\int_{(t,\infty)}\d[\delta\eta,\delta k^r]_s \nonumber\\
	&\quad \leq  2\int_{(t,\infty)} \delta y_{s} \delta f_s  \d C_s - 2 \int_{(t,\infty)} \delta y_{s-} \d (\delta \eta)_s + \sum_{s \in (t,\infty)} (\delta f_s\Delta C_s)^2, \; t \in [0,\infty], \; \text{$\P$--a.s.}
\end{align}
Since
\begin{equation*}
	\sum_{s \in (t,\infty)} (\delta f_s\Delta C_s)^2 = \int_{(t,\infty)} |\delta f_s|^2\d[C]_s, \; t \in [0,\infty], \; \text{$\P$--a.s.,}
\end{equation*}
we find for any stopping time $S \in \cT_{0,\infty}$ and by taking conditional expectation in \eqref{eq::inequality_pathwise_t} that
\begin{align*}
	|\delta y_S|^2 + \E\bigg[\int_S^T \d[\delta\eta]_s \bigg| \cG_S\bigg] + \E\bigg[\int_S^T \d[\delta k^r]_s \bigg| \cG_S\bigg] 
	&+ \E\bigg[\int_{S-}^T  \d[\delta k^\ell]_s \bigg| \cG_S\bigg] \\
	&\quad \leq 2 \E\bigg[\int_S^T \delta y_s \delta f_s \d C_s \bigg| \cG_S\bigg] + \E\bigg[\int_S^T (\delta f_s)^2\d[C]_s \bigg| \cG_S\bigg], \; \text{$\P$--a.s.}
\end{align*}
Analogously, in case $S \in \cT^p_{0,\infty}$, we find by taking left-hand limits in \eqref{eq::inequality_pathwise_t} that

\begin{align*}
	|\delta y_{S-}|^2 + \E\bigg[\int_{S-}^T \d[\delta\eta]_s \bigg| \cG_{S-}\bigg] + \E\bigg[\int_{S-}^T \d[\delta k^r]_s \bigg| \cG_{S-}\bigg] 
	&+ \E\bigg[\int_{S-}^T  \d[\delta k^\ell]_s \bigg| \cG_{S-}\bigg] \\
	&\quad \leq 2 \E\bigg[\int_{S-}^T \delta y_s \delta f_s \d C_s \bigg| \cG_{S-}\bigg] + \E\bigg[\int_{S-}^T (\delta f_s)^2\d[C]_s \bigg| \cG_{S-}\bigg], \; \text{$\P$--a.s.}
\end{align*}
This yields \eqref{eq::cond_deltaeta0} and \eqref{eq::cond_deltaeta_pred0} since
\begin{equation*}
	\E \bigg[ \int_S^T \d \langle\delta \eta\rangle_u \bigg| \cG_S \bigg] = \E \bigg[ \int_S^T \d [\delta \eta]_u \bigg| \cG_S \bigg] 
	\; \text{and} \; 
	\E \bigg[ \int_{S-}^T \d \langle\delta \eta\rangle_u \bigg| \cG_{S-} \bigg] = \E \bigg[ \int_{S-}^T \d [\delta \eta]_u \bigg| \cG_{S-} \bigg], \; \text{$\P$--a.s.}
\end{equation*}

\medskip
Before turning to the remaining inequalities, it is worth noting the following. With \Cref{lem::bound_delta_y_s2} and \Cref{lem::cond_doob}, we find that
\begin{equation*}
	\E\bigg[\sup_{u \in [t^\prime,\infty]}|y_u|^2 \bigg| \cG_{t^\prime}\bigg] \leq 12 \E\bigg[|\xi_T|^2 + \sup_{s \in [t^\prime,\infty]}|\xi^+_s\mathbf{1}_{\{s<T\}}|^2 + \bigg(\int_{(t^\prime,\infty)} |f_s| \d C_s\bigg)^2 \bigg| \cG_{t^\prime} \bigg], \; \text{$\P$--a.s.}, \; t^\prime \in (0,\infty).
\end{equation*}
By taking the conditional expectation with respect to $\cG_t$ for $t \in [0,t^\prime)$, and then letting $t^\prime$ tend to $t$, we find
\begin{equation*}\label{eq::conditional_sup_y1}
	\E\bigg[\sup_{u \in (t,\infty]}|y_u|^2 \bigg| \cG_t\bigg] \leq 12 \E\bigg[|\xi_T|^2 + \sup_{s \in (t,\infty]}|\xi^+_s\mathbf{1}_{\{s<T\}}|^2 + \bigg(\int_{(t,\infty)} |f_s| \d C_s\bigg)^2 \bigg| \cG_t \bigg], \; \text{$\P$--a.s.}, \; t \in [0,\infty].
\end{equation*}
As before, the processes within the conditional expectations are $\P$--a.s. right-continuous, and therefore, by the $\P$--a.s. right-continuity of their respective optional projections, we even have
\begin{equation}\label{eq::ineq_sup_y_cond}
	\E\bigg[\sup_{u \in (S,\infty]}|y_u|^2 \bigg| \cG_S\bigg] \leq 12 \E\bigg[|\xi_T|^2 + \sup_{u \in (S,\infty]}|\xi^+_u\mathbf{1}_{\{u<T\}}|^2 + \bigg(\int_{(S,\infty)} |f_u| \d C_u\bigg)^2 \bigg| \cG_S \bigg], \; \text{$\P$--a.s.}, \; S \in \cT_{0,\infty}.
\end{equation}

\medskip
We now turn to \eqref{eq::cond_eta1}. It is enough to show the bound for $i = 1$, and we thus also drop the superscript in what follows. An analogous argument to the one which lead to \eqref{eq::equality_without_skorokhod} yields by letting $t^\prime$ tend to infinity that
\begin{align}\label{eq::apriori4}
	&|y_t|^2 + \int_{(t,\infty)}\d[\eta]_s + \int_{(t,\infty)}\d[k^r]_s + \int_{[t,\infty)} \d[k^\ell]_s - 2\int_{(t,\infty)}\d[\eta, k^r]_s \nonumber\\
	&\quad = |y_{\infty-}|^2 + 2\int_{(t,\infty)} y_s f_s  \d C_s - 2 \int_{(t,\infty)} y_{s-} \d\eta_s + \sum_{s \in (t,\infty)} (f_s\Delta C_s)^2 + 2 \int_{(t,\infty)} y_{s-} \d k^r_s + 2 \int_{[t,\infty)} y_s \d k^\ell_s \nonumber\\
	&\quad \leq |\xi_T|^2 + 2\int_{(t,\infty)} y_s f_s  \d C_s - 2 \int_{(t,\infty)} y_{s-} \d\eta_s + \sum_{s \in (t,\infty)} (f_s\Delta C_s)^2 + 2 \int_{(t,\infty)} y_{s-} \d k^r_s + 2 \int_{[t,\infty)} y_s \d k^\ell_s, \; t \in [0,\infty], \; \text{$\P$--a.s.}
\end{align}
Here the inequality follows from \eqref{eq::cond_bound_delta_y_s_minus}. Now the Skorokhod condition implies that
\begin{align*}
	\int_{(t,\infty)} y_{s-} \d  k^r_s &= \int_{(t,\infty)} y_{s-} \mathbf{1}_{\{y_{\smalltext{s}\smalltext{-}} = \overline\xi_\smalltext{s}\}} \1_{\{s \leq T\}} \d  k^r_s = \int_{(t,\infty)} \overline\xi_s \mathbf{1}_{\{y_{\smalltext{s}\smalltext{-}} = \overline\xi_\smalltext{s}\}} \1_{\{s \leq T\}} \d  k^r_s = \int_{(t,\infty)} \overline\xi_s \1_{\{s \leq T\}} \d k^r_s
\end{align*}
and
\begin{align*}
	\int_{[t,\infty)} y_s \d k^\ell_s &= \int_{[t,\infty)} y_s \mathbf{1}_{\{y_\smalltext{s} = \xi_\smalltext{s}\}}\1_{\{s < T\}}\d k^\ell_s = \int_{[t,\infty)} \xi_s \mathbf{1}_{\{y_\smalltext{s} = \xi_\smalltext{s}\}} \1_{\{s < T\}} \d k^\ell_s = \int_{[t,\infty)} \xi_s \1_{\{s < T\}} \d k^\ell_s.
\end{align*}
Thus
\begin{align}\label{eq::cond_kr}
	2\int_{(t,\infty)} y_{s-} \d  k^r_s = 2\int_{(t,\infty)} \overline{\xi}_s \mathbf{1}_{\{s \leq T\}} \d k^r_s &\leq 2 \sup_{s \in (t,\infty)} \big\{\xi^+_s \mathbf{1}_{\{s < T\}}\big\} \int_{(t,\infty)} \d k^r_s \nonumber\\
	&\leq \frac{1}{\kappa} \sup_{s \in (t,\infty)} |\xi^+_s\mathbf{1}_{\{s < T\}}|^2  + \kappa \bigg(\int_{(t,\infty)} \d k^r_s \bigg)^2,
\end{align}
\begin{align}\label{eq::cond_kell}
	2\int_{[t,\infty)} y_s\d k^\ell_s = 2\int_{[t,\infty)} \xi_s \mathbf{1}_{\{s < T\}} \d k^\ell_s &\leq 2 \sup_{s \in [t,\infty)} \big\{\xi^+_s \mathbf{1}_{\{s < T\}} \big\}\int_{[t,\infty)} \d k^\ell_s \nonumber\\
	&\leq \frac{1}{\kappa} \sup_{s \in [t,\infty)} |\xi^+_s \mathbf{1}_{\{s < T\}}|^2  + \kappa \bigg(\int_{[t,\infty)} \d k^\ell_s \bigg)^2, 
\end{align}
for every $\kappa \in (0,\infty)$. Similarly, we find
\begin{equation}\label{eq::int_y_f_cs}
	2\int_{(t,\infty)} y_s f_s  \d C_s \leq \frac{1}{\varepsilon}\sup_{s \in (t,\infty)} |y_s|^2 + \varepsilon\bigg(\int_{(t,\infty)} |f_s| \d C_s\bigg)^2,
\end{equation}
for every $\varepsilon \in (0,\infty)$. Since
\begin{equation*}
	\int_{(t,\infty)} \d k^r_s + \int_{[t,\infty)} \d k^\ell_s = y_t - y_{\infty-} - \int_{(t,\infty)} f_s \d C_s + \int_{(t,\infty)} \d \eta_s, \; t \in [0,\infty), \; \text{$\P$--a.s.},
\end{equation*}
we also have
\begin{equation}\label{eq::cond_k1}
	\bigg(\int_{(t,\infty)} \d k^r_s + \int_{[t,\infty)} \d k^\ell_s \bigg)^2 \leq 4 \bigg( |y_t|^2 + |y_{\infty-}|^2 + \bigg(\int_{(t,\infty)} f_s \d C_s\bigg)^2 + \bigg(\int_{(t,\infty)} \d \eta_s\bigg)^2 \bigg), \; t \in [0,\infty), \; \text{$\P$--a.s.}
\end{equation}

Combining \eqref{eq::apriori4} with \Cref{eq::cond_kr,eq::cond_kell,eq::int_y_f_cs,eq::cond_k1} yields

\begin{align}\label{eq::inequality_pathwise_without_delta}
	&|y_t|^2 + \int_{(t,\infty)}\d[\eta]_s + \int_{(t,\infty)}\d[k^r]_s + \int_{[t,\infty)} \d[k^\ell]_s - 2\int_{(t,\infty)}\d[\eta, k^r]_s \nonumber\\ 
	&\quad \leq |\xi_T|^2 + \frac{1}{\varepsilon}\sup_{s \in (t,\infty)} |y_s|^2 + \varepsilon\bigg(\int_{(t,\infty)} |f_s| \d C_s\bigg)^2 - 2 \int_{(t,\infty)} y_{s-} \d\eta_s + \sum_{s \in (t,\infty)} (f_s\Delta C_s)^2 \nonumber\\
	&\quad\quad + \frac{1}{\kappa} \sup_{s \in (t,\infty)} |\xi^+_s\mathbf{1}_{\{s < T\}}|^2  + \kappa \bigg(\int_{(t,\infty)} \d k^r_s \bigg)^2 + \frac{1}{\kappa} \sup_{s \in [t,\infty)} |\xi^+_s \mathbf{1}_{\{s < T\}}|^2  + \kappa \bigg(\int_{[t,\infty)} \d k^\ell_s \bigg)^2 \nonumber\\
	&\quad \leq |\xi_T|^2 + \frac{1}{\varepsilon}\sup_{s \in (t,\infty)} |y_s|^2 +(\varepsilon+4\kappa)\bigg(\int_{(t,\infty)} |f_s| \d C_s\bigg)^2 - 2 \int_{(t,\infty)} y_{s-} \d\eta_s + \sum_{s \in (t,\infty)} (f_s\Delta C_s)^2 \nonumber\\
	&\quad\quad + \frac{2}{\kappa} \sup_{s \in [t,\infty)} |\xi^+_s\mathbf{1}_{\{s < T\}}|^2 + 4\kappa \bigg( |y_t|^2 + |y_{\infty-}|^2 + \bigg(\int_{(t,\infty)} \d \eta_s\bigg)^2 \bigg) \nonumber\\
	&\quad \leq |\xi_T|^2 + \frac{1}{\varepsilon}\sup_{s \in (t,\infty)} |y_s|^2 +(\varepsilon+4\kappa)\bigg(\int_{(t,\infty)} |f_s| \d C_s\bigg)^2 - 2 \int_{(t,\infty)} y_{s-} \d\eta_s + \sum_{s \in (t,\infty)} (f_s\Delta C_s)^2 \nonumber\\
	&\quad\quad + \frac{2}{\kappa} \sup_{s \in [t,\infty)} |\xi^+_s\mathbf{1}_{\{s < T\}}|^2 + 4\kappa \bigg( |y_t|^2 + |\xi_T|^2 + \bigg(\int_{(t,\infty)} \d \eta_s\bigg)^2 \bigg) \nonumber\\
	&\quad \leq (1+4\kappa)|\xi_T|^2 + \frac{1}{\varepsilon}\sup_{s \in (t,\infty)} |y_s|^2 +(\varepsilon+4\kappa)\bigg(\int_{(t,\infty)} |f_s| \d C_s\bigg)^2 - 2 \int_{(t,\infty)} y_{s-} \d\eta_s + \sum_{s \in (t,\infty)} (f_s\Delta C_s)^2 \nonumber\\
	&\quad\quad + \frac{2}{\kappa} \sup_{s \in [t,\infty)} |\xi^+_s\mathbf{1}_{\{s < T\}}|^2 + 4\kappa \bigg( |y_t|^2 + \bigg(\int_{(t,\infty)} \d \eta_s\bigg)^2 \bigg), \; t \in [0,\infty], \; \text{$\P$--a.s.}
\end{align}
Let $\kappa \in (0,\infty)$ be such that $0 < 1-4\kappa \leq 1$, and let $S \in \cT_{0,\infty}$. By taking conditional expectation, rearranging the terms and using $\E[(\int_{(S,\infty)}\d\eta_u)^2|\cG_S] = \E[\int_{(S,\infty)}\d[\eta]_u|\cG_S]$ and $\sum_{s \in (S,\infty)} (f_s\Delta C_s)^2 \leq (\int_S^T |f_s| \d C_s)^2$, we find
\begin{align*}
	& (1-4\kappa)\Bigg(|y_S|^2 + \E\bigg[\int_{(S,\infty)}\d[\eta]_u\bigg|\cG_S\bigg]\Bigg) + \E\bigg[\int_{(S,\infty)}\d[k^r]_u\bigg|\cG_S\bigg] + \E\bigg[\int_{[S,\infty)} \d[k^\ell]_u\bigg|\cG_S\bigg] \\
	& \leq (1+4\kappa)\E\big[|\xi_T|^2\big|\cG_S\big] + \frac{1}{\varepsilon}\E\bigg[\sup_{u \in (S,\infty)} |y_u|^2\bigg|\cG_S\bigg] +(1 + \varepsilon+4\kappa)\E\bigg[\bigg(\int_{(S,\infty)} |f_u| \d C_u\bigg)^2\bigg|\cG_S\bigg]  + \frac{2}{\kappa} \E\bigg[\sup_{u \in [S,\infty)} |\xi^+_u\mathbf{1}_{\{u < T\}}|^2\bigg|\cG_S\bigg] \\
	& \leq (1+4\kappa + 12/\varepsilon)\E\big[|\xi_T|^2\big|\cG_S\big] + (1 + \varepsilon+4\kappa + 12/\varepsilon)\E\bigg[\bigg(\int_{(S,\infty)} |f_u| \d C_u\bigg)^2\bigg|\cG_S\bigg]  + \bigg(\frac{12}{\varepsilon} +\frac{2}{\kappa}\bigg) \E\bigg[\sup_{u \in [S,\infty)} |\xi^+_u\mathbf{1}_{\{u < T\}}|^2\bigg|\cG_S\bigg].
\end{align*}
Here the second inequality follows from \eqref{eq::ineq_sup_y_cond}. Since $0 < (1-4\kappa) \leq 1$, we can thus divide both sides by $(1-4\kappa)$ and find
\begin{align*}
	& |y_S|^2 + \E\bigg[\int_{(S,\infty)}\d[\eta]_u\bigg|\cG_S\bigg] + \E\bigg[\int_{(S,\infty)}\d[k^r]_u\bigg|\cG_S\bigg] + \E\bigg[\int_{[S,\infty)} \d[k^\ell]_u\bigg|\cG_S\bigg] \\
	& \leq \frac{\max\{1 + \varepsilon+4\kappa + 12/\varepsilon,12/\varepsilon + 2/\kappa\}}{1-4\kappa} \Bigg( \E\big[|\xi_T|^2\big|\cG_S\big] + \E\bigg[\bigg(\int_{(S,\infty)} |f_u| \d C_u\bigg)^2\bigg|\cG_S\bigg] + \E\bigg[\sup_{u \in [S,\infty)} |\xi^+_u\mathbf{1}_{\{u < T\}}|^2\bigg|\cG_S\bigg]\Bigg).
\end{align*}

\medskip
Finally, as explained at the beginning of the proof of \Cref{lem::bound_delta_y_s2}, the inequality \eqref{eq::cond_eta1_pred} follows from \eqref{eq::cond_eta1}, which completes the proof.
\end{proof}

\section{Proofs of the main results}\label{sec::existence_rbsde}

We are now in a position to prove \Cref{thm::main} and \Cref{thm::main_bsde}. The proofs are based on the optimal stopping theory we revisited in \Cref{sec::optimal_stopping} and on the \emph{a priori} estimates we established in \Cref{sec::a_priori}. The existence and uniqueness of the BSDE and the reflected BSDE are based on defining a contraction map on the weighted spaces of \Cref{sec::weighted_spaces}. Therefore, we first need to show that such a contraction map is well-defined, meaning that it maps its domain into itself.

\begin{proposition}\label{lem::exist_simple_gen}
Suppose that $f$ does not depend on $(y,\mathrm{y},z,u)$. There exists a unique triple
$(Z,U,N) \in \mathbb{H}^{2}_{T}(X) \times \mathbb{H}^{2}_{T}(
\mu ) \times \mathcal H^{2,\perp}_{0,T}(X,\mu )$ and a, up to $\mathbb{P}$--indistinguishability, unique triple $(Y,K^{r},K^{\ell})$ such that the collection $(Y,Z,U,N,K^{r},K^{\ell})$ satisfies~\ref{cond::integ_mart} up to~\ref{cond::repre_Y}. Moreover,
$Y \in \mathcal S^{2}_{T}$. If, in addition,
$(X,\mathbb{G},T,\xi ,f,C)$ is standard data for some
$\hat\beta \in (0,\infty )$, then
$(\alpha Y,\alpha Y_{-},Z,U,N) \in \mathbb{H}^{2}_{T,\hat\beta}
\times \mathbb{H}^{2}_{T,\hat\beta} \times \mathbb{H}^{2}_{T,
\hat\beta}(X) \times \mathbb{H}^{2}_{T,\hat\beta}(\mu ) \times
\mathcal H^{2,\perp}_{T,\hat\beta}(X,\mu )$.
\end{proposition}

\begin{proof}
	Let $(Y,Z,U,N,K^\ell,K^r)$ be the collection of processes constructed in \Cref{lem::snell_decomposition}, which clearly is the unique collection of processes satisfying \ref{cond::integ_mart} up to \ref{cond::repre_Y}. That $Y \in \cS^2_T$ follows from \Cref{lem::bound_delta_y_s2}.

	\medskip
	That $(\alpha Y, \alpha Y_-,Z,U,N) \in \H^2_{T,\hat\beta} \times \H^2_{T,\hat\beta} \times \H^2_{T,\hat\beta}(X) \times \H^2_{T,\hat\beta}(\mu) \times \cH^{2,\perp}_{T,\hat\beta}(X,\mu)$ in case $(X,\G,T,\xi,f,C)$ is standard data under some $\hat\beta \in (0,\infty)$ follows from Proposition \ref{prop::apriori}. This completes the proof.
\end{proof}

Before proving the next result, recall from \Cref{rem::after_formulation} that the first component of the reflected BSDE is an optional semimartingale indexed by $[0,\infty]$ since $Y = Y_0 + M + A$, $\P$--a.s., where
\begin{equation*}
	M_t \coloneqq \int_0^t Z_s \d X_s + \int_0^t \int_E U_s(x)\tilde\mu(\d s, \d x) + \int_0^t \d N_t, \; \text{and} \; A_t \coloneqq -\int_0^t f_s\big(Y_s,Y_{s-},Z_s,U_s(\cdot)\big) \d C_s - K^r_t - K^\ell_{t-}, \; t \in [0,\infty].
\end{equation*}
The integrals above do not include the point $\infty$ in their domain of integration, yet their values at infinity are determined by letting $t \uparrow\uparrow \infty$.  

\begin{lemma}\label{lem::cond_repre}
	Let $\hat\beta \in (0,\infty)$. Suppose that $(Y,Z,U,N,K^r,K^\ell)$ satisfy \ref{cond::integ_mart} up to \ref{cond::increasing_proc} and one of the following conditions holds:
	
	\medskip
	$(i)$ $(\alpha Y,\alpha Y_-,Z,U,N) \in \H^2_{T,\hat\beta} \times \H^2_{T,\hat\beta} \times \H^2_{T,\hat\beta}(X) \times \H^2_{T,\hat\beta}(\mu);$
	
	\medskip
	$(ii)$ the generator $f$ does not depend on $Y_{s-}$ and $(\alpha Y,Z,U,N) \in \H^2_{T,\hat\beta} \times \H^2_{T,\hat\beta}(X) \times \H^2_{T,\hat\beta}(\mu);$
	
	\medskip
	$(iii)$ the generator $f$ does not depend on $Y_s$ and $(\alpha Y_-,Z,U,N) \in \H^2_{T,\hat\beta} \times \H^2_{T,\hat\beta}(X) \times \H^2_{T,\hat\beta}(\mu)$.
	
	\medskip
	Then $Y \in \cS^2_T$ and \ref{cond::repre_Y} holds.
\end{lemma}

\begin{proof}
	We prove this result under the assumption that $(ii)$ holds. The other cases follow analogously. Note first that 
	\begin{align*}
		\E\bigg[\bigg(\int_0^T |f_s\big(Y_s,Z_s,U_s(\cdot)\big)| \d C_s\bigg)^2\bigg] 
		&\leq \E\bigg[\bigg(\int_0^T\frac{1}{\cE(\hat\beta A)_{s}}\d A_s\bigg)\bigg(\int_0^T \cE(\hat\beta A)_{s}\frac{|f_s\big(Y_s,Z_s,U_s(\cdot)\big)|^2}{\alpha^2_s} \d C_s\bigg)\bigg] \\
		&\leq  \frac{1}{\hat\beta}\E\bigg[\int_0^T \cE(\hat\beta A)_{s}\frac{|f_s\big(Y_s,Z_s,U_s(\cdot)\big)|^2}{\alpha^2_s} \d C_s\bigg] < \infty.
	\end{align*}
	Here we first used the Cauchy--Schwarz inequality, then \eqref{eq::integ_stoch_exp_inverse}, and finally that $(Y,Z,U) \in \H^2_{T,\hat\beta} \times \H^2_{T,\hat\beta}(X) \times \H^2_{T,\hat\beta}(\mu)$ together with the Lipschitz property of $f$ and \ref{data::integ_f0}. We now show that $Y \in \cS^2_T$. Let
	\begin{align*}
		J_t 
		&\coloneqq Y_{t \land T} - Y_0 + \int_0^{t \land T} f_s\big(Y_s,Z_s,U_s(\cdot)\big)\d C_s + K^r_{t \land T} + K^\ell_{(t \land T)-} \\
		&= \int_0^{t \land T} Z_s \d X_s + \int_0^{t \land T}\int_E U_s(x)\tilde\mu(\d s, \d x) + N_{t \land T}, \; t \in [0,\infty), \; \text{$\P$--a.s.}
	\end{align*}
	Then $J \in \cS^2_T$ since
	\begin{equation*}
		\E\bigg[\sup_{t \in [0,T]}|J_t|^2\bigg] = \E\bigg[\sup_{t \in [0,\infty]}|J_t|^2\bigg] = \E\bigg[\sup_{t \in [0,\infty)}|J_t|^2\bigg] 
		\leq 4 \E\big[|J_\infty|^2\big] 
		= 4\Big(\|Z\|^2_{\H^\smalltext{2}_{\smalltext{T}\smalltext{,}\smalltext{0}}(X)} + \|U\|^2_{\H^\smalltext{2}_{\smalltext{T}\smalltext{,}\smalltext{0}}(\mu)} + \|N\|^2_{\cH^\smalltext{2}_{\smalltext{T}\smalltext{,}\smalltext{0}}}\Big) < \infty,
	\end{equation*}
	and
	\begin{align*}
		\|Y - J\|_{\cS^\smalltext{2}_\smalltext{T}} 
		&\leq \|Y_0\|_{\L^\smalltext{2}} + \E\bigg[\bigg(\int_0^T |f_s\big(Y_s,Z_s,U_s(\cdot)\big)| \d C_s\bigg)^2\bigg]^{1/2} + \|K^r_T + K^\ell_T\|_{\L^2} \\
		&\leq \|\xi_T - J_T\|_{\L^2} + 2 \E\bigg[\bigg(\int_0^T |f_s\big(Y_s,Z_s,U_s(\cdot)\big)| \d C_s\bigg)^2\bigg]^{1/2} + 2 \|K^r_T + K^\ell_T\|_{\L^2} \\
		&\leq \|\xi_T\|_{\L^2} + \|J_T\|_{\L^2} + 2 \E\bigg[\bigg(\int_0^T |f_s\big(Y_s,Z_s,U_s(\cdot)\big)| \d C_s\bigg)^2\bigg]^{1/2} + 2 \|K^r_T + K^\ell_T\|_{\L^2} < \infty.
	\end{align*}
	Here we used $Y_T = \xi_T$ in the second-to-last line. This yields
	\begin{equation*}
		\|Y\|_{\cS^\smalltext{2}_\smalltext{T}} \leq \|Y - J\|_{\cS^\smalltext{2}_\smalltext{T}} + \|J\|_{\cS^\smalltext{2}_\smalltext{T}} < \infty.
	\end{equation*}

	We turn to the proof of \ref{cond::repre_Y}. Let $(\tilde Y,\tilde Z, \tilde U, \tilde N, \tilde K^r, \tilde K^\ell)$ be the solution to the reflected BSDE satisfying \ref{cond::integ_mart} up to \ref{cond::repre_Y} with $\tilde Y \in \cS^2_T$ and generator $f(Y,Z,U)$ given by \Cref{lem::exist_simple_gen}. Applying the Gal'chouk--It\^o--Lenglart formula to $|Y - \tilde Y|^2$ yields
	\begin{align}\label{eq::ito_representation}
		|Y_t - \tilde Y_t|^2 
		&= |Y_{t^\prime} - \tilde Y_{t^\prime}|^2 + 2 \int_{(t,t^\prime]} (Y_{s-}-\tilde Y_{s-}) \d (K^r - \tilde K^r)_s - 2 \int_{(t,t^\prime]} (Y_{s-}-\tilde Y_{s-}) \d (M - \tilde M)_s - \int_{(t,t^\prime]} \d [M^c - \tilde M^c]_s \nonumber \\
		& \quad - \sum_{s \in (t,t^\prime]} \big( Y_s - \tilde Y_s  - (Y_{s-} - \tilde Y_{s-})\big)^2 + 2 \int_{[t,t^\prime)} (Y_s - \tilde Y_s) \d (K^\ell - \tilde K^\ell)_s - \sum_{s \in [t,t^\prime)} \big( Y_{s+} - \tilde Y_{s+} - (Y_s - \tilde Y_s)\big)^2 \nonumber \\
		&\leq |Y_{t^\prime} - \tilde Y_{t^\prime}|^2 + 2 \int_{(t,t^\prime]} \big(Y_{s-} - \overline\xi_s - (\tilde Y_{s-}-\overline\xi_s)\big) \d (K^r - \tilde K^r)_s - 2 \int_{(t,t^\prime]} (Y_{s-}-\tilde Y_{s-}) \d (M - \tilde M)_s \nonumber \\
		& \quad  +2 \int_{[t,t^\prime)} \big(Y_s - \xi_s - (\tilde Y_s-\xi_s)\big) \d (K^\ell - \tilde K^\ell)_s \nonumber \\
		&\leq |Y_{t^\prime} - \tilde Y_{t^\prime}|^2- 2 \int_{(t,t^\prime]} (Y_{s-}-\tilde Y_{s-}) \d (M - \tilde M)_s, \; 0 \leq t \leq t^\prime < \infty, \; \text{$\P$--a.s.,}
	\end{align}
	where
	\begin{equation*}
		M_t \coloneqq \int_0^t Z_s \d X_s + \int_0^t\int_E U_s(x)\tilde\mu(\d s, \d x) + N_t, \; 
		\text{and} 
		\; \tilde M_t \coloneqq \int_0^t \tilde Z_s \d X_s + \int_0^t\int_E \tilde U_s(x)\tilde\mu(\d s, \d x) + \tilde N_t, \; t \in [0,\infty].
	\end{equation*}	
	Note that we have $M_\infty = \lim_{t \uparrow\uparrow \infty}M_t$ and $\tilde M_\infty =\lim_{t \uparrow\uparrow \infty}M_t$ up to a $\P$--null set.	Since $Y- \tilde Y \in \cS^2_T$, the local martingale
	\begin{equation*}
		\int_0^\cdot (Y_{s-}-\tilde Y_{s-}) \d (M - \tilde M)_s = \int_0^{\cdot \land T} (Y_{s-}-\tilde Y_{s-}) \d (M - \tilde M)_s,
	\end{equation*}
	is a uniformly integrable martingale by the Burkholder--Davis--Gundy inequality. Taking the conditional expectation in \eqref{eq::ito_representation} yields $|Y_S - \tilde Y_S|^2 \leq \E[|Y_{S^\prime} - \tilde Y_{S^\prime}|^2|\cG_S]$, $\P$--a.s., for two finite stopping times $S$ and $S^\prime$ with $S \leq S^\prime$. In particular, since $(Y,\tilde Y) \in \cS^2_T$, and by choosing $S^\prime = S \lor n$ and then letting $n$ tend to infinity, this yields
	\begin{equation*}
		|Y_S - \tilde Y_S|^2 \leq \E[|Y_{\infty-} - \tilde Y_{\infty-}|^2|\cG_S], \; \text{$\P$--a.s.}
	\end{equation*}
	Suppose for the moment that $\P[A] = 0$, where $A = \{T = \infty\} \cap \{|Y_{\infty-}-\tilde Y_{\infty-}| > 0\}$. Then $\E[|Y_{\infty-} - \tilde Y_{\infty-}|^2|\cG_S] = 0$, $\P$--almost surely. \Cref{prop::optional_ineq} together with $Y_T = \xi_T = \tilde Y_T$, $\P$--a.s., implies that $Y = \tilde Y$ up to $\P$--indistinguishability. Hence
	\begin{equation*}
		Y_S = \tilde Y_S = {\esssup_{\tau \in \cT_{\smalltext{S}\smalltext{,}\smalltext{\infty}}}}^{\cG_\smalltext{S}} \E\bigg[ \int_S^{\tau \land T} f_s(Y_s,Z_s,U_s(\cdot)) \d C_s + \xi_{\tau \land T} \bigg| \cG_S\bigg], \;\text{$\P$--a.s.}, \; S \in \cT_{0,\infty}.
	\end{equation*}
	
	\medskip
	It remains to prove $\P[A] = 0$, and we suppose, for the sake of reaching a contradiction, that $\P[A] > 0$. On $B \coloneqq A \cap \{\Delta Y_\infty = Y_\infty - Y_{\infty-} = 0\}$, we have $-\Delta Y_\infty = \Delta K^r_\infty = K^r_\infty - K^r_{\infty-} = 0$ and $\overline\xi_\infty \leq Y_{\infty-} = Y_\infty = \xi_\infty$, $\P$--a.s. on $B$. This implies $|\tilde Y_\infty - \tilde Y_{\infty-}| = |\xi_\infty - \tilde Y_{\infty-}| = |Y_{\infty-} - \tilde Y_{\infty-}| > 0$, $\P$--a.s. on $B$, therefore $-\Delta \tilde Y_\infty = \Delta\tilde K^r_\infty > 0$, $\P$--a.s. on $B$, and thus $\tilde Y_{\infty-} = \overline\xi_\infty$, $\P$--a.s. on $B$. However, this also yields $|Y_{\infty-} - \overline\xi_\infty| = |Y_{\infty-} - \tilde Y_{\infty-}| > 0$, $\P$--a.s. on $B$, which now implies $\tilde Y_{\infty-} = \overline\xi_\infty < Y_{\infty-} = Y_\infty = \xi_\infty = \tilde Y_\infty$, $\P$--a.s. on $B$. Therefore, we must have $0 > -\Delta\tilde K^r_\infty = \Delta\tilde Y_\infty > 0$, $\P$--a.s. on $B$, which now yields $\P[B] = 0$. Let $B^\prime \coloneqq A \cap \{\Delta Y_\infty \neq 0\}$. Then $Y_{\infty} - Y_{\infty-} = \Delta Y_\infty = -\Delta K^r_\infty < 0$ and thus $Y_{\infty-} = \overline\xi_\infty$, $\P$--a.s. on $B^\prime$. Since $|\overline\xi_\infty - \tilde Y_{\infty-}| = |Y_{\infty-} - \tilde Y_{\infty-}| > 0$, $\P$--a.s. on $B^\prime$, this implies $\Delta \tilde K^r_\infty = 0$ and therefore $\Delta\tilde Y_\infty = 0$, $\P$--a.s. on $B^\prime$. Hence $Y_{\infty-} > Y_\infty = \xi_\infty = \tilde Y_{\infty} = \tilde Y_{\infty-} > \overline \xi_\infty = Y_{\infty-}$, $\P$--a.s. on $B^\prime$, and therefore $\P[B^\prime] = 0$. This now yields $\P[A] = \P[B] + \P[B^\prime] = 0$ and completes the proof.
\end{proof}

We now turn to the proof of our main result.

\begin{proof}[Proof of Theorem \ref{thm::main}.]\label{proof::main}
	We prove the theorem under assumption $(ii)$. Let $\cL^{(2)}_{T,\hat\beta}$ be the collection of processes $(Y,Z,U,N)$ for which $Y = Y_{\cdot \land T}$ is optional, $(Z,U,N) \in \H^2_T(X) \times \H^2_T(\mu) \times \cH^{2,\perp}_{0,T}(X,\mu)$, and
	\begin{equation*}
		\|(Y,Z,U,N)\|^2_{\cL^\smalltext{(2)}_{\smalltext{T}\smalltext{,}\smalltext{\hat\beta}}} \coloneqq \|\alpha Y\|^2_{\H^\smalltext{2}_{\smalltext{T}\smalltext{,}\smalltext{\hat\beta}}} + \|Z\|^2_{\H^\smalltext{2}_{\smalltext{T}\smalltext{,}\smalltext{\hat\beta}}(X)} + \|U\|^2_{\H^\smalltext{2}_{\smalltext{T}\smalltext{,}\smalltext{\hat\beta}}(\mu)} + \|N\|^2_{\cH^\smalltext{2}_{\smalltext{T}\smalltext{,}\smalltext{\hat\beta}}} < \infty.
	\end{equation*}
	Then $\cL^{(2)}_{T,\hat\beta}$ together with the semi-norm $\|\cdot\|_{\cL^\smalltext{(2)}_{\smalltext{T}\smalltext{,}\smalltext{\hat\beta}}}$ is a Banach space after identifying processes $(Y,Z,U,N)$ and $(Y^\prime,Z^\prime,U^\prime,N^\prime)$ for which $\|(Y,Z,U,N) - (Y^\prime,Z^\prime,U^\prime,N^\prime)\|_{\cL^\smalltext{(2)}_{\smalltext{T}\smalltext{,}\smalltext{\hat\beta}}} = 0$ holds. Let $(y,z,u,n) \in \cL^{(2)}_{T,\hat\beta}$, and note that
	\begin{align*}
		&\E \bigg[ \int_0^T \cE(\hat\beta A)_s \frac{|f_s\big(y_s,z_s,u_s(\cdot)\big)|^2}{\alpha^2_s} \d C_s\bigg] \\ 
		&\quad\leq 2 \E \bigg[ \int_0^T \cE(\hat\beta A)_s \frac{|f_s\big(y_s,z_s,u_s(\cdot)\big) - f_s(0,0,{\bf 0})|^2}{\alpha^2_s} \d C_s 
		+ \int_0^T \cE(\hat\beta A)_s \frac{|f_s(0,0,{\bf 0})|^2}{\alpha^2_s} \d C_s \bigg] \\
		&\quad\leq 2\Bigg( \|\alpha y\|^2_{\H^\smalltext{2}_{\smalltext{T}\smalltext{,}\smalltext{\hat\beta}}} + \|z\|^2_{\H^\smalltext{2}_{\smalltext{T}\smalltext{,}\smalltext{\hat\beta}}(X)} + \|u\|^2_{\H^\smalltext{2}_{\smalltext{T}\smalltext{,}\smalltext{\hat\beta}}(\mu)} + \bigg\|\frac{f(0,0,\mathbf{0})}{\alpha}\bigg\|^2_{\H^\smalltext{2}_{\smalltext{T}\smalltext{,}\smalltext{\hat\beta}}}\Bigg) < \infty.
	\end{align*}
	We denote by $(Y,Z,U,N,K^r,K^\ell)$ the collection satisfying \ref{cond::integ_mart} up to \ref{cond::repre_Y} with generator $f_s\big(y_s,z_s,u_s(\cdot)\big)$ constructed in \Cref{lem::snell_decomposition} (or in \Cref{lem::exist_simple_gen}). Then $(Y,Z,U,N) \in \cL^{(2)}_{T,\hat\beta}$ by the bounds of \Cref{prop::apriori}. The map $\Upsilon_2 : \cL^{(2)}_{T,\hat\beta} \longrightarrow \cL^{(2)}_{T,\hat\beta}$ given by $\Upsilon_2(y, z, u, n) \coloneqq (Y, Z, U, N)$ is thus well-defined. 
	
	\medskip
	We prove that $\Upsilon_2$ is a contraction. For $i \in \{1,2\}$, let $(y^i,z^i,u^i,n^i) \in \cL^{(2)}_{T,\hat\beta}$ and let $(Y^i, Z^i,U^i,N^i) \coloneqq \Upsilon_2(y^i,z^i,u^i,n^i)$. Let $\delta y \coloneqq y^1 - y^2, \; \delta z = z^1 - z^2, \; \delta u = u^1 - u^2 ,\; \text{and} \; \delta n = n^1 - n^2,$
	and define $\delta Y$, $\delta Z$, $\delta U$ and $\delta N$ similarly. Denote by $\psi = (\psi_t)_{t \in [0,\infty)}$ the process $\psi_t \coloneqq f_t \big(y^1_t,z^1_t,u^1_t(\cdot)\big) - f_t\big(y^2_t,z^2_t,u^2_t(\cdot)\big).$ With \Cref{prop::apriori}, we find
	\begin{align*}
		\big\| \Upsilon_2(y^1,z^1,u^1,n^1) - \Upsilon_2(y^2,z^2,u^2,n^2) \big\|^2_{\cL^\smalltext{(2)}_{\smalltext{T}\smalltext{,}\smalltext{\hat\beta}}} 
		&= \|\alpha\delta Y\|^2_{\H^\smalltext{2}_{\smalltext{T}\smalltext{,}\smalltext{\hat\beta}}} + \|\delta Z\|^2_{\H^\smalltext{2}_{\smalltext{T}\smalltext{,}\smalltext{\hat\beta}}(X)} + \|\delta U\|^2_{\H^\smalltext{2}_{\smalltext{T}\smalltext{,}\smalltext{\hat\beta}}(\mu)} + \|\delta N\|^2_{\cH^{\smalltext{2}}_{\smalltext{T}\smalltext{,}\smalltext{\hat\beta}}} \nonumber\\
		&\leq M^\Phi_2(\hat\beta) \Big\| \frac{\psi}{\alpha} \Big\|^2_{\H^\smalltext{2}_{\smalltext{T}\smalltext{,}\smalltext{\hat\beta}}} \nonumber\\
		&\leq M^\Phi_2(\hat\beta) \Big( \|\alpha\delta y\|^2_{\H^\smalltext{2}_{\smalltext{T}\smalltext{,}\smalltext{\hat\beta}}} + \|\delta z\|^2_{\H^\smalltext{2}_{\smalltext{T}\smalltext{,}\smalltext{\hat\beta}}(X)} + \|\delta u\|^2_{\H^\smalltext{2}_{\smalltext{T}\smalltext{,}\smalltext{\hat\beta}}(\mu)} \Big)  \nonumber\\
		&\leq M^\Phi_2(\hat\beta) \big\|(y^1, z^1, u^1, n^1) - (y^2, z^2, u^2, n^2)\big\|^2_{\cL^\smalltext{(2)}_{\smalltext{T}\smalltext{,}\smalltext{\hat\beta}}}.
	\end{align*}
	Here, we used the Lipschitz-continuity of the generator $f$ in the third line. Since $M^\Phi_2(\hat\beta) < 1$, the map $\Upsilon_2$ is indeed a contraction on $\cL^{(2)}_{T,\hat\beta}$. By Banach's fixed-point theorem, there exists a unique fixed-point of $\Upsilon_2$, which we denote by $(Y,Z,U,N)$. Denoting by $(K^r,K^\ell)$ the two corresponding non-decreasing processes coming from the decomposition of the Snell envelope $Y_\cdot + \int_0^{\cdot \land T} f_s(Y_s,Z_s,U_s(\cdot))\d C_s,$ we see that $(Y,Z,U,N,K^r,K^\ell)$ satisfies \ref{cond::integ_mart} up to \ref{cond::repre_Y}. 
	
	\medskip
	Suppose that $(Y^\prime,Z^\prime,U^\prime,N^\prime,K^{\prime,r},K^{\prime,\ell})$ is a solution satisfying \ref{cond::integ_mart} up to \ref{cond::increasing_proc} such that $(\alpha Y^\prime,Z^\prime,U^\prime,N^\prime)$ is in $\H^2_{T,\hat\beta} \times \H^2_{T,\hat\beta}(X) \times \H^2_{T,\hat\beta}(\mu) \times \cH^{2,\perp}_{0,T,\hat\beta}(X,\mu)$. Then \ref{cond::repre_Y} holds by \Cref{lem::cond_repre} and thus $(Y^\prime,Z^\prime,U^\prime,N^\prime)$ is the fixed-point of $\Upsilon_2$. Hence $(Y^\prime,Z^\prime,U^\prime,N^\prime) = (Y,Z,U,N)$ in $\cL^{(2)}_{T,\hat\beta}$, and $Y = Y^\prime$ up to $\P$--indistinguishability by \Cref{prop::apriori}. That $(K^{\prime,r},K^{\prime,\ell}) = (K^r,K^\ell)$ up to indistinguishability follows from \Cref{lem::exist_simple_gen}. 
	This implies the stated uniqueness. That $Y$ is in $\cS^2_T$ follows from \Cref{lem::bound_delta_y_s2} since
	\begin{align*}
		\E\bigg[\bigg(\int_0^T |f_s\big(Y_s,Z_s,U_s(\cdot)\big)| \d C_s\bigg)^2\bigg] 
		&\leq \frac{1}{\hat\beta}\E\bigg[\int_0^T \cE(\hat\beta A)_s\frac{|f_s\big(Y_s,Z_s,U_s(\cdot)\big)|^2}{\alpha^2_s} \d C_s\bigg] < \infty
	\end{align*}
	by the Cauchy--Schwarz inequality and \eqref{eq::integ_stoch_exp_inverse}.
	This completes the proof of $(ii)$. 
	
	\medskip
	For $(i)$ and $(iii)$ we define $\cL^{(1)}_{T,\hat\beta}$ and $\cL^{(3)}_{T,\hat\beta}$ as the spaces of processes $(Y,Z,U,N)$ for which $(Y,Z,U,N) \in  \cS^2_T \times \H^2_T(X) \times \H^2_T(\mu) \times \cH^{2,\perp}_{0,T}(X,\mu)$, $\P$--a.e. path of $Y$ is l\`adl\`ag, and
	\begin{align*}
		\|(Y,Z,U,N)\|^2_{\cL^\smalltext{(1)}_{\smalltext{T}\smalltext{,}\smalltext{\hat\beta}}} &\coloneqq \|Y\|^2_{\cS^\smalltext{2}_\smalltext{T}} + \|\alpha Y\|^2_{\H^\smalltext{2}_{\smalltext{T}\smalltext{,}\smalltext{\hat\beta}}} + \|\alpha Y_-\|^2_{\H^\smalltext{2}_{\smalltext{T}\smalltext{,}\smalltext{\hat\beta}}} + \|Z\|^2_{\H^\smalltext{2}_{\smalltext{T}\smalltext{,}\smalltext{\hat\beta}}(X)} + \|U\|^2_{\H^\smalltext{2}_{\smalltext{T}\smalltext{,}\smalltext{\hat\beta}}(\mu)} + \|N\|^2_{\cH^\smalltext{2}_{\smalltext{T}\smalltext{,}\smalltext{\hat\beta}}} < \infty,\\
		\|(Y,Z,U,N)\|^2_{\cL^\smalltext{(3)}_{\smalltext{T}\smalltext{,}\smalltext{\hat\beta}}} &\coloneqq \|Y\|^2_{\cS^\smalltext{2}_\smalltext{T}} + \|\alpha Y_-\|^2_{\H^\smalltext{2}_{\smalltext{T}\smalltext{,}\smalltext{\hat\beta}}} + \|Z\|^2_{\H^\smalltext{2}_{\smalltext{T}\smalltext{,}\smalltext{\hat\beta}}(X)} + \|U\|^2_{\H^\smalltext{2}_{\smalltext{T}\smalltext{,}\smalltext{\hat\beta}}(\mu)} + \|N\|^2_{\cH^\smalltext{2}_{\smalltext{T}\smalltext{,}\smalltext{\hat\beta}}} < \infty,
	\end{align*}
	respectively. We turn both $\cL^{(1)}_{T,\hat\beta}$ and $\cL^{(3)}_{T,\hat\beta}$ into Banach spaces by identifying $(Y,Z,U,N)$ and $(Y^\prime,Z^\prime,U^\prime,N^\prime)$ for which 
	\begin{equation*}
		\|(Y,Z,U,N) - (Y^\prime,Z^\prime,U^\prime,N^\prime)\|_{\cL^\smalltext{(1)}_{\smalltext{T}\smalltext{,}\smalltext{\hat\beta}}} = 0, \; \text{and} \; \|(Y,Z,U,N) - (Y^\prime,Z^\prime,U^\prime,N^\prime)\|_{\cL^\smalltext{(3)}_{\smalltext{T}\smalltext{,}\smalltext{\hat\beta}}} = 0,
	\end{equation*}
	respectively. The approach to deduce the existence of a unique fixed-point is then analogous to our previous argument. For $(i)$ and $(iii)$ we define the maps $\Upsilon_1 : \cL^{(1)}_{T,\hat\beta} \longrightarrow \cL^{(1)}_{T,\hat\beta}$ and $\Upsilon_3 : \cL^{(3)}_{T,\hat\beta} \longrightarrow \cL^{(3)}_{T,\hat\beta}$ analogously to $\Upsilon_2$, and then note that by \Cref{prop::apriori} the maps are well-defined. By the Lipschitz property of $f$, the \emph{a priori} estimates of \Cref{prop::apriori}, we find
	\begin{align*}
		\big\| \Upsilon_1(y^1,z^1,u^1,n^1) - \Upsilon_1(y^2,z^2,u^2,n^2) \big\|^2_{\cL^\smalltext{(1)}_{\smalltext{T}\smalltext{,}\smalltext{\hat\beta}}} 
		&= \|\delta Y\|^2_{\cS^\smalltext{2}_\smalltext{T}} + \|\alpha\delta Y\|^2_{\H^\smalltext{2}_{\smalltext{T}\smalltext{,}\smalltext{\hat\beta}}} + \|\delta Z\|^2_{\H^\smalltext{2}_{\smalltext{T}\smalltext{,}\smalltext{\hat\beta}}(X)} + \|\delta U\|^2_{\H^\smalltext{2}_{\smalltext{T}\smalltext{,}\smalltext{\hat\beta}}(\mu)} + \|\delta N\|^2_{\cH^{\smalltext{2}}_{\smalltext{T}\smalltext{,}\smalltext{\hat\beta}}} \nonumber\\
		&\leq M^\Phi_1(\hat\beta) \Big\| \frac{\psi}{\alpha} \Big\|^2_{\H^\smalltext{2}_{\smalltext{T}\smalltext{,}\smalltext{\hat\beta}}} \nonumber\\
		&\leq M^\Phi_1(\hat\beta) \Big(\|\alpha\delta y\|^2_{\H^\smalltext{2}_{\smalltext{T}\smalltext{,}\smalltext{\hat\beta}}} + \|\delta z\|^2_{\H^\smalltext{2}_{\smalltext{T}\smalltext{,}\smalltext{\hat\beta}}(X)} + \|\delta u\|^2_{\H^\smalltext{2}_{\smalltext{T}\smalltext{,}\smalltext{\hat\beta}}(\mu)} \Big)  \nonumber\\
		&\leq M^\Phi_1(\hat\beta) \Big( \|\delta y\|^2_{\cS^\smalltext{2}_\smalltext{T}} + \|\alpha\delta y\|^2_{\H^\smalltext{2}_{\smalltext{T}\smalltext{,}\smalltext{\hat\beta}}} + \|\delta z\|^2_{\H^\smalltext{2}_{\smalltext{T}\smalltext{,}\smalltext{\hat\beta}}(X)} + \|\delta u\|^2_{\H^\smalltext{2}_{\smalltext{T}\smalltext{,}\smalltext{\hat\beta}}(\mu)} \Big)  \nonumber\\
		&\leq M^\Phi_1(\hat\beta) \big\|(y^1, z^1, u^1, n^1) - (y^2, z^2, u^2, n^2)\big\|^2_{\cL^\smalltext{(1)}_{\smalltext{T}\smalltext{,}\smalltext{\hat\beta}}},
	\end{align*}
	and
	\begin{align*}
		\big\| \Upsilon_3(y^1,z^1,u^1,n^1) - \Upsilon_3(y^2,z^2,u^2,n^2) \big\|^2_{\cL^\smalltext{(3)}_{\smalltext{T}\smalltext{,}\smalltext{\hat\beta}}} 
		&= \|\delta Y\|^2_{\cS^\smalltext{2}_\smalltext{T}} + \|\alpha\delta Y_-\|^2_{\H^\smalltext{2}_{\smalltext{T}\smalltext{,}\smalltext{\hat\beta}}} + \|\delta Z\|^2_{\H^\smalltext{2}_{\smalltext{T}\smalltext{,}\smalltext{\hat\beta}}(X)} + \|\delta U\|^2_{\H^\smalltext{2}_{\smalltext{T}\smalltext{,}\smalltext{\hat\beta}}(\mu)} + \|\delta N\|^2_{\cH^{\smalltext{2}}_{\smalltext{T}\smalltext{,}\smalltext{\hat\beta}}} \nonumber\\
		&\leq M^\Phi_3(\hat\beta)\Big\| \frac{\psi}{\alpha} \Big\|^2_{\H^\smalltext{2}_{\smalltext{T}\smalltext{,}\smalltext{\hat\beta}}} \nonumber\\
		&\leq M^\Phi_3(\hat\beta) \Big(\|\alpha\delta y_-\|^2_{\H^\smalltext{2}_{\smalltext{T}\smalltext{,}\smalltext{\hat\beta}}} + \|\delta z\|^2_{\H^\smalltext{2}_{\smalltext{T}\smalltext{,}\smalltext{\hat\beta}}(X)} + \|\delta u\|^2_{\H^\smalltext{2}_{\smalltext{T}\smalltext{,}\smalltext{\hat\beta}}(\mu)} \Big)  \nonumber\\
		&\leq M^\Phi_3(\hat\beta) \Big( \|\delta y\|^2_{\cS^\smalltext{2}_\smalltext{T}} + \|\alpha\delta y_-\|^2_{\H^\smalltext{2}_{\smalltext{T}\smalltext{,}\smalltext{\hat\beta}}} + \|\delta z\|^2_{\H^\smalltext{2}_{\smalltext{T}\smalltext{,}\smalltext{\hat\beta}}(X)} + \|\delta u\|^2_{\H^\smalltext{2}_{\smalltext{T}\smalltext{,}\smalltext{\hat\beta}}(\mu)} \Big)  \nonumber\\
		&\leq M^\Phi_3(\hat\beta)\big\|(y^1, z^1, u^1, n^1) - (y^2, z^2, u^2, n^2)\big\|^2_{\cL^\smalltext{(3)}_{\smalltext{T}\smalltext{,}\smalltext{\hat\beta}}}.
	\end{align*} 
	Thus $\Upsilon_1$ (resp. $\Upsilon_3$) is a contraction if $(i)$ (resp. $(iii)$) holds. For both $(i)$ and $(iii)$ the representation \ref{cond::repre_Y} is immediate by \Cref{lem::cond_repre}. The stated uniqueness can be deduced similarly to before, we thus omit the details.
	
	\medskip
	Finally, if, in addition, $\xi^{+}\mathbf{1}_{[0,T)} \in \cS^2_{T,\beta}$ for some $\beta \in (0,\hat\beta)$, then $(K^r,K^\ell) \in \cI^2_{T,\beta} \times \cI^2_{T,\beta}$ in $(i)$, $(ii)$ and $(iii)$ by \Cref{prop::weighted_k_estimate}. This completes the proof.
\end{proof}

\begin{remark}\label{rem::missing_pieces}
	As we saw in the previous proof, the $Y$-component of the unique fixed-point of $\Upsilon_2$ necessarily has to satisfy \ref{cond::repre_Y}. So the class in which uniqueness can be deduced from the fixed-point property of $\Upsilon_2$ is necessarily defined by \ref{cond::repre_Y} as well. It turns out that the proof of existence and uniqueness in {\rm\cite{grigorova2020optimal}} of their reflected {\rm BSDE} overlooks this intricate point, as an argument like our {\rm\Cref{lem::cond_repre}} is missing.
\end{remark}

\begin{proof}[Proof of \Cref{thm::main_bsde}]
	The proof of this result is analogous to the proof of \Cref{thm::main}. The main difference is that we do not use the \emph{a priori} estimates of \Cref{prop::apriori}, but instead use \Cref{prop::apriori_bsde} to deduce the existence of a unique fixed-point in the three cases $(i)$, $(ii)$ and $(iii)$. Let us show that in all three cases $(i)$, $(ii)$ and $(iii)$, the component $Y$ is in $\cS^2_{T,\hat\beta}$. Note first that from $Y_S = \E\big[\xi_T + \int_S^T f_s\big(Y_s,Y_{s-},Z_s,U_s(\cdot)\big)\d C_s \big| \cG_S\big],$ it follows that
	\begin{align*}
		\cE(\hat\beta A)^{1/2}_S|Y_S| 
		&\leq \sqrt{2}\E\Bigg[\sqrt{\cE(\hat\beta A)_S|\xi_T|^2 + \cE(\hat\beta A)_S\bigg(\int_S^T  \big|f_s\big(Y_s,Y_{s-},Z_s,U_s(\cdot)\big)\big|\d C_s\bigg)^2} \Bigg| \cG_S\Bigg] \\
		&\leq \sqrt{2}\E\Bigg[\sqrt{\cE(\hat\beta A)_S|\xi_T|^2 + \frac{1}{\hat\beta}\int_S^T \cE(\hat\beta A)_s \frac{\big|f_s\big(Y_s,Y_{s-},Z_s,U_s(\cdot)\big)\big|^2}{\alpha^2_s}\d C_s} \Bigg| \cG_S\Bigg] \\
		&\leq \sqrt{2}\E\Bigg[\sqrt{\cE(\hat\beta A)_T|\xi_T|^2 + \frac{1}{\hat\beta}\int_0^T \cE(\hat\beta A)_s \frac{\big|f_s\big(Y_s,Y_{s-},Z_s,U_s(\cdot)\big)\big|^2}{\alpha^2_s}\d C_s} \Bigg| \cG_S\Bigg].
	\end{align*}
	Here the second line follows from the same arguments we used to deduce \eqref{eq::F_t_squared}. Thus, by \Cref{lem::cond_doob}, it follows that
	\begin{equation*}
		\|Y\|^2_{\cS^\smalltext{2}_{\smalltext{T}\smalltext{,}\smalltext{\hat\beta}}} 
		= \E\bigg[\sup_{s \in [0,T]}\big|\cE(\hat\beta A)^{1/2}_sY_s\big|^2\bigg] 
		\leq 8 \|\xi_T\|^2_{\L^\smalltext{2}_\smalltext{\hat\beta}} + \frac{8}{\hat\beta} \bigg\|\frac{f\big(Y,Y_-,Z,U(\cdot)\big)}{\alpha}\bigg\|^2_{\H^\smalltext{2}_{\smalltext{T}\smalltext{,}\smalltext{\hat\beta}}}.
	\end{equation*}
	Since $\|\xi_T\|_{\L^\smalltext{2}_\smalltext{\hat\beta}} < \infty$ and 
	\begin{align*}
		\bigg\|\frac{f\big(Y,Y_-,Z,U(\cdot)\big)}{\alpha}\bigg\|_{\H^\smalltext{2}_{\smalltext{T}\smalltext{,}\smalltext{\hat\beta}}} 
		&\leq \bigg\|\frac{f\big(Y,Y_-,Z,U(\cdot)\big)-f(0,0,0,\mathbf{0})}{\alpha}\bigg\|_{\H^\smalltext{2}_{\smalltext{T}\smalltext{,}\smalltext{\hat\beta}}} + \bigg\|\frac{f(0,0,0,\mathbf{0})}{\alpha}\bigg\|_{\H^\smalltext{2}_{\smalltext{T}\smalltext{,}\smalltext{\hat\beta}}} \\
		&= \|\alpha Y\|_{\H^\smalltext{2}_{\smalltext{T}\smalltext{,}\smalltext{\hat\beta}}} + \|\alpha Y_-\|_{\H^\smalltext{2}_{\smalltext{T}\smalltext{,}\smalltext{\hat\beta}}} + \|Z\|_{\H^\smalltext{2}_{\smalltext{T}\smalltext{,}\smalltext{\hat\beta}}(X)} + \|U\|_{\H^\smalltext{2}_{\smalltext{T}\smalltext{,}\smalltext{\hat\beta}}(\mu)} + \bigg\|\frac{f(0,0,0,\mathbf{0})}{\alpha}\bigg\|_{\H^\smalltext{2}_{\smalltext{T}\smalltext{,}\smalltext{\hat\beta}}} < \infty,
	\end{align*}
	we deduce that $\|Y\|_{\cS^\smalltext{2}_{\smalltext{T}\smalltext{,}\smalltext{\hat\beta}}} < \infty$. This completes the proof.
	
\end{proof}

\begin{remark}\label{rem::proof_main_stopping_norm}
	$(i)$ In the proof of {\rm\Cref{thm::main}} and {\rm\ref{thm::main_bsde}}, we could substitute the $\cS^2_T$--norm $\|\cdot\|_{\cS^\smalltext{2}_\smalltext{T}}$ by the $\cT^2_T$--norm $\|\cdot\|_{\cT^\smalltext{2}_\smalltext{T}}$ introduced in \Cref{rem::reflection_stopping_norm}. We merely need to keep in mind the changes to the \emph{a priori} estimates in both {\rm\Cref{prop::apriori}} and {\rm\ref{prop::apriori_bsde}} described in {\rm\Cref{rem::apriori_stopping_norm}}. Therefore, modifying the contraction constants as described in {\rm\Cref{rem::reflection_stopping_norm}} and \ref{rem::bsde_stopping_norm} would still ensure the desired well-posedness of our reflected BSDE and BSDE, as stated in {\rm\Cref{thm::main}} and {\rm\ref{thm::main_bsde}}, respectively.

	\medskip
	$(ii)$ In {\rm\cite[Theorem 3.5]{papapantoleon2018existence}} and its corresponding proof, the claim is that a weighted $\cS^2$-type norm for the $Y$-component is also sufficient to construct a contraction map in the \emph{BSDE} case. This corresponds to the norm $\|\cdot\|_{\smalltext{\star}\smalltext{,}\smalltext{\hat\beta}}$ in their notation. It does not seem as though this is actually possible since by the Lipschitz property of the generator, we are forced to use an $\H^2$-type norm on $Y$ in a fixed-point argument. Although in the classical cases this is possible since weighted $\cS^2$-norms and $\H^2$-norms on $Y$ are comparable, in their and our generality, the norms are not comparable.	
\end{remark}

\section{A comparison principle for BSDEs}\label{sec::comparison}

Comparison principles for BSDEs play a crucial role in the study of stochastic control problems as they give rise to necessary conditions optimisers ought to satisfy. In this section, we prove a comparison principle for our BSDEs. Given the several counterexamples for BSDEs with jumps in \citeauthor*{barles1997backward} \cite{barles1997backward}, \citeauthor*{royer2006backward} \cite{royer2006backward} and \citeauthor*{quenez2013bsdes} \cite{quenez2013bsdes}, we are forced to impose stronger assumptions on the generator. The conditions we lay out in this section will allow us to conclude that the operator which maps $\xi_T$ to the first component $Y(\xi_T)$ of the solution to the BSDE with generator $f$ and terminal condition $\xi_T$ is monotone, that is, if $\xi_T \leq \xi^\prime_T$, $\P$--a.s., then $Y(\xi_T) \leq Y(\xi^\prime_T)$, $\P$--a.s. The method of proof we use is the classical linearisation and change of measure argument. Here, we suppose that we are given another generator $f^\prime : \bigsqcup_{(\omega,t) \in \Omega \times [0,\infty)} \big(\R \times \R \times \R^m \times \fH_{\omega,t}\big) \longrightarrow \R$ such that $\Omega \times [0,\infty) \ni (\omega,t) \longmapsto f^\prime_t\big(Y_t(\omega),Y_{t-}(\omega),Z_t(\omega),U_t(\omega;\cdot)\big) \in \R,$ is optional for each $(Y,Z,U) \in \cS^2_T \times \H^2_T(X) \times \H^2_T(\mu)$. The Lipschitz property of $f$ then allows us to write
\begin{align*}
	f_s\big(\omega,y,\mathrm{y},z,u_s(\omega;\cdot)\big) - f^\prime_s\big(\omega,y^\prime,\mathrm{y}^\prime,z^\prime,u^\prime_s(\omega;\cdot)\big) 
	&\geq \lambda^{y,y^\smalltext{\prime}}_s(\omega)(y-y^\prime) + \widehat\lambda^{\mathrm{y},\mathrm{y}^\smalltext{\prime}}_s(\omega)(\mathrm{y}-\mathrm{y}^\prime)  + \eta^{z,z^\smalltext{\prime},\top}_s(\omega)c_s(\omega)(z-z^\prime) \\
	&\quad+ f_s\big(\omega,y^\prime,\mathrm{y}^\prime,z^\prime,u_s(\omega;\cdot)\big) - f_s\big(\omega,y^\prime,\mathrm{y}^\prime,z^\prime,u^\prime_s(\omega;\cdot)\big) \\
	&\quad + f_s\big(\omega,y^\prime,\mathrm{y}^\prime,z^\prime,u^\prime_s(\omega;\cdot)\big) - f^\prime_s\big(\omega,y^\prime,\mathrm{y}^\prime,z^\prime,u^\prime_s(\omega;\cdot)\big),
\end{align*}
where
\begin{equation*}
	\lambda^{y,y^\smalltext{\prime}}_s(\omega) \coloneqq -\sqrt{r_s(\omega)}\sgn(y-y^\prime), \; \widehat\lambda^{\mathrm{y},\mathrm{y}^\smalltext{\prime}}_s(\omega) \coloneqq - \sqrt{\mathrm{r}_s(\omega)}\sgn(\mathrm{y}-\mathrm{y}^\prime),
\end{equation*}
\begin{equation*}
	\eta^{z,z^\smalltext{\prime}}_s(\omega) \coloneqq -\sqrt{\theta^X_s(\omega)} \frac{(z-z^\prime)}{\big\|c^{1/2}_s(\omega)(z-z^\prime)\big\|}\mathbf{1}_{\{c^{\smalltext{1}\smalltext{/}\smalltext{2}}_\smalltext{s}(z-z^\smalltext{\prime})\neq 0\}}(\omega).
	\footnote{Recall from \Cref{sec::weighted_spaces} that $c^{1/2}$ is the unique (predictable) square-root matrix-valued process of $c$.}
\end{equation*}

The comparison result we present will be based on the following assumption.

\begin{assumption}\label{ass::comparison}
	The following conditions hold:
	\begin{enumerate}
		\item[$(i)$] $\Phi < 1$, and the non-negative random variables $\int_0^T \sqrt{r_s} \d C_s$, $\int_0^T \sqrt{\mathrm{r}_s} \d C_s$ and $\int_0^T \theta^X_s \d C_s$ are $\P$--{\rm a.s.} bounded$;$
		\item[$(ii)$] For each $\P$--{\rm a.s.} l\`adl\`ag process $Y \in \cS^2_T$,  and $(Z,Z^\prime,U,U^\prime) \in \H^2_T(X) \times \H^2_T(X) \times \H^2_T(\mu) \times \H^2_T(\mu)$, there exists $\rho = \rho^{Y,Z,U,U^\smalltext{\prime}} \in \H^2_T(\mu)$ such that\footnote{Here, $\eta^{Z,Z^\smalltext{\prime}} = (\eta^{Z,Z^\smalltext{\prime}}_t)_{t \in [0,\infty)}$ denotes the predictable process defined by $\eta^{Z,Z^\smalltext{\prime}}_t(\omega) \coloneqq \eta^{Z_t(\omega),Z^\smalltext{\prime}_t(\omega)}_t(\omega)$.} $\eta^{Z,Z^\smalltext{\prime}}_t\Delta X_{t \land T} + \Delta (\rho \star \tilde\mu)_{t \land T} > -1$, $\P$--{\rm a.s.}, $t \in [0,\infty)$, the random variable $\langle \rho \star \tilde\mu\rangle_T$ is bounded, $\P$--{\rm a.s.}, and
			\begin{equation}\label{eq::comparison_inequality}
				f_s\big(Y_s,Y_{s-},Z_s,U_s(\cdot)\big) - f_s\big(Y_s,Y_{s-},Z_s,U^\prime_s(\cdot)\big) 
				\geq \frac{\d\langle \rho \star\tilde\mu,(U-U^\prime)\star\tilde\mu\rangle_s}{\d C_s}, \; \text{$\P \times \d C_s$--{\rm a.e.}}
			\end{equation}
	\end{enumerate}
\end{assumption}

\begin{remark}\label{rem::comparison_assumption}
$(i)$ In the standard case, where there is only a Brownian motion $X$, the stopping time $T$ is deterministic and finite, $C$ satisfies $\d C_s = \d s$, and the Lipschitz-coefficients of the generator $f$ are bounded, $\mathrm{\Cref{ass::comparison}}$ is clearly satisfied. However, if additionally there is an integer-valued random measure $\mu$ such that its compensator can be written in the form $\nu(\d s, \d x) = F(\d x)\d t$, then \eqref{eq::comparison_inequality} turns into
	\begin{equation*}
		f_s\big(Y_s,Y_{s-},Z_s,U_s(\cdot)\big) - f_s\big(Y_s,Y_{s-},Z_s,U^\prime_s(\cdot)\big) \geq \frac{\d\langle \rho \star\tilde\mu,(U-U^\prime)\star\tilde\mu\rangle_s}{\d C_s} = \int_E \rho_s(x) U_s(x) F(\d x), \; \text{$\P\times\d s$--a.e.}
	\end{equation*}
	This is now reminiscent of the classical $(A_\gamma)$--condition in {\rm\cite[Section 2.2, p. 1362]{royer2006backward}} or the assumption in {\rm\cite[Theorem 4.1]{quenez2013bsdes}}.

\medskip
$(ii)$ The conditions in {\rm\Cref{ass::comparison}} are simple enough to check in practice, but are not necessary, as some of them can be weakened by carefully redoing the proof of the comparison principle in {\rm\Cref{prop::comparison}}. We will discuss this in {\rm\Cref{rem::comparison_proof}}.

\medskip
$(iii)$ If the condition $\eta^{Z,Z^{\prime}} \Delta X_{\cdot \land T} + \Delta (\rho \star \tilde\mu)_{\cdot \land T} > -1$ in {\rm\Cref{ass::comparison}.$(ii)$} fails to hold, then the comparison principle for our {\rm BSDE} is false in general. See {\rm\cite[Example 3.1]{quenez2013bsdes}} for a counterexample in a Brownian--Poisson setting.
\end{remark}

The following comparison principle is the main result of this section.

\begin{proposition}\label{prop::comparison}
	Suppose that {\rm\Cref{ass::comparison}} holds. Let $(\xi_T,\xi^\prime_T) \in (\L^2(\cG_T))^2$, and suppose that $(Y,Z,U,N)$ and $(Y^\prime,Z^\prime,U^\prime,N^\prime)$ are solutions in $\cS^2_T \times \H^2_T(X) \times \H^2_T(\mu) \times \cH^{2,\perp}_T(X,\mu)$ to the {\rm BSDEs} with generator $f$ and terminal condition $\xi_T$ and generator $f^\prime$ and terminal condition $\xi^\prime_T$, respectively. If $\xi^\prime_T \leq \xi_T$, $\P${\rm--a.s.}, and
	\begin{equation*}
		f^\prime_s\big(Y^\prime_s,Y^\prime_{s-},Z^\prime_s,U^\prime_s(\cdot)\big) \leq f_s\big(Y^\prime_s,Y^\prime_{s-},Z^\prime_s,U^\prime_s(\cdot)\big), \; \text{$\P \otimes \d C_s$\rm--a.e.},
	\end{equation*}
	 then $Y^\prime \leq Y$ up to $\P$--indistinguishability.	
\end{proposition}

The proof of this result is based on the following lemma.

\begin{lemma}\label{lem::stoch_exp}
	Suppose that $M \in \cM_\text{\rm loc}$ with $M_0 = 0$ is such that $\langle M \rangle$ is bounded, where $\langle M \rangle$ denotes the compensator of the optional quadratic variation $[M]$. Then $\cE(M) \in \cH^2$.
\end{lemma}

\begin{proof}
	Although this result follows from \cite[Proposition 8.27]{jacod1979calcul}, we would like to present another argument by following the proof of \cite[Th\'eor\`eme II.3]{lepingle1978sur}. By \cite[Proposition II.1]{lepingle1978sur2}, we can write
	\begin{equation*}
		|\cE(M)|^2 = \cE(M)\cE(M) = \cE(2M + [M]) = \cE(2M + [M] - \langle M \rangle + \langle M \rangle) = \cE(\widetilde N) \cE(\langle M \rangle),
	\end{equation*}
	for some $\widetilde N \in \cM_\text{\rm loc}$. Since $\langle M \rangle$ is non-decreasing, we have that $\cE(\langle M \rangle) \leq \mathrm{e}^{\langle M \rangle}$. In particular, the local martingale satisfies $\cE(\widetilde N) \geq 0$ and is thus a supermartingale. Since $\langle M \rangle$ is bounded, say by $a \in (0,\infty)$, we can write $|\cE(M)|^2 \leq \cE(\widetilde N) \mathrm{e}^a$. Let $(\tau_n)_{n \in \N}$ be a localising sequence such that each stopped process $\cE(M)_{\cdot \land \tau_\smalltext{n}}$ is a uniformly integrable martingale, and fix $t \in (0,\infty)$. We find
	\begin{equation*}
		\E\bigg[\sup_{s \in [0,t]} |\cE(M)_{s \land \tau_\smalltext{n}}|^2\bigg] \leq 4 \E[|\cE(M)_{t \land \tau_\smalltext{n}}|^2] \leq 4 \mathrm{e}^a \E[\cE(\widetilde N)_{t \land \tau_\smalltext{n}}] \leq 4 \mathrm{e}^a\E[\cE(\widetilde N)_0] = 4 \mathrm{e}^a.
	\end{equation*}
	Here the first inequality follows from Doob's $\L^2$-inequality, and the third inequality follows from the optional stopping theorem (see \cite[Theorem 3.2.7]{weizsaecker1990stochastic}) together with the fact that $\cE(\widetilde N)$ is a supermartingale. By applying Fatou's lemma twice, we deduce that $\E\big[\sup_{s \in [0,\infty)} |\cE(M)_s|^2\big] \leq 4 \mathrm{e}^a.$ In particular, the stochastic exponential $\cE(M)$ is a uniformly integrable martingale and is in $\cH^2$. This completes the proof.
	\end{proof}

\begin{proof}[Proof of \Cref{prop::comparison}]
	First, let us write 
	\begin{equation*}
		\delta Y \coloneqq Y-Y^\prime, \; \delta Z \coloneqq Z-Z^\prime, \; \delta U \coloneqq U - U^\prime, \; \delta N \coloneqq N - N^\prime, \; \delta \xi_T \coloneqq \xi_T - \xi^\prime_T,
	\end{equation*}
	\begin{equation*}
		\delta f \coloneqq f\big(Y,Y_-,Z,U(\cdot)) - f^\prime\big(Y^\prime,Y^\prime_-,Z^\prime,U^\prime(\cdot)\big),
	\end{equation*}
	and
	\begin{equation*}
		\lambda_s(\omega) \coloneqq \lambda^{Y_\smalltext{s}(\omega),Y^\smalltext{\prime}_\smalltext{s}(\omega)}_s(\omega), 
		\; \widehat\lambda_s(\omega) \coloneqq \widehat\lambda^{Y_{\smalltext{s}\smalltext{-}}(\omega),Y^\smalltext{\prime}_{\smalltext{s}\smalltext{-}}(\omega)}_s(\omega), 
		\; \eta_s(\omega) \coloneqq \eta^{Z_\smalltext{s}(\omega),Z^\smalltext{\prime}_\smalltext{s}(\omega)}_s(\omega),
		\; \rho_s(\omega;x) = \rho^{Y^\smalltext{\prime},Z^\smalltext{\prime},U,U^\smalltext{\prime}}_s(\omega;x),
	\end{equation*}
	for simplicity. Now consider $v \coloneqq \int_0^{\cdot \land T} \gamma_s \d C_s,$ where  $\gamma \coloneqq \frac{\widehat\lambda}{1-\widehat\lambda\Delta C}.$ Here the process $\gamma$ is predictable and the integral is well-defined by \Cref{ass::comparison}.$(i)$ since $\widehat\lambda_s \Delta C_s \leq |\widehat\lambda_s| \Delta C_s \leq \sqrt{\mathrm{r}_s}\Delta C_s \leq \Phi < 1 $ and thus
	\begin{equation*}
		|\gamma| \leq \frac{|\widehat\lambda|}{1-\Phi} \leq \frac{\sqrt{\mathrm{r}_s}}{1-\Phi}.
	\end{equation*}
	Moreover, $\Delta v = \gamma\Delta C = \widehat\lambda\Delta C/(1-\widehat\lambda\Delta C) > -1$. Thus, the stochastic exponential $\cE(v)$ is positive and because $v$ is predictable and of finite variation, the stochastic exponential $\cE(v)$ is predictable and satisfies $0 < \cE(v) \leq \mathrm{e}^{v}$. In particular, $\cE(v)$ is bounded by \Cref{ass::comparison}.$(i)$.	With the integration by parts formula, we derive
	\begin{align}\label{eq::ito_comparison}
		\d \big(\cE(v)\delta Y\big)_s &= \cE(v)_{s-}\d (\delta Y)_s + \delta Y_{s-} \d \cE(v)_s + \d [\cE(v),\delta Y]_s \nonumber\\
		&= -\cE(v)_{s-} \delta f_s \d C_s + \cE(v)_{s-} \d(\delta Z \bcdot X)_s + \cE(v)_{s-}\d (\delta U \star\tilde\mu)_s + \cE(v)_{s-}\d \delta N_s + \cE(v)_{s-}\delta Y_{s-}\gamma_s\d C_s \nonumber\\
		&\quad + \cE(v)_{s-}\gamma_s\d [C,\delta Y]_s \nonumber\\
		&= -\cE(v)_{s-} \delta f_s \d C_s + \cE(v)_{s-} \d(\delta Z \bcdot X)_s + \cE(v)_{s-}\d (\delta U \star\tilde\mu)_s + \cE(v)_{s-}\d \delta N_s + \cE(v)_{s-}\delta Y_{s-}\gamma_s\d C_s \nonumber\\
		&\quad - \cE(v)_{s-}\gamma_s\delta f_s \Delta C_s \d C_s + \cE(v)_{s-}\gamma_s \d [C,\delta Z \bcdot X]_s + \cE(v)_{s-}\gamma_s \d [C,\delta U \star\tilde\mu]_s + \cE(v)_{s-}\gamma_s \d [C,\delta N]_s \nonumber\\
		&= - \cE(v)_{s-} \big(\delta f_s(1+\gamma_s\Delta C_s) - \gamma_s\delta Y_{s-}\big)\d C_s + \cE(v)_{s-} \d(\delta Z \bcdot X)_s + \cE(v)_{s-}\d (\delta U \star\tilde\mu)_s + \cE(v)_{s-}\d \delta N_s \nonumber\\
		&\quad + \cE(v)_{s-}\gamma_s \d [C,\delta Z \bcdot X]_s + \cE(v)_{s-}\gamma_s \d [C,\delta U \star\tilde\mu]_s + \cE(v)_{s-}\gamma_s \d [C,\delta N]_s.
	\end{align}
	Note that the processes on the last line are local martingales by \cite[Proposition I.4.49.(c)]{jacod2003limit} since $C$ is predictable. We now define the probability measure $\Q$ on $(\Omega,\cG)$ through the density
	\begin{equation}\label{eq::stoch_exp_measure_change}
		\frac{\d\Q}{\d\P} \coloneqq \cE(L)_\infty \coloneqq \cE\bigg(\int_0^{\cdot \land T}\eta_s\d X_s + \rho \star \tilde\mu_{\cdot \land T}\bigg)_\infty.
	\end{equation}
	\Cref{ass::comparison} implies that $\eta \in \H^2(X^T)$, that $\cE(L)$ is non-negative since $\Delta L = \eta \Delta X + \Delta (\rho\star\tilde\mu) > -1$ and that $\langle L \rangle$ is bounded since
	\begin{align*}
		\langle L \rangle_T = \int_0^T \eta^\top_s c_s \eta_s \d C_s + \langle \rho\star\tilde\mu\rangle_T = \int_0^T \theta^X_s \d C_s + \langle \rho\star\tilde\mu\rangle_T.
	\end{align*}
	Therefore, $\cE(L) \in \cH^2$ by \Cref{lem::stoch_exp}, and $\cE(L)_\infty = \cE(L)_T \in \L^2(\cG_T;\P)$. We now rewrite the $\P$--local martingales appearing in \eqref{eq::ito_comparison} as $\Q$--semimartingales. By an application of Girsanov's theorem \cite[Proposition 7.25 and 7.26]{jacod1979calcul}, we find the $\Q$--semimartingale decompositions
	\begingroup
	\allowdisplaybreaks
	\begin{align*}
		\cE(v)_{s-} \d (\delta Z \bcdot X)_s 
		&= \cE(v)_{s-}\d (\delta Z\bcdot X - \langle \delta Z \bcdot X,L \rangle)_s + \cE(v)_{s-}\d\langle \delta Z \bcdot X,L \rangle_s \\
		&= \cE(v)_{s-}\d (\delta Z \bcdot X - \langle \delta Z \bcdot X,L \rangle)_s + \cE(v)_{s-}\d\langle \delta Z \bcdot X, \eta \bcdot X \rangle_s \\
		&= \cE(v)_{s-}\d (\delta Z \bcdot X - \langle \delta Z\bcdot X,L \rangle)_s + \cE(v)_{s-} \eta^\top_s c_s\delta Z_s \d C_s,\\
		\cE(v)_{s-}\d(\delta U \star \tilde\mu)_s &= \cE(v)_{s-}\d\big(\delta U \star \tilde\mu - \langle \delta U \star\tilde\mu,L\rangle\big)_s + \cE(v)_{s-}\d \langle \delta U \star\tilde\mu,L\rangle_s \\
		&= \cE(v)_{s-}\d\big(\delta U \star \tilde\mu - \langle \delta U \star\tilde\mu,L\big)_s + \cE(v)_{s-}\d \langle \delta U \star\tilde\mu,\rho\star\tilde\mu\rangle_s,\\
		\cE(v)_{s-}\d \delta N_s &= \cE(v)_{s-}\d(\delta N-\langle \delta N, L \rangle_s) + \cE(v)_{s-}\d\langle \delta N,L\rangle_s = \cE(v)_{s-}\d\delta N_s,
	\end{align*}
	where the last equality follows from the $\P$--orthogonality of $N$ with respect to $X$ and $\mu$ from \ref{data::martingale},
	\begin{align*}
		\cE(v)_{s-}\gamma_s \d[C,\delta Z \bcdot X]_s 
		&= \cE(v)_{s-}\gamma_s \d\big([C,\delta Z \bcdot X]-\big\langle [C,\delta Z \bcdot X],L \big\rangle\big)_s + \cE(v)_{s-}\gamma_s \d\big\langle [C,\delta Z \bcdot X],L \big\rangle_s \\
		&= \cE(v)_{s-}\gamma_s \d\big([C,\delta Z \bcdot X]-\big\langle [C,\delta Z \bcdot X],L \big\rangle\big)_s + \cE(v)_{s-}\gamma_s \Delta C_s \d\langle \delta Z \bcdot X,L \rangle_s \\
		&= \cE(v)_{s-}\gamma_s \d\big([C,\delta Z \bcdot X]-\big\langle [C,\delta Z \bcdot X],L \big\rangle\big)_s + \cE(v)_{s-} \gamma_s \Delta C_s \eta_s^\top c_s \delta Z_s \d C_s,
	\end{align*}
	here the second to last equality follows from $\d [C,\delta Z \bcdot X]_s = \Delta C \d(\delta Z \bcdot X)_s$, see \cite[Proposition I.4.49]{jacod2003limit}, and similarly,
	\begin{align*}
		\cE(v)_{s-}\gamma_s \d[C,\delta U \star\tilde\mu]_s 
		&= \cE(v)_{s-}\gamma_s\d\big([C,\delta U \star\tilde\mu]-\big\langle [C,\delta U \star\tilde\mu],L \big\rangle\big)_s + \cE(v)_{s-} \gamma_s \d\big\langle [C,\delta U \star\tilde\mu],L \big\rangle_s \\
		&= \cE(v)_{s-}\gamma_s \d\big([C,\delta U \star\tilde\mu]-\big\langle [C,\delta U \star\tilde\mu],L \big\rangle\big)_s + \cE(v)_{s-}\gamma_s \Delta C_s \d\langle \delta U \star\tilde\mu,\rho\star\tilde\mu \rangle_s,\\
		\cE(v)_{s-}\gamma_s \d [C,N]_s 
		&= \cE(v)_{s-}\gamma_s \d \big([C,\delta N] - \big\langle [C,\delta N],L\big\rangle \big)_s + \cE(v)_{s-}\d\big\langle [C,\delta N],L\big\rangle_s \\
		&= \cE(v)_{s-}\gamma_s \d \big([C,\delta N] - \big\langle [C,\delta N],L\big\rangle \big)_s + \cE(v)_{s-}\Delta C_s\d\big\langle \delta N,L\big\rangle_s \\
		&= \cE(v)_{s-}\gamma_s \d \big([C,\delta N] - \big\langle [C,\delta N],L\big\rangle \big)_s \\
		&= \cE(v)_{s-}\gamma_s \d [C,\delta N]_s.
	\end{align*}
	Inserting the previous identities into \eqref{eq::ito_comparison} yields
	\begin{align}\label{eq::ito_comparison2}
		\d \big(\cE(v)\delta Y\big)_s 
		&= - \cE(v)_{s-} \big(\delta f_s(1+\gamma_s\Delta C_s) - \gamma_s\delta Y_{s-}\big)\d C_s + \cE(v)_{s-} \d(\delta Z \bcdot X)_s \nonumber\\
		&\quad + \cE(v)_{s-}\d (\delta U \star\tilde\mu)_s + \cE(v)_{s-}\d \delta N_s +\cE(v)_{s-}\gamma_s \d [C,\delta Z \bcdot X]_s \nonumber\\
		&\quad + \cE(v)_{s-}\gamma_s \d [C,\delta U \star\tilde\mu]_s + \cE(v)_{s-}\gamma_s \d [C,\delta N]_s \nonumber\\
		&= - \cE(v)_{s-} \big(\delta f_s(1+\gamma_s\Delta C_s) - \gamma_s\delta Y_{s-}\big)\d C_s + \cE(v)_{s-}\d (\delta Z \bcdot X_s - \langle \delta Z \bcdot X,L \rangle)_s \nonumber\\
		&\quad + \cE(v)_{s-} \eta^\top_s c_s\delta Z_s \d C_s + \cE(v)_{s-}\d\big(\delta U \star \tilde\mu - \langle \delta U \star\tilde\mu,L\rangle\big)_s + \cE(v)_{s-}\d \langle \delta U \star\tilde\mu,\rho\star\tilde\mu\rangle_s \nonumber\\
		&\quad + \cE(v)_{s-}\d \delta N_s + \cE(v)_{s-}\gamma_s \d\big([C,\delta Z \bcdot X]-\big\langle [C,\delta Z \bcdot X],L \big\rangle\big)_s + \cE(v)_{s-}\eta^\top_s c_s \delta Z_s \gamma_s \Delta C_s \d C_s \nonumber\\
		&\quad + \cE(v)_{s-}\gamma_s \d\big([C,\delta U \star\tilde\mu]-\big\langle [C,\delta U \star\tilde\mu],L \big\rangle\big)_s + \cE(v)_{s-}\gamma_s \Delta C_s \d\langle \delta U \star\tilde\mu,\rho\star\tilde\mu \rangle_s \nonumber\\
		&\quad + \cE(v)_{s-}\gamma_s \d [C,\delta N]_s \nonumber\\
		&= -  \cE(v)_{s-} \big(\delta f_s(1+\gamma_s\Delta C_s) - \gamma_s\delta Y_{s-}\big)\d C_s + \cE(v)_{s-}\d (\delta Z \bcdot X_s - \langle \delta Z \bcdot X,L \rangle)_s \nonumber\\
		&\quad + \cE(v)_{s-} \eta^\top_s c_s\delta Z_s (1+\gamma_s\Delta C_s) \d C_s + \cE(v)_{s-}\d\big(\delta U \star \tilde\mu - \langle \delta U \star\tilde\mu,L\rangle\big)_s \nonumber\\
		&\quad + \cE(v)_{s-}(1+\gamma_s \Delta C_s)\d \langle \delta U \star\tilde\mu,\rho\star\tilde\mu\rangle_s  + \cE(v)_{s-}\d \delta N_s \nonumber\\
		&\quad + \cE(v)_{s-}\gamma_s \d\big([C,\delta Z \bcdot X]-\big\langle [C,\delta Z \bcdot X],L \big\rangle\big)_s + \cE(v)_{s-}\gamma_s \d\big([C,\delta U \star\tilde\mu]-\big\langle [C,\delta U \star\tilde\mu],L \big\rangle\big)_s \nonumber\\
		&\quad + \cE(v)_{s-}\gamma_s \d [C,\delta N]_s.
	\end{align}
	Note now that $\cE(v)_{-}(1+\gamma\Delta C) = \cE(v) > 0$, and thus
	\begin{align*}
		\cE(v)_{s-}\delta f_s(1+\gamma_s\Delta C_s) 
		&\geq \cE(v)_{s-}\bigg(\lambda_s\delta Y_s + \widehat\lambda_s\delta Y_{s-} + \eta^\top_s c_s \delta Z_s + \frac{\d\langle \delta U \star\tilde\mu, \rho\star\tilde\mu\rangle}{\d C_s}\bigg)(1+\gamma_s\Delta C_s) \\
		&= \cE(v)_{s-}\bigg(\lambda_s\delta Y_s + \eta^\top_s c_s \delta Z_s + \frac{\d\langle \delta U \star\tilde\mu, \rho\star\tilde\mu\rangle}{\d C_s}\bigg)(1+\gamma_s\Delta C_s) + \cE(v)_{s-}\gamma_s\delta Y_{s-}.
	\end{align*}
	Therefore
	\begin{align*}
		\d \big(\cE(v)\delta Y\big)_s 
		& \leq - \cE(v)_{s-}\lambda_s \delta Y_s (1+\gamma_s\Delta C_s) \d C_s + \d M^\Q_s  = - \cE(v)_s\lambda_s \delta Y_s \d C_s + \d M^\Q_s.
	\end{align*}
	Here we use $\cE(v) = \cE(v)_-(1+\gamma\Delta C)$ in the last line, and we denote by $M^\Q$ the sum of the $\Q$--local martingales appearing in \eqref{eq::ito_comparison2}. Consider now $w \coloneqq \int_0^{\cdot \land T} \lambda_s \d C_s$. Since $|\lambda_s \Delta C_s| = |\lambda_s|\Delta C_s \leq \sqrt{r_s} \Delta C_s \leq \Phi < 1$, we have $\cE(w) > 0$. Moreover, since $|\lambda_s| \leq \sqrt{r_s}$, \Cref{ass::comparison}.$(i)$ ensures that $w$ is bounded. As before, this also implies that $\cE(w)$ is bounded since $\cE(w)\leq\mathrm{e}^w$. Another application of the integration by parts formula yields
	\begin{align*}
		\d (\cE(w)\cE(v)\delta Y)_s 
		&= \cE(w)_{s-}\d (\cE(v)\delta Y)_s + (\cE(v)\delta Y)_{s-}\d \cE(w)_s + \d [\cE(w),\cE(v)\delta Y] \\
		&= \cE(w)_{s-}\d (\cE(v)\delta Y)_s + (\cE(v)\delta Y)_s\d \cE(w)_s \\
		&= \cE(w)_{s-}\d (\cE(v)\delta Y)_s + \cE(w)_{s-}\cE(v)_s \lambda_s \delta Y_s\d C_s \\
		&\leq -\cE(w)_{s-}\cE(v)_s\lambda_s\delta Y_s \d C_s + \cE(w)_{s-}\d M^\Q_s + \cE(w)_{s-}\cE(v)_s \lambda_s \delta Y_s\d C_s = \cE(w)_{s-}\d M^\Q_s.
	\end{align*}
	This implies that, $\P$--a.s., for each $(t,t^\prime) \in [0,\infty)$ with $t \leq t^\prime$, we have
	\begin{equation}\label{eq::comparison_ineq_martingale}
		\cE(w)_{t \land T}\cE(v)_{t \land T}\delta Y_{t \land T} \geq \cE(w)_{t^\smalltext{\prime} \land T}\cE(v)_{t^\smalltext{\prime} \land T}\delta Y_{t^\smalltext{\prime} \land T} - \int_{t \land T}^{t^\smalltext{\prime} \land T}\cE(w)_{s-}\d M^\Q_s.
	\end{equation}
	Let $(\tau_n)_{n \in \N}$ be a $\Q$--localising sequence such that for each $n \in \N$, the stopped process $\int_0^{\cdot \land \tau_\smalltext{n}} \cE(w)_{s-} \d M^\Q_s,$ is a uniformly integrable martingale under $\Q$. By taking the conditional expectation in \eqref{eq::comparison_ineq_martingale}, we find
	\begin{equation}\label{eq::comp_limit_n}
		\cE(w)_{t \land \tau_\smalltext{n} \land T}\cE(v)_{t \land \tau_\smalltext{n} \land T}\delta Y_{t \land \tau_\smalltext{n} \land T} \geq \E^\Q\bigg[ \cE(w)_{t^\smalltext{\prime} \land \tau_\smalltext{n} \land T}\cE(v)_{t^\smalltext{\prime} \land \tau_\smalltext{n} \land T}\delta Y_{t^\smalltext{\prime} \land \tau_\smalltext{n} \land T} \bigg| \cG_t\bigg].
	\end{equation}
	Since $\cE(v)$ and $\cE(w)$ are bounded, $\delta Y \in \cS^2_T$ and $\d\Q/\d\P \in \L^2(\cG,\P)$, we can apply dominated convergence with $n \longrightarrow \infty$ and find
	\begin{equation}\label{eq::comp_limit_t}
		\cE(w)_{t \land T}\cE(v)_{t \land T}\delta Y_{t \land \tau_\smalltext{n} \land T} \geq \E^\Q\bigg[ \cE(w)_{t^\smalltext{\prime} \land T}\cE(v)_{t^\smalltext{\prime} \land T}\delta Y_{t^\smalltext{\prime} \land T} \bigg| \cG_t\bigg].
	\end{equation}
	Applying dominated convergence again, now letting $t^\prime \uparrow\uparrow \infty$, yields
	\begin{equation}\label{eq::comparison_cond_exp}
		\cE(w)_{t \land T}\cE(v)_{t \land T}\delta Y_{t \land T} \geq \E^\Q\bigg[ \cE(w)^T_{\infty-}\cE(v)^T_{\infty-}\delta Y^T_{\infty-} \bigg| \cG_t\bigg].
	\end{equation}
	Since $\cG_{\infty} = \cG_{\infty-} = \bigvee_{t \in [0,\infty)} \cG_t$, we deduce from
	\begin{equation*}
		\int_0^t \delta f_s \mathbf{1}_{[0,T]} \d C_s + \delta Y^T_t = \E \bigg[\delta \xi_T + \int_0^T \delta f_s \d C_s \bigg| \cG_t\bigg] \eqqcolon M_t, \; t \in [0,\infty),
	\end{equation*}
	that 
	\begin{equation*}
		\int_0^T \delta f_s \d C_s + \delta Y^T_{\infty-} = M_{\infty-} = \delta \xi_T + \int_0^T \delta f_s \d C_s, \;\text{$\P$--a.s.},
	\end{equation*}
	and thus $\delta Y^T_{\infty-} = \delta\xi_T = \xi_T - \xi^\prime_T \geq 0$, $\P$--a.s. Then \eqref{eq::comparison_cond_exp} implies 
	\begin{equation}\label{eq::conluding_comparison}
		\cE(w)_{t \land T}\cE(v)_{t \land T}\delta Y_{t \land T} \geq 0, \; \text{$\Q$--a.s.}
	\end{equation}
	Because $\cE(w)$ and $\cE(v)$ are positive under $\P$ and thus under $\Q$, this implies that $\delta Y_{t \land T} \geq 0$, $\Q$--a.s. Since $\Q$ and $\P$ are equivalent probability measures on each $(\Omega,\cG_t)$, $t \in [0,\infty)$, we also have $\delta Y_{t \land T} \geq 0$, $\P$--a.s., and therefore $Y \geq Y^\prime$, $\P$--a.s. by ($\P$--a.s.) right-continuity of $Y$ and $Y^\prime$. This completes the proof.
	\endgroup
\end{proof}

\begin{remark}\label{rem::comparison_proof}
	As we mentioned in {\rm\Cref{rem::comparison_assumption}}, one could weaken {\rm\Cref{ass::comparison}}. For the arguments laid out in the proof of {\rm\Cref{prop::comparison}} to go through, we have to ensure the following$:$
	
	\medskip
	$(i)$ the probability measure $\Q$ defined in \eqref{eq::stoch_exp_measure_change} has to be locally equivalent to $\P$, that is, the restriction $\Q_t$ of $\Q$ to $\cG_t$ is equivalent to the restriction $\P_t$ of $\P$ to $\cG_t$ for each $t \in [0,\infty);$
	
	\medskip
	$(ii)$ we need to be able to let $n \rightarrow \infty$ and $t^\prime \uparrow\uparrow \infty$ inside the conditional expectations of \eqref{eq::comp_limit_n} and \eqref{eq::comp_limit_t}$;$
	
	\medskip
	$(iii)$ the stochastic exponentials $\cE(w)$ and $\cE(v)$ have to be positive, so that we can divide both sides of \eqref{eq::conluding_comparison} by $\cE(w)\cE(v)$.
\end{remark}

\appendix

\section{Proofs of Section \ref{sec::preliminaries} and \ref{sec::main_results}}\label{app::proofs_preliminaries}

\begin{proof}[Proof of \Cref{prop::completeness_ivm}]\label{proof::completeness_ivm}
	That $\|\cdot\|_{\H^\smalltext{2}(\mu)}$ is a norm and that the stated equalities hold is clear from our preceding discussion. It remains to show completeness. Let $(U^k)_{k \in \N}$ be a Cauchy sequence in $\H^2(\mu)$. 
	Then, in particular,
	\begin{equation*}
		\lim_{(k,\ell) \rightarrow \infty}\|U^k - U^\ell\|^2_{\H^\smalltext{2}(\mu)} = \lim_{(k,\ell) \rightarrow \infty}\E\big[\langle (U^k - U^\ell) \star \tilde\mu\rangle_\infty \big] = \lim_{(k,\ell) \rightarrow \infty}\E\big[\langle U^k \star \tilde\mu - U^\ell \star \tilde\mu\rangle_\infty \big] = 0.
	\end{equation*}
	Hence the sequence $(U^k \star \tilde\mu)_{k \in \N}$ is Cauchy in $\cK^2(\mu) \subseteq \cH^2$. Since $\cK^2(\mu)$ is closed in $\cH^2$, there exists $U \star \tilde\mu \in \cK^2(\mu)$ with
	\begin{equation*}
		\lim_{k \rightarrow \infty}\E\big[\langle U^k \star \tilde\mu - U \star \tilde\mu \rangle_\infty \big] = 0.
	\end{equation*}
	Therefore $\lim_{k \rightarrow \infty}\|U^k - U\|^2_{\H^\smalltext{2}(\mu)} = \lim_{k \rightarrow \infty}\E\big[\langle (U^k - U) \star \tilde\mu\rangle_\infty \big] = \lim_{k \rightarrow \infty}\E\big[\langle U^k \star \tilde\mu - U \star \tilde\mu\rangle_\infty \big]= 0.$
\end{proof}

\begin{proof}[Proof of \Cref{prop::orthogonal_decomposition}]\label{proof::orthogonal_decomposition}
	Let us start by showing that $\cL^2(X) \cap \cK^2(\mu)$ is the null subspace of $\cH^2$. For $M \in \cL^2(X) \cap \cK^2(\mu)$, we can write
	\begin{equation*}
		M = \int_{(0,\cdot]} Z_s \d X_s = \int_{(0,\cdot] \times E} U_s(x)\tilde\mu(\d s, \d x).
	\end{equation*}
	Therefore
	\begin{equation*}
		\langle M, M \rangle = \bigg\langle M, \int_{(0,\cdot] \times E} U_s(x)\tilde\mu(\d s, \d x) \bigg\rangle = 0,
	\end{equation*}
	by \cite[Theorem 13.3.16]{cohen2015stochastic} or \cite[Lemme 7.39]{jacod1979calcul} since $M_{\mu}\big[\Delta M \big| \widetilde\cP\big ] = M_{\mu} \big[Z^\top \Delta X \big| \widetilde\cP \big] = \sum_{i = 1}^n Z^i M_{\mu} \big[\Delta X^i \big| \widetilde\cP \big] = 0,$ which implies that $M = M_0 = 0$. We now define $\cH^{2,\perp}(X,\mu) = \big( \cL^2(X) \oplus \cK^2(\mu) \big)^\perp$, which by the previous considerations lead to a decomposition $\cH^2 = \cL^2(X) \oplus \cK^2(\mu) \oplus \cH^{2,\perp}(X,\mu).$ We now fix $M \in \cH^2$. By the previous considerations, we can decompose $M$ uniquely as
	\begin{equation*}
		M = \int_{(0,\cdot]} Z_s \d X_s + \int_{(0,\cdot] \times E} U_s(x) \tilde\mu(\d s, \d x) + N,
	\end{equation*}
	where uniqueness is meant in the spaces $\cL^2(X)$, $\cK^2(\mu)$ and $\cH^{2,\perp}(X,\mu)$, respectively. In particular, since $X^i \in \cL^2(X)$, we have that $\langle X^i, N \rangle = 0$. Moreover, by \cite[Theorem III.4.20]{jacod2003limit}, we have $M_{\mu}[\Delta N | \widetilde\cP] = 0$. 
	
	\medskip
	We now turn to the claimed uniqueness. Suppose that there exists $(Z^\prime,U^\prime,N^\prime) \in \H^2(X) \times \H^2(\mu) \times \cH^{2,\perp}(X,\mu)$ satisfying
	\begin{equation*}
		M = \int_{(0,\cdot]} Z^\prime_s \d X_s + \int_{(0,\cdot] \times E} U^\prime_s(x) \tilde\mu(\d s, \d x) + N^\prime, \; \langle N^\prime, X^i \rangle = 0, \; \text{for each $i \in \{1,\ldots,n\}$}, \; \text{and} \; M_{\mu}\big[\Delta N^\prime \big| \widetilde\cP\big] = 0.
	\end{equation*}
	We now show that $N^\prime \in \cH^{2,\perp}$. Since for any $H \in \H^2(X)$,
	\begin{equation*}
		\bigg\langle N^\prime, \int_{(0,\cdot]} H_s \d X_s \bigg\rangle = \int_{(0,\cdot]}\bigg( \sum_{i = 1}^m H^i_s c^{N^\smalltext{\prime},i}_s \bigg) \d C_s,
	\end{equation*}
	for a predictable process $c^{N,i}$ satisfying $0 = \langle N^\prime, X^i \rangle_\cdot = \int_{(0,\cdot]} c^{N^\smalltext{\prime},i}_s \d C_s,$ we see that $c^{N,i}_s = 0$, $\d C_s$--a.e., $\P$--a.s., which then implies  
	\begin{equation*}
		\bigg\langle N^\prime, \int_{(0,\cdot]} H_s \d X_s \bigg\rangle = 0.
	\end{equation*}
	Hence $N^\prime \in \big(\cL^2(X)\big)^\perp$. Next, since $M_{\mu}[\Delta N^\prime | \widetilde\cP] = 0$, we have again by \cite[Theorem 13.3.16]{cohen2015stochastic} or \cite[Lemme 7.39]{jacod1979calcul} that $\langle N^\prime, V \star \tilde\mu \rangle = 0$ for each $V \star \tilde\mu$, and therefore $N^\prime \in (\cK^2(\mu))^\perp$, so $N \in \cH^{2,\perp}$. This implies that
	\begin{equation*}
		M = \int_{(0,\cdot]} Z^\prime_s \d X_s + \int_{(0,\cdot] \times E} U^\prime_s(x) \tilde\mu(\d s, \d x) + N^\prime,
	\end{equation*}
	is a decomposition of $M$ in $\cL^2(X) \oplus \cK^2(\mu) \oplus \cH^{2,\perp}$. We therefore have $(Z,U,N)= (Z^\prime,U^\prime,N^\prime)$ in $\H^2(X) \times \H^2(\mu) \times\cH^{2,\perp}$ since 
	\begin{equation*}
		\int_{(0,\cdot]} Z_s \d X_s = \int_{(0,\cdot]} Z^\prime_s \d X_s, \; \text{and} \; \int_{(0,\cdot] \times E} U_s(x) \tilde\mu(\d s, \d x) = \int_{(0,\cdot] \times E} U^\prime_s(x) \tilde\mu(\d s, \d x),
	\end{equation*}
	implies that $\|Z-Z^\prime\|_{\H^\smalltext{2}(X)} = 0$ and $\|U - U^\prime\|_{\H^\smalltext{2}(\mu)} = 0$, which then also implies $N = N^\prime$. This completes the proof.
\end{proof}

\begin{proof}[Proof of \Cref{lem::integrability_xi}]
	Note first that
	\begin{align*}
		\|\alpha \brs{\xi}\|_{\H^\smalltext{2}_{\smalltext{T}\smalltext{,}\smalltext{\hat\beta}}}^2 
		&= \E\bigg[\int_0^T\cE(\hat\beta A)_s \lim_{s^\prime \uparrow\uparrow s} \Big\{\sup_{t \in [s^\prime,\infty]}\big|\xi^+_s\mathbf{1}_{\{s < T\}}\big|^2\Big\} \d A_s \bigg] \\
		&= \E\bigg[\int_0^T\cE(\hat\beta A)_s \lim_{s^\prime \uparrow\uparrow s} \Big\{\sup_{t \in [s^\prime,\infty]}\big|\cE(\beta^\star)^{-1/2}_s\cE(\beta^\star)^{1/2}_s\xi^+_s\mathbf{1}_{\{s < T\}}\big|^2\Big\} \d A_s \bigg] \\
		&\leq \E\bigg[\int_0^T\cE(\hat\beta A)_s \lim_{s^\prime \uparrow\uparrow s} \cE(\beta^\star)^{-1}_{s^\prime}\Big\{\sup_{t \in [s^\prime,\infty]}\big|\cE(\beta^\star)^{1/2}_s\xi^+_s\mathbf{1}_{\{s < T\}}\big|^2\Big\} \d A_s \bigg] \\
		&= \E\bigg[\int_0^T\cE(\hat\beta A)_s \cE(\beta^\star)^{-1}_{s-} \lim_{s^\prime \uparrow\uparrow s} \Big\{\sup_{t \in [s^\prime,\infty]}\big|\cE(\beta^\star)^{1/2}_s\xi^+_s\mathbf{1}_{\{s < T\}}\big|^2\Big\} \d A_s \bigg] \\
		&= \E\bigg[\int_0^T\cE(\hat\beta A)_{s-}(1+\hat\beta\Delta A_s) \cE(\beta^\star)^{-1}_{s-} \lim_{s^\prime \uparrow\uparrow s} \Big\{\sup_{t \in [s^\prime,\infty]}\big|\cE(\beta^\star)^{1/2}_s\xi^+_s\mathbf{1}_{\{s < T\}}\big|^2\Big\} \d A_s \bigg] \\
		&\leq (1+\hat\beta \Phi)\E\bigg[\sup_{s \in [0,\infty]}\big|\cE(\beta^\star A)^{1/2}_s\xi^+_s\mathbf{1}_{\{s < T\}}\big|^2 \int_0^T\cE(\hat\beta A)_{s-}\cE(\beta^\star A)^{-1}_{s-}\d A_s\bigg] \\
		&= (1+\hat\beta \Phi)\E\bigg[\sup_{s \in [0,\infty]}\big|\cE(\beta^\star A)^{1/2}_s\xi^+_s\mathbf{1}_{\{s < T\}}\big|^2 \int_0^T\cE(A^{\hat\beta,\beta^\star})_{s-}\d A_s\bigg],
	\end{align*}
	where $A^{\hat\beta,\beta^\star} = -(\beta^\star - \hat\beta)A^c - \sum_{s \in (0,\cdot]} \frac{(\beta^\star- \hat\beta)\Delta A_s}{1+\beta^\star \Delta A_s},$ by \Cref{lem::stoch_exp_rules}.$(ii)$ and $A^c$ denotes the continuous part of $A$. Next,
	\begin{align*}
		\int_0^{t \land T}\cE(A^{\hat\beta,\beta^\star})_{s-}\d A_s 
		&= \int_0^{t \land T}\cE(A^{\hat\beta,\beta^\star})_{s-}\d A^c_s + \sum_{s \in (0,t \land T]}\cE(A^{\hat\beta,\beta^\star})_{s-}\Delta A_s \\
		&= -\frac{1}{(\beta^\star-\hat\beta)}\int_0^{t \land T}\cE(A^{\hat\beta,\beta^\star})_{s-}\d (A^{\hat\beta,\beta^\star})^c_s - \sum_{s \in (0,t \land T]}\cE(A^{\hat\beta,\beta^\star})_{s-} \frac{(1+\beta^\star \Delta A_s)}{(\beta^\star-\hat\beta)} \Delta (A^{\hat\beta,\beta^\star})_s \\
		&= \int_0^{t \land T}\cE(A^{\hat\beta,\beta^\star})_{s-}\frac{(1+\beta^\star\Delta A_s)}{(\beta^\star-\hat\beta)}\d (-A^{\hat\beta,\beta^\star})^c_s + \sum_{s \in (0,t \land T]}\cE(A^{\hat\beta,\beta^\star})_{s-} \frac{(1+\beta^\star \Delta A_s)}{(\beta^\star-\hat\beta)} \Delta (-A^{\hat\beta,\beta^\star})_s.
	\end{align*}
	Since $-A^{\hat\beta,\beta^\star}$ is non-decreasing, the stochastic exponential $\cE(A^{\hat\beta,\beta^\star}) = \cE(\hat\beta A)/\cE(\beta^\star A)$ is non-negative, and $\beta^\star - \hat\beta > 0$, we deduce that
	\begin{equation*}
		\int_0^{t \land T}\cE(A^{\hat\beta,\beta^\star})_{s-}\d A_s \leq \frac{(1+\beta^\star \Phi)}{(\beta^\star-\hat\beta)} \int_0^{t \land T} \cE(A^{\hat\beta,\beta^\star})_{s-}\d (-A^{\hat\beta,\beta^\star})_s = \frac{(1+\beta^\star \Phi)}{(\beta^\star-\hat\beta)} \bigg(1- \cE(A^{\hat\beta,\beta^\star})_{t \land T}\bigg) \leq \frac{(1+\beta^\star \Phi)}{(\beta^\star-\hat\beta)},
	\end{equation*}
	where in the last inequality we used that $0 < \cE(A^{\hat\beta,\beta^\star}) \leq 1$ since $1 \leq \cE(\hat\beta A) \leq \cE(\beta^\star A)$. We thus conclude that
	\begin{equation*}
		\|\alpha \brs{\xi}\|_{\H^\smalltext{2}_{\smalltext{T}\smalltext{,}\smalltext{\hat\beta}}}^2 
		\leq \frac{(1+\beta^\star \Phi)(1+\hat\beta \Phi)}{(\beta^\star-\hat\beta)}\E\bigg[\sup_{s \in [0,\infty]}\big|\cE(\beta^\star A)^{1/2}_s\xi^+_s\mathbf{1}_{\{s < T\}}\big|^2\bigg] 
		= \frac{(1+\beta^\star \Phi)(1+\hat\beta \Phi)}{(\beta^\star-\hat\beta)}\|\xi^+_\cdot \mathbf{1}_{\{\cdot < T\}}\|^2_{\cS^\smalltext{2}_{\smalltext{T}\smalltext{,}\smalltext{\beta^\star}}}.
	\end{equation*}
	This completes the proof.
\end{proof}

\section{Proofs of technical lemmata}\label{app::proofs_sec_main_results}\label{sec::optimal_stopping_proofs_technical}

\begin{lemma}\label{lem::analysis_ff_fg}
	Let $\Psi \in [0,\infty)$, and let $(\ff^\Psi,\fg^\Psi) : (0,\infty) \rightarrow \R \times \R$ be defined by
	\begin{equation*}
		\ff^\Psi(\beta) \coloneqq \inf_{\gamma\in (0,\beta)} \bigg\{\frac{(1+\beta\Psi)}{\gamma(\beta-\gamma)}\bigg\} \; \text{and} \; \fg^\Psi(\beta) \coloneqq \inf_{\gamma \in (0,\beta)} \bigg\{\frac{(1+\gamma\Psi)}{\gamma(\beta-\gamma)}\bigg\}.
	\end{equation*}
	Then, for $\beta \in (0,\infty)$,
	\begin{equation*}
		\ff^\Psi(\beta) = \frac{4(1+\beta\Psi)}{\beta^2} \; \text{and} \; \fg^\Psi(\beta) =  \frac{4}{\beta^2}\1_{\{\Psi = 0\}} + \frac{\Psi^2  \sqrt{1+\beta\Psi}}{\big(1+\beta\Psi -\sqrt{1+\beta\Psi}\big)\big(\sqrt{1+\beta\Psi}-1\big)}\1_{\{\Psi > 0\}}.
	\end{equation*}
\end{lemma}

\begin{proof}
	Fix $\beta \in (0,\infty)$, and let $f : (0,\beta) \rightarrow \R$ be given by $f(\gamma) \coloneqq \frac{(1+\beta\Psi)}{\gamma(\beta-\gamma)},$ so that $\ff^\Psi(\beta) = \inf f$. Note that $f$ is strictly convex since 
	\begin{equation*}
		\frac{\partial^2 f}{\partial \gamma^2} = \frac{2(\beta^2 - 3\gamma\beta + 3 \gamma^2)(1+\beta\Psi)}{\gamma^3(\beta-\gamma)^3},
	\end{equation*}
	because
	\begin{equation}\label{eq::numerator_2partial_ff}
		\beta^2-3\gamma\beta + 3\gamma^2 = \beta^2 - 2 \frac{3}{2}\gamma \beta + \frac{9}{4}\gamma^2 - \frac{9}{4}\gamma^2 + 3\gamma^2 = \bigg(\beta - \frac{3}{2}\gamma\bigg)^2 - \frac{9}{4}\gamma^2 + \frac{12}{4}\gamma^2 = \bigg(\beta - \frac{3}{2}\gamma\bigg)^2 + \frac{3}{4}\gamma^2 > 0,
	\end{equation}
	and as $f$ tends to infinity at the boundaries of the interval $(0,\beta)$, there is a unique critical point $\gamma^\star$ in $(0,\beta)$ at which $f$ attains its minimum. The point $\gamma^\star$ satisfies
	\begin{equation*}
		\frac{\partial f}{\partial \gamma} (\gamma^\star) = \frac{(2\gamma^\star- \beta)(1+\beta\Psi)}{(\gamma^\star)^2(\beta-\gamma^\star)^2} = 0,
	\end{equation*}
	and therefore $\gamma^\star = \beta/2$. This implies $\ff^\Psi(\beta) = \inf f = f(\gamma^\star) = \frac{4(1+\beta\Psi)}{\beta^2}$. 
	
	\medskip
	We turn to $\fg^\Psi$. Let $g : (0,\beta) \longrightarrow \R$ be given by $g(\gamma) \coloneqq \frac{(1+\gamma\Psi)}{\gamma(\beta-\gamma)}.$ Then $\fg^\Psi(\beta) = \inf g$. We also note here that $g$ tends to infinity at the boundary of $(0,\beta)$. Similar to before, $g$ is strictly convex since
	\begin{equation*}
		\frac{\partial^2 g}{\partial \gamma^2} = \frac{2(\beta^2-3\gamma\beta + 3\gamma^2 + \gamma^3\Psi)}{\gamma^3(\beta-\gamma)^3} > 0,
	\end{equation*}
	where the strict inequality follows from \eqref{eq::numerator_2partial_ff}. As before, we now only need to find the unique critical point $\gamma^\star$ of $g$ in $(0,\beta)$. Suppose first that $\Psi \in (0,\infty)$. From
	\begin{equation*}
		\frac{\partial g}{\partial \gamma}(\gamma^\star) = \frac{2\gamma - \beta + \gamma^2\Psi}{\gamma^2(\beta-\gamma)^2} = 0,
	\end{equation*}
	we find that $\gamma^\star_{1,2} = \frac{-1 \pm \sqrt{1+\beta\Psi}}{\Psi},$ and thus the critical point we are looking for is $\gamma^\star = (-1 +\sqrt{1+\beta\Psi})/\Psi$. This implies
	\begin{equation*}
		\fg^\Psi(\beta) = \inf g = g(\gamma^\star) = \frac{\Psi^2  \sqrt{1+\beta\Psi}}{\big(1+\beta\Psi -\sqrt{1+\beta\Psi}\big)\big(\sqrt{1+\beta\Psi}-1\big)}.
	\end{equation*}
	In case $\Psi = 0$, we have $\fg^0(\beta) = \ff^0(\beta) = 4/\beta^2$. This completes the proof.
\end{proof}

\begin{lemma}\label{lem::contraction_constants_reflected_bsde}
	Let $\Psi \in [0,\infty)$. Then
	\begin{equation*}
		\lim_{\beta \uparrow\uparrow\infty} \ff^\Psi(\beta) = 0 \; \text{\rm and} \; \lim_{\beta \uparrow\uparrow\infty} \beta \fg^\Psi(\beta) = \Psi.
	\end{equation*}
	Furthermore, for each $i \in \{1,2,3\}$, the constant $M^\Psi_i(\beta)$ is decreasing in $\beta$ and
	\begin{equation*}
		\lim_{\beta\uparrow\uparrow\infty} M^\Psi_1(\beta) =\max\{1,\Psi\}\Psi, \; \lim_{\beta\uparrow\uparrow\infty} M^\Psi_2(\beta) = \Psi \; \text{and} \; \lim_{\beta\uparrow\uparrow\infty} M^\Psi_3(\beta) = \max\{1,\Psi\}\Psi.
	\end{equation*}
\end{lemma}

\begin{proof}
	It is clear that $(1+\beta\Psi)/\beta$, $\ff^\Psi(\beta)$ and $\fg^\Psi(\beta)$ are decreasing in $\beta$. Moreover, $(1+\beta\Psi)/\beta$ converges to $\Psi$, $\ff^\Psi(\beta)$ converges to zero and $\beta\fg^\Psi(\beta)$ converges to $\Psi$ by \Cref{lem::analysis_ff_fg} as $\beta$ tends to infinity. The stated limits of $M^\Psi_1$, $M^\Psi_2$ and $M^\Psi_3$ thus follow immediately.
\end{proof}

The following result can be deduced similarly, we thus omit its proof.

\begin{lemma}\label{lem::contraction_constants_bsde}
	For each $i \in \{1,2,3\}$, the constant $\widetilde M^\Psi(\beta)$ is decreasing in $\beta$, and
	\begin{equation*}
		\lim_{\beta\rightarrow\infty} \widetilde M^\Psi_1(\beta) =
			\max\{1,\Psi\}\Psi,
		\; 
		\lim_{\beta\rightarrow\infty} \widetilde M^\Psi_2(\beta) = \Psi, \; \text{\rm and} \; 
		 \lim_{\beta\rightarrow\infty} \widetilde M^\Psi_3(\beta) = \max\{1,\Psi\}\Psi.
	\end{equation*}
\end{lemma}

We now prove the technical lemmata of \Cref{sec::optimal_stopping}.

\begin{proof}[Proof of \Cref{lem::snell_reformulation}]
	We clearly have $L = L_{\cdot \land T}$, $\P$--a.s., and since $-|M| - 1 \leq L \leq \xi_{\cdot \land T}^+ + \int_0^T |f_u| \d C_u + |M| + 1,$ up to $\P$--indistinguishability, we find using \eqref{eq::integrability_optimal_stopping} and Doob's $\L^2$-inequality that
	\begin{equation*}
		\E\bigg[\sup_{s \in [0,T]} |L_s|^2\bigg] = \E\bigg[\sup_{s \in [0,\infty]} |L_s|^2\bigg] < \infty.
	\end{equation*}
	Next, it is clear that $L \geq J$ and thus $V(S) \leq {\esssup_{\tau \in \cT_{\smalltext{S}\smalltext{,}\smalltext{\infty}}}}^{\cG_\smalltext{S}} \E[L_\tau |\cG_S], \; \text{$\P$--a.s.}, \; S \in \cT_{0,\infty}.$ We turn to the converse inequality. Fix $S \in \cT_{0,\infty}$, let $\tau \in \cT_{S,\infty}$ and let $\hat\tau \coloneqq \tau \mathbf{1}_{\{J_\smalltext{\tau} \geq L_\smalltext{\tau}\}} + \infty \mathbf{1}_{\{J_\smalltext{\tau} < L_\smalltext{\tau}\}} \in \cT_{S,\infty}.$ Then
	\begin{align*}
		\E[L_{\tau}| \cG_S] = \E[L_{\tau} \mathbf{1}_{\{J_\smalltext{\tau} \geq L_\smalltext{\tau}\}} + \big(M_\tau - \mathbf{1}_{\{\tau < T\}}\big)\mathbf{1}_{\{J_\smalltext{\tau} < L_\smalltext{\tau}\}}| \cG_S]  
		&\leq \E[L_\tau \mathbf{1}_{\{J_\smalltext{\tau} \geq L_\smalltext{\tau}\}} + M_\tau\mathbf{1}_{\{J_\smalltext{\tau} < L_\smalltext{\tau}\}}| \cG_S] \\
		&\leq \E[J_\tau \mathbf{1}_{\{J_\smalltext{\tau} \geq L_\smalltext{\tau}\}} + \E[J_\infty | \cG_\tau]\mathbf{1}_{\{J_\smalltext{\tau} < L_\smalltext{\tau}\}}| \cG_S] \\
		&= \E[J_\tau \mathbf{1}_{\{J_\smalltext{\tau} \geq L_\smalltext{\tau}\}} + J_\infty\mathbf{1}_{\{J_\smalltext{\tau} < L_\smalltext{\tau}\}}| \cG_S] = \E[J_{\hat\tau}| \cG_S] \leq V(S), \; \text{$\P$--a.s.},
	\end{align*}
	and therefore ${\esssup^{\cG_\smalltext{S}}_{\tau \in \cT_{\smalltext{S}\smalltext{,}\smalltext{\infty}}}} \E\big[ L_\tau \big| \cG_S\big] \leq V(S), \; \text{$\P$--a.s.}, S \in \cT_{0,\infty},$ which completes the proof.
\end{proof}

\begin{proof}[Proof of \Cref{lem::supermartingale_family}]
	We only show that $V(S) \in \L^2(\cG_S)$, for each $S \in \cT_{0,\infty}$, as the rest follows from Proposition 1.3 and Proposition 1.5 in \cite{kobylanski2012optimal}\footnote{The fact that the filtration is complete and the horizon is finite in \cite{kobylanski2012optimal} is not relevant for the proof of their Proposition 1.3 and Proposition 1.5.} together with the argument in \cite[Footnote 4]{grigorova2020optimal}. Fix $S \in \cT_{0,\infty}$, and note that
	\begin{equation*}
		-\E\big[|\xi_T|\big| \cG_S\big] - \E\bigg[\int_0^T|f_u|\d C_u \bigg| \cG_S\bigg] \leq V(S) \leq \E\bigg[ |\xi_T| + \sup_{u \in [0,T)} \xi^+_{u} + \int_0^T |f_u| \d C_u \bigg|\cG_S\bigg], \; \text{$\P$--a.s.},
	\end{equation*}
	implies
	\begin{equation*}
		|V(S)| \leq \E\bigg[ |\xi_T| + \sup_{u \in [0,T)} \xi^+_u + \int_0^T |f_u| \d C_u \bigg|\cG_S\bigg], \; \text{$\P$--a.s.},
	\end{equation*}
	from which we deduce that $V(S) \in \L^2(\cG_S)$ by \eqref{eq::integrability_optimal_stopping}. This completes the proof.
\end{proof}

\begin{proof}[Proof of \Cref{lem::snell_aggregation}]
	The existence of the process $V = (V_t)_{t \in [0,\infty]}$ is a mere application \Cref{lem::supermartingale_family} and \cite[Appendix I, Remark 23(b)]{dellacherie1982probabilities}. Next, let $S \in \cT_{0,\infty}$. Since $V_{T} = V(T) = V(S \lor T) = V_{S \lor T}$, $\P$--a.s., we find
	\begin{equation*}
		V_{S \land T} = V_S \mathbf{1}_{\{S \leq T\}} + V_T \mathbf{1}_{\{S > T\}} = V_S\mathbf{1}_{\{S \leq T\}} + V_{S \lor T}\mathbf{1}_{\{S > T\}} = V_S, \; \text{$\P$--a.s.}
	\end{equation*}
	In particular, $V_\cdot = V_{\cdot \land T}$ up to $\P$--indistinguishability by \Cref{prop::optional_ineq}.
	
	\medskip
	The fact that $V$ is in $\cS^2_T$ can be argued as follows. Let $N = (N_t)_{t \in [0,\infty]}$ be the martingale satisfying
	\begin{equation*}
		N_S = \E\Bigg[|\xi_T| + \sup_{u \in [0,T)}|\xi^+_u| + \int_0^T|f_u|\d C_u \bigg| \cG_S\bigg], \; \text{$\P$--a.s.}, \; S \in \cT_{0,\infty}.
	\end{equation*} 
	Note that $N$ is square-integrable by \eqref{eq::integrability_optimal_stopping}. As in the proof of \Cref{lem::supermartingale_family}, we find that
	\begin{equation*}
		|V_S| \leq \E\Bigg[|\xi_T| + \sup_{u \in [0,T)}|\xi^+_u| + \int_0^T|f_u|\d C_u \bigg| \cG_S\bigg] = N_S,\; \text{$\P$--a.s.}, \; S \in \cT_{0,\infty},
	\end{equation*}
	Thus $|V| \leq N$ up to $\P$--indistinguishability by \Cref{prop::optional_ineq}, and with Doob's $\L^2$-inequality for martingales, we find
	\begin{equation*}
		\E\bigg[\sup_{s \in [0,\infty]} |V_{s \land T}|^2\bigg] \leq \E\bigg[\sup_{s \in [0,\infty]} |N_s|^2\bigg] \leq 4 \E\big[|N_\infty|^2\big] < \infty,
	\end{equation*}
	which completes the proof.
\end{proof}

\begin{proof}[Proof of \Cref{lem::snell_indicator}]
	We start with $(i)$. Let $S \in \cT_{0,\infty}$. Note that on $\{S < T\}$, we have $\mathbf{1}_{\{V_\smalltext{S} = L_\smalltext{S}\}} = \mathbf{1}_{\{V_\smalltext{S} = J_\smalltext{S}\}}$, $\P$--a.s., since $V_S \geq M_S > M_S - 1$, $\P$--almost surely. On $\{S \geq T\}$, we have that $V_S = V_T = J_T$, $\P$--a.s., and therefore $\mathbf{1}_{\{V_\smalltext{S} = L_\smalltext{S}\}}= \mathbf{1}_{\{V_\smalltext{S} = J_\smalltext{S}\}}$, $\P$--almost surely. 
	
	\medskip
	We turn to $(ii)$. Let $S \in \cT^p_{0,\infty}$. On $\{S \leq T\}$, we have $\mathbf{1}_{\{V_{\smalltext{S}\smalltext{-}} = \bar L_\smalltext{S}\}} = \mathbf{1}_{\{V_{\smalltext{S}\smalltext{-}}=\bar J_\smalltext{S}\}}$ since $V_{S-} \geq M_{S-} > M_{S-} - 1$, $\P$--almost surely. On $\{S > T\}$, we have $V_{S-} = V_T = J_T = J_{S-}$, $\P$--a.s., and therefore $\mathbf{1}_{\{V_{\smalltext{S}\smalltext{-}} = \bar L_\smalltext{S}\}} = \mathbf{1}_{\{V_{\smalltext{S}\smalltext{-}}=\bar J_\smalltext{S}\}}$, $\P$--a.s., which completes the proof.
\end{proof}

\section{Miscellaneous}

\begin{lemma}\label{lem::stoch_exp_rules}
	Let $A$ and $B$ be two optional, real-valued processes with $\P$--{\rm a.s.} right-continuous and non-decreasing paths such that $A_0 = B_0 = 0$, $\P$--almost surely. Then the following holds:
	\begin{enumerate}
		\item[$(i)$] $\cE(A)^{-1} = \cE(-\overline A)$, where $\overline A = A - \sum_{s \in (0,\cdot]} \frac{(\Delta A_s)^2}{1+\Delta A_s};$
		\item[$(ii)$] $\cE(A)^{-1}\cE(B) = \cE(C)$, where $C = B^c - A^c + \sum_{s \in (0,\cdot]} \frac{\Delta B_s - \Delta A_s}{1+\Delta A_s};$
		\item[$(iii)$] $\cE(A)^{1/2} = \cE(D)$, where $D = \frac{1}{2} A^c + \sum_{s \in (0,\cdot]} \big(\sqrt{1+\Delta A_s}-1\big).$
	\end{enumerate}
\end{lemma}

\begin{proof}
Assertion $(i)$ follows from \cite[Lemma 4.4]{cohen2012existence}. We prove $(ii)$. The product formula for stochastic exponentials yields $\cE(A)^{-1}\cE(B) = \cE(C),$ where $C = \overline A + B + [\overline A,B]$. By differentiating the continuous part $C^c$ of $C$ from the purely discontinuous part $C^d$, we can explicitly write
\begin{align*}
			C &= -\overline A + B - [\overline A,B] \\
			&= -A^c + B^c + \sum_{s \in (0,\cdot]} \bigg(-\Delta A_s + \frac{(\Delta A_s)^2}{1+\Delta A_s} + \Delta B_s - \bigg(\Delta A_s - \frac{(\Delta A_s)^2}{1+\Delta A_s}\bigg)\Delta B_s\bigg) \\
			&= B^c - A^c + \sum_{s \in (0,\cdot]} \bigg(-\frac{\Delta A_s}{1+\Delta A_s} + \frac{\Delta B_s + \Delta A_s \Delta B_s}{1+\Delta A_s} - \frac{\Delta A_s \Delta B_s}{1+\Delta A_s}\bigg) = B^c - A^c + \sum_{s \in (0,\cdot]} \frac{\Delta B_s - \Delta A_s}{1+\Delta A_s}.
\end{align*}
Assertion $(iii)$ follows by squaring $\cE(D)$, using the product formula for stochastic exponentials, and then checking that this coincides with $\cE(A)$. This completes the proof.
\end{proof}

\begin{proposition}\label{lem::leb_stj}
	Let $A$ be a non-decreasing, $[0,\infty]$-valued function with the conventions $A_{0-} \coloneqq 0$, $A_{\infty-} \coloneqq \lim_{t \uparrow\uparrow \infty} A_t$ and $A_\infty \coloneqq \infty$.  Let $R = (R_t)_{t\in[0,\infty]}$ and $L = (L_t)_{t\in[0,\infty]}$ be defined by
		\begin{equation*}
			R_t \coloneqq 
			\begin{cases}
				\inf\{ s \in [0,\infty) : A_s > t\}, \; t \in [0,\infty), \\
				\infty, \; t = \infty, 
			\end{cases}
		\end{equation*}
		with the conventions $\inf \varnothing = \infty$ and $R_{0-} = 0$, and $L_t \coloneqq R_{t-}$, $t\in[0,\infty]$. For any $t\in [0,\infty]$
		\begin{enumerate}
			\item[$(i)$] $t \leq A_{L_\smalltext{t}} = A_{R_{\smalltext{t}\smalltext{-}}} \leq A_{R_\smalltext{t}};$
			\item[$(ii)$] $A_{L_\smalltext{t}-} = A_{R_{\smalltext{t}\smalltext{-}}-} \leq A_{R_{\smalltext{t}}-} \leq t;$
			\item[$(iii)$] for any $s\in[0,\infty]$, $t \leq A_s$ if and only if $L_t \leq s;$
			\item[$(iv)$] for any $s\in[0,\infty]$, if $t < A_s$, then $R_t \leq s$ $($the other direction is false in general, and equality is possible$);$
			\item [$(v)$] for any function $f$ on $[0,\infty)$ which is non-negative and Borel-measurable, if $A$ is finite on $[0,\infty)$, then
			\begin{equation*}
				\int_{[0,\infty)} f(s) \d A_s = \int_{[0,\infty)} f(L_s) \mathbf{1}_{\{L_\cdot < \infty\}}(s) \d s = \int_{[0,\infty)} f(R_s) \mathbf{1}_{\{R_\cdot < \infty\}}(s) \d s;
			\end{equation*}
			\item[$(vi)$] for any non-decreasing, Borel-measurable, $\ell$--sub-multiplicative\footnote{This means that $\ell \in (0,\infty)$ and
				\begin{equation*}
					g(x+y) \leq \ell g(x)g(y),\; (x,y) \in [0,\infty)^2.
				\end{equation*}} $g : [0,\infty) \longrightarrow [0,\infty)$, and if $A$ is finite on $[0,\infty)$
				\begin{equation*}
					\int_{(0,t]} g(A_s) \d A_s \leq \ell g\bigg( \max_{\{s : L_\smalltext{s} < \infty\}} \Delta A_{L_\smalltext{s}} \bigg) \int_{(A_\smalltext{0}, A_\smalltext{t}]} g(s) \d s.
				\end{equation*}
		\end{enumerate}
\end{proposition}

\begin{proof}
	For $(i)$ through $(v)$ see \cite[page 119--120]{dellacherie1982probabilities} or \cite[page 21--22]{he1992semimartingale}, for $(vi)$ see the proof of Lemma 2.14 in \cite{papapantoleon2018existence}.
\end{proof}

Since we have not been able to locate a reference for the following result, we prove it for the convenience of the reader.

\begin{proposition}\label{prop::optional_ineq}
	Suppose that $X = (Y_t)_{t \in [0,\infty]}$ and $Y = (Y_t)_{t \in [0,\infty]}$ are two optional processes for which $X_S \leq Y_S$, $\P${\rm--a.s.}, holds for each $S \in \cT_{0,\infty}$. Then $X \leq Y$ holds up to $\P$-evanescence.
\end{proposition}

\begin{proof}
	Let $\varepsilon > 0$, and let $A \coloneqq \{(\omega,t) \in \Omega \times [0,\infty) : X_s(\omega) > Y_s(\omega) \}$, which is an optional subset of $\Omega \times [0,\infty)$. By \cite[Theorem IV.84]{ dellacherie1978probabilities}, there is a $\G$--stopping time $S$ such that for all $\omega \in \Omega$ with $S(\omega) < \infty$, we have $(\omega,S(\omega)) \in A $ and $\P[S < \infty] \geq \P[\pi_\Omega(A)] - \varepsilon.$ Here, $\pi_\Omega(A)$ is the projection of $A$ onto $\Omega$, which is $\cG^U_\infty$-measurable. Since $\P[S < \infty] \leq \P[S < \infty, X_S > Y_S] = 0$ and $\varepsilon > 0$ was arbitrary, it follows that $\P[\pi_\Omega(A)] = 0$. This, with $\P[X_\infty > Y_\infty] = 0$, implies the claim.
\end{proof}

\begin{lemma}\label{lem::cond_doob}
	Suppose that $M = (M_t)_{t \in [0,\infty]}$ is a non-negative, $\P${\rm--a.s.} right-continuous, $(\F,\P)$-martingale with $M_\infty \in \L^2(\cG_\infty)$. Then
	\begin{equation*}
		\E \bigg[ \sup_{s \in [t,\infty]} M^2_s \bigg| \cG_t \bigg] \leq 4 \E[ M^2_\infty | \cG_t ], \; \textnormal{$\P$--a.s.}, \; t \in [0,\infty].
	\end{equation*}
\end{lemma}

\begin{proof}
	First, the condition that $M_\infty \in \L^2(\cG_\infty)$ implies already that $M_t \in \L^2(\cG_t)$ for each $t \in [0,\infty)$. Moreover, the result is trivial for $t = \infty$. So we fix $t \in [0,\infty)$, and consider a subdivision $t = t_0 < t_1 < \cdots  < t_N = \infty$. By Proposition 2.1 in \cite{acciaio2013trajectorial}, we have that
	\begin{equation*}
		\max_{k \in \{0,\ldots,N\}} M^2_{t_\smalltext{k}} \leq -4 \sum_{\ell = 0}^{N-1}  \max_{k \in \{0,\ldots,\ell\}} M_{t_\smalltext{k}} \big(M_{t_\smalltext{\ell+1}} - M_{t_\smalltext{\ell}}\big) - 2 M^2_{t_\smalltext{0}} + 4 M^2_{t_\smalltext{N}}.
	\end{equation*}
	For each $\ell \in \{0,1,\ldots,N-1\}$,
	\begin{equation*}
		\bigg|\max_{k \in \{0,\ldots,\ell\}} M_{t_\smalltext{k}} \big(M_{t_\smalltext{\ell+1}} - M_{t_\smalltext{\ell}}\big)\bigg| \leq \frac{1}{2}\max_{k \in \{0,\ldots,\ell\}} M^2_{t_\smalltext{k}} + \frac{1}{2}\big(M_{t_\smalltext{\ell+1}} - M_{t_\smalltext{\ell}}\big)^2,
	\end{equation*}
	and the right-hand side in the previous inequality is integrable. Therefore, by the martingale property of $M$, we deduce that
	\begin{align*}
		\E\bigg[ \max_{k \in \{0,\ldots,N\}} M^2_{t_\smalltext{k}} \bigg| \cG_{t_0} \bigg] \leq \E\bigg[ -4 \sum_{\ell = 0}^{N-1}  \max_{k \in \{0,\ldots,\ell\}} M_{t_\smalltext{k}} \big(M_{t_\smalltext{\ell+1}} - M_{t_\smalltext{\ell}}\big) \bigg| \cG_{t_\smalltext{0}} \bigg] - \E\big[ 2 M^2_{t_\smalltext{0}} \big| \cG_{t_\smalltext{0}}\big] + 4 \E\big[ M^2_{t_\smalltext{N}} \big| \cG_{t_\smalltext{0}} \big] \leq 4 \E\big[ M^2_{t_\smalltext{N}} \big| \cG_{t_\smalltext{0}} \big].
	\end{align*}
	The claim then follows by approximating the supremum of $M$ on $[t,\infty]$ and then using Fatou's lemma for conditional expectations, which completes the proof.
\end{proof}

\begin{lemma}\label{lem::increas_cond_mean}
	Let $L = (L_t)_{t \in [0,\infty]}$ and $J = (J_t)_{t\in[0,\infty]}$ be two product-measurable processes whose $\P$--almost all paths admit limits from the left on $(0,\infty]$. Suppose that $\E[\sup_{t \in [0,\infty]}|L_t|] + \E[\sup_{t \in [0,\infty]}|J_t|] < \infty$ and that $\E[L_t|\cG_t] \leq \E[J_t|\cG_t]$, $\P$--a.s., $t \in [0,\infty]$. Then $\E[L_{t-}|\cG_{t-}] \leq \E[J_{t-}|\cG_{t-}]$, $\P$--a.s., $t \in [0,\infty]$, where $L_{0-} \coloneqq J_{0-} \coloneqq 0$ and $\cG_{0-} \coloneqq \cG_0$.
\end{lemma}

\begin{proof}
	The assertion trivially holds for $t = 0$. So suppose that $t \in (0,\infty]$. Let $(t_n)_{n \in \N} \subseteq [0,t)$ be a sequence that converges increasingly to $t$. We have
	\begin{equation*}
		\lim_{n \rightarrow \infty}\E\big[|\E[L_{t_\smalltext{n}}-L_{t-}|\cG_{t_\smalltext{n}}]|\big] \leq \lim_{n \rightarrow \infty} \E\big[|L_{t_\smalltext{n}}-L_{t-}|\big] = 0.
	\end{equation*}
	Here the last equality follows by dominated convergence. Therefore, upon choosing a suitable subsequence of $(t_n)_{n \in \N}$ if necessary, we have
	\begin{equation*}
		\lim_{n \rightarrow \infty}\E[L_{t_\smalltext{n}}|\cG_{t_\smalltext{n}}] = \E[L_{t-}|\cG_{t-}], \; \text{$\P$--a.s.}
	\end{equation*}
	The same argument applied to $\E[J_{t_\smalltext{n}}|\cG_{t_\smalltext{n}}]$, $n \in \N$, and upon choosing a further subsequence if necessary, then yields
	\begin{equation*}
		\E[L_{t-}|\cG_{t-}] = \lim_{n \rightarrow \infty}\E[L_{t_\smalltext{n}}|\cG_{t_\smalltext{n}}] \leq \lim_{n \rightarrow \infty}\E[J_{t_\smalltext{n}}|\cG_{t_\smalltext{n}}]  = \E[J_{t-}|\cG_{t-}], \; \text{$\P$--a.s.},
	\end{equation*}
	which completes the proof.
\end{proof}

\begin{proposition}\label{prop::pred_left_limit}
	Let $\xi = (\xi_t)_{t \in [0,\infty]}$ be $(\cG_t)_{t \in [0,\infty]}$-predictable. Then the process $\overline\xi = (\overline\xi_t)_{t \in [0,\infty]}$ defined by
\begin{equation*}
	\overline\xi_0 \coloneqq \xi_0, \; \text{\rm and} \; \overline\xi_t \coloneqq \limsup_{s \uparrow\uparrow t} \xi_s \; \text{\rm for $t \in (0,\infty]$},
\end{equation*}
is $(\cG^U_t)_{t \in [0,\infty]}$-predictable.
\end{proposition}

\begin{proof}
	For each $n \in \mathbb{N}$, define $\xi^n = (\xi^n_t)_{t \in [0,\infty]}$ by $\xi^n_0 \coloneqq \xi_0$ and for $t \in (0,\infty]$ by
	\begin{equation*}
		\xi^n_t \coloneqq \bigg(\sup_{s \in [n,t)} \xi_s \bigg)\mathbf{1}_{(n,\infty]}(t) + \sum_{k = 0}^{n2^\smalltext{n}-1} \bigg(\sup_{s \in [k2^{\smalltext{-}\smalltext{n}},t)}\xi_s\bigg) \mathbf{1}_{(k2^{\smalltext{-}\smalltext{n}}, (k+1)2^{\smalltext{-}\smalltext{n}}]}(t).
	\end{equation*}
	Each $\xi^n$ is $(\cG^U_t)_{t \in [0.\infty]}$-adapted by \cite[Proposition 2.21]{karoui2013capacities} and left-continuous on $(0,\infty]$. The claim now follows from the fact that $\overline\xi$ is the limit of the sequence $(\xi^n)_{n \in \N}$, which completes the proof.
\end{proof}

\begin{proposition}\label{prop::equiv_integ_xi}
	Let $\xi = (\xi_t)_{t \in [0,\infty]}$ be a real-valued, optional process, $T$ be a stopping time. Suppose that $\xi_\cdot = \xi_{\cdot \land T}$. Then $\sup_{s \in [0,T]}|\xi_s|$ is $\cG^U_\infty$-measurable and for each $p \in (1,\infty)$
	\begin{equation}\label{eq::ineq_esssup_sup}
		\E\bigg[ \underset{\tau \in \cT_{\smalltext{0}\smalltext{,}\smalltext{T}}}{\rm ess\,sup}^{\cG_\smalltext{\infty}}|\xi_\tau|^p \bigg] \leq \E\bigg[\sup_{s \in [0,T]} |\xi_s|^p\bigg] \leq \bigg(\frac{p}{p-1}\bigg)^p \E\bigg[ \underset{\tau \in \cT_{\smalltext{0}\smalltext{,}\smalltext{T}}}{\rm ess\,sup}^{\cG_\smalltext{\infty}}|\xi_\tau|^p \bigg]
	\end{equation}
\end{proposition}

\begin{proof} 
	Fix $p \in (1,\infty)$. Since $\xi_{\cdot \land T}$ is optional, and thus $\cB([0,\infty]) \otimes \cG_\infty$-measurable, we have that $\sup_{s \in [0,T]} |\xi_s| = \sup_{s \in [0,\infty]} |\xi_{s \land T}|$ is $\cG^U_\infty$-measurable by an application of \cite[Proposition 2.21]{karoui2013capacities}. The first claimed inequality follows from
	\begin{equation*}
		\underset{\tau \in \cT_{\smalltext{0}\smalltext{,}\smalltext{T}}}{\rm ess\,sup}^{\cG_\smalltext{\infty}} |\xi_\tau|^p \leq \sup_{s \in [0,T]} |\xi_s|^p, \; \text{$\P$--a.s.}
	\end{equation*}
	Next, we suppose without loss of generality, that the right hand side in \eqref{eq::ineq_esssup_sup} is finite, otherwise the second inequality trivially holds. Let $M = (M_t)_{t \in [0,\infty]}$ be the non-negative martingale satisfying
	\begin{equation*}
		M_S = \E\bigg[ \underset{\tau \in \cT_{\smalltext{0}\smalltext{,}\smalltext{T}}}{\rm ess\,sup}^{\cG_\smalltext{\infty}} |\xi_\tau| \bigg| \cG_S \bigg], \; \text{$\P$--a.s.}, \; S \in \cT_{0,\infty}.
	\end{equation*}
	Note that $|\xi_S| \leq M_S$, $\P$--a.s., for each $S \in \cT_{0,\infty}$ and thus in particular $\sup_{s \in [0,T]} |\xi_s| = \sup_{s \in [0,\infty]} |\xi_{s \land T}| \leq \sup_{s \in [0,\infty]} |M_s|$
	by \Cref{prop::optional_ineq}. By using Doob's inequality for martingales, we thus find
	\begin{equation*}
		\E\bigg[\sup_{s \in [0,T]} |\xi_s|^p\bigg] \leq \E\bigg[\sup_{s \in [0,\infty]} M_s^p\bigg] \leq \bigg(\frac{p}{p-1}\bigg)^p\E[ M_\infty^p ] = \bigg(\frac{p}{p-1}\bigg)^p \E\bigg[\underset{\tau \in \cT_{\smalltext{0}\smalltext{,}\smalltext{T}}}{\rm ess\,sup}^{\cG_\smalltext{\infty}} |\xi_\tau|^p \bigg],
	\end{equation*}
	which concludes the proof.
\end{proof}

{\small
\bibliography{bibliographyDylan}}

\begin{thebibliography}{143}
\providecommand{\natexlab}[1]{#1}
\providecommand{\url}[1]{\texttt{#1}}
\expandafter\ifx\csname urlstyle\endcsname\relax
  \providecommand{\doi}[1]{doi: #1}\else
  \providecommand{\doi}{doi: \begingroup \urlstyle{rm}\Url}\fi

\bibitem[Acciaio et~al.(2013)Acciaio, Beiglb{\"o}ck, Penkner, Schachermayer,
  and Temme]{acciaio2013trajectorial}
B.~Acciaio, M.~Beiglb{\"o}ck, F.~Penkner, W.~Schachermayer, and J.~Temme.
\newblock A trajectorial interpretation of {D}oob's martingale inequalities.
\newblock \emph{The Annals of Applied Probability}, 23\penalty0 (4):\penalty0
  1494--1505, 2013.

\bibitem[Akdim et~al.(2020)Akdim, Haddadi, and Ouknine]{akdim2020strong}
K.~Akdim, M.~Haddadi, and Y.~Ouknine.
\newblock Strong {S}nell envelopes and {RBSDE}s with regulated trajectories
  when the barrier is a semimartingale.
\newblock \emph{Stochastics: An International Journal of Probability and
  Stochastic Processes}, 92\penalty0 (3):\penalty0 335--355, 2020.

\bibitem[Aksamit et~al.(2023)Aksamit, Li, and
  Rutkowski]{aksamit2023generalized}
A.~Aksamit, L.~Li, and M.~Rutkowski.
\newblock Generalized {BSDE} and reflected {BSDE} with random time horizon.
\newblock \emph{Electronic Journal of Probability}, 28\penalty0 (40):\penalty0
  1--41, 2023.

\bibitem[Alsheyab and Choulli(2023)]{alsheyab2023optimal}
S.~Alsheyab and T.~Choulli.
\newblock Optimal stopping problem under random horizon: mathematical
  structures and linear {RBSDE}s.
\newblock \emph{ArXiv preprint arXiv:2301.09836}, 2023.

\bibitem[Arharas and Ouknine(2024)]{arharas2021reflected}
I.~Arharas and Y.~Ouknine.
\newblock {R}eflected and {D}oubly {R}eflected {B}ackward {S}tochastic
  {D}ifferential {E}quations with {I}rregular {O}bstacles and a {L}arge set of
  {S}topping {S}trategies.
\newblock \emph{Journal of Theoretical Probability}, 2024.

\bibitem[Baadi(2022)]{baadi2022reflected}
B.~Baadi.
\newblock Reflected {BSDE}s with two completely separated barriers and
  regulated trajectories in general filtration.
\newblock \emph{Random Operators and Stochastic Equations}, 30\penalty0
  (2):\penalty0 97--111, 2022.

\bibitem[Baadi and Ouknine(2017)]{baadi2017reflected}
B.~Baadi and Y.~Ouknine.
\newblock Reflected {BSDE}s when the obstacle is not right-continuous in a
  general filtration.
\newblock \emph{ALEA}, 14:\penalty0 201--218, 2017.

\bibitem[Baadi and Ouknine(2018)]{baadi2018reflected}
B.~Baadi and Y.~Ouknine.
\newblock Reflected {BSDE}s with optional barrier in a general filtration.
\newblock \emph{Afrika Matematika}, 29\penalty0 (7--8):\penalty0 1049--1064,
  2018.

\bibitem[Bally et~al.(2002)Bally, Caballero, Fernandez, and
  El~Karoui]{bally2002reflected}
V.~Bally, M.E. Caballero, B.~Fernandez, and N.~El~Karoui.
\newblock Reflected {BSDE}'s, {PDE}'s and variational inequalities.
\newblock Rapport de recherche 4455, Unit{\'e} de recherche INRIA Rocquencourt,
  2002.

\bibitem[Bandini(2015)]{bandini2015existence}
E.~Bandini.
\newblock Existence and uniqueness for backward stochastic differential
  equations driven by a random measure, possibly non quasi--left continuous.
\newblock \emph{Electronic Communications in Probability}, 20\penalty0
  (71):\penalty0 1--13, 2015.

\bibitem[Barles et~al.(1997)Barles, Buckdahn, and Pardoux]{barles1997backward}
G.~Barles, R.~Buckdahn, and \'E. Pardoux.
\newblock Backward stochastic differential equations and integral--partial
  differential equations.
\newblock \emph{Stochastics: An International Journal of Probability and
  Stochastic Processes}, 60\penalty0 (1--2):\penalty0 57--83, 1997.

\bibitem[Bayraktar and Yao(2012)]{bayraktar2012quadratic}
E.~Bayraktar and S.~Yao.
\newblock Quadratic reflected {BSDE}s with unbounded obstacles.
\newblock \emph{Stochastic Processes and their Applications}, 122\penalty0
  (4):\penalty0 1155--1203, 2012.

\bibitem[Becherer et~al.(2019)Becherer, B{\"u}ttner, and
  Kentia]{becherer2019monotone}
D.~Becherer, M.~B{\"u}ttner, and K.~Kentia.
\newblock On the monotone stability approach to {BSDE}s with jumps: extensions,
  concrete criteria and examples.
\newblock In \emph{Frontiers in stochastic analysis--BSDEs, SPDEs and their
  applications, Edinburgh, July 2017, selected, revised and extended
  contributions}, volume 289 of \emph{Springer proceedings in mathematics \&
  statistics}, pages 1--41, 2019.

\bibitem[Bismut(1973{\natexlab{a}})]{bismut1973analyse}
J.-M. Bismut.
\newblock \emph{Analyse convexe et probabilit{\'e}s}.
\newblock PhD thesis, Facult{\'e} des sciences de Paris, 1973{\natexlab{a}}.

\bibitem[Bismut(1973{\natexlab{b}})]{bismut1973conjugate}
J.-M. Bismut.
\newblock Conjugate convex functions in optimal stochastic control.
\newblock \emph{Journal of Mathematical Analysis and Applications}, 44\penalty0
  (2):\penalty0 384--404, 1973{\natexlab{b}}.

\bibitem[Bouchard et~al.(2018)Bouchard, Possama{\"\i}, Tan, and
  Zhou]{bouchard2018unified}
B.~Bouchard, D.~Possama{\"\i}, X.~Tan, and C.~Zhou.
\newblock A unified approach to {\it a priori} estimates for supersolutions of
  {BSDE}s in general filtrations.
\newblock \emph{Annales de l'institut Henri Poincar{\'e}, Probabilit{\'e}s et
  Statistiques $(${\rm B}$)$}, 54\penalty0 (1):\penalty0 154--172, 2018.

\bibitem[Bouhadou and Ouknine(2018{\natexlab{a}})]{bouhadou2018non}
S.~Bouhadou and Y.~Ouknine.
\newblock Non linear optimal stopping problem and reflected {BSDE}s in the
  predictable setting.
\newblock \emph{ArXiv preprint arXiv:1811.00695}, 2018{\natexlab{a}}.

\bibitem[Bouhadou and Ouknine(2018{\natexlab{b}})]{bouhadou2018optimal}
S.~Bouhadou and Y.~Ouknine.
\newblock Optimal stopping in general predictable framework.
\newblock \emph{ArXiv preprint arXiv:1812.01759}, 2018{\natexlab{b}}.

\bibitem[Bouhadou et~al.(2022)Bouhadou, Hilbert, and
  Ouknine]{bouhadou2022rbsdes}
S.~Bouhadou, A.~Hilbert, and Y.~Ouknine.
\newblock {RBSDE}s with optional barriers: monotone approximation.
\newblock \emph{Probability, Uncertainty and Quantitative Risk}, 7\penalty0
  (2):\penalty0 67--84, 2022.

\bibitem[Brachetta et~al.(2022)Brachetta, Callegaro, Ceci, and
  Sgarra]{brachetta2022optimal}
M.~Brachetta, G.~Callegaro, C.~Ceci, and C.~Sgarra.
\newblock Optimal reinsurance via {BSDE}s in a partially observable model with
  jump clusters.
\newblock \emph{ArXiv preprint arXiv:2207.05489}, 2022.

\bibitem[Briand and Hu(1998)]{briand1998stability}
P.~Briand and Y.~Hu.
\newblock Stability of {BSDE}s with random terminal time and homogenization of
  semilinear elliptic {PDE}s.
\newblock \emph{Journal of Functional Analysis}, 155\penalty0 (2):\penalty0
  455--494, 1998.

\bibitem[Briand et~al.(2001)Briand, Delyon, and M\'emin]{briand2001donsker}
P.~Briand, B.~Delyon, and J.~M\'emin.
\newblock Donsker-type theorem for {BSDE}s.
\newblock \emph{Electronic Communications in Probability}, 6:\penalty0 1--14,
  2001.

\bibitem[Briand et~al.(2002)Briand, Delyon, and
  M{\'e}min]{briand2002robustness}
P.~Briand, B.~Delyon, and J.~M{\'e}min.
\newblock On the robustness of backward stochastic differential equations.
\newblock \emph{Stochastic Processes and their Applications}, 97\penalty0
  (2):\penalty0 229--253, 2002.

\bibitem[Briand et~al.(2021)Briand, Geiss, Geiss, and
  Labart]{briand2021donsker}
P.~Briand, C.~Geiss, S.~Geiss, and C.~Labart.
\newblock Donsker-type theorem for {BSDE}s: rate of convergence.
\newblock \emph{Bernoulli}, 27\penalty0 (2):\penalty0 899--929, 2021.

\bibitem[Buckdahn(1993)]{buckdahn1993backward}
R.~Buckdahn.
\newblock Backward stochastic differential equations driven by a martingale.
\newblock Pr{\'e}publication 93--05, URA 225 Universit{\'e} de Provence,
  Marseille, 1993.

\bibitem[Carbone et~al.(2008)Carbone, Ferrario, and
  Santacroce]{carbone2008backward}
R.~Carbone, B.~Ferrario, and M.~Santacroce.
\newblock Backward stochastic differential equations driven by c\`adl\`ag
  martingales.
\newblock \emph{Theory of Probability \& Its Applications}, 52\penalty0
  (2):\penalty0 304--314, 2008.

\bibitem[Cheridito and Stadje(2013)]{cheridito2013bs}
P.~Cheridito and M.~Stadje.
\newblock {BS$\Delta$E}s and {BSDE}s with non-{L}ipschitz drivers: comparison,
  convergence and robustness.
\newblock \emph{Bernoulli}, 19\penalty0 (3):\penalty0 1047--1085, 2013.

\bibitem[Cheridito et~al.(2007)Cheridito, Soner, Touzi, and
  Victoir]{cheridito2007second}
P.~Cheridito, H.M. Soner, N.~Touzi, and N.~Victoir.
\newblock Second-order backward stochastic differential equations and fully
  nonlinear parabolic {PDE}s.
\newblock \emph{Communications on Pure and Applied Mathematics}, 60\penalty0
  (7):\penalty0 1081--1110, 2007.

\bibitem[Cohen and Elliott(2010)]{cohen2010general2}
S.~Cohen and R.~Elliott.
\newblock A general theory of finite state backward stochastic difference
  equations.
\newblock \emph{Stochastic Processes and their Applications}, 120\penalty0
  (4):\penalty0 442--466, 2010.

\bibitem[Cohen and Elliott(2011)]{cohen2011backward}
S.N. Cohen and R.J. Elliott.
\newblock Backward stochastic difference equations and nearly time-consistent
  nonlinear expectations.
\newblock \emph{SIAM Journal on Control and Optimization}, 49\penalty0
  (1):\penalty0 125--139, 2011.

\bibitem[Cohen and Elliott(2012)]{cohen2012existence}
S.N. Cohen and R.J. Elliott.
\newblock Existence, uniqueness and comparisons for {BSDE}s in general spaces.
\newblock \emph{The Annals of Probability}, 40\penalty0 (5):\penalty0
  2264--2297, 2012.

\bibitem[Cohen and Elliott(2015)]{cohen2015stochastic}
S.N. Cohen and R.J. Elliott.
\newblock \emph{Stochastic calculus and applications}.
\newblock Probability and its applications. Springer New York, 2015.

\bibitem[Confortola et~al.(2016)Confortola, Fuhrman, and
  Jacod]{confortola2014backward}
F.~Confortola, M.~Fuhrman, and J.~Jacod.
\newblock Backward stochastic differential equation driven by a marked point
  process: an elementary approach with an application to optimal control.
\newblock \emph{The Annals of Applied Probability}, 26\penalty0 (3):\penalty0
  1743--1773, 2016.

\bibitem[Cr{\'e}pey and Matoussi(2008)]{crepey2008reflected}
S.~Cr{\'e}pey and A.~Matoussi.
\newblock Reflected and doubly reflected {BSDE}s with jumps: {\it a priori}
  estimates and comparison.
\newblock \emph{The Annals of Applied Probability}, 18\penalty0 (5):\penalty0
  2041--2069, 2008.

\bibitem[Cvitani\'c and Karatzas(1996)]{cvitanic1996backward}
J.~Cvitani\'c and I.~Karatzas.
\newblock Backward stochastic differential equations with reflection and
  {D}ynkin games.
\newblock \emph{The Annals of Probability}, 24\penalty0 (4):\penalty0
  2024--2056, 1996.

\bibitem[Cvitani{\'c} et~al.(2018)Cvitani{\'c}, Possama{\"\i}, and
  Touzi]{cvitanic2018dynamic}
J.~Cvitani{\'c}, D.~Possama{\"\i}, and N.~Touzi.
\newblock Dynamic programming approach to principal--agent problems.
\newblock \emph{Finance and Stochastics}, 22\penalty0 (1):\penalty0 1--37,
  2018.

\bibitem[Darling and Pardoux(1997)]{darling1997backwards}
R.W.R. Darling and \'E. Pardoux.
\newblock Backwards {SDE} with random terminal time and applications to
  semilinear elliptic {PDE}.
\newblock \emph{The Annals of Probability}, 25\penalty0 (3):\penalty0
  1135--1159, 1997.

\bibitem[Darrell and Epstein(1992)]{duffie1992stochastic}
D.~Darrell and L.G. Epstein.
\newblock Stochastic differential utility.
\newblock \emph{Econometrica}, 60\penalty0 (2):\penalty0 353--394, 1992.

\bibitem[Davis and Varaiya(1973)]{davis1973dynamic}
M.H.A. Davis and P.~Varaiya.
\newblock Dynamic programming conditions for partially observable stochastic
  systems.
\newblock \emph{SIAM Journal on Control and Optimization}, 11\penalty0
  (2):\penalty0 226--261, 1973.

\bibitem[Dellacherie and Meyer(1978)]{dellacherie1978probabilities}
C.~Dellacherie and P.-A. Meyer.
\newblock \emph{Probabilities and potential}, volume~29 of \emph{Mathematics
  studies}.
\newblock North-Holland, 1978.

\bibitem[Dellacherie and Meyer(1982)]{dellacherie1982probabilities}
C.~Dellacherie and P.-A. Meyer.
\newblock \emph{Probabilities and potential {B}: theory of martingales}.
\newblock North-Holland Mathematics Studies. Elsevier Science, 1982.

\bibitem[di~Nunno(2022)]{dinunno2022stochastic}
G.~di~Nunno.
\newblock On stochastic control for time changed {L}\'{e}vy dynamics.
\newblock \emph{SeMA Journal}, 79\penalty0 (3):\penalty0 529--547, 2022.

\bibitem[Duffie and Epstein(1992)]{duffie1992asset}
D.~Duffie and L.G. Epstein.
\newblock Asset pricing with stochastic differential utility.
\newblock \emph{The Review of Financial Studies}, 5\penalty0 (3):\penalty0
  411--436, 1992.

\bibitem[Dumitrescu et~al.(2016)Dumitrescu, Quenez, and
  Sulem]{dumitrescu2016generalized}
R.~Dumitrescu, M.-C. Quenez, and A.~Sulem.
\newblock Generalized {D}ynkin games and doubly reflected {BSDE}s with jumps.
\newblock \emph{Electronic Journal of Probability}, 21\penalty0 (64):\penalty0
  1--32, 2016.

\bibitem[El~Asri and Ourkiya(2023)]{el2023infinite}
B.~El~Asri and N.~Ourkiya.
\newblock Infinite horizon multi-dimensional {BSDE} with oblique reflection and
  switching problem.
\newblock \emph{Stochastics and Dynamics}, 23\penalty0 (5):\penalty0 2350035,
  2023.

\bibitem[El~Karoui(1978)]{karoui1978arret}
N.~El~Karoui.
\newblock Arr\^{e}t optimal pr\'{e}visible.
\newblock In G.~Kallianpur and D.~K{\"o}lzow, editors, \emph{Measure theory.
  Applications to stochastic analysis. Proceedings, Oberwolfach conference,
  Germany, July 3--9, 1977}, volume 695 of \emph{Lecture notes in mathematics},
  pages 1--11. Springer Berlin, Heidelberg, 1978.

\bibitem[El~Karoui(1981)]{el1981aspects}
N.~El~Karoui.
\newblock Les aspects probabilistes du contr{\^o}le stochastique.
\newblock In \emph{\'Ecole d'{\'e}t{\'e} de probabilit{\'e}s de Saint-Flour
  IX--1979}, volume 876 of \emph{Lecture notes in mathematics}, pages 73--238.
  Springer, 1981.

\bibitem[El~Karoui and Huang(1997)]{elkaroui1997general}
N.~El~Karoui and S.-J. Huang.
\newblock A general result of existence and uniqueness of backward stochastic
  differential equations.
\newblock In N.~El~Karoui and L.~Mazliak, editors, \emph{Backward stochastic
  differential equations}, volume 364 of \emph{Pitman research notes in
  mathematics}, pages 27--36. Longman, 1997.

\bibitem[El~Karoui and Jeanblanc-Picqu{\'e}(1998)]{el1998optimization}
N.~El~Karoui and M.~Jeanblanc-Picqu{\'e}.
\newblock Optimization of consumption with labor income.
\newblock \emph{Finance and Stochastics}, 2\penalty0 (4):\penalty0 409--440,
  1998.

\bibitem[El~Karoui and Mazliak(1997)]{el1997backward2}
N.~El~Karoui and L.~Mazliak, editors.
\newblock \emph{Backward stochastic differential equations}, volume 364 of
  \emph{Pitman research notes in mathematics}.
\newblock Longman, 1997.

\bibitem[El~Karoui and Quenez(1997)]{el1997non}
N.~El~Karoui and M.-C. Quenez.
\newblock Non-linear pricing theory and backward stochastic differential
  equations.
\newblock In W.J. Runggaldier, editor, \emph{Financial mathematics. Lectures
  given at the 3rd session of the centro internazionale matematico estivo
  $($C.I.M.E.$)$ held in Bressanone, Italy, July 8--13, 1996}, volume 1656 of
  \emph{Lecture notes in mathematics}, pages 191--246. Springer, 1997.

\bibitem[El~Karoui and Tan(2013)]{karoui2013capacities}
N.~El~Karoui and X.~Tan.
\newblock Capacities, measurable selection and dynamic programming part {I}:
  abstract framework.
\newblock Technical report, \'Ecole Polytechnique and universit{\'e}
  Paris--Dauphine, 2013.

\bibitem[El~Karoui et~al.(1997{\natexlab{a}})El~Karoui, Kapoudjian, Pardoux,
  Peng, and Quenez]{el1997reflected}
N.~El~Karoui, C.~Kapoudjian, \'E. Pardoux, S.~Peng, and M.-C. Quenez.
\newblock Reflected solutions of backward {SDE}'s, and related obstacle
  problems for {PDE}'s.
\newblock \emph{The Annals of Probability}, 25\penalty0 (2):\penalty0 702--737,
  1997{\natexlab{a}}.

\bibitem[El~Karoui et~al.(1997{\natexlab{b}})El~Karoui, Pardoux, and
  Quenez]{el1997reflected2}
N.~El~Karoui, \'E. Pardoux, and M.-C. Quenez.
\newblock Reflected backward {SDE}s and {A}merican options.
\newblock In L.C.G. Rogers and D.~Talay, editors, \emph{Numerical methods in
  finance}, pages 215--231. Cambridge University Press, 1997{\natexlab{b}}.

\bibitem[El~Karoui et~al.(1997{\natexlab{c}})El~Karoui, Peng, and
  Quenez]{el1997backward}
N.~El~Karoui, S.~Peng, and M.-C. Quenez.
\newblock Backward stochastic differential equations in finance.
\newblock \emph{Mathematical Finance}, 7\penalty0 (1):\penalty0 1--71,
  1997{\natexlab{c}}.

\bibitem[El~Karoui et~al.(2016)El~Karoui, Matoussi, and
  Ngoupeyou]{karoui2016quadratic}
N.~El~Karoui, A.~Matoussi, and A.~Ngoupeyou.
\newblock Quadratic exponential semimartingales and application to {BSDE}s with
  jumps.
\newblock Technical report, Universit{\'e} Pierre \& Marie Curie,
  Universit{\'e} du Maine, and banque des \'Etats de l'Afrique centrale, 2016.

\bibitem[Essaky and Hassani(2013)]{essaky2013generalized}
E.H. Essaky and M.~Hassani.
\newblock Generalized {BSDE} with $2$-reflecting barriers and stochastic
  quadratic growth.
\newblock \emph{Journal of Differential Equations}, 254\penalty0 (3):\penalty0
  1500--1528, 2013.

\bibitem[Foresta(2021)]{foresta2021optimal}
N.~Foresta.
\newblock Optimal stopping of marked point processes and reflected backward
  stochastic differential equations.
\newblock \emph{Applied Mathematics \& Optimization}, 83\penalty0 (3):\penalty0
  1219--1245, 2021.

\bibitem[Gal'{\v{c}}uk(1981)]{gal1981optional}
L.I. Gal'{\v{c}}uk.
\newblock Optional martingales.
\newblock \emph{Mathematics of the USSR--Sbornik}, 40\penalty0 (4):\penalty0
  435--468, 1981.

\bibitem[Grigorova et~al.(2017)Grigorova, Imkeller, Offen, Ouknine, and
  Quenez]{grigorova2017reflected}
M.~Grigorova, P.~Imkeller, E.~Offen, Y.~Ouknine, and M.-C. Quenez.
\newblock Reflected {BSDE}s when the obstacle is not right-continuous and
  optimal stopping.
\newblock \emph{The Annals of Applied Probability}, 27\penalty0 (5):\penalty0
  3153--3188, 2017.

\bibitem[Grigorova et~al.(2018)Grigorova, Imkeller, Ouknine, and
  Quenez]{grigorova2018doubly}
M.~Grigorova, P.~Imkeller, Y.~Ouknine, and M.-C. Quenez.
\newblock Doubly reflected {BSDE}s and {$\mathcal{E}^{f}$}-{D}ynkin games:
  beyond the right-continuous case.
\newblock \emph{Electronic Journal of Probability}, 23\penalty0 (122):\penalty0
  1--38, 2018.

\bibitem[Grigorova et~al.(2020)Grigorova, Imkeller, Ouknine, and
  Quenez]{grigorova2020optimal}
M.~Grigorova, P.~Imkeller, Y.~Ouknine, and M.-C. Quenez.
\newblock Optimal stopping with $f$-expectations: the irregular case.
\newblock \emph{Stochastic Processes and their Applications}, 130\penalty0
  (3):\penalty0 1258--1288, 2020.

\bibitem[Gu et~al.(2023{\natexlab{a}})Gu, Lin, and Xu]{gu2023mean}
Z.~Gu, Y.~Lin, and K.~Xu.
\newblock Mean reflected {BSDE} driven by a marked point process and
  application in insurance risk management.
\newblock \emph{ArXiv preprint arXiv:2310.15203}, 2023{\natexlab{a}}.

\bibitem[Gu et~al.(2023{\natexlab{b}})Gu, Lin, and Xu]{gu2023reflected}
Z.~Gu, Y.~Lin, and K.~Xu.
\newblock Reflected {BSDE} driven by a marked point process with a
  convex/concave generator.
\newblock \emph{ArXiv preprint arXiv:2310.20361}, 2023{\natexlab{b}}.

\bibitem[Hamad{\`e}ne(2002)]{hamadene2002reflected}
S.~Hamad{\`e}ne.
\newblock Reflected {BSDE}'s with discontinuous barrier and application.
\newblock \emph{Stochastics: An International Journal of Probability and
  Stochastic Processes}, 74\penalty0 (3--4):\penalty0 571--596, 2002.

\bibitem[Hamad{\`e}ne(2006)]{hamadene2006mixed}
S.~Hamad{\`e}ne.
\newblock Mixed zero-sum stochastic differential game and {A}merican game
  options.
\newblock \emph{SIAM Journal on Control and Optimization}, 45\penalty0
  (2):\penalty0 496--518, 2006.

\bibitem[Hamad\`ene and Hassani(2005)]{hamadene2005bsdes}
S.~Hamad\`ene and M.~Hassani.
\newblock B{SDE}s with two reflecting barriers: the general result.
\newblock \emph{Probab. Theory Related Fields}, 132\penalty0 (2):\penalty0
  237--264, 2005.

\bibitem[Hamad\`ene and Hassani(2006)]{hamadene2006bsdes}
S.~Hamad\`ene and M.~Hassani.
\newblock B{SDE}s with two reflecting barriers driven by a {B}rownian and a
  {P}oisson noise and related {D}ynkin game.
\newblock \emph{Electron. J. Probab.}, 11:\penalty0 no. 5, 121--145, 2006.

\bibitem[Hamad\`ene and Lepeltier(1995)]{hamadene1995zero}
S.~Hamad\`ene and J.-P. Lepeltier.
\newblock Zero-sum stochastic differential games and backward equations.
\newblock \emph{Systems \& Control Letters}, 24\penalty0 (4):\penalty0
  259--263, 1995.

\bibitem[Hamad{\`e}ne and Ouknine(2003)]{hamadene2003reflected}
S.~Hamad{\`e}ne and Y.~Ouknine.
\newblock Reflected backward stochastic differential equation with jumps and
  random obstacle.
\newblock \emph{Electronic Journal of Probability}, 8\penalty0 (2):\penalty0
  1--20, 2003.

\bibitem[Hamad\`ene et~al.(1997)Hamad\`ene, Lepeltier, and
  Matoussi]{hamadene1997double}
S.~Hamad\`ene, J.-P. Lepeltier, and A.~Matoussi.
\newblock Double barrier backward {SDE}s with continuous coefficient.
\newblock In N.~El~Karoui and L.~Mazliak, editors, \emph{Backward stochastic
  differential equations}, volume 364 of \emph{Pitman research notes in
  mathematics}, pages 161--175. Longman, 1997.

\bibitem[He et~al.(1992)He, Wang, and Yan]{he1992semimartingale}
S.~He, J.~Wang, and J.A. Yan.
\newblock \emph{Semimartingale theory and stochastic calculus}.
\newblock Science Press, 1992.

\bibitem[Horn and Johnson(2013)]{horn2013matrix}
R.A. Horn and C.R. Johnson.
\newblock \emph{Matrix analysis}.
\newblock Cambridge University Press, second edition, 2013.

\bibitem[Hu et~al.(2005)Hu, Imkeller, and M{\"u}ller]{hu2005utility}
Y.~Hu, P.~Imkeller, and M.~M{\"u}ller.
\newblock Utility maximization in incomplete markets.
\newblock \emph{The Annals of Applied Probability}, 15\penalty0 (3):\penalty0
  1691--1712, 2005.

\bibitem[Jackson and \v{Z}itkovi\'{c}(2022)]{jackson2022characterization}
J.~Jackson and G.~\v{Z}itkovi\'{c}.
\newblock A characterization of solutions of quadratic {BSDE}s and a new
  approach to existence.
\newblock \emph{Stochastic Processes and their Applications}, 147:\penalty0
  210--225, 2022.

\bibitem[Jacod(1979)]{jacod1979calcul}
J.~Jacod.
\newblock \emph{Calcul stochastique et probl{\`e}mes de martingales}, volume
  714 of \emph{Lecture notes in mathematics}.
\newblock Springer, 1979.

\bibitem[Jacod and Shiryaev(2003)]{jacod2003limit}
J.~Jacod and A.N. Shiryaev.
\newblock \emph{Limit theorems for stochastic processes}, volume 288 of
  \emph{Grundlehren der mathematischen Wissenschaften}.
\newblock Springer--Verlag Berlin Heidelberg, 2003.

\bibitem[Jeanblanc et~al.(2012)Jeanblanc, Matoussi, and
  Ngoupeyou]{jeanblanc2012robust}
M.~Jeanblanc, A.~Matoussi, and A.~Ngoupeyou.
\newblock Robust utility maximization problem in model with jumps and unbounded
  claim.
\newblock Technical report, Universit{\'e} d'{\'E}vry-Val-d'Essonne,
  universit{\'e} du Maine, and banque des \'Etats de l'Afrique centrale, 2012.

\bibitem[Kazi-Tani et~al.(2015{\natexlab{a}})Kazi-Tani, Possama{\"\i}, and
  Zhou]{kazi2015quadratic1}
N.~Kazi-Tani, D.~Possama{\"\i}, and C.~Zhou.
\newblock Quadratic {BSDE}s with jumps: a fixed-point approach.
\newblock \emph{Electronic Journal of Probability}, 20\penalty0 (66):\penalty0
  1--28, 2015{\natexlab{a}}.

\bibitem[Kazi-Tani et~al.(2015{\natexlab{b}})Kazi-Tani, Possama{\"\i}, and
  Zhou]{kazi2015second}
N.~Kazi-Tani, D.~Possama{\"\i}, and C.~Zhou.
\newblock Second-order {BSDE}s with jumps: formulation and uniqueness.
\newblock \emph{The Annals of Applied Probability}, 25\penalty0 (5):\penalty0
  2867--2908, 2015{\natexlab{b}}.

\bibitem[Kazi-Tani et~al.(2015{\natexlab{c}})Kazi-Tani, Possama{\"\i}, and
  Zhou]{kazi2015second2}
N.~Kazi-Tani, D.~Possama{\"\i}, and C.~Zhou.
\newblock Second order {BSDE}s with jumps: existence and probabilistic
  representation for fully-nonlinear {PIDE}s.
\newblock \emph{Electronic Journal of Probability}, 20\penalty0 (65):\penalty0
  1--31, 2015{\natexlab{c}}.

\bibitem[Kazi-Tani et~al.(2016)Kazi-Tani, Possama{\"\i}, and
  Zhou]{kazi2016quadratic}
N.~Kazi-Tani, D.~Possama{\"\i}, and C.~Zhou.
\newblock Quadratic {BSDE}s with jumps: related nonlinear expectations.
\newblock \emph{Stochastics and Dynamics}, 16\penalty0 (04):\penalty0 1650012,
  2016.

\bibitem[Klimsiak(2012)]{klimsiak2012reflected2}
T.~Klimsiak.
\newblock Reflected {BSDE}s with monotone generator.
\newblock \emph{Electronic Journal of Probability}, 17\penalty0 (107):\penalty0
  1--25, 2012.

\bibitem[Klimsiak(2013)]{klimsiak2013bsdes}
T.~Klimsiak.
\newblock {BSDE}s with monotone generator and two irregular reflecting
  barriers.
\newblock \emph{Bulletin des Sciences Math{\'e}matiques}, 137\penalty0
  (3):\penalty0 268--321, 2013.

\bibitem[Klimsiak(2015)]{klimsiak2015reflected}
T.~Klimsiak.
\newblock Reflected {BSDE}s on filtered probability spaces.
\newblock \emph{Stochastic Processes and their Applications}, 125\penalty0
  (11):\penalty0 4204--4241, 2015.

\bibitem[Klimsiak(2021)]{klimsiak2021non}
T.~Klimsiak.
\newblock Non-semimartingale solutions of reflected {BSDE}s and applications to
  {D}ynkin games.
\newblock \emph{Stochastic Processes and their Applications}, 134:\penalty0
  208--239, 2021.

\bibitem[Klimsiak and Rzymowski(2021)]{klimsiak2021reflected}
T.~Klimsiak and M.~Rzymowski.
\newblock Reflected {BSDE}s with two optional barriers and monotone coefficient
  on general filtered space.
\newblock \emph{Electronic Journal of Probability}, 26\penalty0 (91):\penalty0
  1--24, 2021.

\bibitem[Klimsiak and Rzymowski(2022)]{klimsiak2022nonlinear}
T.~Klimsiak and M.~Rzymowski.
\newblock Nonlinear {BSDE}s with two optional {D}oob's class barriers
  satisfying weak {M}okobodzki's condition and extended {D}ynkin games.
\newblock \emph{ArXiv preprint arXiv:2205.06222}, 2022.

\bibitem[Klimsiak et~al.(2019)Klimsiak, Rzymowski, and
  S{\l}omi{\'n}ski]{klimsiak2019reflected}
T.~Klimsiak, M.~Rzymowski, and L.~S{\l}omi{\'n}ski.
\newblock Reflected {BSDE}s with regulated trajectories.
\newblock \emph{Stochastic Processes and their Applications}, 129\penalty0
  (4):\penalty0 1153--1184, 2019.

\bibitem[Klimsiak et~al.(2020)Klimsiak, Rzymowski, and
  S{\l}omi{\'n}ski]{klimsiak2020reflected}
T.~Klimsiak, M.~Rzymowski, and L.~S{\l}omi{\'n}ski.
\newblock Reflected backward stochastic differential equations with two
  optional barriers.
\newblock \emph{Bulletin des Sciences Math{\'e}matiques}, 158:\penalty0 102820,
  2020.

\bibitem[Kobylanski(1997)]{kobylanski1997resultats}
M.~Kobylanski.
\newblock R{\'e}sultats d'existence et d'unicit{\'e} pour des {\'e}quations
  diff{\'e}rentielles stochastiques r{\'e}trogrades avec des
  g{\'e}n{\'e}rateurs {\`a} croissance quadratique.
\newblock \emph{Comptes Rendus de l'Acad{\'e}mie des Sciences--Series
  I--Mathematics}, 324\penalty0 (1):\penalty0 81--86, 1997.

\bibitem[Kobylanski and Quenez(2012)]{kobylanski2012optimal}
M.~Kobylanski and M.-C. Quenez.
\newblock Optimal stopping time problem in a general framework.
\newblock \emph{Electronic Journal of Probability}, 17\penalty0 (72):\penalty0
  1--28, 2012.

\bibitem[Kobylanski et~al.(2002)Kobylanski, Lepeltier, Quenez, and
  Torres]{kobylanski2002reflected}
M.~Kobylanski, J.-P. Lepeltier, M.-C. Quenez, and S.~Torres.
\newblock Reflected {BSDE} with superlinear quadratic coefficient.
\newblock \emph{Probability and Mathematical Statistics}, 22\penalty0
  (1):\penalty0 51--83, 2002.

\bibitem[Lauri{\`e}re and Tangpi(2019)]{lauriere2019backward}
M.~Lauri{\`e}re and L.~Tangpi.
\newblock Backward propagation of chaos.
\newblock \emph{ArXiv preprint arXiv:1911.06835}, 2019.

\bibitem[Lenglart(1980)]{lenglart1980tribus}
\'E. Lenglart.
\newblock Tribus de {M}eyer et th{\'e}orie des processus.
\newblock \emph{S\'eminaire de probabilit\'es de Strasbourg}, XIV:\penalty0
  500--546, 1980.

\bibitem[Lepeltier and San~Mart{\'\i}n(2004)]{lepeltier2004backward}
J.-P. Lepeltier and J.~San~Mart{\'\i}n.
\newblock Backward {SDE}s with two barriers and continuous coefficient: an
  existence result.
\newblock \emph{Journal of Applied Probability}, 41\penalty0 (1):\penalty0
  162--175, 2004.

\bibitem[Lepeltier and Xu(2005)]{lepeltier2005penalization}
J.-P. Lepeltier and M.~Xu.
\newblock Penalization method for reflected backward stochastic differential
  equations with one r.c.l.l. barrier.
\newblock \emph{Statistics \& Probability Letters}, 75\penalty0 (1):\penalty0
  58--66, 2005.

\bibitem[Lepeltier and Xu(2007)]{lepeltier2007reflected}
J.-P. Lepeltier and M.~Xu.
\newblock Reflected {BSDE} with quadratic growth and unbounded terminal value.
\newblock Technical report, Universit\'e du Maine and Fudan University, 2007.

\bibitem[L\'epingle and M\'emin(1978)]{lepingle1978sur}
D.~L\'epingle and J.~M\'emin.
\newblock Int\'egrabilit{\'e} uniforme et dans ${L}^r$ des martingales
  exponentielles.
\newblock \emph{Publications des s\'eminaires de math\'ematiques et
  informatique de Rennes}, 1:\penalty0 1--14, 1978.

\bibitem[L\'{e}pingle and M\'{e}min(1978)]{lepingle1978sur2}
D.~L\'{e}pingle and J.~M\'{e}min.
\newblock Sur l'int\'{e}grabilit\'{e} uniforme des martingales exponentielles.
\newblock \emph{Zeitschrift f{\"u}r Wahrscheinlichkeitstheorie und verwandte
  Gebiete}, 42\penalty0 (3):\penalty0 175--203, 1978.

\bibitem[Li(2023)]{li2023backward}
H.~Li.
\newblock Backward stochastic differential equations with double mean
  reflections.
\newblock \emph{ArXiv preprint arXiv:2307.05947}, 2023.

\bibitem[Li and Liu(2023)]{li2023multi}
H.~Li and G.~Liu.
\newblock Multi-dimensional reflected {BSDE}s driven by {$G$}--{B}rownian
  motion with diagonal generators.
\newblock \emph{ArXiv preprint arXiv:2108.09012}, 2023.

\bibitem[Li and Ning(2024)]{li2024propagation}
H.~Li and N.~Ning.
\newblock Propagation of chaos for doubly mean reflected {BSDE}s.
\newblock \emph{ArXiv preprint arXiv:2401.16617}, 2024.

\bibitem[Li et~al.(2022)Li, Liu, and Rutkowski]{li2022well}
L.~Li, R.~Liu, and M.~Rutkowski.
\newblock Well-posedness and penalization schemes for generalized {BSDE}s and
  reflected generalized {BSDE}s.
\newblock \emph{ArXiv preprint arXiv:2212.12854}, 2022.

\bibitem[Lin and Xu(2024)]{lin2024mean}
Y.~Lin and K.~Xu.
\newblock Mean-field reflected {BSDE}s driven by a marked point process.
\newblock \emph{ArXiv preprint arXiv:2401.07723}, 2024.

\bibitem[Lin et~al.(2020)Lin, Ren, Touzi, and Yang]{lin2020second}
Y.~Lin, Z.~Ren, N.~Touzi, and J.~Yang.
\newblock Second order backward {SDE} with random terminal time.
\newblock \emph{Electronic Journal of Probability}, 25\penalty0 (99):\penalty0
  1--43, 2020.

\bibitem[Madan et~al.(2016)Madan, Pistorius, and Stadje]{madan2015convergence}
D.~Madan, M.~Pistorius, and M.~Stadje.
\newblock Convergence of {BS$\Delta$E}s driven by random walks to {BSDE}s: the
  case of (in)finite activity jumps with general driver.
\newblock \emph{Stochastic Processes and their Applications}, 126\penalty0
  (5):\penalty0 1553--1584, 2016.

\bibitem[Maingueneau(1978)]{maingueneau1978temps}
M.A. Maingueneau.
\newblock Temps d'arr\^{e}t optimaux et th\'{e}orie g\'{e}n\'{e}rale.
\newblock \emph{S\'eminaire de probabilit\'es de Strasbourg}, XII:\penalty0
  457--467, 1978.

\bibitem[Matoussi and Salhi(2019)]{matoussi2019exponential}
A.~Matoussi and R.~Salhi.
\newblock Exponential quadratic {BSDE}s with infinite activity jumps.
\newblock \emph{ArXiv preprint arXiv:1904.08666}, 2019.

\bibitem[Meyer(1976)]{meyer1976cours}
P.-A. Meyer.
\newblock Un cours sur les int\'{e}grales stochastiques.
\newblock \emph{S\'eminaire de probabilit\'es de Strasbourg}, X:\penalty0
  245--400, 1976.

\bibitem[Nie and Rutkowski(2022)]{nie2022reflected}
T.~Nie and M.~Rutkowski.
\newblock Reflected and doubly reflected {BSDE}s driven by {RCLL} martingales.
\newblock \emph{Stochastics and Dynamics}, 22\penalty0 (05):\penalty0 2250012,
  2022.

\bibitem[Nutz(2015)]{nutz2015robust}
M.~Nutz.
\newblock Robust superhedging with jumps and diffusion.
\newblock \emph{Stochastic Processes and their Applications}, 125\penalty0
  (12):\penalty0 4543--4555, 2015.

\bibitem[Papapantoleon et~al.(2018)Papapantoleon, Possama{\"\i}, and
  Saplaouras]{papapantoleon2018existence}
A.~Papapantoleon, D.~Possama{\"\i}, and A.~Saplaouras.
\newblock Existence and uniqueness for {BSDE}s with jumps: the whole nine
  yards.
\newblock \emph{Electronic Journal of Probability}, 23\penalty0 (121):\penalty0
  1--68, 2018.

\bibitem[Papapantoleon et~al.(2019)Papapantoleon, Possama{\"\i}, and
  Saplaouras]{papapantoleon2019stability}
A.~Papapantoleon, D.~Possama{\"\i}, and A.~Saplaouras.
\newblock Stability results for martingale representations: the general case.
\newblock \emph{Transactions of the American Mathematical Society},
  372\penalty0 (8):\penalty0 5891--5946, 2019.

\bibitem[Papapantoleon et~al.(2023)Papapantoleon, Possama{\"\i}, and
  Saplaouras]{papapantoleon2021stability}
A.~Papapantoleon, D.~Possama{\"\i}, and A.~Saplaouras.
\newblock Stability of backward stochastic differential equations: the general
  {L}ipschitz case.
\newblock \emph{Electronic Journal of Probability}, 28\penalty0 (51):\penalty0
  1--56, 2023.

\bibitem[Papapantoleon et~al.(2024)Papapantoleon, Saplaouras, and
  Theodorakopoulos]{papapantoleon2024existence}
A.~Papapantoleon, A.~Saplaouras, and S.~Theodorakopoulos.
\newblock Existence, uniqueness and propagation of chaos for general
  {M}c{K}ean--{V}lasov and mean-field {BSDE}s.
\newblock \emph{In preparation}, 2024.

\bibitem[Pardoux and Peng(1990)]{pardoux1990adapted}
\'E. Pardoux and S.~Peng.
\newblock Adapted solution of a backward stochastic differential equation.
\newblock \emph{System and Control Letters}, 14\penalty0 (1):\penalty0 55--61,
  1990.

\bibitem[Pardoux and Peng(1992)]{pardoux1992backward}
\'E. Pardoux and S.~Peng.
\newblock Backward stochastic differential equations and quasilinear parabolic
  partial differential equations.
\newblock In B.L. Rozovskii and R.B. Sowers, editors, \emph{Stochastic partial
  differential equations and their applications. Proceedings of IFIP WG 7/1
  international conference University of North Carolina at Charlotte, NC June
  6--8, 1991}, volume 176 of \emph{Lecture notes in control and information
  sciences}, pages 200--217. Springer, 1992.

\bibitem[Peng(1991)]{peng1991probabilistic}
S.~Peng.
\newblock Probabilistic interpretation for systems of quasilinear parabolic
  partial differential equations.
\newblock \emph{Stochastics and Stochastic Reports}, 37\penalty0
  (1--2):\penalty0 61--74, 1991.

\bibitem[Peng and Xu(2005)]{peng2005smallest}
S.~Peng and M.~Xu.
\newblock The smallest $g$-supermartingale and reflected {BSDE} with single and
  double ${L}^2$ obstacles.
\newblock \emph{Annales de l'institut Henri Poincar{\'e}, Probabilit{\'e}s et
  Statistiques $(${\rm B}$)$}, 41\penalty0 (3):\penalty0 605--630, 2005.

\bibitem[Peng and Xu(2010)]{peng2010reflected}
S.~Peng and M.~Xu.
\newblock Reflected {BSDE} with a constraint and its applications in an
  incomplete market.
\newblock \emph{Bernoulli}, 16\penalty0 (3):\penalty0 614--640, 2010.

\bibitem[Perninge(2021)]{perninge2021note}
M.~Perninge.
\newblock A note on reflected {BSDE}s in infinite horizon with stochastic
  {L}ipschitz coefficients.
\newblock \emph{ArXiv preprint arXiv:2103.05603}, 2021.

\bibitem[Perninge(2024{\natexlab{a}})]{perninge2023optimal}
M.~Perninge.
\newblock Optimal stopping of {BSDE}s with constrained jumps and related
  zero-sum games.
\newblock \emph{Stochastic Processes and their Applications}, 173:\penalty0
  104355, 2024{\natexlab{a}}.

\bibitem[Perninge(2024{\natexlab{b}})]{perninge2024optimal}
M.~Perninge.
\newblock Optimal stopping of {BSDE}s with constrained jumps and related double
  obstacle {PDE}s.
\newblock \emph{ArXiv preprint arXiv:2402.17541}, 2024{\natexlab{b}}.

\bibitem[Pham and Zhang(2013)]{pham2013some}
T.~Pham and J.~Zhang.
\newblock Some norm estimates for semimartingales.
\newblock \emph{Electronic Journal of Probability}, 18\penalty0 (109):\penalty0
  1--25, 2013.

\bibitem[Popier(2021)]{popier2021backward}
A.~Popier.
\newblock Backward stochastic {V}olterra integral equations with jumps in a
  general filtration.
\newblock \emph{ESAIM: Probability and Statistics}, 25:\penalty0 133--203,
  2021.

\bibitem[Possama{\"\i} and Tan(2015)]{possamai2015weak}
D.~Possama{\"\i} and X.~Tan.
\newblock Weak approximation of second-order {BSDE}s.
\newblock \emph{The Annals of Applied Probability}, 25\penalty0 (5):\penalty0
  2535--2562, 2015.

\bibitem[Possama{\"\i} and Tangpi(2021)]{possamai2021non}
D.~Possama{\"\i} and L.~Tangpi.
\newblock Non-asymptotic convergence rates for mean-field games: weak
  formulation and {M}c{K}ean--{V}lasov bsdes.
\newblock \emph{ArXiv preprint arXiv:2105.00484}, 2021.

\bibitem[Possama{\"\i} et~al.(2018)Possama{\"\i}, Tan, and
  Zhou]{possamai2018stochastic}
D.~Possama{\"\i}, X.~Tan, and C.~Zhou.
\newblock Stochastic control for a class of nonlinear kernels and applications.
\newblock \emph{The Annals of Probability}, 46\penalty0 (1):\penalty0 551--603,
  2018.

\bibitem[Protter(2005)]{protter2005stochastic}
P.E. Protter.
\newblock \emph{Stochastic integration and differential equations}, volume~21
  of \emph{Stochastic modelling and applied probability}.
\newblock Springer--Verlag Berlin Heidelberg, 2nd edition, 2005.

\bibitem[Qian(2023)]{qian2023multi}
R.~Qian.
\newblock Multi-dimensional reflected {M}c{K}ean--{V}lasov {BSDE}s with the
  obstacle depending on both the first unknown and its distribution.
\newblock \emph{ArXiv preprint arXiv:2309.10427}, 2023.

\bibitem[Quenez and Sulem(2013)]{quenez2013bsdes}
M.-C. Quenez and A.~Sulem.
\newblock {BSDE}s with jumps, optimization and applications to dynamic risk
  measures.
\newblock \emph{Stochastic Processes and their Applications}, 123\penalty0
  (8):\penalty0 3328--3357, 2013.

\bibitem[Quenez and Sulem(2014)]{quenez2014reflected}
M.-C. Quenez and A.~Sulem.
\newblock Reflected {BSDE}s and robust optimal stopping for dynamic risk
  measures with jumps.
\newblock \emph{Stochastic Processes and their Applications}, 124\penalty0
  (9):\penalty0 3031--3054, 2014.

\bibitem[Royer(2004)]{royer2004bsdes}
M.~Royer.
\newblock {BSDE}s with a random terminal time driven by a monotone generator
  and their links with {PDE}s.
\newblock \emph{Stochastics and Stochastic Reports}, 76\penalty0 (4):\penalty0
  281--307, 2004.

\bibitem[Royer(2006)]{royer2006backward}
M.~Royer.
\newblock Backward stochastic differential equations with jumps and related
  non-linear expectations.
\newblock \emph{Stochastic Processes and their Applications}, 116\penalty0
  (10):\penalty0 1358--1376, 2006.

\bibitem[Rozkosz and S{\l}omi{\'n}ski(2012)]{rozkosz2012solutions}
A.~Rozkosz and L.~S{\l}omi{\'n}ski.
\newblock {$\Bbb{L}^p$} solutions of reflected {BSDE}s under monotonicity
  condition.
\newblock \emph{Stochastic Processes and their Applications}, 122\penalty0
  (12):\penalty0 3875--3900, 2012.

\bibitem[Soner et~al.(2011)Soner, Touzi, and Zhang]{soner2011quasi}
H.M. Soner, N.~Touzi, and J.~Zhang.
\newblock Quasi-sure stochastic analysis through aggregation.
\newblock \emph{Electronic Journal of Probability}, 16\penalty0 (2):\penalty0
  1844--1879, 2011.

\bibitem[Soner et~al.(2012)Soner, Touzi, and Zhang]{soner2012wellposedness}
H.M. Soner, N.~Touzi, and J.~Zhang.
\newblock Wellposedness of second order backward {SDE}s.
\newblock \emph{Probability Theory and Related Fields}, 153\penalty0
  (1--2):\penalty0 149--190, 2012.

\bibitem[Tevzadze(2008)]{tevzadze2008solvability}
R.~Tevzadze.
\newblock Solvability of backward stochastic differential equations with
  quadratic growth.
\newblock \emph{Stochastic Processes and their Applications}, 118\penalty0
  (3):\penalty0 503--515, 2008.

\bibitem[Touzi(2013)]{touzi2013optimal}
N.~Touzi.
\newblock \emph{Optimal stochastic control, stochastic target problems, and
  backward {SDE}}, volume~29 of \emph{Fields institute monographs}.
\newblock Springer--Verlag New York, 2013.

\bibitem[von Weizs{\"a}cker and Winkler(1990)]{weizsaecker1990stochastic}
H.~von Weizs{\"a}cker and G.~Winkler.
\newblock \emph{Stochastic integrals}.
\newblock Advanced lectures in mathematics. Vieweg+Teubner Verlag Wiesbaden,
  1990.

\bibitem[Zhang(2017)]{zhang2017backward}
J.~Zhang.
\newblock \emph{Backward stochastic differential equations---from linear to
  fully nonlinear theory}, volume~86 of \emph{Probability theory and stochastic
  modelling}.
\newblock Springer--Verlag New York, 2017.

\bibitem[Zheng et~al.(2023)Zheng, Zhang, and Meng]{zheng2023one}
S.~Zheng, L.~Zhang, and X.~Meng.
\newblock One-dimensional reflected {BSDE}s and {BSDE}s with quadratic growth
  generators.
\newblock \emph{ArXiv preprint arXiv:2110.06907}, 2023.

\end{thebibliography}

\end{document}